\documentclass[1996/10/24]{amsart}

\usepackage{color}

\allowdisplaybreaks[3]

\usepackage{amsfonts}
\usepackage{epsf,latexsym,amsfonts,amsbsy,mathrsfs}
\usepackage{amsmath, amssymb, amsthm,amscd,amsxtra}
\usepackage[]{graphicx,subfigure}
\usepackage{epstopdf}
\usepackage{float}
\usepackage{multirow}
\usepackage{stmaryrd}
\usepackage{epsfig}
\usepackage{tikz}
\usetikzlibrary{positioning} 
\usepackage{xcolor}
\usepackage{marginnote}
\usepackage{sidenotes}

\usepackage{changes}
\usepackage{lipsum}

\newtheorem{theorem}{Theorem}[section]
\newtheorem{lemma}[theorem]{Lemma}

\theoremstyle{definition}

\newtheoremstyle{italicremark}
{3pt}
{3pt}
{\itshape}
{}
{\bfseries}
{.}
{ }
{}

\theoremstyle{italicremark}
\newtheorem{remark}[theorem]{Remark}

\numberwithin{equation}{section}

\newcommand{\ignore}[1]{}

\newcommand{\T}{\mathcal{T}}
\newcommand{\E}{\mathcal{E}}
\newtheorem{rem}{Remark}[section]

\begin{document}
	
	\title[]{An Energy-Stable, Bound-Preserving and Locally Conservative Numerical Framework for Multicomponent Gas Flow in Poroelastic Media}
	
	\author{Huangxin Chen}
	\address{School of Mathematical Sciences and Fujian Provincial Key Laboratory on Mathematical Modeling and
		High Performance Scientific Computing, Xiamen University, Fujian, 361005, China}
	\email{chx@xmu.edu.cn}
	\author{Yuxiang Chen}
	\address{School of Mathematical Sciences and Fujian Provincial Key Laboratory on Mathematical Modeling and
		High Performance Scientific Computing, Xiamen University, Fujian, 361005, China}
	\email{chenyuxiang@stu.xmu.edu.cn}
	\author{Jisheng Kou}
	\address{State Key Laboratory of Intelligent Deep Metal Mining and Equipment, Zhejiang Key Laboratory of Rock Mechanics and Geohazards, School of Civil Engineering, Shaoxing University, Shaoxing 312000, Zhejiang, China} 
	\email{jishengkou@163.com}
	\author{Shuyu Sun}
	\address{School of Mathematical Sciences, Tongji University, Shanghai 200092, China.}
	\email{suns@tongji.edu.cn}
	\subjclass[2010]{}
	

	\date{}
	
	\begin{abstract}
		{In this paper, we propose a robust and efficient numerical framework for simulating multicomponent gas flow in poroelastic media, with a focus on preserving fundamental thermodynamic principles and ensuring computational reliability. The model captures the complex nonlinear coupling between multicomponent transport and solid deformation, while addressing critical numerical challenges such as mass conservation, energy stability, and molar density boundedness. To achieve this, we develop a stabilized discretization approach that guarantees the preservation of the original energy dissipation law and ensures the boundedness of each gas component's molar density. Furthermore, the proposed method incorporates an adaptive time-stepping strategy that dynamically adjusts the time step size based on the system's dynamics, significantly enhancing computational efficiency without compromising stability or accuracy. For spatial discretization, a mixed finite element method combined with an upwind scheme is employed for the flow and transport equations to ensure local mass conservation, while a discontinuous Galerkin (DG) method is utilized for discretizing the momentum equation of poroelasticity to effectively overcome numerical locking phenomena. Numerical experiments are presented to demonstrate the performance, robustness, and applicability of the method in simulating multicomponent gas flow under various scenarios.}
	\end{abstract}
	
	\maketitle
	
	\section{Introduction}
	The flow of multicomponent gases through poroelastic media is a fundamental process in numerous geophysical and engineering applications, including carbon sequestration, hydrogen storage, natural gas recovery, and subsurface contaminant transport \cite{El2018, Guo2018, Mikyska2014}. These processes are characterized by complex nonlinear couplings between multicomponent transport, phase behavior, and the mechanical response of the solid skeleton. A critical challenge in modeling such systems lies in the incorporation of the solid matrix deformation. Inspired by Biot's poroelasticity theory, contemporary models often solve mechanical equilibrium equations to obtain solid displacements, from which porosity changes are derived. While significant advances have been made in modeling single-phase and two-phase flow in poroelastic media \cite{Chencmame2024}, the study of multicomponent compressible flow in poroelastic media remains relatively scarce. Moreover, the thermodynamic consistency of these models, particularly their adherence to the second law of thermodynamics, has often been overlooked, despite its critical role in ensuring model reliability and numerical stability \cite{Chen2006,Kou2018,kousun2018,Casas-Vazquez2008}. Beyond mechanical coupling, accurate description of diffusion in multicomponent mixtures is essential. The Maxwell-Stefan (MS) equations, which balance driving forces with inter-component friction, provide a robust framework for modeling diffusion \cite{Bothe2011, Fowler2018, Hoteit2013, Kou2023, Leahy-Dios2007, Runstedtler2006, Wesselingh2000}. Recent extensions have incorporated friction between fluid components and the solid matrix \cite{Koupof2023, Krishna2018, Leonardi2010, Lipnizki2001}. In this work, we rigorously derive a novel generalized MS model that couples fluid transport with solid mechanics, ensuring that the resulting coeﬀicients satisfy irreversible thermodynamics principles, which is a cornerstone of non-equilibrium thermodynamics and include the linear phenomenological assumption, Onsager’s reciprocal relations, the energy dissipation law, the minimum dissipation principle, Curie's principle, and others.
	
	From a numerical perspective, the design of efficient and stable discretization methods for coupled nonlinear systems (e.g., those governing multiphysics processes in porous media) presents substantial challenges. An efficient scheme must concurrently satisfy several stringent requirements rooted in physical principles: it should preserve a discrete energy dissipation law to guarantee nonlinear stability and thermodynamic conformity  \cite{Eyre1998, Shen2018}, enforce the strict boundedness of physical quantities like molar density to preclude non-physical solutions \cite{Wise2012, Yang2022cutoff}, ensure local mass conservation for all components, and retain high computational efficiency. As our physical model is derived from strict thermodynamic consistency, the numerical methods must also uphold this property to avoid unstable or non-physical results \cite{Chen2006, Kou2018}.
	
	Various established strategies exist for handling Helmholtz free energy in gradient flow type systems. These range from the rigorously nonlinear convex splitting method \cite{Eyre1998, Qiao2014}, which guarantees stability at the cost of solving nonlinear algebraic systems, to more efficient linear strategies. The latter category includes versatile stabilization methods \cite{XuT2006, Ju2021, Li2023}, structure-preserving exponential time-differencing (ETD) schemes \cite{Beylkin1998, Cox2002}, and the widely used invariant energy quadratization (IEQ) \cite{Ju2017, Wang2017} and scalar auxiliary variable (SAV) approaches \cite{Shen2018, Shen2019}. While the IEQ and SAV schemes offer linear, easily implementable energy-stable schemes, they typically preserve a modified energy functional rather than the original one \cite{Shen2021}. Other notable methods, such as the energy factorization (EF) approach \cite{kou2020, kou2022}, operate within an efficient linear framework while successfully preserving the dissipation characteristics of the original energy. Building upon this foundation, the present work employs a tailored stabilization approach \cite{Kou2023} that enables a linear, energy-stable scheme while exactly preserving the original energy dissipation structure.
	
	Beyond energy stability, the boundedness of molar density is a fundamental requirement in simulating compressible flow in porous media. Designing numerical schemes that respect these bounds is crucial for obtaining physically reasonable solutions. The exploration of provably bound-preserving schemes has led to various strategies, including Lagrange multiplier \cite{Cheng2022, IICheng2022}, variational inequality \cite{Joshi2018, Yang2017}, post-processing strategy \cite{Zhang2010}, cut-off scheme \cite{Yang2022cutoff}, and nonlinear convex splitting approaches \cite{Wise2012, Dong2021, Shen2021}. Recent advances have further yielded high-order methods that simultaneously ensure energy stability and solution boundedness, including stabilized linear schemes \cite{Tang2016,Shen2016} and SAV methods combined with cut-off operations \cite{Akrivis2019, Shen2019}. Notably, a significant advancement by Ju et al. \cite{Ju2022, JuJCS2022} ingeniously combined SAV method with stabilized ETD schemes to construct numerical methods preserving both the original energy-dissipation law and the maximum bound principle for a class of Allen-Cahn type gradient flows. However, in the specific context of advection-dominated transport in porous media, many overshoot/undershoot correction techniques may compromise numerical accuracy and critical properties like local mass conservation.
		
    To address the intertwined challenges of computational efficiency, exact discrete energy dissipation, and strict solution boundedness, in this work we combine a robust adaptive time-stepping strategy with a modified linear stabilization method (cf. \cite{Kou2023}). The boundedness requirement entails the simultaneous control of each component molar density and the total molar density. The key design is to use the total molar density $c$ together with $M-1$ component densities as primary variables, which allows us to enforce both the total and component-wise bounds in a unified manner. At each time step, the resulting nonlinear fully discrete system is solved by a stabilized linear fixed-point iteration. We prove that the iteration is contractive and hence converges to a unique limit. Furthermore, the unique limit coincides with the solution of the original nonlinear discrete system. The boundedness properties established at the iterative level are inherited by the converged solution. In particular, all molar density bounds and the discrete energy dissipation law hold for the final iterate. To enable an efficient and fully explicit time step size update, the velocity term in the mass conservation equation is treated explicitly within the linear iterations, which yields an explicit formula for computing the admissible adaptive time step.
	
	For spatial discretization, we employ a mixed finite element method combined with an upwind scheme to ensure local mass conservation and stability when handling advection-dominated transport. Furthermore, the momentum balance equation for the solid skeleton is discretized using a discontinuous Galerkin method to eﬀiciently avoid numerical locking phenomena  \cite{Phillips2008}. The resulting framework aims to provide a reliable, efficient, and physically faithful numerical tool for simulating complex coupled nonlinear processes in porous media.

	The remainder of this paper is organized as follows. In Section \ref{sec-model}, we present the governing equations of our thermodynamically consistent model for multicomponent gas flow in poroelastic media. Section \ref{sec-semi} elaborates the semi-discrete numerical scheme and shows the energy stability of the proposed scheme. Section \ref{sec-fully} is devoted to the construction of a fully discrete numerical scheme, where its essential properties including energy dissipation and boundedness of molar densities are rigorously established, and an adaptive time-stepping algorithm is introduced. Numerical experiments in Section \ref{sec-num} validate the performance and robustness of the proposed method. Finally,
	we provide some concluding remarks in Section \ref{sec-con}.

	\section{Mathematical Model}\label{sec-model}

We consider a compressible fluid mixture composed of $M$ components. The molar density vector is denoted as $\mathbf{c} = [c_1, c_2, \ldots, c_M]^T$, with the total molar density given by $c = \sum_{i=1}^M c_i$. Assuming constant temperature $T$, we employ the Peng–Robinson equation of state to construct the Helmholtz free energy of the mixture.

For each pure component, the energy parameters $\alpha_i$ and $\beta_i$ are determined from its critical properties

\begin{equation}
	\alpha_i = 0.45724 \frac{R^2 T_{c,i}^2}{P_{c,i}} \left[1 + \chi_i \left(1 - \sqrt{T / T_{c,i}} \right) \right]^2, \quad
	\beta_i = 0.07780 \frac{R T_{c,i}}{P_{c,i}},
\end{equation}
where $T_{c,i}$ and $P_{c,i}$ are the critical temperature and pressure of component $i$, respectively. The parameter $\chi_i$ depends on the acentric factor $\omega_i$ as follows
\begin{equation}
	\chi_i = 
	\begin{cases}
		0.37464 + 1.54226\omega_i - 0.26992\omega_i^2, & \omega_i \leq 0.49, \\
		0.379642 + 1.485030\omega_i - 0.164423\omega_i^2 + 0.016666\omega_i^3, & \omega_i > 0.49.
	\end{cases}
\end{equation}

The parameters $\alpha(T)$ and $\beta$ are computed using the following mixing rules

\begin{equation}
	\alpha(T) = \sum_{i=1}^{M} \sum_{j=1}^{M} \frac{c_i c_j}{c^2} \left( \alpha_i \alpha_j \right)^{1/2} (1 - \iota_{ij}), \quad
	\beta = \sum_{i=1}^{M} \frac{c_i}{c} \beta_i,
\end{equation}
where $\iota_{ij}$ is the binary interaction parameter, satisfying $\iota_{ii} = 0$ and $\iota_{ij} = \iota_{ji}$. The mixture parameter $\beta$ is a convex combination of the pure component parameters $\beta_i$. Consequently, it is bounded by the minimum and maximum values of the $\beta_i$
\begin{equation}
	\min_{i} \beta_i \leq \beta \leq \max_{i} \beta_i.
\end{equation}

The Helmholtz free energy density determined by the Peng-Robinson equation of state $f(\mathbf{c})$ is expressed as \cite{Robinson1976}
\begin{equation}
	\begin{aligned}
		f(\mathbf{c}) &= RT \sum_{i=1}^{M} c_i (\ln c_i - 1) - cRT \ln(1 - \beta c) \\
		&\quad + \frac{\alpha(T)c}{2\sqrt{2} \beta} \ln \left( \frac{1 + (1 - \sqrt{2})\beta c}{1 + (1 + \sqrt{2})\beta c} \right) + \zeta_c(T)cRT,
	\end{aligned}
\end{equation}
where $R$ is the universal gas constant, taken to be 8.314 J/(mol$\cdot$ K). The term $\zeta_c(T)$ represents a temperature-dependent correction related to the heat capacity of the mixture. For isothermal systems, we assume $\zeta_c(T) = 0$.

The chemical potential of component $i$ is defined by
$
\mu_i := \frac{\partial f(\mathbf{c})}{\partial c_i},
$
and the pressure satisfies the thermodynamic identity
\begin{equation}\label{eq-p}
p = \sum_{i=1}^{M} c_i \mu_i - f(\mathbf{c}).
\end{equation}

To account for the mechanical response of the poroelastic media, we also incorporate an elastic deformation model for the solid matrix. The solid displacement field is denoted by $\boldsymbol{w}_s$, and the associated strain tensor is given by

\begin{equation}
	\boldsymbol{\varepsilon}(\boldsymbol{w}_s) = \frac{1}{2} \left( \nabla \boldsymbol{w}_s + \nabla \boldsymbol{w}_s^T \right).
\end{equation}

Under the premise of infinitesimal skeleton deformations \cite{coussy2004}, mathematically represented by the condition $|\nabla \boldsymbol{w}_s|\ll 1$, Biot’s poroelasticity framework is adopted to model the resultant variations in rock volume as well as rock deformation. Assuming a linear isotropic elastic solid, the Cauchy stress tensor is expressed as

\begin{equation}
	\boldsymbol{\sigma}_e = \gamma\, \text{tr}(\boldsymbol{\varepsilon}(\boldsymbol{w}_s)) \mathbf{I} + 2\eta \boldsymbol{\varepsilon}(\boldsymbol{w}_s),
\end{equation}
where $\gamma$ and $\eta$ are  Lam$\acute{e}$ first and second parameters, respectively, and $\text{tr}(\cdot)$ denotes the trace operator. The governing equation for the quasi-static deformation of the solid is

\begin{equation}
	-\nabla \cdot \boldsymbol{\sigma} = \mathbf{f}_{\text{ext}},
\end{equation}	
where 	$ \boldsymbol{\sigma} = \boldsymbol{\sigma}_e - \alpha p$, and  $\mathbf{f}_{\text{ext}}$ is the external body force. Here $\alpha $ is Biot's constant, quantifying the degree of coupling between the volumetric deformation of the porous solid skeleton and the pore fluid pressure change.

To more accurately describe mass transport in multicomponent mixtures, especially in poroelastic media, we adopt the Maxwell--Stefan--Darcy (MSD) framework as follows \cite{Kou2023}, rather than the classical Darcy model which cannot resolve interspecies interactions:
\begin{align}
	\sum_{j=1, j \neq i}^M \frac{c_i c_j\left(\mathbf{u}_{i}-\mathbf{u}_{j}\right)}{c^2 \phi D_{i,j}}+\frac{(1 - \phi)\mathbf{u}_{i}}{\phi D_{i, s}}=-c_i \nabla \mu_i, \quad i = 1,\cdots, M.
\end{align}
The MSD model captures frictional forces arising between different components and between fluid and solid phases by introducing binary diffusion coefficients \( D_{i,j} \) and fluid--solid interaction coefficients \( D_{i,s} \). These coefficients satisfy \( D_{i,j} = D_{j,i} > 0 \) for all \( i \neq j \), and \( D_{i,s} > 0 \), ensuring thermodynamic consistency and positive definiteness of diffusion. The model is consistent with Onsager's reciprocal principle, which demands symmetric relations between fluxes and driving forces, as established in the analysis of the derived coefficient matrices \cite{Koupof2023}.

To simplify the notation and facilitate the analysis, we rescale the MSD diffusion coefficients by introducing
\[
\mathscr{D}_{i,j} = \phi D_{i,j}, \quad \mathscr{D}_{i,s} = \frac{\phi D_{i,s}}{1 - \phi},
\]
where \( \phi \in (0,1) \) is the porosity of porous media. 
Based on the Darcy's law, we define
\begin{align*}
	\mathscr{D}_{i,s} = \frac{\phi D_{i,s}}{1 - \phi} = \frac{\kappa(\phi)}{\eta_i}\boldsymbol{K},
\end{align*}
where $\eta_i$ is the viscosity of component $i$, $\boldsymbol{K} = \mathcal{K} \boldsymbol{I}$ is the absolute permeability, $\mathcal{K}$ is the scalar and $\boldsymbol{I}$ is the identity matrix, the function $\kappa(\phi)$ is given by the Kozeny-Carman model as follows \cite{Chen2006}
\begin{equation}
	\kappa(\phi) = \left(\frac{\phi}{\phi_r}\right)^{3} \left(\frac{1-\phi_r}{1-\phi}\right)^{2}.
\end{equation}
Here, $\phi_r$ is the porosity at the reference pressure.

Combining mass conservation, the MSD model, thermodynamic relations, and the momentum conservation law for the solid phase, the governing equations for the multicomponent transport system are given as follows:
\begin{align}
	&\frac{\partial\left(\phi c_i\right)}{\partial t}+\nabla \cdot\left(c_i \boldsymbol{u}_i\right)=q_i, \quad i=1, \cdots, M, \qquad\qquad{\rm in}~\Omega_{t}:=\Omega\times(0,t), \label{eq-1-1}\\
	&\sum_{j=1, j \neq i}^M \frac{c_i c_j\left(\boldsymbol{u}_i-\boldsymbol{u}_j\right)}{c^2 \mathscr{D}_{i,j}}+\frac{\boldsymbol{u}_i}{\mathscr{D}_{i, s}}=-c_i \nabla \mu_i, \quad i=1, \cdots, M, \qquad    {\rm in}~\Omega_{t},\label{eq-1-2-c}\\
	&p=\sum_{i=1}^Mc_i \mu_i-f(\mathbf{c}), ~\qquad   \ \ \  {\rm in}~\Omega_{t},\\
	&\nabla\cdot\boldsymbol{\sigma}(\boldsymbol{w}_{s}, p) = 0, \qquad\qquad  {\rm in}~\Omega_{t},\label{eq-1-2-e}\\
	&\frac{\partial \phi}{\partial t} = \frac{1}{N}\frac{\partial p}{\partial t} + \alpha  \nabla\cdot \frac{\partial\boldsymbol{w}_s}{\partial t},  \qquad  {\rm in}~\Omega_{t}.\label{eq-1-6}
\end{align}
In this work, the computational domain is modeled as a closed system, i.e., no external forces are applied and no mass flux occurs across the boundary. Accordingly, pure Neumann boundary conditions are prescribed
\begin{align}\label{eq-BD}
	\boldsymbol{\sigma}(\boldsymbol{w}_s, p)\cdot\boldsymbol{n} = 0, \qquad {\rm on}~\partial\Omega,
\end{align}
representing the traction-free condition for the solid mechanics part, and
\begin{align}\label{eq-BD-1}
	\boldsymbol{u}_{i}\cdot\boldsymbol{n} = 0, \quad i=1, \cdots, M, \quad {\rm on}~\partial\Omega,
\end{align}
representing the impermeable boundary condition for the fluid transport part.

To close the system, initial conditions are prescribed for the porosity and the
component concentrations. At the initial time $t=0$, we assume
\begin{align}
	&\phi(\boldsymbol{x},0)=\phi_0(\boldsymbol{x}), \qquad {\rm in}~\Omega,\\
	&c_i(\boldsymbol{x},0)=c_{i,0}(\boldsymbol{x}), \quad i=1,\cdots,M, \qquad {\rm in}~\Omega.\label{eq-IC}
\end{align}

The total free energy of the system \eqref{eq-1-1}--\eqref{eq-1-6} in $\Omega$ is defined as \cite{Chencmame2024}
\[
E(t)=\int_{\Omega}\left(\phi f(\boldsymbol{c})
+\frac{1}{2}\boldsymbol{\sigma}_{e}(\boldsymbol{w}_{s}):\boldsymbol{\varepsilon}(\boldsymbol{w}_{s})
+\frac{1}{2N} p^2\right)\, d \boldsymbol{x}.
\]
Here $\phi f(c)$ is the fluid free energy per unit bulk volume, weighted by the porosity $\phi$ since only the pore space is occupied by the fluid; $\tfrac12\,\boldsymbol{\sigma}_e(\boldsymbol w_s):\boldsymbol{\varepsilon}(\boldsymbol w_s)$ is the recoverable elastic strain energy stored in the solid skeleton; and $\tfrac{1}{2N}p^2$ is the storage (compressibility) energy associated with the pore pressure, where $N>0$ is a Biot-type storage modulus.

To verify the thermodynamic consistency of \eqref{eq-1-1}--\eqref{eq-BD-1}, we derive the associated energy dissipation.
Decompose the total energy into the fluid-mixture part $\Psi_f$ and the solid-skeleton part $\Psi_s$ (cf.~\cite{coussy2004}),
\[
E=\Psi_f+\Psi_s,\qquad
\Psi_f:=\int_\Omega \phi f(\mathbf c)\,d\boldsymbol{x},\qquad
\Psi_s:=\int_{\Omega}\left(\frac{1}{2}\boldsymbol{\sigma}_{e}(\boldsymbol{w}_{s}):\boldsymbol{\varepsilon}(\boldsymbol{w}_{s})
+\frac{1}{2N} p^2\right)\, d \boldsymbol{x}.
\]
Differentiating with respect to time gives
\[
\frac{dE}{dt}=\frac{d\Psi_f}{dt}+\frac{d\Psi_s}{dt}.
\]

Using the chain rule, the thermodynamic identity \eqref{eq-p}, the mass conservation law
\eqref{eq-1-1}, and the impermeable boundary condition \eqref{eq-BD-1}, we obtain
\begin{align}
	\frac{d \Psi_f}{dt}
	&=\int_{\Omega}\frac{\partial(\phi f(\mathbf{c}))}{\partial t}\, d \boldsymbol{x}\nonumber\\
	&=\int_{\Omega}\left(f(\mathbf{c}) \frac{\partial \phi}{\partial t}
	+\phi \sum_{i=1}^{M}\mu_i\frac{\partial c_i}{\partial t} \right)\, d \boldsymbol{x}\nonumber\\
	&=\int_{\Omega}\left(\Big(\sum_{i=1}^M c_i \mu_i-p\Big)\frac{\partial \phi}{\partial t}
	+ \phi \sum_{i=1}^{M}\mu_i\frac{\partial c_i}{\partial t} \right)\, d \boldsymbol{x}\nonumber\\
	&=\int_{\Omega}\left(\sum_{i=1}^M\mu_i \frac{\partial (\phi c_i)}{\partial t}
	- p \frac{\partial \phi}{\partial t}\right)\, d \boldsymbol{x}\nonumber\\
	&=\int_{\Omega}\left(\sum_{i=1}^M\mu_i\Big(q_i-\nabla \cdot(c_i \boldsymbol{u}_i)\Big)
	- p \frac{\partial \phi}{\partial t}\right)\, d \boldsymbol{x}\nonumber\\
	&=\int_{\Omega}\left(\sum_{i=1}^M\mu_i q_i
	+\sum_{i=1}^M c_i\nabla\mu_i \cdot \boldsymbol{u}_i
	- p \frac{\partial \phi}{\partial t}\right)\, d \boldsymbol{x}.
	\label{eq:Psif_pre}
\end{align}
Taking the dot product of the MSD relation \eqref{eq-1-2-c} with $\boldsymbol{u}_i$ and summing over $i$ yields the
nonnegative dissipation 
\[
\sum_{i=1}^M c_i\nabla\mu_i \cdot \boldsymbol{u}_i
= -\sum_{i=1}^M \frac{|\boldsymbol{u}_i|^2}{\mathscr{D}_{i, s}}
-\sum_{i=1}^M\sum_{j=1, j \neq i}^{M}\frac{c_ic_j}{2 c^2 \mathscr{D}_{i,j}}\,
\bigl|\boldsymbol{u}_i-\boldsymbol{u}_j\bigr|^2 .
\]
Substituting this identity into \eqref{eq:Psif_pre} and recalling that the system is closed (thus $q_i\equiv 0$) give
\begin{align}
	\frac{d \Psi_f}{dt}
	&= 
	-\int_{\Omega}\left(\sum_{i=1}^M \frac{|\boldsymbol{u}_i|^2}{\mathscr{D}_{i, s}}
	+\sum_{i=1}^M\sum_{j=1, j \neq i}^{M}\frac{c_ic_j}{2 c^2 \mathscr{D}_{i,j}}\,
	\bigl|\boldsymbol{u}_i-\boldsymbol{u}_j\bigr|^2
	+ p \frac{\partial \phi}{\partial t}\right)\, d \boldsymbol{x}.
	\label{eq:Psif_final}
\end{align}

Let $\boldsymbol{v}_s:=\partial_t\boldsymbol{w}_s$.  By \eqref{eq-1-2-e} and \eqref{eq-1-6}, we get
\begin{align}
	\frac{d\Psi_s}{dt}
	&=\frac{d}{dt}\int_{\Omega}\left(\frac{1}{2}\boldsymbol{\sigma}_{e}(\boldsymbol{w}_{s}):\boldsymbol{\varepsilon}(\boldsymbol{w}_{s})
	+\frac{1}{2N} p^2\right)\, d \boldsymbol{x}\nonumber\\
	&=\int_{\Omega}\boldsymbol{\sigma}_{e}(\boldsymbol{w}_{s}):\boldsymbol{\varepsilon}(\boldsymbol{v}_s)\, d \boldsymbol{x}
	+\int_{\Omega}\frac{1}{N}p\,\partial_t p\, d \boldsymbol{x}.
	\label{eq:Psis_start}
\end{align}
We now express the first term in \eqref{eq:Psis_start} using the momentum equation \eqref{eq-1-2-e} and the traction-free
condition \eqref{eq-BD}. Testing \eqref{eq-1-2-e} by $\boldsymbol{v}_s$ and integrating by parts, we get
\begin{align*}
	0
	&=\int_\Omega (\nabla\cdot\boldsymbol{\sigma}(\boldsymbol{w}_s,p))\cdot\boldsymbol{v}_s\,d\boldsymbol{x}\\
	&= -\int_\Omega \boldsymbol{\sigma}(\boldsymbol{w}_s,p):\nabla\boldsymbol{v}_s\,d\boldsymbol{x}
	+\int_{\partial\Omega}(\boldsymbol{\sigma}(\boldsymbol{w}_s,p)\boldsymbol{n})\cdot\boldsymbol{v}_s\,ds\\
	&= -\int_\Omega \boldsymbol{\sigma}(\boldsymbol{w}_s,p):\boldsymbol{\varepsilon}(\boldsymbol{v}_s)\,d\boldsymbol{x}.
\end{align*}
Substituting $\boldsymbol{\sigma}=\boldsymbol{\sigma}_e-\alpha p\boldsymbol I$ into the above equation, we obtain
\begin{equation}
	\int_\Omega \boldsymbol{\sigma}_e(\boldsymbol{w}_s):\boldsymbol{\varepsilon}(\boldsymbol{v}_s)\,d\boldsymbol{x}
	=\alpha\int_\Omega p\,\nabla\cdot\boldsymbol{v}_s\,d\boldsymbol{x}.
	\label{eq:sigmae_identity}
\end{equation}
Next, using \eqref{eq-1-6} we rewrite the storage term as
\begin{equation}
	\frac{1}{N}p\,\partial_t p
	= p\left(\partial_t\phi-\alpha\nabla\cdot\boldsymbol{v}_s\right),
\end{equation}
then, we get
\begin{equation}
	\int_\Omega\frac{1}{N}p\,\partial_t p\,d\boldsymbol{x}
	=\int_\Omega p\,\partial_t\phi\,d\boldsymbol{x}
	-\alpha\int_\Omega p\,\nabla\cdot\boldsymbol{v}_s\,d\boldsymbol{x}.
	\label{eq:storage_rewrite}
\end{equation}
Combining \eqref{eq:Psis_start}--\eqref{eq:storage_rewrite}, the two $\alpha$-terms cancel exactly, yielding
\begin{equation}
	\frac{d\Psi_s}{dt}=\int_\Omega p\,\frac{\partial\phi}{\partial t}\,d\boldsymbol{x}.
	\label{eq:Psis_final}
\end{equation}

Adding \eqref{eq:Psif_final} and \eqref{eq:Psis_final} gives
\[
\frac{dE}{dt}
=
-\int_{\Omega}\left(\sum_{i=1}^M \frac{|\boldsymbol{u}_i|^2}{\mathscr{D}_{i, s}}
+\sum_{i=1}^M\sum_{j=1, j \neq i}^{M}\frac{c_ic_j}{2 c^2 \mathscr{D}_{i,j}}\,
\bigl|\boldsymbol{u}_i-\boldsymbol{u}_j\bigr|^2\right)\, d \boldsymbol{x}.
\]
In particular, for the closed system considered in this work,
\[
\frac{dE}{dt}
=
-\int_{\Omega}\left(\sum_{i=1}^M \frac{|\boldsymbol{u}_i|^2}{\mathscr{D}_{i, s}}
+\sum_{i=1}^M\sum_{j=1, j \neq i}^{M}\frac{c_ic_j}{2 c^2 \mathscr{D}_{i,j}}\,
\bigl|\boldsymbol{u}_i-\boldsymbol{u}_j\bigr|^2\right)\, d \boldsymbol{x}
\le 0,
\]
which confirms that the model satisfies the second law of thermodynamics.

\section{Semi-discrete Scheme}\label{sec-semi}

To develop a thermodynamically consistent numerical scheme, we now introduce a semi-discrete time approximation for the model \eqref{eq-1-1}--\eqref{eq-1-6}. Let $t^n = \sum^{n}_{j = 0}\tau_j$, where $\tau_j$ denotes the $j$-th time step, and denote the numerical approximation of a quantity $c_i$ at time $t^n$ by $c_i^n$. $c^n = \sum_{i=1}^Mc^n_i$ denotes the total molar density.  For the chemical potential, we consider a first-order semi-implicit approximation inspired by the stabilization method of the free energy \cite{kou2023}. Specifically, the chemical potential at time level $n+1$ is approximated by	
\begin{align}
	\mu^{n+1}_{i} = \mu_i(\mathbf{c}^{n}) + \theta_n R T\frac{c^{n+1}_{i} - c^{n}_{i}}{c^{n}(1-\beta^{\ast} c^{n})},\label{eq-3-7}\\
	\mu^{n+1} = \sum^{M}_{i=1}\mu_i(\mathbf{c}^{n}) + \theta_n R T\frac{c^{n+1}- c^{n}}{c^{n}(1-\beta^{\ast} c^{n})},\label{eq-3-8}
\end{align}
where $\theta_n$ is a stabilization parameter, $\beta^{\ast} = \max_{i} \beta_i$. This formula provides a stabilized discrete approximation of the chemical potential based on the previous step and a stabilization term involving the change in molar density.
To analyze the associated discrete energy behavior, we consider the Taylor expansion of the Helmholtz free energy density $f(\boldsymbol{c}^{n+1})$ around $\boldsymbol{c}^n$
as follows:	
\begin{align}
	f(\boldsymbol{c}^{n+1}) = f(\boldsymbol{c}^{n}) + \boldsymbol{\mu}(\boldsymbol{c}^n) \cdot (\boldsymbol{c}^{n+1} - \boldsymbol{c}^{n}) +\frac{1}{2} (\boldsymbol{c}^{n+1} - \boldsymbol{c}^{n})^{\top} f^{\prime\prime}(\boldsymbol{\varkappa}) (\boldsymbol{c}^{n+1} - \boldsymbol{c}^{n}), \label{f_n_1}
\end{align}
where $\boldsymbol{c}^{n} := [c^n_1, \cdots, c^n_M]$, $\boldsymbol{\mu}(\boldsymbol{c}^n) := [\mu_1(\boldsymbol{c}^n), \cdots, \mu_M(\boldsymbol{c}^n)]$, and $\varkappa$ is a weighted average of the two vectors  $\boldsymbol{c}^n$ and $\boldsymbol{c}^{n+1}$ with a certain scalar weight.

On the other hand, using the above definition of $\boldsymbol{\mu}^{n+1}$, we obtain
\begin{align}
	\boldsymbol{\mu}^{n+1} = \boldsymbol{\mu}(\boldsymbol{c}^n) + \frac{\theta_n RT}{c^n(1 - \beta^{\ast} c^n)} \mathbf{I}_{M \times M} (\boldsymbol{c}^{n+1} - \boldsymbol{c}^n), \label{mu_n_1}
\end{align}
where $\mathbf{I}_{M \times M}$ is the identity matrix in $\mathbb{R}^{M \times M}$. Then, substituting (\ref{mu_n_1}) into the energy expansion (\ref{f_n_1}), we arrive at
\begin{align}
	f(\boldsymbol{c}^{n+1}) &= f(\boldsymbol{c}^{n}) + \boldsymbol{\mu}^{n+1} \cdot (\boldsymbol{c}^{n+1} - \boldsymbol{c}^{n}) \nonumber\\
	& \quad + (\boldsymbol{c}^{n+1} - \boldsymbol{c}^{n})^{\top} \left(\frac{1}{2}f^{\prime\prime}(\boldsymbol{\varkappa}) - \frac{\theta_n RT}{c^n(1 - \beta^{\ast} c^n)} \mathbf{I}_{M \times M} \right) (\boldsymbol{c}^{n+1} - \boldsymbol{c}^{n}).
\end{align}
By Gershgorin's circle theorem \cite{Horn2003}, all eigenvalues of 
$
f^{\prime\prime}(\boldsymbol{\varkappa}) - \frac{\theta_n RT}{c^n(1 - \beta^{\ast} c^n)} \mathbf{I}_{M \times M}
$
lie within the union of intervals centered at the diagonal entries, with radii equal to the sum of the absolute values of the corresponding off-diagonal entries. 
If the stabilization parameter $\theta_n$ is chosen sufficiently large, each Gershgorin interval is shifted far enough to the left of the origin, ensuring that all eigenvalues are non-positive. Consequently, the matrix is negative semi-definite. As a result, the remaining quadratic term in the Taylor expansion is non-positive, and the following discrete energy inequality holds
\begin{align}\label{eq-3-6}
	f(\boldsymbol{c}^{n+1}) - f(\boldsymbol{c}^{n}) \leq \boldsymbol{\mu}^{n+1} \cdot (\boldsymbol{c}^{n+1} - \boldsymbol{c}^{n}),
\end{align}
which guarantees that the discrete free energy does not increase in time. This property is essential for ensuring the thermodynamic consistency and stability of the numerical scheme.

Based on the governing system \eqref{eq-1-1}–\eqref{eq-IC} and the energy stability considerations discussed above, we now propose a semi-discrete-in-time numerical scheme. This scheme is constructed using a stabilized chemical potential formulation \eqref{eq-3-7}-\eqref{eq-3-8}, ensuring the dissipation of a discrete energy functional. Let $\tau_n$ denote the time step size at step $n$, and we define the backward difference operator $D_{\tau_{n}}a^{n+1} := \frac{a^{n+1} - a^n}{\tau_{n}}$ for any time-dependent quantity $a$. The semi-discrete system at time level $n+1$ is given as follows:
\begin{align}
	&D_{\tau_n}\left(\phi^{n+1} c^{n+1}_i\right)+\nabla \cdot\left(c^{n}_i \boldsymbol{u}^{n+1}_i\right)= 0, \quad i=1, \cdots, M , \label{eq-2-1}\\
	&\sum_{j=1, j \neq i}^M \frac{c^{n}_i c^{n}_j\left(\boldsymbol{u}^{n+1}_i-\boldsymbol{u}^{n+1}_j\right)}{(c^{n})^2 \mathscr{D}_{i, j}}+\frac{\boldsymbol{u}^{n+1}_i}{\mathscr{D}_{i, s}}=-c^{n}_i \nabla \mu^{n+1}_i, \quad i=1, \cdots, M,\label{eq-2-2-c}\\
	&p^{n+1}=\sum_{i=1}^Mc^{n}_i \mu^{n+1}_i-f(\mathbf{c}^{n}),\\
	&\nabla\cdot\boldsymbol{\sigma}(\boldsymbol{w}^{n+1}_{s}, p^{n+1}) = 0, \label{eq-2-2-e}\\
	&D_{\tau_n}\phi^{n+1} = \frac{1}{N}D_{\tau_n} p^{n+1}+ \alpha  D_{\tau_n}\nabla\cdot \boldsymbol{w}^{n+1}_{s}.\label{eq-2-6}
\end{align}
\begin{theorem}
	Consider the semi-discrete system \eqref{eq-2-1}–\eqref{eq-2-6} and the boundary condition \eqref{eq-BD}-\eqref{eq-BD-1}.  The discrete total energy at time level $n$ is defined as
	\[
	E^n = \int_{\Omega} \left( \phi^n f(\boldsymbol{c}^n) + \frac{1}{2} \boldsymbol{\sigma}_e(\boldsymbol{w}^n_s) : \varepsilon(\boldsymbol{w}^n_s) + \frac{1}{2N}(p^n)^2 \right) \, d\boldsymbol{x}.
	\]
	Then the following inequality holds
	\[
	E^{n+1} - E^n \leq - \tau_n \int_{\Omega} \left( \sum_{i=1}^M\frac{|\boldsymbol{u}^{n+1}_i|^2}{\mathscr{D}_{i,s}} + \sum_{i=1}^M\sum_{\substack{j=1 \\ j \neq i}}^M \frac{c^{n}_i c^{n}_j|\boldsymbol{u}^{n+1}_i - \boldsymbol{u}^{n+1}_j|^2}{2(c^n)^2 \mathscr{D}_{i,j}} \right) \, d\boldsymbol{x}\leq 0.
	\]
	Hence, the discrete energy $E^n$ is non-increasing with respect to $n$. It is strictly increasing unless flow and diffusion reaches equilibrium, i.e. all $u_i^{n+1}$ are zero, which is equivalent to the condition that the pressure as well as the chemical potentials of all components become spatially constants.
\end{theorem}

\begin{proof}
	The total energy consists of two parts
	\[
	E^n = \Psi_{\text{f}}^n + \Psi_{\text{s}}^n,
	\]
	where 
	$\Psi_{\text{f}}^n = \int_\Omega \phi^n f(\boldsymbol{c}^n)\, d\boldsymbol{x}, \quad
	\Psi_{\text{s}}^n = \int_\Omega \frac{1}{2}\boldsymbol{\sigma}_e(\boldsymbol{w}^n_s):\varepsilon(\boldsymbol{w}^n_s) + \frac{1}{2N}(p^n)^2 \, d\boldsymbol{x}.$
	
	By multiplying both sides of  \eqref{eq-2-1} by the corresponding chemical potential \( \mu_i^{n+1} \), summing over all components \( i = 1, \dots, M \), and integrating over the domain \( \Omega \), and further applying boundary conditions \eqref{eq-BD-1}, we obtain
	\begin{align}\label{eq-3-14}
		\frac{1}{\tau_n} \int_\Omega \sum_{i=1}^M\mu^{n+1}_i (\phi^{n+1} c^{n+1}_i - \phi^n c^n_i) \, d\boldsymbol{x}
		= -\int_\Omega \sum_{i=1}^Mc^n_i \boldsymbol{u}^{n+1}_i \cdot \nabla \mu^{n+1}_i \, d\boldsymbol{x}.
	\end{align} 
	We multiply both sides of \eqref{eq-2-2-c} by $\boldsymbol{u}^{n+1}_i$, then integrate over the domain $\Omega$, Summing over all $i$, we get
	\begin{align}\label{eq-3-15}
		\int_\Omega \sum_{i=1}^Mc^n_i \boldsymbol{u}^{n+1}_i \cdot \nabla \mu^{n+1}_i \, d\boldsymbol{x}
		= -\int_\Omega \sum_{i=1}^M\left( \frac{|\boldsymbol{u}^{n+1}_i|^2}{\mathscr{D}_{i,s}} + \sum_{j\neq i}^{M}\frac{c^{n}_i c^{n}_j|\boldsymbol{u}^{n+1}_i - \boldsymbol{u}^{n+1}_j|^2}{2(c^n)^2 \mathscr{D}_{i,j}} \right) \, d\boldsymbol{x}.
	\end{align}
	By using \eqref{eq-3-6}, we can obtain
\begin{align}\label{eq-3-16}
	&\int_{\Omega} \left( (\phi^{n+1}f(c^{n+1}) - \phi^n f(c^n)) + p^{n+1}(\phi^{n+1} - \phi^n) \right) \, d\boldsymbol{x} \\\nonumber
	&\leq \frac{1}{\tau_n} \int_{\Omega} \Bigg( f(c^n)(\phi^{n+1} - \phi^n) + \phi^{n+1}\boldsymbol{\mu}^{n+1} \cdot (c^{n+1} - c^n) \\\nonumber
	&\quad + ( \sum_{i=1}^M c_i^n \mu_i^{n+1} - f(c^n) ) (\phi^{n+1} - \phi^n) \Bigg) \, d\boldsymbol{x} \\\nonumber
	&= \frac{1}{\tau_n} \int_{\Omega} \sum_{i=1}^M \mu_i^{n+1} \left( \phi^{n+1}c_i^{n+1} - \phi^n c_i^n \right) \, d\boldsymbol{x}.
\end{align}
	By \eqref{eq-3-14}, \eqref{eq-3-15} and \eqref{eq-3-16}, we get
	
	\begin{align}\label{eq-3-17}
		&\frac{1}{\tau_n}(\Psi_{\text{f}}^{n+1} - \Psi_{\text{f}}^n + \int_\Omega p^{n+1}(\phi^{n+1}-\phi^n)\,d\boldsymbol{x})\\\nonumber
		&\leq -\int_\Omega \left( \sum_{i=1}^M\frac{|\boldsymbol{u}^{n+1}_i|^2}{\mathscr{D}_{i,s}} + \sum_{j\neq i}^{M} \frac{c^{n}_i c^{n}_j|\boldsymbol{u}^{n+1}_i - \boldsymbol{u}^{n+1}_j|^2}{2(c^n)^2 \mathscr{D}_{i,j}} \right) \, d\boldsymbol{x}.
	\end{align}	
	From the energy estimate derived from \eqref{eq-2-2-e} and \eqref{eq-2-6}, we have
	\begin{align}\label{eq-3-18}
		&\frac{1}{\tau_n} \left( \eta\|\varepsilon(\boldsymbol{w}^{n+1}_{s})\|^{2}_{L^{2}} - \eta\|\varepsilon(\boldsymbol{w}^{n}_{s})\|^{2}_{L^{2}}
		+ \frac{\gamma}{2}\|\nabla\cdot\boldsymbol{w}_{s}^{n+1}\|^{2} - \frac{\gamma}{2}\|\nabla\cdot\boldsymbol{w}^{n}_{s}\|^{2} \right. \\\nonumber
		&\left.\quad + \frac{1}{2N}\|p^{n+1}\|^{2}_{L^2} - \frac{1}{2N}\|p^{n}\|^{2}_{L^2} - (p^{n+1}, \phi^{n+1} - \phi^n) \right) \leq 0.
	\end{align}
	Adding \eqref{eq-3-17} and \eqref{eq-3-18}, we obtain
	\[
	\frac{1}{\tau_n}(E^{n+1} - E^n) \leq -\int_\Omega \left( \sum_{i=1}^M\frac{|\boldsymbol{u}^{n+1}_i|^2}{\mathscr{D}_{i,s}} + \sum_{j\neq i}^{M} \frac{c^{n}_i c^{n}_j|\boldsymbol{u}^{n+1}_i - \boldsymbol{u}^{n+1}_j|^2}{2(c^n)^2 \mathscr{D}_{i,j}} \right) \, d\boldsymbol{x}.
	\]
	
	Since all coefficients in the dissipation term are nonnegative, $E^{n+1}=E^n$ implies
	$\boldsymbol{u}_i^{n+1}\equiv \boldsymbol{0}$ for all $i$.
	Conversely, if $\boldsymbol{u}_i^{n+1}\equiv \boldsymbol{0}$, then \eqref{eq-2-2-c} reduces to
	$c_i^n\nabla\mu_i^{n+1}=\boldsymbol{0}$, hence $\nabla\mu_i^{n+1}=\boldsymbol{0}$ on $\{c_i^n>0\}$.
	Using the Gibbs--Duhem relation at constant temperature, $dp=\sum_{i=1}^M c_i\,d\mu_i$, we obtain
	$\nabla p^{n+1}=\sum_{i=1}^M c_i^{n+1}\nabla\mu_i^{n+1}=\boldsymbol{0}$, i.e., $p^{n+1}$ is spatially constant.
	This completes the proof.
\end{proof}

\section{Fully discrete Scheme}\label{sec-fully}
This section presents the fully discrete numerical scheme for the governing equation system. Temporal discretization is achieved using a semi-implicit discrete scheme. For spatial discretization, we employ a strategic combination of finite element methods tailored to the specific physics of each sub-problem. The mass conservation equations and the MSD model are discretized using a mixed finite element method. This approach allows for the accurate computation of molar density and velocity, ensuring local mass conservation. The momentum equation for the elastic solid is discretized using a discontinuous Galerkin (DG) method. This choice is particularly effective in mitigating numerical locking phenomena, a common issue in computational elasticity when modeling nearly incompressible materials. 

To rigorously enforce boundedness for multicomponent mixtures, a primary challenge stems from the need to satisfy a set of interdependent constraints: each molar density $c_i$ must satisfy the physical bound $0 <= \beta_i c_i < 1$, and concurrently, the total molar density $c = \sum_{i=1}^{M} c_i$ must obey the same strict bound $0 < \beta c < 1$, a condition intrinsic to the form of the Helmholtz free energy. In the conventional formulation, the $M$ component densities $(c_1, \dots, c_M)$ are taken as independent primary variables, and the essential constraints, in particular the upper bounds on $c_i$ and the total density $c$ are implicit and strongly coupled. This makes it extremely challenging to design algorithms that can simultaneously and strictly enforce this complete set of bounds. To circumvent this limitation, we develop a carefully tailored change of variables. A key observation is that the molar density of the $M$-th component can be algebraically eliminated, expressing it as a function of the remaining $M-1$ component densities and the total molar density $c$:
\begin{equation}
	c_{M} = c - \sum_{i=1}^{M-1} c_{i}.
\end{equation}
This simple yet pivotal relation allows us to reformulate the governing system, adopting the set $(c_1, \dots, c_{M-1}, c)$ as the new primary unknowns.


Next, we introduce the finite element spaces and associated notation required for the discretization of the multicomponent flow system. Let $\mathcal{T}_h$ be a family of conforming, shape-regular, and quasi-uniform triangulations of the domain $\Omega$, where $h$ denotes the maximum element diameter. For any element $K \in \mathcal{T}_h$, we denote by $\mathcal{E}_h$ the set of all facets (edges in 2D, faces in 3D) of the triangulation, and by $\mathcal{E}_h^I$ the subset of interior facets. For two adjacent elements $K^+, K^-$ sharing a common facet $e = \partial K^+ \cap \partial K^- \in \mathcal{E}_h^I$, we fix a unit normal vector $\boldsymbol{n}_e$ pointing from $K^+$ to $K^-$.

We define the average and jump operators for any piecewise smooth function $\psi$ on an interior face $e$
\[
\{\psi\}_e := \frac{1}{2}\left(\psi^+ + \psi^-\right), \quad [\psi]_e := \psi^+ - \psi^-,
\]
where $\psi^\pm$ denote the traces of $\psi$ from $K^\pm$ onto $e$, respectively.

The $L^2$ inner products over domains and facets are defined in the standard manner. For any domain $D \subseteq \Omega$ and functions $\xi, \zeta$ (scalar- or vector-valued), we define
\[
(\xi, \zeta)_D := \int_D \xi \zeta  d\boldsymbol{x}, \quad \langle \xi, \zeta \rangle_e := \int_e \xi \zeta  ds.
\]
The corresponding norms are denoted by $\|\cdot\|_{L^2(D)}$ and $\|\cdot\|_{L^2(e)}$, respectively.

We employ a mixed finite element approximation that uses different function spaces tailored to the specific physical fields, the solid displacement $\boldsymbol{w}_s$ is approximated in the space of piecewise linear vector-valued functions
\[
\boldsymbol{\mathcal{V}}_h := \left\{ \boldsymbol{v}_h \in L^2(\Omega)^d : \boldsymbol{v}_h|_K \in [\mathbb{P}_1(K)]^d, \ \forall K \in \mathcal{T}_h \right\},
\]
the Darcy velocity $\boldsymbol{u}_i$ is discretized using the lowest-order Raviart-Thomas space to ensure exact mass conservation at the discrete level
\[
\boldsymbol{\mathcal{U}}_h := \left\{ \boldsymbol{v}_h \in H(\text{div}; \Omega) : \boldsymbol{v}_h|_K \in RT_0(K), \ \forall K \in \mathcal{T}_h \right\},
\]
where $RT_0(K) := [\mathbb{P}_0(K)]^d + \boldsymbol{x} \mathbb{P}_0(K)$ denotes the local Raviart-Thomas space on element $K$. We also define its subspace with homogeneous normal flux on the boundary
\[
\boldsymbol{\mathcal{U}}_h^0 := \left\{ \boldsymbol{v}_h \in \boldsymbol{\mathcal{U}}_h : \boldsymbol{v}_h \cdot \boldsymbol{n} = 0 \text{ on } \partial\Omega \right\}.
\]
The scalar unknowns, molar density $c_i$, pressure $p$, and porosity $\phi$, are approximated in the space of piecewise constant functions
\[
\mathcal{Q}_h := \left\{ q_h \in L^2(\Omega) : q_h|_K \in \mathbb{P}_0(K), \ \forall K \in \mathcal{T}_h \right\}.
\]

To discretize the advective terms in the mass conservation equations, we employ an upwind scheme. The upwind value of a density $c_h$ on a facet $e$ is defined based on the direction of the velocity flux
\begin{align}
	c^{n*}_{h}= \begin{cases}c^{n}_{h}|_{K^{+}}, & \boldsymbol{u}_{i,h}^{n} \cdot \boldsymbol{n}_{e} \geq 0, \\ c^{n}_{h}|_{K^{-}}, & \boldsymbol{u}_{i,h}^{n} \cdot \boldsymbol{n}_{e} <0. \end{cases}
\end{align}

The fully discrete problem seek $(\boldsymbol{w}^{n+1}_{s,h}, \boldsymbol{u}^{n+1}_{i,h} | _{i=1,\cdots, M}, c^{n+1}_{i,h}| _{i=1,\cdots, M-1}, c^{n+1}_{h}, p^{n+1}_h$, $\phi^{n+1}_h)$ $\in$ $\boldsymbol{\mathcal{V}}_h$ $\times \boldsymbol{\mathcal{U}}_h$ $\times \mathcal{Q}_h$ $\times \mathcal{Q}_h$ $\times \mathcal{Q}_h$ $\times \mathcal{Q}_h$ such that the weak form of the governing equations is satisfied for all admissible test functions in the corresponding discrete spaces.
\begin{subequations}\label{eq-full-discrete}
    \begin{align}
	&(D_{\tau_n} (\phi^{n+1}_{h}c^{n+1}_{h}), q_h) + 
	\sum_{e\in \mathcal{E}_h^I}\langle \sum_i c_{i,h}^{n*}\boldsymbol{u}_{i,h}^{n+1}\cdot\boldsymbol{n}, [q_h]\rangle_e + \sum_{e\in \mathcal{E}_h^I}\frac{\varsigma}{h_e}\langle \mathcal{K}_{e} [\mu^{n+1}_{h}], [q_h] \rangle_e= 0,\label{eq-4-1-1}\\ 
	&(D_{\tau_n} (\phi^{n+1}_{h}c^{n+1}_{i,h}), q_h) + 
	\sum_{e\in \mathcal{E}_h^I}\langle c_{i,h}^{n*}\boldsymbol{u}_{i,h}^{n+1}\cdot\boldsymbol{n}, [q_h]\rangle_e \label{eq-4-1-2}\\\nonumber
	&\qquad\qquad+ \sum_{e\in \mathcal{E}_h^I}\frac{\varsigma}{h_e}\langle \mathcal{K}_{e}[\mu^{n+1}_{i,h}], [q_h] \rangle_e= 0, \quad i=1, \cdots, M-1, \\
	&\sum_{j=1}^{i} \left(\frac{c^{n}_{i,h} c^{n}_{j,h}\left(\boldsymbol{u}^{n+1}_{i,h} -\boldsymbol{u}^{n+1}_{j,h}\right)}{(c^{n}_{h})^2 \mathscr{D}_{i, j}}, \mathbf{v}_{h}\right)+\sum_{j=i+1}^{M} \left(\frac{c^{n}_{i,h} c^{n}_{j,h}\left(\boldsymbol{u}^{n+1}_{i,h} -\boldsymbol{u}^{n+1}_{j,h}\right)}{(c^{n}_{h})^2 \mathscr{D}_{i, j}}, \mathbf{v}_{h}\right)\label{eq-4-1-3}\\\nonumber
	&\qquad\qquad\qquad+\left(\frac{\boldsymbol{u}^{n+1}_{i,h}}{\mathscr{D}_{i, s}},\mathbf{v}_{h}\right)=\sum_{e\in \mathcal{E}_h^I}\langle [\mu_{i,h}^{n+1}],c^{n*}_{i,h}\mathbf{v}_{h}\cdot\boldsymbol{n}\rangle_e , \quad i=1, \cdots, M,\\
	&(p^{n+1}_h, z_h) = (\sum_{i=1}^{M}c_{i,h}^{n}\mu_{i,h}^{n+1} - f(\mathbf{c}_{h}^{n}), z_h),\label{eq-4-1-4}\\
	&\mathcal{A}(\boldsymbol{w}^{n+1}_{s,h},  p_h^{n+1},\boldsymbol{\upsilon}_h) = 0, \label{eq-4-1-5} \\
	&(D_{\tau_n} \phi^{n+1}_h,\varphi_h) = \frac{1}{N}(D_{\tau_n} p^{n+1}_h, \varphi_h) + \alpha(D_{\tau_n} (\nabla\cdot \boldsymbol{w}_{s,h}^{n+1}), \varphi_h)\label{eq-4-1-6}\\\nonumber
	&\qquad\qquad\qquad\qquad\qquad\qquad\qquad-\alpha\sum_{e \in \mathcal{E}_h^I}\langle\{\varphi_h\boldsymbol{n}_e\}, [D_{\tau_n}\boldsymbol{w}^{n+1}_{s,h}]\rangle_e,
\end{align}
\end{subequations}
where $\mathcal{K}_{e} > 0$ represents the average of the components of $\boldsymbol{K}$ on the edge $e$ and $\mathcal{A}(\boldsymbol{w}_{s,h}, p_h,\textbf{v}_h)$ is the bilinear form defined as 
	\begin{align}
		\mathcal{A}(\boldsymbol{w}_{s,h}, p_h,\textbf{v}_h):= &\sum_{K \in \mathcal{K}_h}(\mathbf{\sigma}_e(\boldsymbol{w}_{s,h}), \varepsilon(\textbf{v}_h))_K-\sum_{e \in \mathcal{E}^{I}_h}\langle\{\mathbf{\sigma}_e(\boldsymbol{w}_{s,h}) \boldsymbol{n}_e\}, [\textbf{v}_h]\rangle_e \\\nonumber
		&- \alpha\sum_{K \in \mathcal{K}_h}(p_h, \nabla\cdot \mathbf{v}_h)_{K}+\alpha\sum_{e \in \mathcal{E}^{I}_h}\langle\{p_h\boldsymbol{n}_e\}, [\textbf{v}_h]\rangle_e\\\nonumber
		&-\sum_{e \in \mathcal{E}^{I}_h}\langle [\boldsymbol{w}_{s,h}],\left\{\mathbf{\sigma}_{e}(\textbf{v}_h) \boldsymbol{n}_e\right\}\rangle_e+\sum_{e \in \mathcal{E}^{I}_h} \frac{\varsigma_1}{h_e}\langle[\boldsymbol{w}_{s,h}], [\textbf{v}_h]\rangle_e.
	\end{align}.
\begin{rem}\label{re-4-1}
	In \eqref{eq-4-1-1}--\eqref{eq-4-1-2} we solve for the total molar density $c_h^{n+1}$ together with the first $M\!-\!1$
	component densities $\{c_{i,h}^{n+1}\}_{i=1}^{M-1}$. The remaining component is recovered by the algebraic constraint
	\[
	c_{M,h}^{n+1}:=c_h^{n+1}-\sum_{i=1}^{M-1}c_{i,h}^{n+1}\qquad(\text{elementwise in }\mathcal Q_h),
	\]
	so that the full set $\{c_{i,h}^{n+1}\}_{i=1}^M$ is always available.
	
	This formulation is unbiased with respect to the choice of the dependent component. Indeed, subtracting the sum of
	\eqref{eq-4-1-2} over $i=1,\dots,M-1$ from \eqref{eq-4-1-1} yields the discrete mass balance for $c_{M,h}^{n+1}$.
	Hence one may equivalently enforce the $M$-th component balance in place of the total-balance equation.
	The same argument applies if any other component is selected as the recovered one.
	Therefore, \eqref{eq-4-1-1}--\eqref{eq-4-1-2} provides a nonredundant representation of the full $M$-component conservation system:
	it is algebraically equivalent to solving all $M$ component conservation laws simultaneously, while avoiding redundant unknowns.
\end{rem}
\begin{theorem}
	Assume the boundary condition \eqref{eq-BD}-\eqref{eq-BD-1} holds. Then, the total free energy defined by the fully discrete scheme is dissipated over time, i.e.,
	$$D_{\tau_n} E^{n+1}_{h}  \leq 0,$$
	where the discrete energy $E^{n+1}_h$ is given by
	\begin{align}
		E^{n+1}_h =& \sum_{K \in \mathcal{T}_h}\int_{K}\left(\phi^{n+1}_{h}f(\boldsymbol{c}^{n+1}_h)+\frac{1}{2}\boldsymbol{\sigma}_{e}(\boldsymbol{w}^{n+1}_{s,h}):\varepsilon(\boldsymbol{w}^{n+1}_{s,h}) + \frac{1}{2N}|p^{n+1}_{h}|^2\right) ~ d \boldsymbol{x}\\\nonumber
		&+\sum_{e \in \mathcal{E}_{h}^I} \frac{\varsigma_1}{2h_e}\langle[\boldsymbol{w}^{n+1}_{s,h}], [ \boldsymbol{w}^{n+1}_{s,h}]\rangle_e
		-\sum_{e \in \mathcal{E}_{h}^I}\langle\{\boldsymbol{\sigma}_e(\boldsymbol{w}^{n+1}_{s,h}) \boldsymbol{n}_e\}, [\boldsymbol{w}^{n+1}_{s,h}]\rangle_e.
	\end{align}
\end{theorem}
\begin{proof}
	The proof proceeds by estimating the difference in the total free energy between two consecutive time steps. From \eqref{eq-3-6} , we derive the following key estimate that
	\begin{align}\label{eq-2-2-13}
		& \frac{1}{\tau_n}\sum_{K \in \mathcal{T}_h}\int_{K}\left((\phi_h^{n+1} f(\boldsymbol{c}_h^{n+1})-\phi_h^n f(\boldsymbol{c}_h^n)) + (\phi_h^{n+1}-\phi_h^{n}) p_h^{n+1}\right) ~d\boldsymbol{x}\\\nonumber
		\leq &\frac{1}{\tau_n}\sum_{K \in \mathcal{T}_h}\int_{K} \left(f(\boldsymbol{c}_h^n)(\phi_h^{n+1}-\phi_h^n) + \phi_h^{n+1} \boldsymbol{\mu}_{h}^{n+1}\cdot(\boldsymbol{c}_h^{n+1}-\boldsymbol{c}_h^n)\right.\\\nonumber
		&\left. + \left(\sum_{i=1}^{M}c^{n}_{i,h}\mu_{i ,h}^{n+1}-f(\boldsymbol{c}_h^n)\right)(\phi_h^{n+1}-\phi_h^n)\right) ~d\boldsymbol{x}.
	\end{align}
	Combining the right-hand side of the above inequality and \eqref{eq-4-1-2} allows us to rewrite it as follows:
	\begin{align*}
		&\frac{1}{\tau_n} \sum_{K \in \mathcal{T}_h}\int_{K}\sum_{i=1}^{M}\mu_{i,h}^{n+1}\left(\phi_h^{n+1} c_{i,h}^{n+1}-\phi_h^n c_{i, h}^n\right) d \boldsymbol{x}\\
		=&-\sum_{e \in \mathcal{E}_h^I}\int_{e} \sum_{i=1}^{M}\left([\mu_{i,h}^{n+1}] c^{n*}_h\boldsymbol{u}_{i,h}^{n+1}\cdot\boldsymbol{n}_e + \frac{\varsigma}{h_e} \mathcal{K}_{e}[\mu_{i,h}^{n+1}]^2\right) ~d \boldsymbol{x}.
	\end{align*}
	Finally, employing \eqref{eq-4-1-3} allows us to express this solely as a negative dissipation term, thereby confirming the energy decay induced by the transport process
	\begin{align}\label{eq-dissipation-transport}
		&\frac{1}{\tau_n} \sum_{K \in \mathcal{T}_h}\int_{K}\sum_{i=1}^{M}\mu_{i,h}^{n+1}\left(\phi_h^{n+1} c_{i,h}^{n+1}-\phi_h^n c_{i, h}^n\right) d \boldsymbol{x}\\\nonumber
		=&-\sum_{e \in \mathcal{E}_h^I}\int_{e} \sum_{i=1}^{M}\left([\mu_{i,h}^{n+1}] c^{n*}_h\boldsymbol{u}_{i,h}^{n+1}\cdot\boldsymbol{n}_e + \frac{\varsigma}{h_e} \mathcal{K}_{e}[\mu_{i,h}^{n+1}]^2\right) ~d \boldsymbol{x}\\\nonumber
		=&-\int_{\Omega} \left( \sum_{i=1}^{M}\frac{|\boldsymbol{u}^{n+1}_{i, h}|^2}{\mathscr{D}_{i, s}} + \sum_{i=1}^M\sum_{j=1, j \neq i}^M\frac{ c^{n}_{i,h}c^{n}_{j,h} |\boldsymbol{u}^{n+1}_{i,h}-\boldsymbol{u}^{n+1}_{j,h}|^2}{2(c^n_h)^2 \mathscr{D}_{i,j}} \right) ~d \boldsymbol{x} -\sum_{e \in \mathcal{E}^I_h}\int_{e}\sum_{i=1}^{M} \frac{\varsigma}{h_e} \mathcal{K}_{e}[\mu_{i,h}^{n+1}]^2 ~d \boldsymbol{x}.
	\end{align}
	As established in \cite{Chencmame2024} for the DG discretization of the momentum equation, the solid free energy evolution satisfies
	\begin{align}\label{eq-2-3130}
		&\frac{1}{2}\sum_{K \in \mathcal{T}_h}D_{\tau_n}(\boldsymbol{\sigma}_e(\boldsymbol{w}_{s,h}^{n+1}), \varepsilon(\boldsymbol{w}_{s,h}^{n+1}))_K\\\nonumber
		& + \sum_{e \in \mathcal{E}_h^I} \frac{\varsigma_1}{2h_e}D_{\tau_n}\langle[\boldsymbol{w}_{s,h}^{n+1}], [ \boldsymbol{w}_{s,h}^{n+1}]\rangle_e + \frac{1}{2N}\sum_{K \in \mathcal{T}_h}D_{\tau_n}(p_{h}^{n+1}, p_{h}^{n+1})\\\nonumber
		& -\sum_{e \in \mathcal{E}_h^I}D_{\tau_n}\left(\langle\{\boldsymbol{\sigma}_e(\boldsymbol{w}_{s,h}^{n+1}) \boldsymbol{n}_e\}, [\boldsymbol{w}_{s,h}^{n+1}]\rangle_e\right) -\sum_{K \in \mathcal{T}_h}(D_{\tau_n} \phi^{n+1}_h,p^{n+1}_h)_K\\\nonumber
		\leq&  \  0.
	\end{align}
	Combining \eqref{eq-dissipation-transport} with \eqref{eq-2-3130}, we directly obtain the desired result
	$$D_{\tau_n} E^{n+1}_{h}  \leq 0.$$
	This completes the proof.
\end{proof}

Furthermore, a linearized iterative strategy is employed within each time step to handle nonlinearities, where the superscript $l$ denotes the value at the $l$-th iteration. The fully discrete scheme is given as follows
\begin{align}
	&(D_{\tau_n} (\phi^{n+1,l}_{h}c^{n+1,l+1}_{h}), q_h) + 
	\sum_{e\in \mathcal{E}_h^I}\langle \sum_i c_{i,h}^{n*}\boldsymbol{u}_{i,h}^{n+1, l}\cdot\boldsymbol{n}, [q_h]\rangle_e + \sum_{e\in \mathcal{E}_h^I}\frac{\varsigma}{h_e}\langle\mathcal{K}_{e}[\mu^{n+1,l+1}_{h}], [q_h] \rangle_e= 0,\label{eq-4-1}\\ 
	&(D_{\tau_n} (\phi^{n+1,l}_{h}c^{n+1,l+1}_{i,h}), q_h) + 
	\sum_{e\in \mathcal{E}_h^I}\langle c_{i,h}^{n*}\boldsymbol{u}_{i,h}^{n+1,l}\cdot\boldsymbol{n}, [q_h]\rangle_e \label{eq-4-2}\\\nonumber
	&\qquad\qquad+ \sum_{e\in \mathcal{E}_h^I}\frac{\varsigma}{h_e}\langle\mathcal{K}_{e}[\mu^{n+1,l+1}_{i,h}], [q_h]\rangle_e= 0, \quad i=1, \cdots, M-1, \\
	&\sum_{j=1}^{i} \left(\frac{c^{n}_{i,h} c^{n}_{j,h}\left(\boldsymbol{u}^{n+1,l+1}_{i,h} -\boldsymbol{u}^{n+1,l+1}_{j,h}\right)}{(c^{n}_{h})^2 \mathscr{D}_{i, j}}, \mathbf{v}_{h}\right)+\sum_{j=i+1}^{M} \left(\frac{c^{n}_{i,h} c^{n}_{j,h}\left(\boldsymbol{u}^{n+1,l+1}_{i,h} -\boldsymbol{u}^{n+1,l}_{j,h}\right)}{(c^{n}_{h})^2 \mathscr{D}_{i, j}}, \mathbf{v}_{h}\right)\label{eq-4-3}\\\nonumber
	&\qquad\qquad\qquad+\left(\frac{\boldsymbol{u}^{n+1,l+1}_{i,h}}{\mathscr{D}_{i, s}},\mathbf{v}_{h}\right)=\sum_{e\in \mathcal{E}_h^I}\langle [\mu_{i,h}^{n+1,l+1}],c^{n*}_{i,h}\mathbf{v}_{h}\cdot\boldsymbol{n}\rangle_e , \quad i=1, \cdots, M,\\
	&(p^{n+1,l+1}_h, z_h) = (\sum_{i=1}^{M}c_{i,h}^{n}\mu_{i,h}^{n+1,l+1} - f(\mathbf{c}_{h}^{n}), z_h),\label{eq-4-4}\\
	&\mathcal{A}(\boldsymbol{w}^{n+1,l+1}_{s,h},  p_h^{n+1,l+1},\boldsymbol{\upsilon}_h) = 0, \label{eq-4-5} \\
	&(D_{\tau_n} \phi^{n+1,l+1}_h,\varphi_h) = \frac{1}{N}(D_{\tau_n} p^{n+1,l+1}_h, \varphi_h) + \alpha(D_{\tau_n} (\nabla\cdot \boldsymbol{w}_{s,h}^{n+1,l+1}), \varphi_h)\label{eq-4-6}\\\nonumber
	&\qquad\qquad\qquad\qquad\qquad\qquad\qquad-\alpha\sum_{e \in \mathcal{E}_h^I}\langle\{\varphi_h\boldsymbol{n}_e\}, [D_{\tau_n}\boldsymbol{w}^{n+1,l+1}_{s,h}]\rangle_e.
\end{align}
\begin{theorem}
	The multicomponent numerical scheme \eqref{eq-4-1-1}--\eqref{eq-4-1-6} ensures both global and local conservation of mole numbers for each component  $i \in \{1, \dots, N\}$.
\end{theorem}
\begin{proof}
	Since the proof follows a standard procedure, we omit the detailed derivation here and refer interested readers to \cite{Chencmame2024}. 
\end{proof}
\begin{lemma}\label{le-1}
	For two given real constants $a, b$ and $a < b$, for any $c_h \in \mathcal{Q}_{h}$, we define $c_{h,-} = \min(c_h+a, 0)$, $c_{h,+} = \max(c_h-b, 0)$. We have the following inequalities hold
	\begin{align}
		\sum_{e\in \mathcal{E}^{I}_h}\langle[c_h], [c_{h, -}]\rangle_e \geq \sum_{e\in \mathcal{E}^{I}_h} \langle[c_{h, -}], [c_{h, -}]\rangle_e,\label{eq-108-1}\\
		\sum_{e\in \mathcal{E}^{I}_h}\langle[c_h], [c_{h, +}]\rangle_e \geq \sum_{e\in \mathcal{E}^{I}_h} \langle[c_{h, +}], [c_{h, +}]\rangle_e.\label{eq-108-2}
	\end{align}
\end{lemma}
A detailed proof of Lemma \ref{le-1} is provided in \cite{Chenpre}.

We now provide a rigorous proof demonstrating that the proposed numerical scheme, combined with the adaptive time-stepping strategy, preserves the physical bounds for all components. Specifically, we prove that the discrete molar densities satisfy $0< c^{n+1}_h < \frac{1}{\beta^{\ast}}$ and  $0 < c^{n+1}_{i,h} \leq c^{n+1}_h < \frac{1}{\beta^{\ast}}$ for $i = 1, \cdots, M$.
\begin{theorem}\label{theo-1}
	Assume that $0 < \beta^{\ast} c^n_h < 1$ and the boundary condition \eqref{eq-BD}-\eqref{eq-BD-1} holds. For $n \geq 0$ and given constants $0 < \delta_{i, 1}, \delta_{i, 2}< \frac{c^{n}_{i,h}}{c^{n}_{h}}< 1$ and $0 <\delta_{1}, \delta_{2} <1$, if the time step size $\tau^{l+1}_n$ satisfies
	\begin{align}\label{eq-tau}
		\tau^{l+1}_n = \min_i\min_{K\in\mathcal{T}_h}\left(\tau^{i, l+1}_{\ast}, \tau^{i, l+1}_{\star}, \tau_{max}\right),
	\end{align}
	where 
	$$\tau^{i,l+1}_{\ast}:= \frac{\left(\phi^{n+1,l}_{h} c^{n}_{h}\left(1-\beta^{\ast} c^{n}_{h}\right)\delta_{i, 1}- (\phi^{n+1,l}_{h}-\phi^{n}_{h})c^{n}_{i,h}\right)|K|}{\sum\limits_{e\in \partial K_{\boldsymbol{u}_{i,h}}^{+}} c^{n}_{i,h}\boldsymbol{u}_{i,h}^{n+1, l}\cdot\boldsymbol{n}|e| + \sum\limits_{e\in \partial K_{\mu_i}^{+}}\frac{\varsigma_1}{h_e}\mathcal{K}_{e}[\mu_i(\boldsymbol{c}^{n}_{h})]|e|+ \epsilon},$$\\
	$$\tau^{i,l+1}_{\star} := \frac{\left(\phi^{n+1,l}_{h} c^{n}_{h}\left(1-\beta^{\ast} c^{n}_{h}\right)\delta_{i, 2}+ (\phi^{n+1,l}_{h}-\phi^{n}_{h})c^{n}_{i,h}\right)|K|}{-\left(\sum\limits_{e\in \partial K_{\boldsymbol{u}_{i,h}}^{-}} c^{n*}_{i,h}\boldsymbol{u}_{i,h}^{n+1, l}\cdot\boldsymbol{n}|e| + \sum\limits_{e\in \partial K_{\mu_i}^{-}}\frac{\varsigma_1}{h_e}\mathcal{K}_{e}[\mu_i(\boldsymbol{c}^{n}_{h})]|e|\right)+ \epsilon}.$$
	Here, $\epsilon > 0$ is a very small constant to avoid zero denominator, $\tau_{max}$ is the allowed maximum time step size to
	guarantee the accuracy of numerical solutions and $\partial K_{\boldsymbol{u}_{i,h}}^{+} = \{e \in \partial K: \boldsymbol{u}_{i,h}^{n+1}\cdot\boldsymbol{n}|_e > 0, \forall K \in \mathcal{T}_{h}\}$, $\partial K_{\boldsymbol{u}_{i,h}}^{-} = \{e\in\partial K: \boldsymbol{u}_{i,h}^{n+1}\cdot\boldsymbol{n}|_e < 0, \forall K \in \mathcal{T}_{h}\}$, $\partial K_{\mu_i}^{-} = \{e\in\partial K: [\mu(c^{n}_{h})] < 0, \forall K \in \mathcal{T}_{h}\}, \partial K_{\mu_i}^{+} = \{e\in\partial K: [\mu(c^{n}_{h})] > 0, \forall K \in \mathcal{T}_{h}\}$. Then $c^{n+1, l+1}_h, c^{n+1, l+1}_{i,h}$ satisfies
	\begin{align}
		0<(1 - \delta_{1}\left(1-\beta^{\ast} c^{n}_{h}\right))c^{n}_{h}\leq c_{h}^{n+1,l+1} \leq (1 + \delta_{2}\left(1 - \beta^{\ast} c^{n}_{h}\right))c^{n}_{h}<\frac{1}{\beta^{\ast}}<\frac{1}{\beta},\label{eq-bound-c}\\
		0< c^{n+1, l+1}_{i,h} \leq  c_{h}^{n+1,l+1}.\label{eq-bound-ci}
	\end{align}
\end{theorem}
\begin{proof}
	
	Let us define $\zeta_{i,h}^{n+1, l+1} =  \frac{c^{n+1, l+1}_{i,h} - c^{n}_{i,h}}{c^{n}_{h}(1-\beta^{\ast} c^{n}_{h})}$, $\zeta_{i, h, -}^{n+1, l+1} = \min(\zeta_{i,h}^{n+1, l+1} + \delta_{i,1}, 0)$ and $\zeta_{i, h, +}^{n+1, l+1} = \max(\zeta_{i, h}^{n+1, l+1} - \delta_{i, 2}, 0)$. Obviously $\zeta_{i, h, -}^{n+1, l+1} \leq 0 $ and $\zeta_{i, h, +}^{n+1, l+1} \geq 0$.
	
	Taking $q_h = \zeta_{i, h, -}^{n+1,l+1}$ in \eqref{eq-4-2} and using \eqref{eq-3-7}, we can obtain
	\begin{align}\label{eq-616-1}
		&\frac{1}{\tau^{l}_n}\left(\phi^{n+1,l}_{h}(c_{i,h}^{n+1,l+1}-c^{n}_{i,h}), \zeta_{i, h, -}^{n+1, l+1}\right) + \frac{1}{\tau^{l}_n}\left((\phi^{n+1,l}_{h}-\phi^{n}_{h})c^{n}_{i,h}, \zeta_{i, h, -}^{n+1, l+1}\right)\\\nonumber
		&\quad+ \sum_{e\in \mathcal{E}^{I}_h} \langle c^{n*}_{i,h}\boldsymbol{u}_{i,h}^{n+1, l}\cdot\boldsymbol{n}, \zeta_{i, h, -}^{n+1, l+1}\rangle_{e} + \sum_{e\in \mathcal{E}^{I}_h}\frac{\varsigma_1}{h_e}\langle\mathcal{K}_{e}[\mu_i(\boldsymbol{c}^{n}_{h})], [\zeta_{i, h, -}^{n+1, l+1}]\rangle_e\\\nonumber
		&\quad + \theta_n R T \sum_{e\in \mathcal{E}^{I}_h}\frac{\varsigma_1}{h_e}\langle\mathcal{K}_{e}[\zeta_{i,h}^{n+1, l+1}], [\zeta_{i, h, -}^{n+1, l+1}]\rangle_e = 0.
	\end{align}
	In terms of the definition of $\zeta_{i, h, -}^{n+1, l+1}$, we deduce that
	\begin{align}\label{eq-616-2}
		\left(\phi^{n+1,l}_{h}(c_{i,h}^{n+1,l+1}-c^{n}_{i,h}), \zeta_{i, h, -}^{n+1, l+1}\right) = & \left(\phi^{n+1,l}_{h} c^{n}_{h}\left(1-\beta^{\ast} c^{n}_{h}\right) \zeta_{i,h}^{n+1,l+1},  \zeta_{i, h, -}^{n+1, l+1}\right) \\\nonumber
		= & \left(\phi^{n+1,l}_{h} c^{n}_{h}\left(1-\beta^{\ast} c^{n}_{h}\right)\left(\zeta_{i, h}^{n+1,l+1}+\delta_{i, 1}\right), \zeta_{i, h, -}^{n+1, l+1}\right)\\\nonumber
		& -\left(\phi^{n+1,l}_{h} c^{n}_{h}\left(1-\beta^{\ast} c^{n}_{h}\right) \delta_{i, 1}, \zeta_{i, h, -}^{n+1, l+1}\right)\\\nonumber
		= & \left(\phi^{n+1,l}_{h} c^{n}_{h}\left(1-\beta^{\ast} c^{n}_{h}\right)  \zeta_{i, h, -}^{n+1, l+1}, \zeta_{i, h, -}^{n+1, l+1}\right)\\\nonumber
		& - \left(\phi^{n+1,l}_{h} c^{n}_{h}\left(1-\beta^{\ast} c^{n}_{h}\right) \delta_{i, 1},  \zeta_{i, h, -}^{n+1, l+1}\right).
	\end{align}
	Taking into account \eqref{eq-tau}, $\zeta_{i, h, -}^{n+1,l+1} \leq 0$, and Lemma \ref{le-1},  we can get  
	
	\begin{align}\label{eq-616-3}
		&  \sum_{K\in \mathcal{T}_h }\sum_{e\in \partial K }\langle c^{n*}_h\boldsymbol{u}_{i,h}^{n+1, l}\cdot\boldsymbol{n}, \zeta_{i, h, -}^{n+1, l+1}\rangle_{e} -\frac{1}{\tau^{l}_n}\sum_{K\in \mathcal{T}_h }\left(\phi^{n+1, l}_{h} c^{n}_{h}\left(1-\beta^{\ast} c^{n}_{h}\right) \delta_{i, 1}, \zeta_{i, h, -}^{n+1, l+1}\right) \\\nonumber
		&\quad + \frac{1}{\tau^{l}_n}\sum_{K\in \mathcal{T}_h }\left((\phi^{n+1,l}_{h}-\phi^{n}_{h})c^{n}_{i, h}, \zeta_{i, h, -}^{n+1, l+1}\right) + \sum_{K\in \mathcal{T}_h }\sum_{e\in \partial K }\frac{\varsigma_1}{h_e}\langle\mathcal{K}_{e}[\mu_i(\boldsymbol{c}^{n}_{h})], \zeta_{i, h, -}^{n+1, l+1}\rangle_e\\\nonumber
		& =\frac{1}{\tau^{l}_n}\sum_{K\in \mathcal{T}_h }\left((\phi^{n+1,l}_{h}-\phi^{n}_{h})c^{n}_{i, h}, \zeta_{i, h, -}^{n+1, l+1}\right) - \frac{1}{\tau^{l}_n}\sum_{K\in \mathcal{T}_h }\left(\phi^{n+1, l}_{h} c^{n}_{h}\left(1-\beta^{\ast} c^{n}_{h}\right) \delta_{i, 1}, \zeta_{i, h, -}^{n+1, l+1}\right) \\\nonumber
		&\quad+ \sum_{K\in \mathcal{T}_h }\sum_{e\in \partial K_{\boldsymbol{u}_{i,h}}^{-}}\langle c^{n*}_h\boldsymbol{u}_{i,h}^{n+1, l}\cdot\boldsymbol{n}, \zeta_{i, h, -}^{n+1, l+1}\rangle_{e} + \sum_{K\in \mathcal{T}_h }\sum_{e\in \partial K_{\boldsymbol{u}_{i,h}}^{+}}\langle c^{n*}_h\boldsymbol{u}_{i,h}^{n+1, l}\cdot\boldsymbol{n}, \zeta_{i, h, -}^{n+1, l+1}\rangle_{e}\\\nonumber
		&\quad + \sum_{K\in \mathcal{T}_h }\sum_{e\in \partial K_{\mu_i}^{-}}\frac{\varsigma_1}{h_e}\langle\mathcal{K}_{e}[\mu_i(\boldsymbol{c}^{n}_{h})], \zeta_{i, h, -}^{n+1, l+1}\rangle_e + \sum_{K\in \mathcal{T}_h }\sum_{e\in \partial K_{\mu_i}^{+}}\frac{\varsigma_1}{h_e}\langle\mathcal{K}_{e}[\mu_i(\boldsymbol{c}^{n}_{h})], \zeta_{i, h, -}^{n+1, l+1}\rangle_e\\\nonumber
		& =\sum_{K\in \mathcal{T}_h }\frac{\left((\phi^{n+1,l}_{h}-\phi^{n}_{h})c^{n}_{i, h}-\phi^{n+1,l}_{h} c^{n}_{h}\left(1-\beta^{\ast} c^{n}_{h}\right)\delta_{i, 1}\right) \zeta_{i, h, -}^{n+1, l+1}|K|}{\tau^{l}_n} \\\nonumber
		&\quad +  \sum_{K\in \mathcal{T}_h }\sum_{e\in \partial K_{\boldsymbol{u}_{i,h}}^{-}}c^{n*}_h\boldsymbol{u}_{i,h}^{n+1, l}\cdot\boldsymbol{n} \zeta_{i, h, -}^{n+1, l+1}|e| +  \sum_{K\in \mathcal{T}_h }\sum_{e\in \partial K_{\boldsymbol{u}_{i,h}}^{+}} c^{n}_h\boldsymbol{u}_{i,h}^{n+1, l}\cdot\boldsymbol{n}\zeta_{i, h, -}^{n+1, l+1}|e|\\\nonumber
		&\quad +  \sum_{K\in \mathcal{T}_h }\sum_{e\in \partial K_{\mu_i}^{-}}\frac{\varsigma_1}{h_e}\mathcal{K}_{e}[\mu_i(\boldsymbol{c}^{n}_{h})]\zeta_{i, h, -}^{n+1, l+1}|e| +  \sum_{K\in \mathcal{T}_h }\sum_{e\in \partial K_{\mu_i}^{+}}\frac{\varsigma_1}{h_e}\mathcal{K}_{e}[\mu_i(\boldsymbol{c}^{n}_{h})]\zeta_{i, h, -}^{n+1, l+1}|e|\geq 0.
	\end{align}
	Here, we assume that $\phi^{n+1,l}_{h}\left(1-\beta^{\ast} c^{n}_{h}\right)\delta_{i, 1} > (\phi^{n+1,l}_{h}-\phi^{n}_{h})$, which means 
	\begin{align}\label{eq-1012-2}
		\frac{\phi^{n+1,l}_{h}}{\phi^{n}_{h}}<\frac{1}{1-\left(1-\beta^{\ast} c^{n}_{h}\right)\delta_{i, 1}}.
	\end{align}
    We note that a rigorous proof of a similar inequality under comparable assumptions can be found in \cite{Chenpre}.
    
	Combining \eqref{eq-616-1}-\eqref{eq-616-3}, we get
	\begin{align}
		\left(\phi^{n+1,l}_{h} c^{n}_{h}\left(1-\beta^{\ast} c^{n}_{h}\right)  \zeta_{i, h, -}^{n+1, l+1}, \zeta_{i, h, -}^{n+1, l+1}\right) + \theta_n R T\sum_{e\in \mathcal{E}^{I}_h}\frac{\varsigma_1}{h_e}\langle\mathcal{K}_{e}[\zeta_{i, h, -}^{n+1, l+1}], [\zeta_{i, h, -}^{n+1, l+1}]\rangle_e \leq 0.
	\end{align}
	Due to $0< c^{n}_{i, h} < \frac{1}{\beta^{\ast}} $, we obtain
	\begin{align}
		\zeta_{i, h, -}^{n+1, l+1} \equiv 0,
	\end{align}
	Due to the assumption $0 < \delta_{i, 1}, \delta_{i, 2}< \frac{c^{n}_{i,h}}{c^{n}_{h}}< 1$, we get
	\begin{align}
		c_{i, h}^{n+1, l+1} \geq (1 - \delta_{i, 1}\frac{c^{n}_{h}}{c^{n}_{i, h}}\left(1-\beta^{\ast} c^{n}_{h}\right))c^{n}_{i, h}>0.
	\end{align}
	The proof of Eq. \eqref{eq-bound-c} could be provided with reference to \cite{Chenpre}.
    In the preceding proof, we have shown that the boundedness of the total molar density $c^{n+1}_{h}$ and of each molar density $c^{n+1}_{i,h}$ for $i = 1,\dots, M-1$ follows from equations \eqref{eq-4-1-1}–\eqref{eq-4-1-2}. According to Remark \ref{re-4-1}, if the condition $\sum_{i=1}^{M} c^{n+1}_{i,h} = c^{n+1}_{h}$ holds, solving the system \eqref{eq-4-1-1}–\eqref{eq-4-1-2} is equivalent to solving the mass conservation equation individually for each of the $M$ components. Consequently, the boundedness of $c^{n+1}_{M,h}$ can be derived from its own mass conservation equation, thereby establishing the boundedness of all component densities $c_1,\dots,c_M$ in a complete sense. Now we complete the proof.
\end{proof}

\begin{theorem}\label{thm:phi_bound}
Assume that $\beta^{\ast} c_h^{n}$ is bounded below and above by constants $\varrho_0$ and $\varrho$, respectively, satisfying $0 < \varrho_0 \le \beta^{\ast} c_h^{n} \le \varrho < 1$.
Given that the stabilization parameter $\gamma$ and the penalty parameter $\varsigma_1$ are assigned sufficiently large values.
Then there exists a constant $C_\epsilon>0$ 
such that, for all $l\ge0$,
\begin{align}\label{eq-phi-bounded}
    0<\phi_h^{n}-C_\epsilon \le \phi_h^{n+1,l+1} \le \phi_h^{n}+C_\epsilon <1,
\end{align}
Let \( C_{\phi, \min} \) and \( C_{\phi, \max} \) denote the minimum and maximum values of \( \phi^{n} \), respectively. Then the constant \( C_{\epsilon} \) depends only on these bounds and on the parameter \( \delta_i \).
\end{theorem}

\begin{proof}
The proof is similar to that in \cite{Chenpre}. We only outline the key idea here.
Testing the DG momentum equation with $\mathbf u_{s,h}^{n+1,l+1}-\mathbf u_{s,h}^{n}$ and using standard trace,
Cauchy--Schwarz and Young inequalities yield an estimate that controls
$\nabla\!\cdot(\mathbf u_{s,h}^{n+1,l+1}-\mathbf u_{s,h}^{n})$ and the jump terms by $\sum_{K\in\mathcal{T}_h}\|p_h^{n+1,l+1}-p_h^{n}\|_{L^{2}_{K}}$,
provided the stabilization parameters are sufficiently large.
Together with the discrete equation of state and the assumption $0<\varrho_0\le \beta^{\ast} c_h^{n}\le \varrho<1$,
this gives a uniform bound for $p_h^{n+1,l+1}$ and hence a one-step perturbation bound
$\|\phi_h^{n+1,l+1}-\phi_h^{n}\|_{L^\infty(\Omega)}\le C_\epsilon$.
Choosing $C_\epsilon$ small enough so that $C_\epsilon<C_{\phi,\min}$ and $C_\epsilon<1-C_{\phi,\max}$
implies $0<\phi_h^{n}-C_\epsilon\le \phi_h^{n+1,l+1}\le \phi_h^{n}+C_\epsilon<1$, and the proof is complete.
\end{proof}

Theorem~\ref{theo-1} proves that, under the adaptive time-step restriction \eqref{eq-tau}, the admissible bounds are preserved at each nonlinear iterate within a fixed time step. 
To transfer this iterate-wise bound preservation to the numerical solution at the time level $t^{n+1}$, it remains to show that the nonlinear splitting iteration converges as $l\to\infty$. 
Once convergence is established, the limit inherits the same bounds.

We now turn to the convergence analysis of the nonlinear splitting iteration. 
To this end, we first collect the error relations between two successive nonlinear iterates. 
Throughout the following analysis, the superscript $l$ denotes the nonlinear iteration index 
within the fixed time step, and we define the iteration errors (increments) between the $(l+1)$-th 
and $l$-th iterates by
\[
e_{\boldsymbol{w}_{s,h}}^{n+1,l+1}
:=\boldsymbol{u}_{s,h}^{n+1,l+1}-\boldsymbol{u}_{s,h}^{n+1,l},\quad
e_{p_h}^{n+1,l+1}:=p_h^{n+1,l+1}-p_h^{n+1,l},\quad
e_{c_{i,h}}^{n+1,l+1}:=c_{i,h}^{n+1,l+1}-c_{i,h}^{n+1,l},
\]
and similarly for 
$e_{\boldsymbol{u}_{i,h}}^{n+1,l+1}$, $e_{\mu_{i,h}}^{n+1,l+1}$, and $e_{\phi_h}^{n+1,l+1}$.
These iteration error variables satisfy the following coupled system, which will serve as the 
starting point for the energy estimates derived below:
\begin{subequations}\label{eq:error_system}
\begin{align}
&\mathcal{A}(e_{\boldsymbol{w}_{s,h}}^{n+1,l+1}, e_{p_{h}}^{n+1,l+1}, \mathbf{v}_h) = 0, \label{eq:error_u_p}\\
&\left(\frac{e_{\phi_{h}}^{n+1, l}c^{n+1,l}_{i,h}+\phi_{h}^{n+1,l}e_{c_{i,h}}^{n+1,l+1}}{\tau}, e_{c_{i,h}}^{n+1,l+1}\right)
+ \sum_{e\in \E_h^I}\langle c^{n*}_h\, e_{\boldsymbol{u}_{i, h}}^{n+1, l}\cdot\boldsymbol{n}, [e_{c_{i,h}}^{n+1,l+1}]\rangle_e \label{eq:error_c}\\\nonumber
&\qquad\qquad\qquad\qquad
+ \sum_{e\in \E_h^I}\frac{\varsigma}{h_e}\langle [e_{\mu_{i,h}}^{n+1,l+1}], [e_{c_{i,h}}^{n+1,l+1}]\rangle_e = 0,
\quad i=1,\dots,M, \\
&\sum_{j=1}^{i} \left(\frac{c^{n}_{i,h} c^{n}_{j,h}\left(e_{\boldsymbol{u}_{i,h}}^{n+1,l+1}-e_{\boldsymbol{u}_{j,h}}^{n+1,l+1}\right)}{(c^{n}_{h})^2 \mathscr{D}_{i, j}}, \mathbf{v}_{h}\right)
+\sum_{j=i+1}^{M} \left(\frac{c^{n}_{i,h} c^{n}_{j,h}\left(e_{\boldsymbol{u}_{i,h}}^{n+1,l+1}-e_{\boldsymbol{u}_{j,h}}^{n+1,l}\right)}{(c^{n}_{h})^2 \mathscr{D}_{i, j}}, \mathbf{v}_{h}\right)\nonumber\\
&\qquad\qquad\qquad+
\left(\frac{e_{\boldsymbol{u}_{i,h}}^{n+1,l+1}}{\mathscr{D}_{i, s}},\mathbf{v}_{h}\right)
=\sum_{e\in \E_h^I}\langle [e_{\mu_{i,h}}^{n+1,l+1}],c^{n*}_{i,h}\,\mathbf{v}_{h}\cdot\boldsymbol{n}\rangle_e ,
\quad i=1,\dots,M, \label{eq:error_ui}\\
&e_{\mu_{i,h}}^{n+1,l+1}
= \frac{\theta_n R T}{ c_h^n(1-\beta^{\ast} c_h^n)}\,e_{c_{i,h}}^{n+1,l+1}, \label{eq:error_mu}\\
&e_{p_{h}}^{n+1,l+1}
= c_h^{n}\sum_{i=1}^{M}e_{\mu_{i,h}}^{n+1,l+1}, \label{eq:error_p}\\
&(e_{\phi_{ h}}^{n+1, l+1}, \varphi_h)
= \frac{1}{N}(e_{p_{h}}^{n+1,l+1}, \varphi_h)
+ \alpha (\nabla\cdot e_{\boldsymbol{w}_{s, h}}^{n+1, l+1}, \varphi_h)
-\alpha \sum_{e \in \E_h^I}\langle\{\varphi_h\boldsymbol{n}_e\}, [e_{\boldsymbol{w}_{s,h}}^{n+1,l+1}]\rangle_e. \label{eq:error_phi}
\end{align}
\end{subequations}
\begin{lemma}\label{lem:matrix_recursion}
Assume that the total concentration satisfies
\begin{equation}\label{eq:assumption_rho}
0<\varrho_0 \le \beta^{\ast} c_h^n \le \varrho <1,
\end{equation}
and the parameters Biot's modulus $N$, penalty parameters $\varsigma,\varsigma_1$, the Lam$\acute{e}$ parameter $\gamma$
are chosen sufficiently large and the MSD diffusion coefficients $\mathscr{D}_{i, s}$ are chosen sufficiently small.
Then there exists a nonnegative $2\times 2$ matrix
\[
A=
\begin{pmatrix}
a_{11} & a_{12}\\
a_{21} & a_{22}
\end{pmatrix},
\qquad a_{ij}\ge 0,
\]
such that the iteration errors satisfy the coupled recursion
\begin{equation}\label{eq:X_recursion}
\mathbf{X}^{\,l+1}\le A\,\mathbf{X}^{\,l},
\end{equation}
where the error vector is defined by
\begin{equation}\label{eq:X_def}
\mathbf{X}^{\,l+1}:=
\begin{pmatrix}
X_1^{l+1}\\[0.3em]
X_2^{l+1}
\end{pmatrix}
:=
\begin{pmatrix}
\displaystyle
\sum_{i=1}^{M}\sum_{K \in \T_h}\|e_{c_{i,h}}^{n+1,l+1}\|^{2}_{L^2(K)}
+ \sum_{i=1}^{M}\sum_{e\in \E_h^I}\frac{1}{h_e}\|[e_{c_{i,h}}^{n+1,l+1}]\|^{2}_{L^2(e)}
\\[1.0em]
\displaystyle
\sum_{i=1}^{M}\sum_{K \in \T_h}\|e_{\boldsymbol{u}_{i,h}}^{n+1,l+1}\|_{L^2(K)}^2
\end{pmatrix}.
\end{equation}
Moreover, the coefficients $a_{ij}$ can be determined as
\begin{subequations}\label{eq:A_entries}
\begin{align}
a_{11}
&=\frac{4M \mathcal{C}_{\phi}}{3C_{\phi,\min}},\label{eq:a11}\\
a_{12}&=\frac13,\label{eq:a12}\\
a_{21}
&=\frac{\displaystyle
\frac{\varrho^2 C_{\mu,\max}^2\,C}{(\beta^{\ast})^2 C_{\min}M}
\frac{M\mathcal{C}_{\phi}}{\left(\varsigma C_{\mu,\min}- \frac{\varrho^{2} C}{(\beta^{\ast})^{2} C_{\phi,\min}}\right)}
}{\displaystyle\Upsilon},\label{eq:a21}\\
a_{22}
&=\frac{\displaystyle
\frac{MC_{\max}}{C_{\min}}
+
\frac{\varrho^2 C_{\mu,\max}^2\,C}{(\beta^{\ast})^2 C_{\min}M}
\frac{C_{\phi,\min}}{4\left(\varsigma C_{\mu,\min}- \frac{\varrho^{2} C}{(\beta^{\ast})^{2} C_{\phi,\min}}\right)}
}{\displaystyle\Upsilon},\label{eq:a22}
\end{align}
\end{subequations}
where
\begin{equation}\label{eq:Delta_def}
\Upsilon:=\left(\frac{M C_{\min}}{2} +\frac{1}{D_{\max}} -  M^2\frac{C_{\max}}{C_{\min}}\right)>0,
\end{equation}
and $C_{\max}, C_{\min}$ are the maximum and minimum value of
$\frac{c^{n}_{i,h}c^{n}_{j,h}}{(c^{n}_{h})^2\mathscr{D}_{i, j}}$,
$C_{\mu,\max}$ is the maximum of $\frac{\theta_n R T}{ c_h^n(1-\beta c_h^n)^2}$,
and $\frac{1}{D_{\max}}$ is a lower bound of $\frac{1}{\mathscr{D}_{i, s}}$.
The constant $\mathcal{C}_{\phi}$ is defined in \eqref{eq:Cphi_def} below.
\end{lemma}

\begin{proof}
The objective is to derive a closed two-component recursion for the error vector $\mathbf{X}^l$.
The proof proceeds in two parts:
(i) derive a concentration/porosity estimate yielding a bound for $X_1^{l+1}$,
(ii) derive a velocity estimate yielding a bound for $X_2^{l+1}$.

Taking $q_h=e_{c_{i,h}}^{n+1,l+1}$ in the concentration error equation \eqref{eq:error_c},
and applying the Cauchy-Schwarz inequality, Young's inequality, we obtain the estimate (see, e.g., \cite{Chenpre}):
\begin{align}\label{eq:c_energy_pre}
&\frac{3C_{\phi, \min}}{4}\sum_{K \in \T_h}\|e_{c_{i,h}}^{n+1,l+1}\|^{2}_{L^2(K)} 
\\\nonumber
&\quad+ \left(\varsigma C_{\mu, \min}- \frac{\varrho^{2} C}{(\beta^{\ast})^{2} C_{\phi, \min}}\right)
\sum_{e\in \E_h^I}\frac{1}{h_e}\|[e_{c_{i,h}}^{n+1,l+1}]\|^{2}_{L^2(e)}\notag\\\nonumber
&\leq
\frac{\varrho^{2}}{(\beta^{\ast})^2C_{\phi, \min}}\sum_{K \in \T_h}\| e_{\phi_{h}}^{n+1, l}\|^{2}_{L^2(K)}
+ \frac{C_{\phi, \min}}{4}\sum_{e\in \E_h^I}\frac{1}{h_e}\|e_{\boldsymbol{u}_{i,h}}^{n+1,l}\|^{2}_{L^{2}(e)}.
\end{align}

Testing the porosity error equation \eqref{eq:error_phi} with $\varphi_h=e_{\phi_h}^{n+1,l}$,
and using the trace inequality together with Cauchy--Schwarz and Young inequalities,
one obtains (cf.~\cite{Chenpre}) the estimate
\begin{align}\label{eq:phi_bound}
\sum_{K \in \T_h}\| e_{\phi_{h}}^{n+1, l}\|^{2}_{L^2(K)}
\le&
\frac{1}{(1-\alpha - \frac{1}{2N})}
\left(
\left(\frac{1}{(\frac{1}{2}-\alpha)\gamma}+\frac{\alpha C}{\varsigma_2-(1+\alpha)C}\right)
\frac{\alpha (1+\gamma)(\varrho C_{\mu, \max})^2}{2\gamma(\beta^{\ast})^2}\right.
\\\nonumber
&\left.+ \frac{(\varrho C_{\mu, \max})^2}{2N(\beta^{\ast})^2}
\right)
\sum_{i=1}^{M}\sum_{K \in \T_h}\| e_{c_{i,h}}^{n+1, l}\|^{2}_{L^2(K)}.
\end{align}

We now introduce the constant $\mathcal{C}_{\phi}$ collecting the coefficients in \eqref{eq:phi_bound}:
\begin{align}\label{eq:Cphi_def}
\mathcal{C}_{\phi} :=
\frac{\varrho^{2}}{(\beta^{\ast})^2C_{\phi, \min}}
\frac{1}{(1-\alpha - \frac{1}{2N})}
&\left(
\left(\frac{1}{(\frac{1}{2}-\alpha)\gamma}+\frac{\alpha C}{\varsigma_2-(1+\alpha)C}\right)\right.\\\nonumber
&\left.\frac{\alpha (1+\gamma)(\varrho C_{\mu, \max})^2}{2\gamma(\beta^{\ast})^2}
+ \frac{(\varrho C_{\mu, \max})^2}{2N(\beta^{\ast})^2}
\right).
\end{align}

\smallskip
For each $i=1,\dots,M$, substituting \eqref{eq:phi_bound} into \eqref{eq:c_energy_pre} yields
\begin{align}\label{eq-109-3}
&\frac{3C_{\phi, \min}}{4}\sum_{K \in \T_h}\|e_{c_{i,h}}^{n+1,l+1}\|^{2}_{L^2(K)}\\\nonumber
&+ \left(\varsigma C_{\mu, \min}- \frac{\varrho^{2} C}{(\beta^{\ast})^{2} C_{\phi, \min}}\right)
\sum_{e\in \E_h^I}\frac{1}{h_e}\|[e_{c_{i,h}}^{n+1,l+1}]\|^{2}_{L^2(e)}\notag\\\nonumber
&\leq \mathcal{C}_{\phi}
\sum_{j=1}^{M}\sum_{K \in \T_h}\| e_{c_{j,h}}^{n+1, l}\|^{2}_{L^2(K)}
+ \frac{C_{\phi, \min}}{4}\sum_{K \in \T_h}\|e_{\boldsymbol{u}_{i,h}}^{n+1,l}\|^{2}_{L^{2}(K)}.
\end{align}

Summing \eqref{eq-109-3} over $i=1,\dots,M$, we obtain
\begin{align}\label{eq-109-3-sum}
&\frac{3C_{\phi, \min}}{4}\sum_{i=1}^{M}\sum_{K \in \T_h}\|e_{c_{i,h}}^{n+1,l+1}\|^{2}_{L^2(K)}\\\nonumber
& + \left(\varsigma C_{\mu, \min}- \frac{\varrho^{2} C}{(\beta^{\ast})^{2} C_{\phi, \min}}\right)
\sum_{i=1}^{M}\sum_{e\in \E_h^I}\frac{1}{h_e}\|[e_{c_{i,h}}^{n+1,l+1}]\|^{2}_{L^2(e)}\notag\\\nonumber
&\leq M\mathcal{C}_{\phi}
\sum_{i=1}^{M}\sum_{K \in \T_h}\| e_{c_{i,h}}^{n+1, l}\|^{2}_{L^2(K)}
+ \frac{C_{\phi, \min}}{4}\sum_{i=1}^{M}\sum_{K \in \T_h}\|e_{\boldsymbol{u}_{i,h}}^{n+1,l}\|^{2}_{L^{2}(K)}.
\end{align}
Next, we use the coercivity condition
\[
\varsigma C_{\mu,\min}- \frac{\varrho^{2} C}{(\beta^{\ast})^{2} C_{\phi, \min}}>0,
\]
which guarantees that the interior-penalty term provides sufficient control of the interelement jumps,
and therefore allows us to bound the jump seminorm.

Up to an equivalent rescaling of constants, the estimate \eqref{eq-109-3-sum} can be rewritten
in the symmetric energy form consistent with the definition of $X_1^{l+1}$:
\begin{align}\label{eq:c_error_first}
&\frac{3C_{\phi,\min}}{4}\sum_{i=1}^{M}\sum_{K \in \T_h}\|e_{c_{i,h}}^{n+1,l+1}\|^{2}_{L^2(K)}
+
\frac{3C_{\phi,\min}}{4}\sum_{i=1}^{M}\sum_{e\in \E_h^I}\frac{1}{h_e}\|[e_{c_{i,h}}^{n+1,l+1}]\|^{2}_{L^2(e)}\notag\\
&\leq
M \mathcal{C}_{\phi}\sum_{i=1}^{M}\sum_{K \in \T_h}\|e_{c_{i,h}}^{n+1,l}\|^{2}_{L^2(K)}
+\frac{C_{\phi,\min}}{4}\sum_{i=1}^{M}\sum_{K \in \T_h}\|e_{\boldsymbol{u}_{i,h}}^{n+1,l}\|_{L^{2}(K)}^{2}.
\end{align}

Dividing \eqref{eq:c_error_first} by $\frac{3C_{\phi,\min}}{4}$ and noting that
$\frac{(C_{\phi,\min}/4)}{(3C_{\phi,\min}/4)}=\frac13$,
we obtain the first recursion component
\begin{equation}\label{eq:X1_recursion}
X_1^{l+1}\le a_{11}X_1^{l}+a_{12}X_2^{l},
\end{equation}
where $a_{11},a_{12}$ are given by \eqref{eq:a11}--\eqref{eq:a12}.

\medskip

Letting $\mathbf{v}_h=e_{\mathbf u_{i,h}}^{n+1,l+1}$ in \eqref{eq:error_ui},
and using the relation \eqref{eq:error_mu} together with the Cauchy--Schwarz and Young inequalities,
we obtain for each $i=1,\dots,M$
\begin{align}\label{eq:u_error_i}
& \left(\frac{M C_{\min}}{2} +\frac{1}{D_{\max}}\right)
\sum_{K \in \T_h}\left\|e_{\boldsymbol{u}_{i,h}}^{n+1,l+1}\right\|_{L^{2}(K)}^2 \\
& \leq
\frac{C_{\max}}{C_{\min}}
\left(
\sum_{j=1}^{i}\sum_{K \in \T_h}\left\|e_{\boldsymbol{u}_{j,h}}^{n+1,l+1}\right\|_{L^2(K)}^2
+\sum_{j=i+1}^{M}\sum_{K \in \T_h}\left\|e_{\boldsymbol{u}_{i,h}}^{n+1,l}\right\|_{L^2(K)}^2
\right)\notag\\\nonumber
&\quad+\frac{\varrho^2 C_{\mu,\max}^2\,C}{(\beta^{\ast})^2 C_{\min}M}
\sum_{e\in \E_h^I}\frac{1}{h_e}\|[e_{c_{i,h}}^{n+1,l+1}]\|_{L^2(e)}^{2}.
\end{align}

\smallskip
Summing \eqref{eq:u_error_i} over $i=1,\dots, M$ yields the intermediate estimate 
\begin{align}
&\sum_{i=1}^{M}
\Biggl(
    \frac{M C_{\min}}{2}
    + \frac{1}{D_{\max}}
    - \frac{(M-i)C_{\max}}{C_{\min}}
\Biggr)
\sum_{K \in \T_h}
\left\|e_{\boldsymbol{u}_{i,h}}^{n+1,l+1}\right\|_{L^{2}(K)}^2
 \\\nonumber
&\leq\;
\sum_{i=1}^{M}
\frac{(i-1)C_{\max}}{C_{\min}}
\sum_{K \in \T_h}
\left\|e_{\boldsymbol{u}_{i,h}}^{n+1,l}\right\|_{L^2(K)}^2
 \\\nonumber
&\quad+ \frac{\varrho^2 C_{\mu,\max}^2\,C}{(\beta^{\ast})^2 C_{\min}M}
\sum_{i=1}^{M}\sum_{e\in \E_h^I}
\frac{1}{h_{e}}
\left\|[e_{c_{i,h}}^{n+1,l+1}]\right\|^{2}_{L^{2}(e)}.
\end{align}

To simplify the inequality into a form suitable for the coupled error recursion, we perform a uniform scaling of the coefficients.
More precisely, we bound the $i$-dependent coefficients on both sides by $i$-independent constants: on the left-hand side, we use a uniform lower bound, while on the right-hand side, we use suitable upper bounds.
This yields a simplified inequality with a single coercivity constant $\Upsilon>0$, namely,
\begin{align}\label{eq:u_error_sum}
\Upsilon\,
\sum_{i=1}^{M}\sum_{K \in \T_h}\left\|e_{\boldsymbol{u}_{i,h}}^{n+1,l+1}\right\|_{L^{2}(K)}^2 
\leq&
\frac{MC_{\max}}{C_{\min}}
\sum_{i=1}^{M}\sum_{K \in \T_h}\left\|e_{\boldsymbol{u}_{i,h}}^{n+1,l}\right\|_{L^2(K)}^2\\\nonumber
&
+
\frac{\varrho^2 C_{\mu,\max}^2\,C}{(\beta^{\ast})^2 C_{\min}M}
\sum_{i=1}^{M}\sum_{e\in \E_h^I}\frac{1}{h_e}\|[e_{c_{i,h}}^{n+1,l+1}]\|_{L^2(e)}^{2}.
\end{align}

\smallskip

The estimate \eqref{eq:u_error_sum} still contains the \emph{$(l+1)$-level} concentration jump term.
To close the recursion, we bound this term using \eqref{eq-109-3}.
Since the left-hand side of \eqref{eq-109-3} is coercive and nonnegative, we can isolate the jump seminorm and obtain
\begin{align}
\sum_{e\in \E_h^I}\frac{1}{h_e}\|[e_{c_{i,h}}^{n+1,l+1}]\|_{L^2(e)}^{2}
\le&
\frac{\mathcal{C}_{\phi}}{\left(\varsigma C_{\mu, \min}- \frac{\varrho^{2} C}{(\beta^{\ast})^{2} C_{\phi, \min}}\right)}
\sum_{i=1}^{M}\sum_{K \in \T_h}\| e_{c_{i,h}}^{n+1, l}\|^{2}_{L^2(K)}
\\\nonumber
&+\frac{C_{\phi, \min}}{4\left(\varsigma C_{\mu, \min}- \frac{\varrho^{2} C}{(\beta^{\ast})^{2} C_{\phi, \min}}\right)}
\sum_{K \in\mathcal{T}_h}\|e_{\boldsymbol{u}_{i,h}}^{n+1,l}\|^{2}_{L^{2}(e)}.  
\end{align}

Substituting this bound into \eqref{eq:u_error_sum} yields the closed form
\begin{align}\label{eq:u_error_sum_closed}
&\Upsilon\,
\sum_{i=1}^{M}\sum_{K \in \T_h}\left\|e_{\boldsymbol{u}_{i,h}}^{n+1,l+1}\right\|_{L^{2}(K)}^2\\
&\leq
\left(\frac{MC_{\max}}{C_{\min}}+\frac{\varrho^2 C_{\mu,\max}^2\,C}{(\beta^{\ast})^2 C_{\min}M}
\frac{C_{\phi,\min}}{4\left(\varsigma C_{\mu,\min}- \frac{\varrho^{2} C}{(\beta^{\ast})^{2} C_{\phi,\min}}\right)}\right)
\sum_{i=1}^{M}\sum_{K \in \T_h}\left\|e_{\boldsymbol{u}_{i,h}}^{n+1,l}\right\|_{L^2(K)}^2 \notag\\\nonumber
&+
\frac{\varrho^2 C_{\mu,\max}^2\,C}{(\beta^{\ast})^2 C_{\min}M}
\frac{M\mathcal{C}_{\phi}}{\left(\varsigma C_{\mu,\min}- \frac{\varrho^{2} C}{(\beta^{\ast})^{2} C_{\phi,\min}}\right)}
\sum_{i=1}^{M}\sum_{K\in \T_h}\|e_{c_{i,h}}^{n+1,l}\|_{L^2(K)}^{2}.
\end{align}
Dividing \eqref{eq:u_error_sum_closed} by $\Upsilon$ and using the definition of $X_1^l,X_2^l$ in \eqref{eq:X_def},
we obtain the second recursion component
\begin{equation}\label{eq:X2_recursion}
X_2^{l+1}\le a_{21}X_1^{l}+a_{22}X_2^{l},
\end{equation}
with $a_{21},a_{22}$ defined in \eqref{eq:a21}--\eqref{eq:a22}.

Combining \eqref{eq:X1_recursion} and \eqref{eq:X2_recursion} yields
\[
\mathbf X^{l+1}\le A\,\mathbf X^l,
\]
with $A$ defined in \eqref{eq:A_entries}. This completes the proof.
\end{proof}

\begin{lemma}\label{lem:contraction}
Let $A$ be the matrix defined in Lemma~\ref{lem:matrix_recursion}, and define the full iterative solution at time level $t^{n+1}$ by
\[
\mathbf W_h^{\,l}
:=
\bigl(
\boldsymbol u_{s,h}^{n+1,l},\, p_h^{n+1,l},\,
\{c_{i,h}^{n+1,l},\,\boldsymbol u_{i,h}^{n+1,l},\,\mu_{i,h}^{n+1,l}\}_{i=1}^M,\,
\phi_h^{n+1,l}
\bigr).
\]
Assume that the following conditions hold:
\begin{subequations}\label{eq:norm_conditions}
\begin{align}
a_{11}+a_{12}<1,\qquad &a_{21}+a_{22}<1,\label{eq:row_sum_cond}\\
a_{11}+a_{21}<1,\qquad &a_{12}+a_{22}<1.\label{eq:col_sum_cond}
\end{align}
\end{subequations}
Then we have $\|A\|_{\infty}<1$ and $\|A\|_{1}<1$. Hence $\rho(A)<1$, where $\rho(A)$ denotes the spectral radius of $A$.
Moreover, the nonlinear splitting map at time level $t^{n+1}$ admits a unique fixed point
\[
\mathbf W_h^\dagger
:=
\bigl(
\boldsymbol u_{s,h}^{n+1,\dagger},\, p_h^{n+1,\dagger},\,
\{c_{i,h}^{n+1,\dagger},\,\boldsymbol u_{i,h}^{n+1,\dagger},\,\mu_{i,h}^{n+1,\dagger}\}_{i=1}^M,\,
\phi_h^{n+1,\dagger}
\bigr),
\]
and the iterative solutions $\{\mathbf W_h^{\,l}\}_{l\ge 0}$ generated by the splitting scheme converge to $\mathbf W_h^\dagger$
from any initial guess.
\end{lemma}

\begin{proof}
Since $a_{ij}\ge 0$, the induced matrix $\infty$-norm and $1$-norm are given by the maximum row sum and the maximum column sum, respectively, i.e.,
\[
\|A\|_\infty=\max\{a_{11}+a_{12},\,a_{21}+a_{22}\},
\qquad
\|A\|_1=\max\{a_{11}+a_{21},\,a_{12}+a_{22}\}.
\]
Assumption~\eqref{eq:norm_conditions} implies $\|A\|_\infty<1$ and $\|A\|_1<1$.
For any consistent matrix norm $\|\cdot\|$, it holds that $\rho(A) \le \|A\|$. Hence, $\rho(A) < 1$.
Therefore, Lemma~\ref{lem:matrix_recursion} yields the componentwise inequality
\begin{equation}\label{eq:X_contract_recursion}
	\mathbf X^{\,l+1}\le A\,\mathbf X^{\,l}\le A^{\,l+1}\mathbf X^{\,0},\qquad l\ge 0.
\end{equation}
Since $\rho(A)<1$, we have $A^{\,l}\to 0$ as $l\to\infty$, and thus
\begin{equation}\label{eq:X_to_zero}
	\mathbf X^{\,l}\to \mathbf 0 \quad\text{as }l\to\infty .
\end{equation}

Consequently, the iterative mapping is a contraction. By the Banach fixed-point theorem,
there exists a unique fixed point $\bigl(c_{i,h}^{n+1,\dagger},\,\mathbf{u}_{i,h}^{n+1,\dagger}\bigr)$ for all $i=1,\ldots,M$ such that
\[
c_{i,h}^{n+1,l}\to c_{i,h}^{n+1,\dagger}\quad\text{in } \mathcal{Q}_h, \qquad l\to\infty,\qquad i=1,\ldots,M,
\]
and
\[
\mathbf{u}_{i,h}^{n+1,l}\to \mathbf{u}_{i,h}^{n+1,\dagger}\quad\text{in } \boldsymbol{\mathcal{U}}_h, \qquad l\to\infty,\qquad i=1,\ldots,M.
\]
Lemma~\ref{lem:matrix_recursion} and the contraction above yield $X_1^{l}\to 0$ and $X_2^{l}\to 0$ as $l\to\infty$.
The remaining error components are controlled by the concentration energy through the stability bounds proved in
\cite{Chenpre}. Concretely, there exists a constant $C>0$, independent of $l$, such that
\begin{align}\label{eq:aux_controls}
\sum_{K \in \mathcal{T}_h}\left(\sum_{i=1}^M \|e_{\mu_{i,h}}^{n+1,l+1}\|_{L^2(K)}^2
+ \|e_{p_h}^{n+1,l+1}\|_{L^2(K)}^2
+ \|e_{\phi_h}^{n+1,l+1}\|_{L^2(K)}^2\right) \\
+ \sum_{K \in \mathcal{T}_h} \|\varepsilon(e_{\boldsymbol{u}_{s,h}}^{n+1,l+1})\|^{2}_{L^2(K)} + \sum_{K \in \mathcal{T}_h} \|\nabla\cdot(e_{\boldsymbol{u}_{s,h}}^{n+1,l+1})\|^{2}_{L^2(K)}
\le C\,X_1^{l+1}, \nonumber
\end{align}
so these components converge as well. Therefore, the full discrete vector $\mathbf W_h^{\,l}$ converges to $\mathbf W_h^\dagger$ strongly in all components,
i.e., each component of $\mathbf W_h^{\,l}$ converges to the corresponding component of $\mathbf W_h^\dagger$
in the associated discrete space as $l\to\infty$.

\end{proof}

\begin{remark}\label{rem:large_parameter}
The contraction regime required in Lemma~\ref{lem:contraction} is compatible with the intended physical setting.
In weakly compressible rocks (large Biot modulus $N$), the porosity--pressure coupling is moderate, which reduces
the coupling constant $\mathcal{C}_{\phi}$ in \eqref{eq:Cphi_def}. 
Moreover, in low-permeability formations the mobility and multicomponent diffusion are small,
corresponding to small MSD diffusion coefficients and thus weaker transport coupling. 
These features support the assumptions under which the entries of $A$ become small enough to ensure $\rho(A)<1$.
\end{remark}

\begin{theorem}\label{theo-final-bound}
Assume that $0<\beta^\ast c_h^n<1$ and that the boundary conditions \eqref{eq-BD}--\eqref{eq-BD-1} hold.
Let the time step size $\tau_n$ be chosen according to the adaptive strategy \eqref{eq-tau}.
Suppose further that the nonlinear splitting iteration at the time level $t^{n+1}$ is convergent, i.e.,
\[
(
c_{i,h}^{n+1,l},\mu_{i,h}^{n+1,l},\phi_h^{n+1,l})
\longrightarrow
(c_{i,h}^{n+1,\dagger},\mu_{i,h}^{n+1,\dagger},\phi_h^{n+1,\dagger})
\quad \text{as } l\to\infty.
\]
Then the converged discrete molar densities satisfy the physical bounds
\begin{align*}
0 < c_h^{n+1,\dagger} < \frac{1}{\beta^\ast},
\quad
0 < c_{i,h}^{n+1,\dagger} \le c_h^{n+1},
\quad i=1,\dots,M.\\
0<\phi_h^{n}-C_\epsilon \le \phi_h^{n+1,\dagger} \le \phi_h^{n}+C_\epsilon <1,
\end{align*}
\end{theorem}

\begin{proof}
By Theorem~\ref{theo-1}, for every nonlinear iterate $l\ge0$ the bounds 
\eqref{eq-bound-c}--\eqref{eq-bound-ci}, and \eqref{eq-phi-bounded} hold for $(c_h^{n+1,l},c_{i,h}^{n+1,l})$, and $\phi^{n+1,l}_{h}$.
Passing to the limit $l\to\infty$ and using the convergence of the nonlinear splitting iteration
yields the desired bounds for the limiting solution $(c_h^{n+1,\dagger},c_{i,h}^{n+1,\dagger})$ and $\phi_{h}^{n+1, \dagger}$.
\end{proof}

\begin{theorem}[The fixed point solves the discrete weak problem]\label{thm:fixed_point_weak_form}
Let $\mathbf W_h^\dagger$ be the fixed point provided by Lemma~\ref{lem:contraction}.
Then $\mathbf W_h^\dagger$ satisfies the coupled discrete weak formulation at time level $t^{n+1}$,
i.e., it is a solution of the fully implicit nonlinear scheme \eqref{eq-full-discrete}.
\end{theorem}

\begin{proof}
	By Lemma~\ref{lem:contraction}, the splitting sequence converges strongly in the discrete spaces as
	$l\to\infty$. Since all discrete spaces are finite dimensional, this convergence holds in any discrete norm.
	We now pass to the limit $l\to\infty$ in \eqref{eq-4-1}--\eqref{eq-4-6}.
	All terms are linear with respect to the unknowns at level $l+1$ except the mixed-level products
	\(
	\phi_h^{n+1,l}c_{i,h}^{n+1,l+1}.
	\)
	Thus it suffices to justify the convergence of these terms; the remaining terms follow directly from
	strong convergence and continuity of the bilinear forms.
	
	Let $q_h\in\mathcal Q_h$ be arbitrary. For the total molar density we write
	\[
	\phi_h^{n+1,l}c_{i,h}^{n+1,l+1}-\phi_h^{n+1,\dagger}c_{i,h}^{n+1,\dagger}
	=
	(\phi_h^{n+1,l}-\phi_h^{n+1,\dagger})c_{i,h}^{n+1,l+1}
	+\phi_h^{n+1,\dagger}(c_{i,h}^{n+1,l+1}-c_{i,h}^{n+1,\dagger}).
	\]
	Hence, by H\"older's inequality,
	\begin{align*}
		&\bigl|(\phi_h^{n+1,l}c_{i,h}^{n+1,l+1}-\phi_h^{n+1,\dagger}c_{i,h}^{n+1,\dagger},\,q_h)\bigr|\\
		&\quad\le
		\sum_{K \in \T_h}\|\phi_h^{n+1,l}-\phi_h^{n+1,\dagger}\|_{L^2(K)}\,
		\|c_{i,h}^{n+1,l+1}\|_{L^\infty(\Omega)}\,\sum_{K \in \T_h}\|q_h\|_{L^2(K)}\\
		&+\|\phi_h^{n+1,\dagger}\|_{L^\infty(\Omega)}\,
		\sum_{K \in \T_h}\|c_{i,h}^{n+1,l+1}-c_{i,h}^{n+1,\dagger}\|_{L^2(K)}\,\sum_{K \in \T_h}\|q_h\|_{L^2(K)}.
	\end{align*}
	By the boundedness assumptions and the discrete maximum principle established earlier,
	$\|c_{i,h}^{n+1,l+1}\|_{L^\infty(\Omega)}$ is uniformly bounded in $l$, and
	$\|\phi_h^{n+1,\dagger}\|_{L^\infty(\Omega)}<\infty$.
	Together with the strong convergence
	\(
	\phi_h^{n+1,l}\to\phi_h^{n+1,\dagger}
	\)
	and
	\(
	c_{i,h}^{n+1,l+1}\to c_{i,h}^{n+1,\dagger}
	\)
	in $L^2(\Omega)$, we obtain
	\begin{equation}\label{eq:phi_c_limit}
		(\phi_h^{n+1,l}c_{i,h}^{n+1,l+1},\,q_h)\to(\phi_h^{n+1,\dagger}c_{i,h}^{n+1,\dagger},\,q_h),
		\qquad \forall\,q_h\in\mathcal Q_h.
	\end{equation}
	Recall that the discrete time difference operator $D_{\tau_n}$ is linear. Therefore, by \eqref{eq:phi_c_limit},
	\[
	(D_{\tau_n}(\phi_h^{n+1,l}c_{i,h}^{n+1,l+1}),\,q_h)
	\to
	(D_{\tau_n}(\phi_h^{n+1,\dagger}c_{i,h}^{n+1,\dagger}),\,q_h),
	\qquad \forall\,q_h\in\mathcal Q_h.
	\]
All other terms in \eqref{eq-4-1}--\eqref{eq-4-2} are linear in the level-($l+1$) unknowns, and their coefficients are
uniformly bounded and converge strongly; hence, they pass to the limit by continuity of the associated forms.
Therefore, letting $l\to\infty$ in \eqref{eq-4-1}--\eqref{eq-4-2} gives the transport equations at the fixed point.
Since \eqref{eq-4-3}--\eqref{eq-4-6} are linear with respect to the level-($l+1$) unknowns, we can pass to the limit
$l\to\infty$ term by term using strong convergence and continuity. The limiting identities coincide with the fully
coupled discrete weak formulation at time level $t^{n+1}$. Consequently, $\mathbf W_h^\dagger$ is a solution of the nonlinear system \eqref{eq-full-discrete}.

\end{proof}

\section{Numerical Examples}\label{sec-num}

In this section, we present a series of numerical experiments to validate the proposed numerical method for multicomponent gas flow in poroelastic media. The tests are designed to verify key properties of the scheme, including: mass conservation for each component; the bounds-preserving property for all molar densities ($0 < \beta^{\ast} c_{h} < 1$); the discrete energy dissipation law; and the robustness and efficiency of the adaptive time-stepping strategy. We consider a ternary gas mixture consisting of Methane (CH$_4$), Carbon Dioxide (CO$_2$), and Ethane (C$_2$H$_6$) flowing through a porous reservoir. The physical parameters and Peng-Robinson equation of state (PR-EoS) parameters for each component are summarized in Table \ref{table1}. The computational domain is set to $\Omega = [0, L]^d$ with $L = 100$ m. A quasi-uniform triangular mesh with approximately $20,000$ elements is used for 2D simulations. The parameter $\delta_{i, 1}= \delta_{i, 2}= \delta_{1} = \delta_{2} = 0.3$, and the time step is dynamically adjusted based on the criterion given in Eq. (\ref{eq-tau}).
\begin{table}[htbp]
	\centering
	\caption {Physical properties of CH$_4$, C$_2$H$_6$ \text{and} CO$_2$}\label{table1}
	\begin{tabular}{cccc}
		\hline \text { Parameter } & CH$_4$ &C$_2$H$_6$ & CO$_2$ \\
		\hline $P_c(\mathrm{bar})$ & 45.99&48.72 & 73.75 \\
		$T_c(\mathrm{~K})$ & 190.56 &305.32 & 304.14 \\
		\text { Acentric factor } & 0.011 &0.099 & 0.239 \\
		\text { Molar weight }($\mathrm{g}$ / \text { mole }) & 16.04&30.07 & 44.01 \\
		\hline
	\end{tabular}
\end{table}
\subsection{Example 1}
This example tests a closed binary gas system (CO$_2$ and CH$_4$) to verify the proposed numerical scheme's capability to preserve fundamental physical properties, including individual species mass conservation, energy dissipation, and the bounds-preserving nature of all molar densities. A spatially heterogeneous permeability field, as illustrated on the left-hand side of Figure \ref{fig1-initial}, is generated using the Perlin noise method to mimic realistic geological conditions.

The initial molar density distributions are defined to create a strong contrast between two distinct square regions. Let the central square domain be $\Omega_1 = [30~\text{m},~70~\text{m}] \times [30~\text{m},~70~\text{m}]$. The initial conditions are set as follows
\begin{align*}
	c_{\text{CO}_2}^0(\boldsymbol{x}) &= \begin{cases}
		300 \text{ mol/m}^3, & \boldsymbol{x} \in \Omega_1, \\
		10 \text{ mol/m}^3, & \boldsymbol{x} \notin \Omega_1.
	\end{cases} \\
	c_{\text{CH}_4}^0(\boldsymbol{x}) &= \begin{cases}
		10 \text{ mol/m}^3, & \boldsymbol{x} \in \Omega_1, \\
		300 \text{ mol/m}^3, & \boldsymbol{x} \notin \Omega_1.
	\end{cases}
\end{align*}
This configuration creates a high molar density of CO$_2$ in the center against a background rich in CH$_4$, and vice versa, establishing sharp initial gradients that will drive the subsequent diffusion and mixing processes. The mechanical parameters are chosen as $N = 10^{11}$ Pa, $\gamma = 10^{15}$ Pa, $\eta = 10^{15}$ Pa$\cdot$s.

Figure \ref{fig1-initial} shows the initial distributions of the molar densities for both CO$_2$ and CH$_4$. Figure \ref{fig1-energy} demonstrates the performance of the proposed scheme: the total free energy decays monotonically, confirming the energy dissipation property; the total mass of each component remains constant over time, verifying discrete mass conservation for each individual species; and the molar densities for both CO$_2$ and CH$_4$ remain strictly within their physical bounds throughout the simulation.  The graph on the right-hand side shows the adaptive time step size, which increases gradually as the sharp initial gradients smooth out and the system evolves toward a homogeneous equilibrium, eventually reaching the prescribed maximum step size $\tau_{\text{max}} = 1000$ s.

In Figures \ref{fig1-co2} and \ref{fig1-ch4}, we illustrate the spatial evolution of the molar densities of CO$_2$ and CH$_4$ at different time steps: $n = 100, 200, 300, 400$. The sequence clearly shows the mutual diffusion process: CO$_2$ diffuses outward from the high-density central region, while CH$_4$ diffuses inward from the surrounding. Figures \ref{fig1-co2-che} and \ref{fig1-ch4-che} present the corresponding chemical potential distributions for both components at different time steps: $n = 100, 200, 300, 400$. Initially, significant chemical potential gradients exist at the interface between $\Omega_1$ and the outer domain, which serve as the primary driving force for the diffusion process. As the system evolves, these gradients gradually diminish until the chemical potentials become uniform throughout the domain at equilibrium, indicating no further net diffusion to occur.

Figure \ref{fig1-pres} shows the evolution of the total pressure field. The initial pressure is non-uniform due to the compositional heterogeneity. The pressure field evolves dynamically as the components interdiffuse, and eventually reaches a homogeneous state at equilibrium. Figure \ref{fig1-poro} presents the evolution of porosity. Porosity adjusts dynamically with the pressure field (Figure \ref{fig1-pres}), increasing where pressure rises and decreasing where pressure drops, following the transient pressure gradients. 

\begin{figure}[htbp]
	\centering
	\includegraphics[width=5.0cm, height=4.0cm,trim=0.5cm 0cm 0cm 0cm,clip]{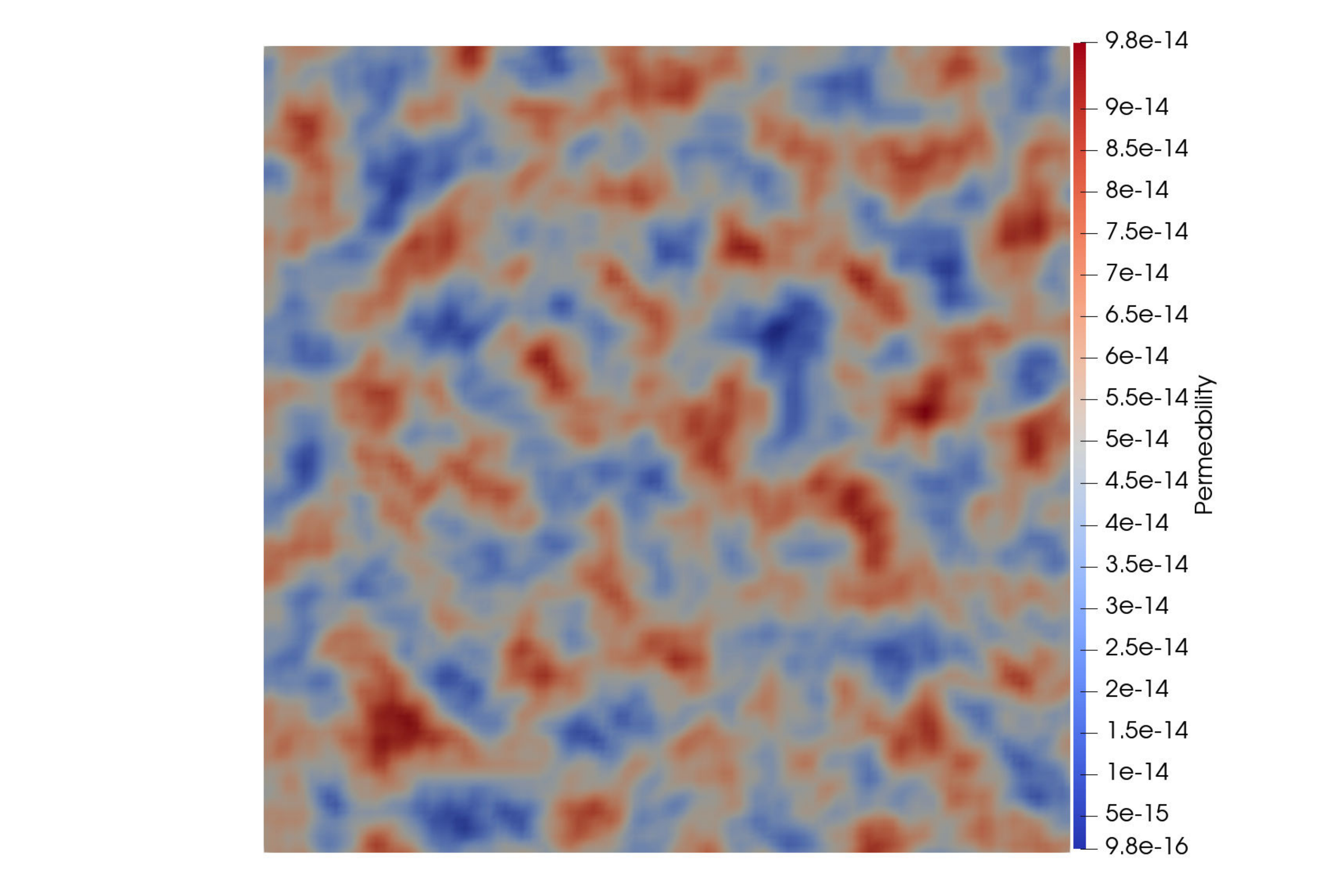}
	\includegraphics[width=5.0cm, height=4.0cm,trim=0.5cm 0cm 0cm 0cm,clip]{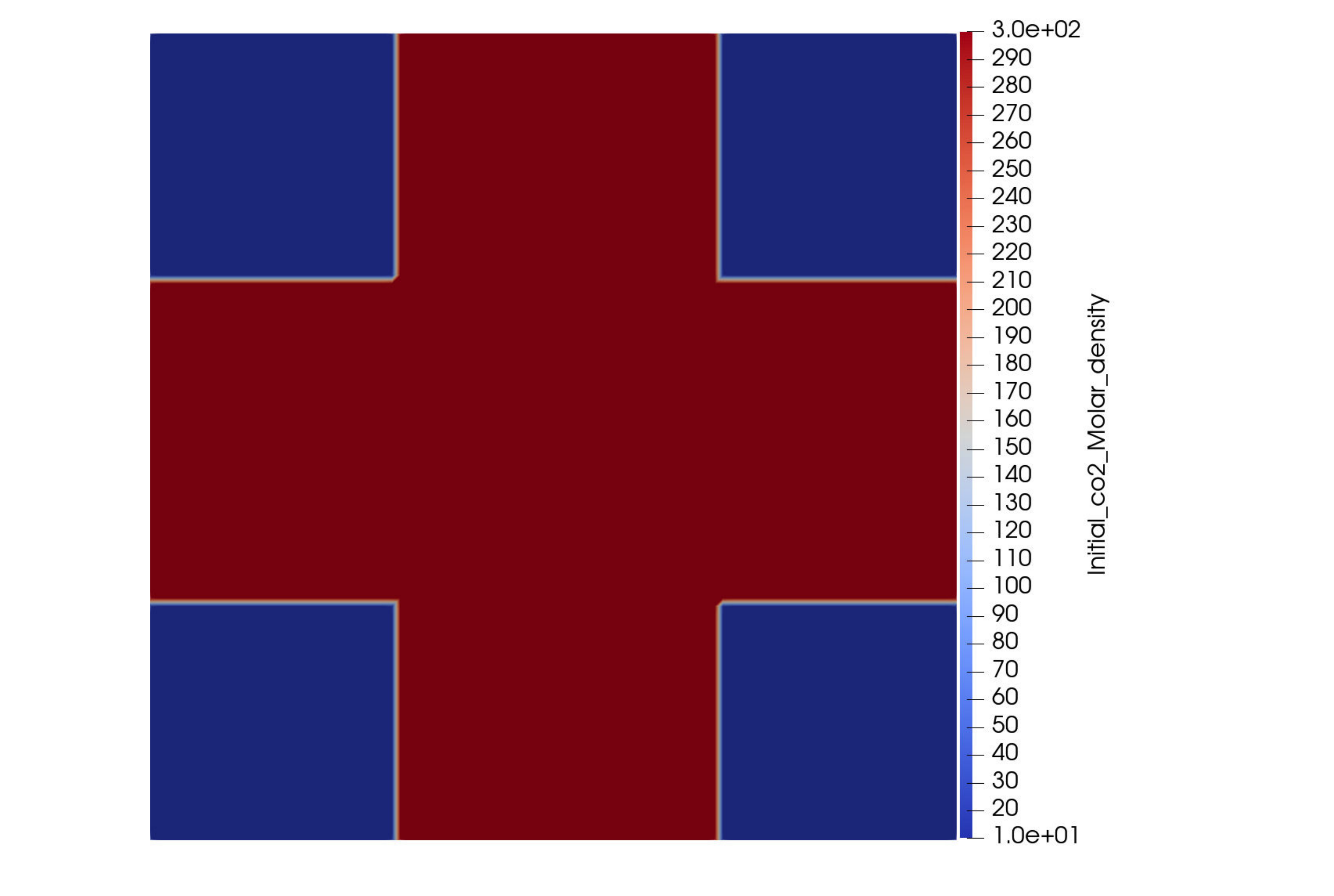}
	\includegraphics[width=5.0cm, height=4.0cm,trim=0.5cm 0cm 0cm 0cm,clip]{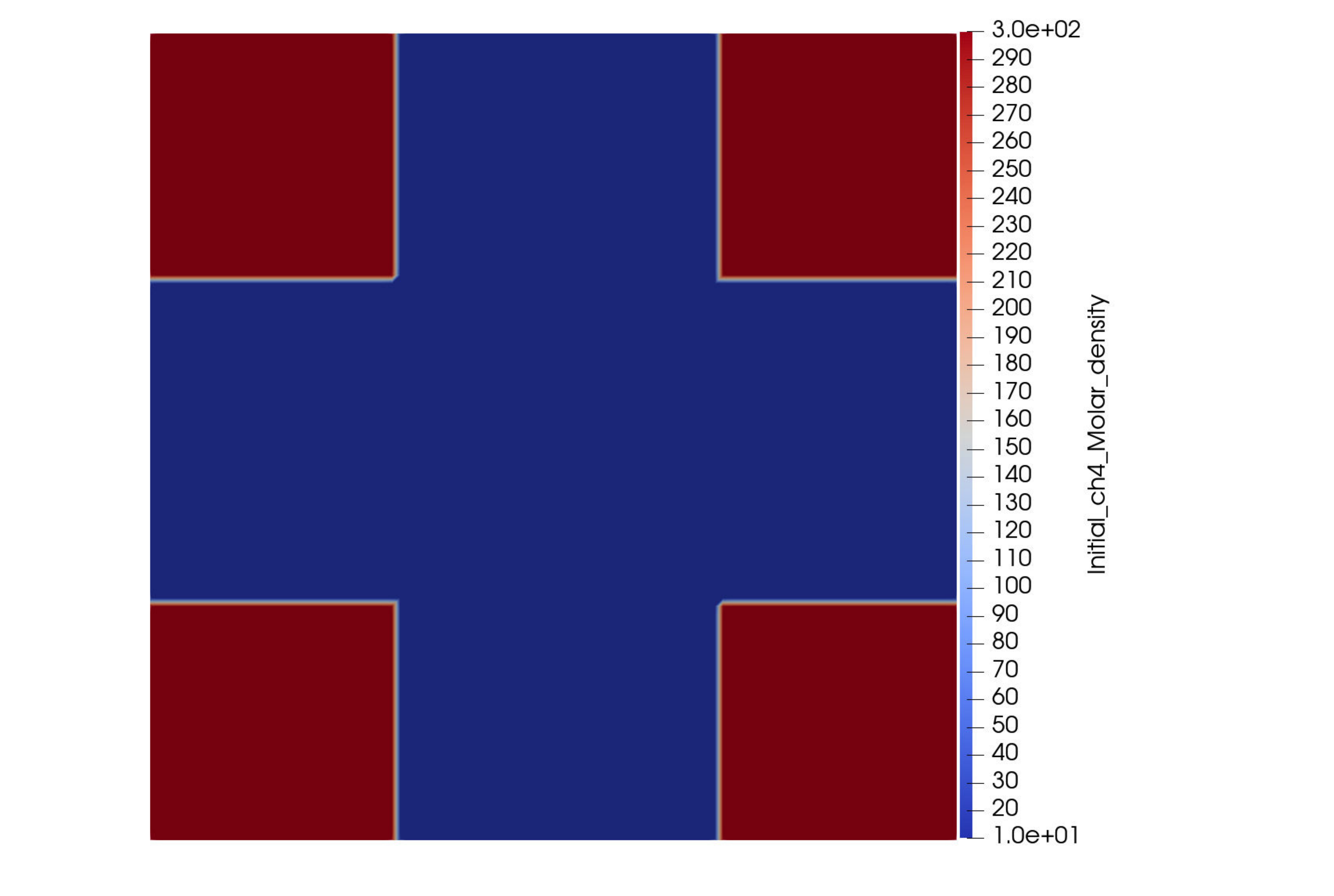}
	\caption{Example 1: Left: Initial distributions of permeability. Middle: Initial distributions of molar density of  CO$_2$. Right: Initial distributions of molar density of CH$_4$.}\label{fig1-initial}
\end{figure}

\begin{figure}[htbp]
	\centering
	\includegraphics[width=5.0cm, height=4.5cm]{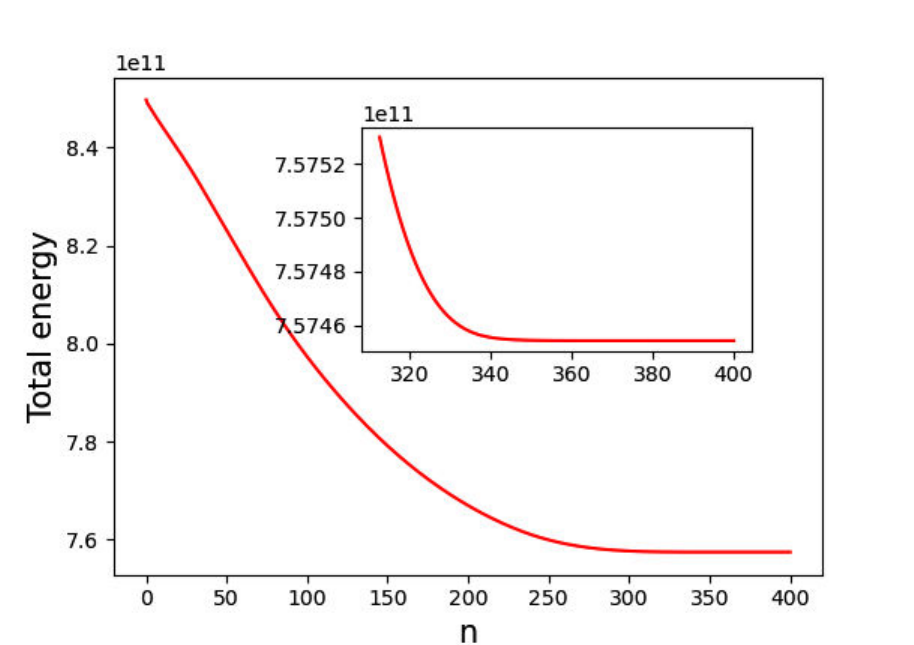}
	\includegraphics[width=5.0cm, height=4.5cm]{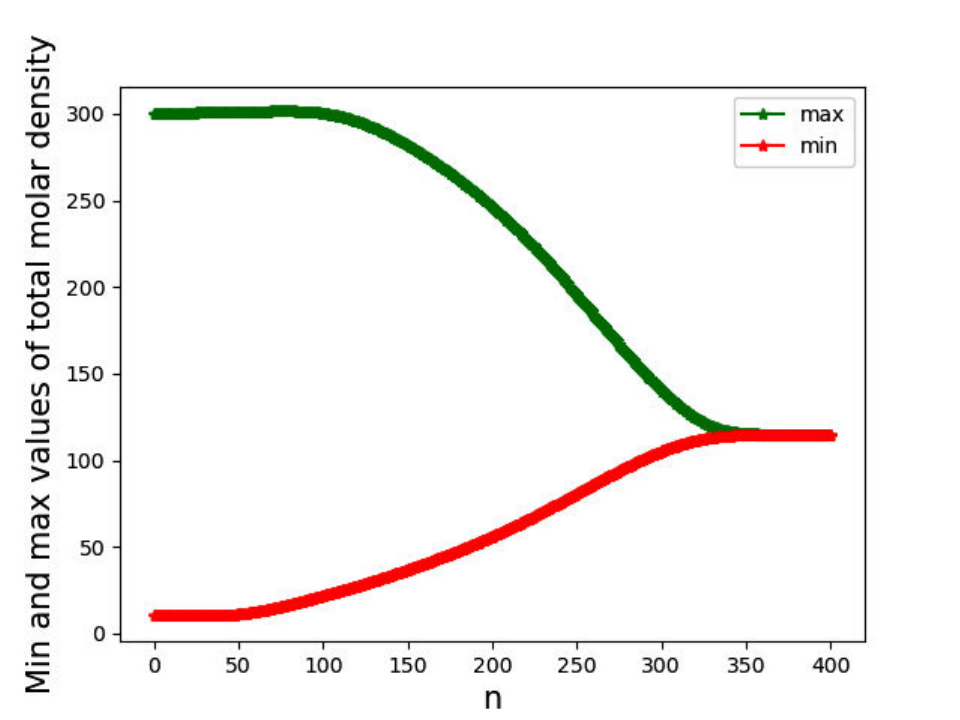}
	\includegraphics[width=5.0cm, height=4.5cm]{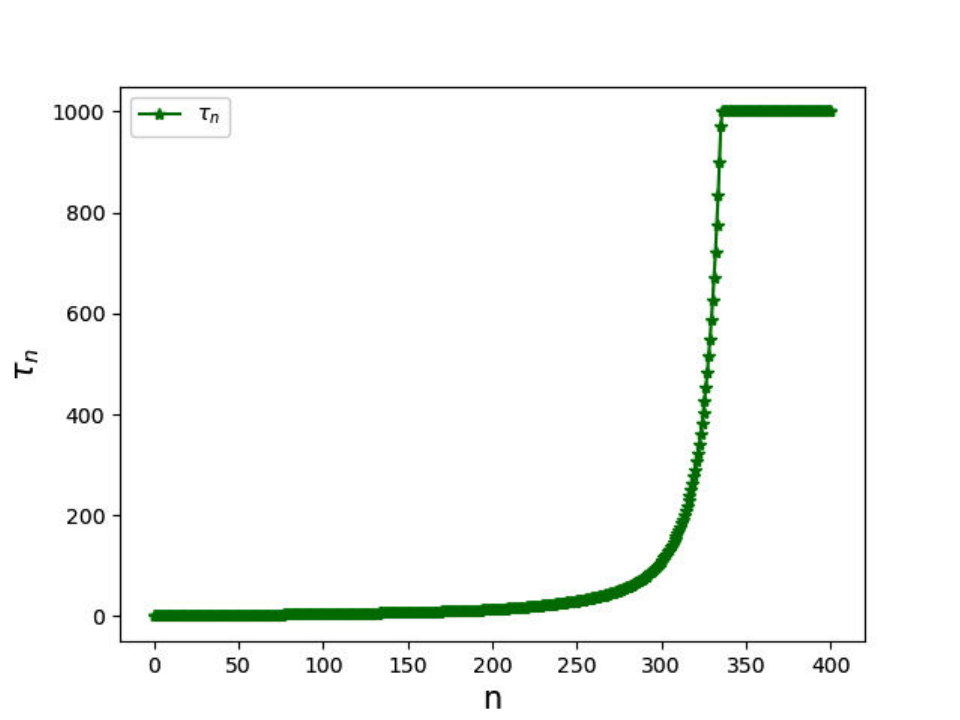}
	\caption{Example 1:  Left: Distributions of energy at different time steps. Middle: Minimum and maximum values of molar density. Right: Adaptive values of the time step size.}\label{fig1-energy}
\end{figure}
\begin{figure}[htbp]
	\centering
	\includegraphics[width=5.0cm, height=4.5cm]{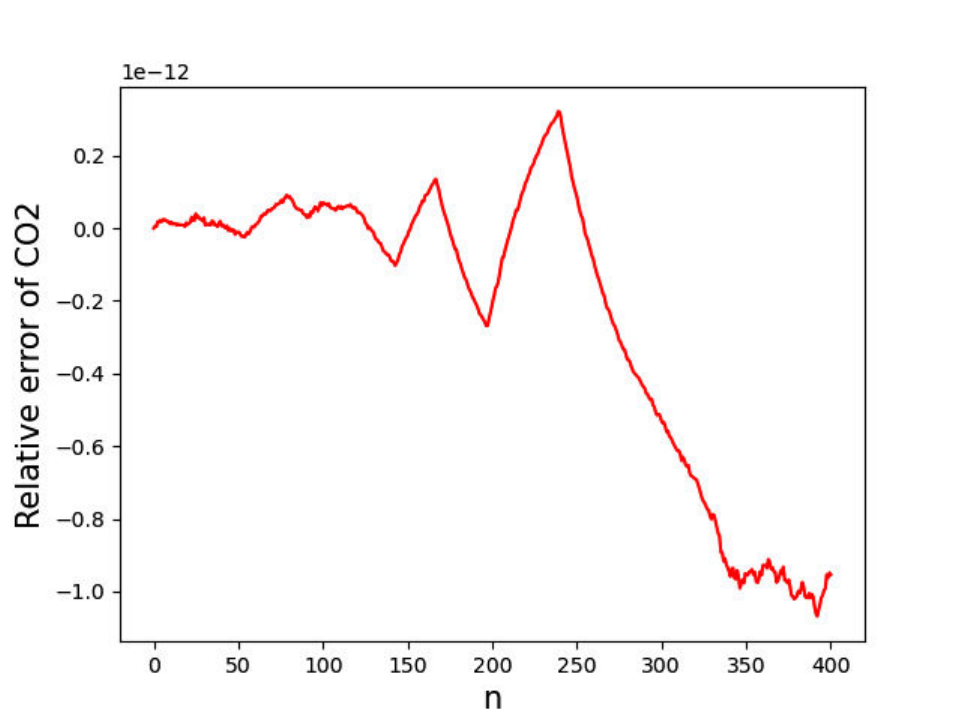}
	\includegraphics[width=5.0cm, height=4.5cm]{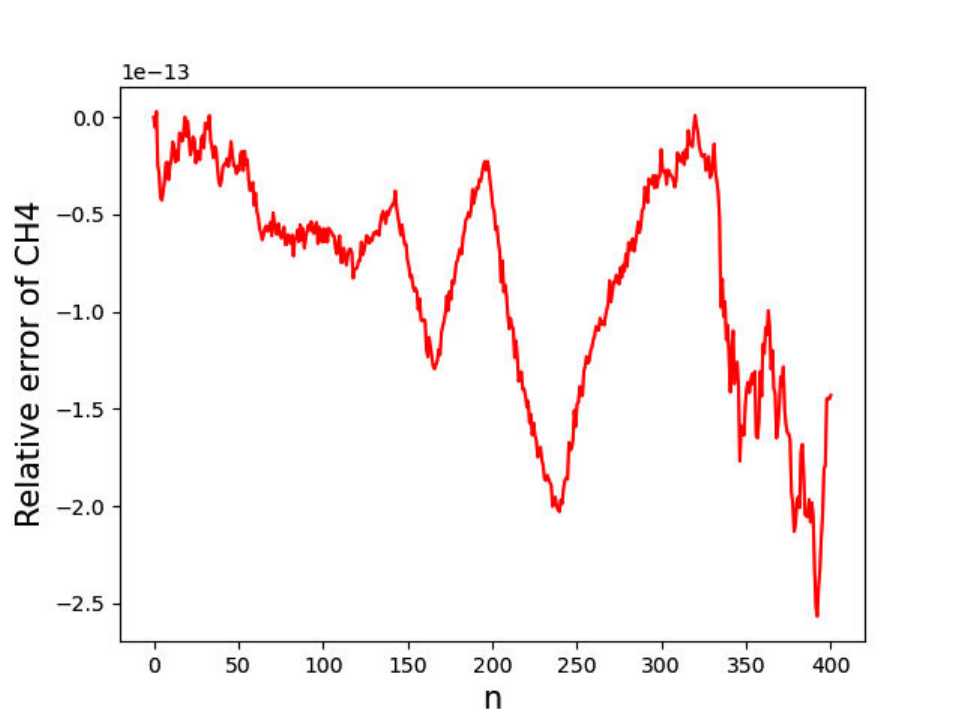}
	\caption{Example 1:  Left: Mass conservation of CO$_2$  at different time steps. Right: Mass conservation of CH$_4$ at different time steps. }\label{fig1-error}
\end{figure}
\begin{figure}[htbp]
	\centering
	\includegraphics[width=5.5cm, height=4cm,trim=0.5cm 0cm 0cm 0cm,clip]{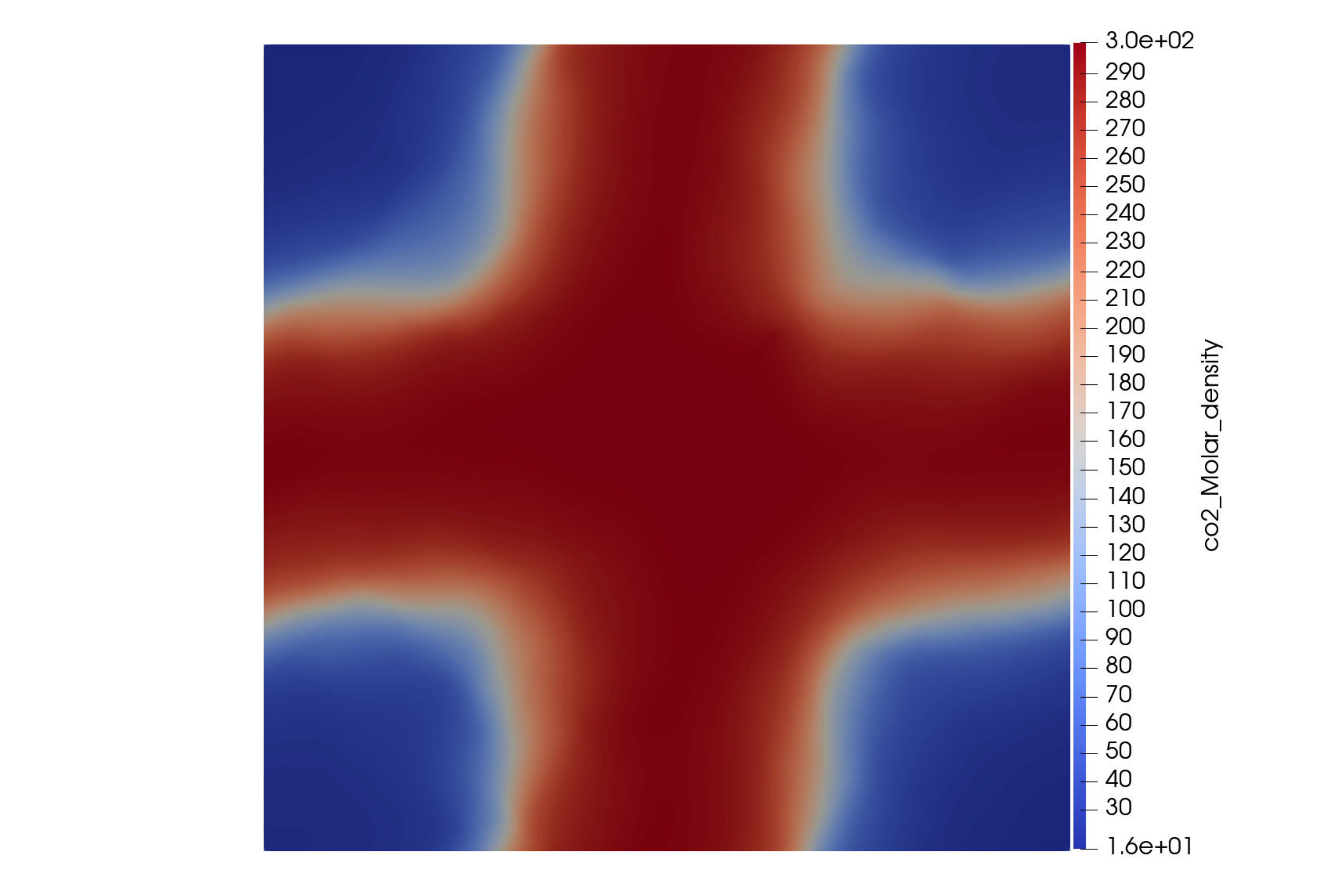}
	\includegraphics[width=5.5cm, height=4cm,trim=0.5cm 0cm 0cm 0cm,clip]{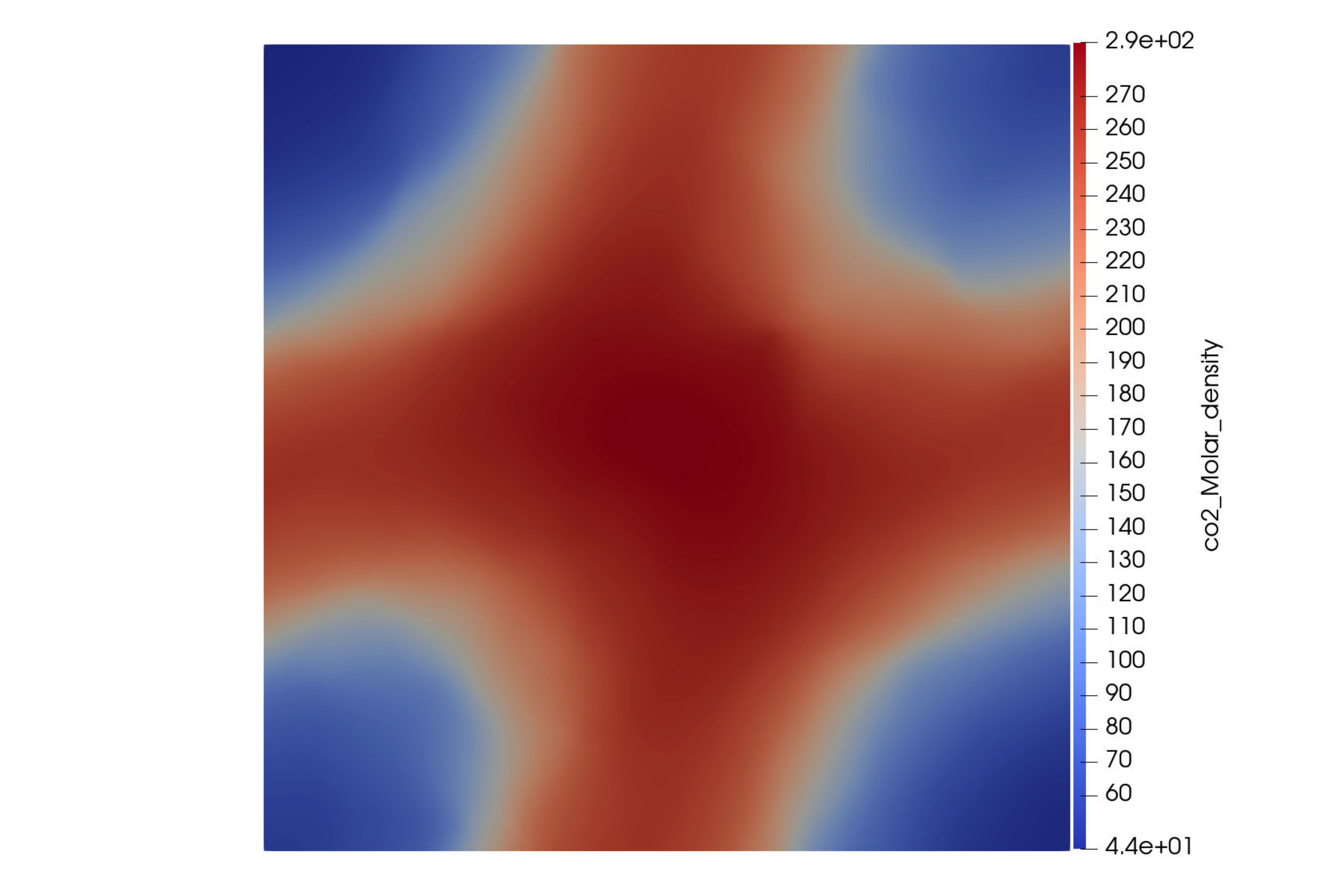}
	
	\includegraphics[width=5.5cm, height=4cm,trim=0.5cm 0cm 0cm 0cm,clip]{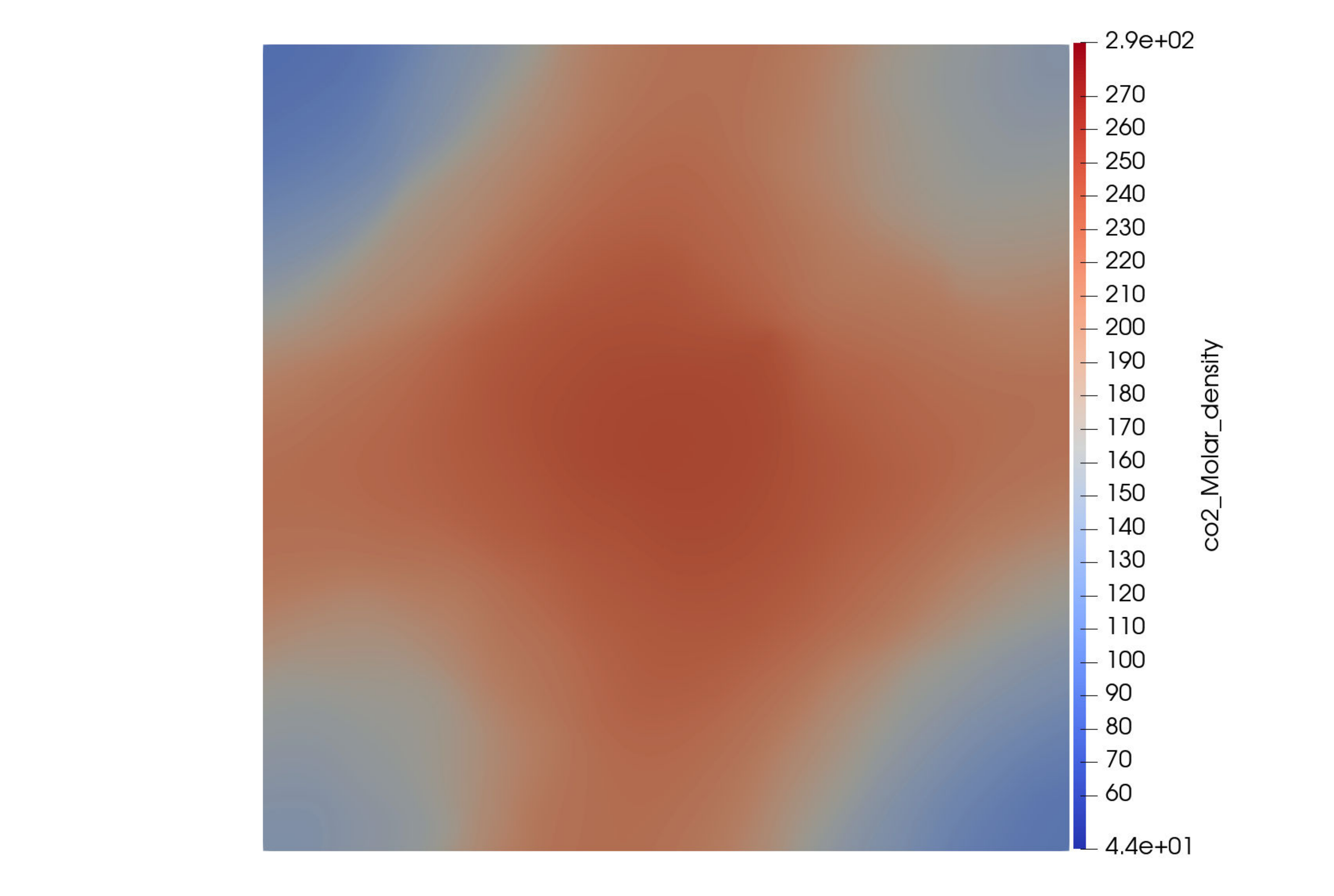}
	\includegraphics[width=5.5cm, height=4cm,trim=0.5cm 0cm 0cm 0cm,clip]{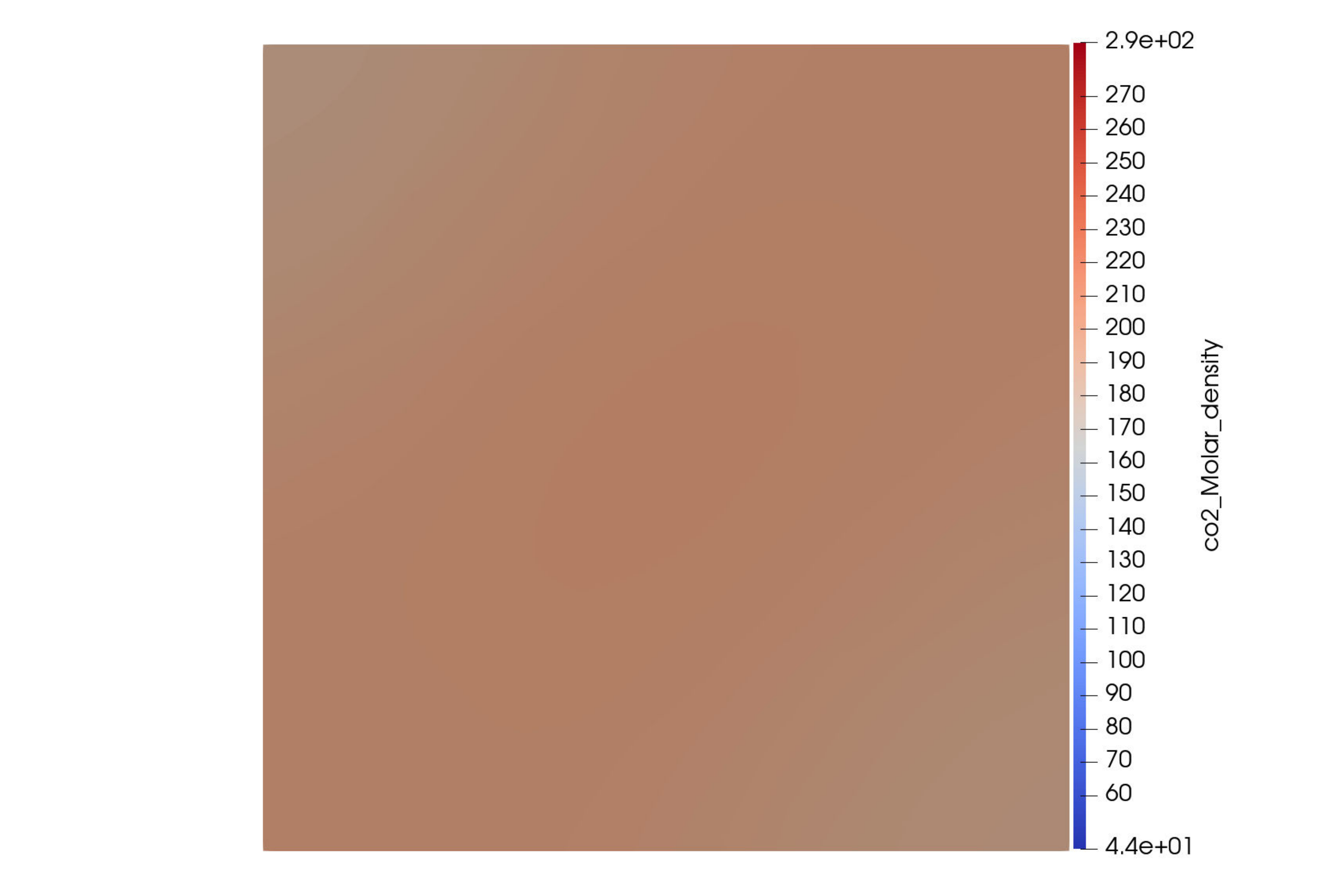}
	\caption{Distributions of molar density of CO$_2$ at different times in Example 1. Top-left: $n = 100$. Top-right: $n = 200$. Bottom-left: $n = 300$. Bottom-right: $n = 400$.}\label{fig1-co2}
\end{figure}

\begin{figure}[htbp]
	\centering
	\includegraphics[width=5.5cm, height=4cm,trim=0.5cm 0cm 0cm 0cm,clip]{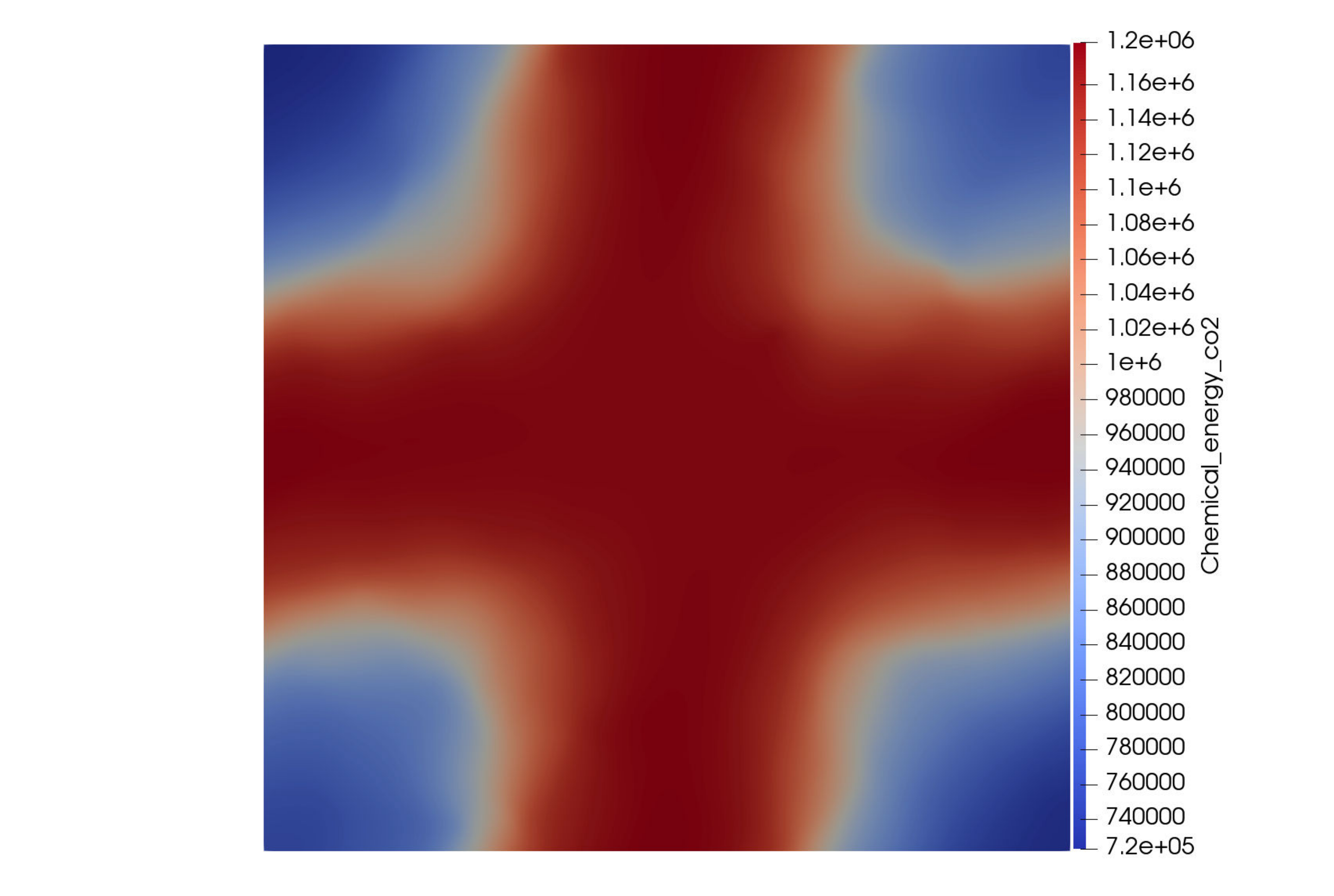}
	\includegraphics[width=5.5cm, height=4cm,trim=0.5cm 0cm 0cm 0cm,clip]{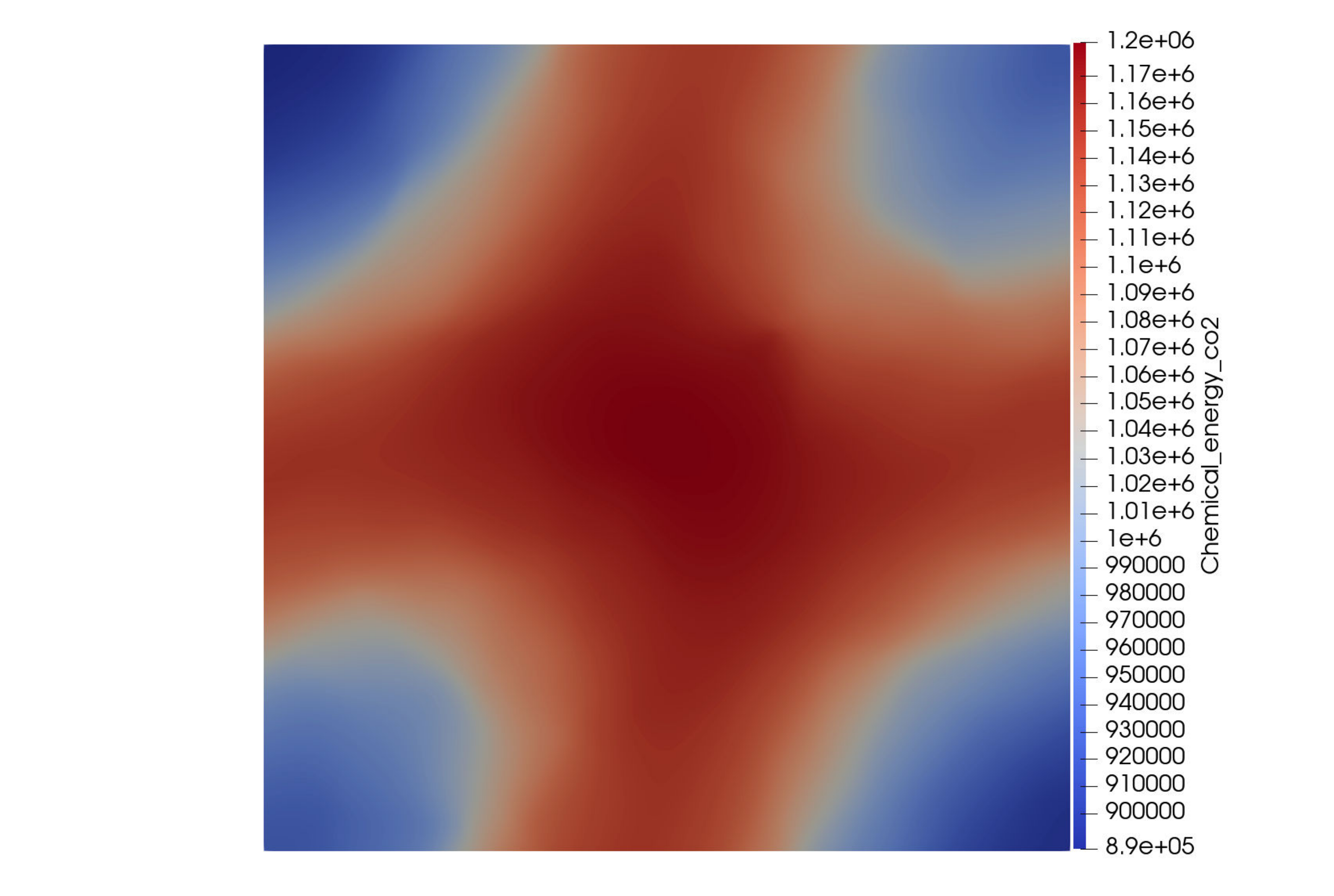}
	
	\includegraphics[width=5.5cm, height=4cm,trim=0.5cm 0cm 0cm 0cm,clip]{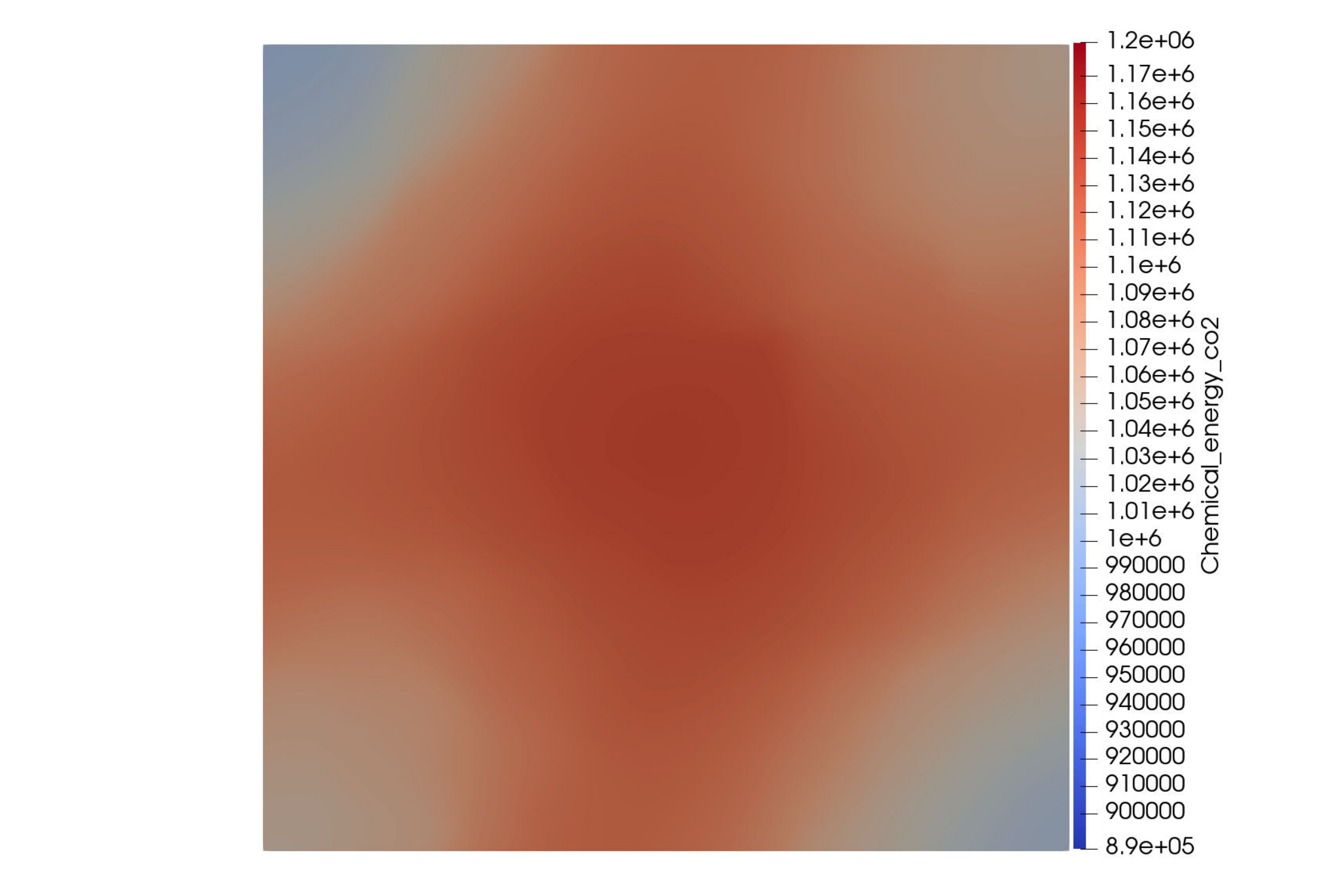}
	\includegraphics[width=5.5cm, height=4cm,trim=0.5cm 0cm 0cm 0cm,clip]{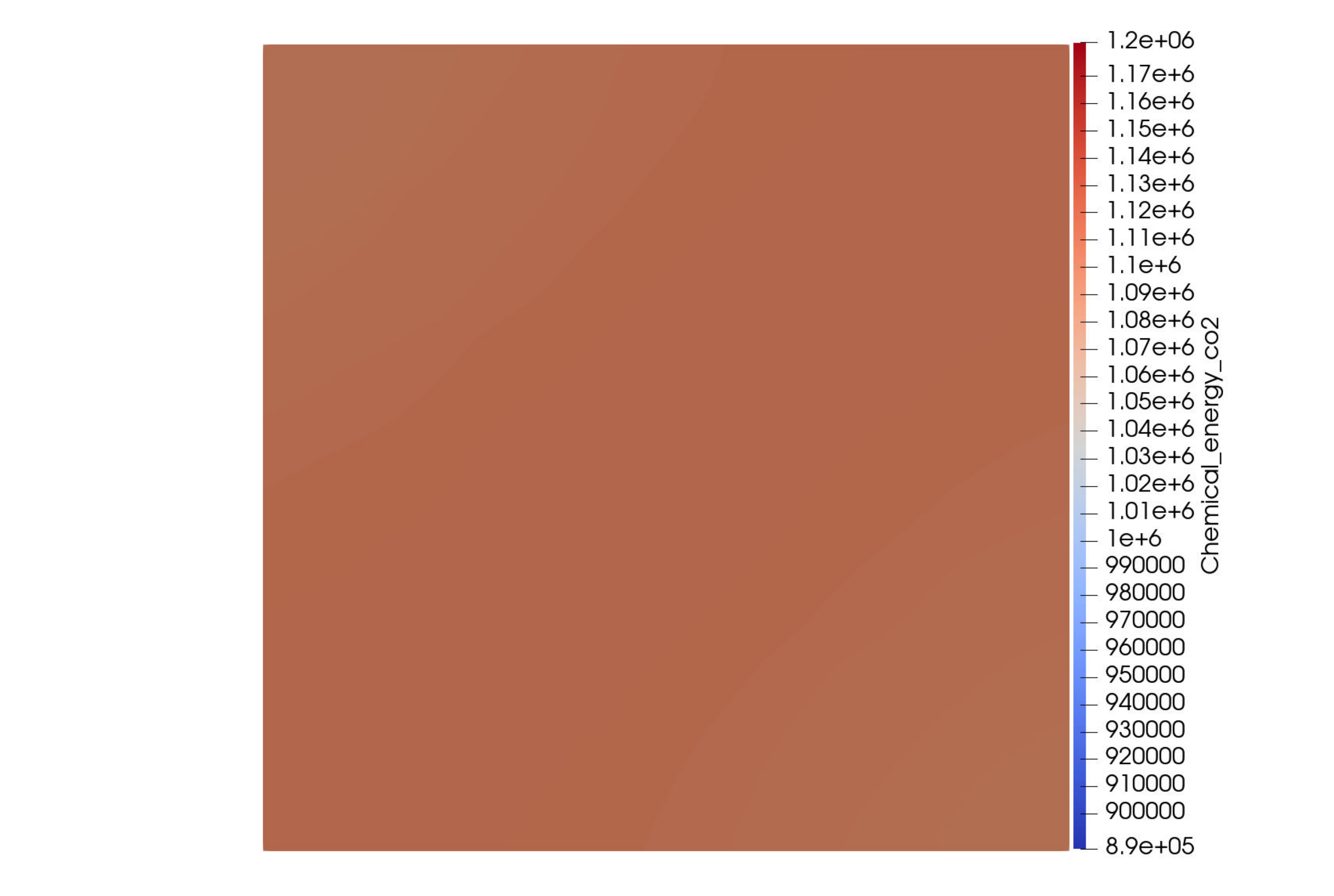}
	\caption{Distributions of chemical potential of CO$_2$ at different times in Example 1. Top-left: $n = 100$. Top-right: $n = 200$. Bottom-left: $n = 300$. Bottom-right: $n = 400$.}\label{fig1-co2-che}
\end{figure}

\begin{figure}[htbp]
	\centering
	\includegraphics[width=5.5cm, height=4cm,trim=0.5cm 0cm 0cm 0cm,clip]{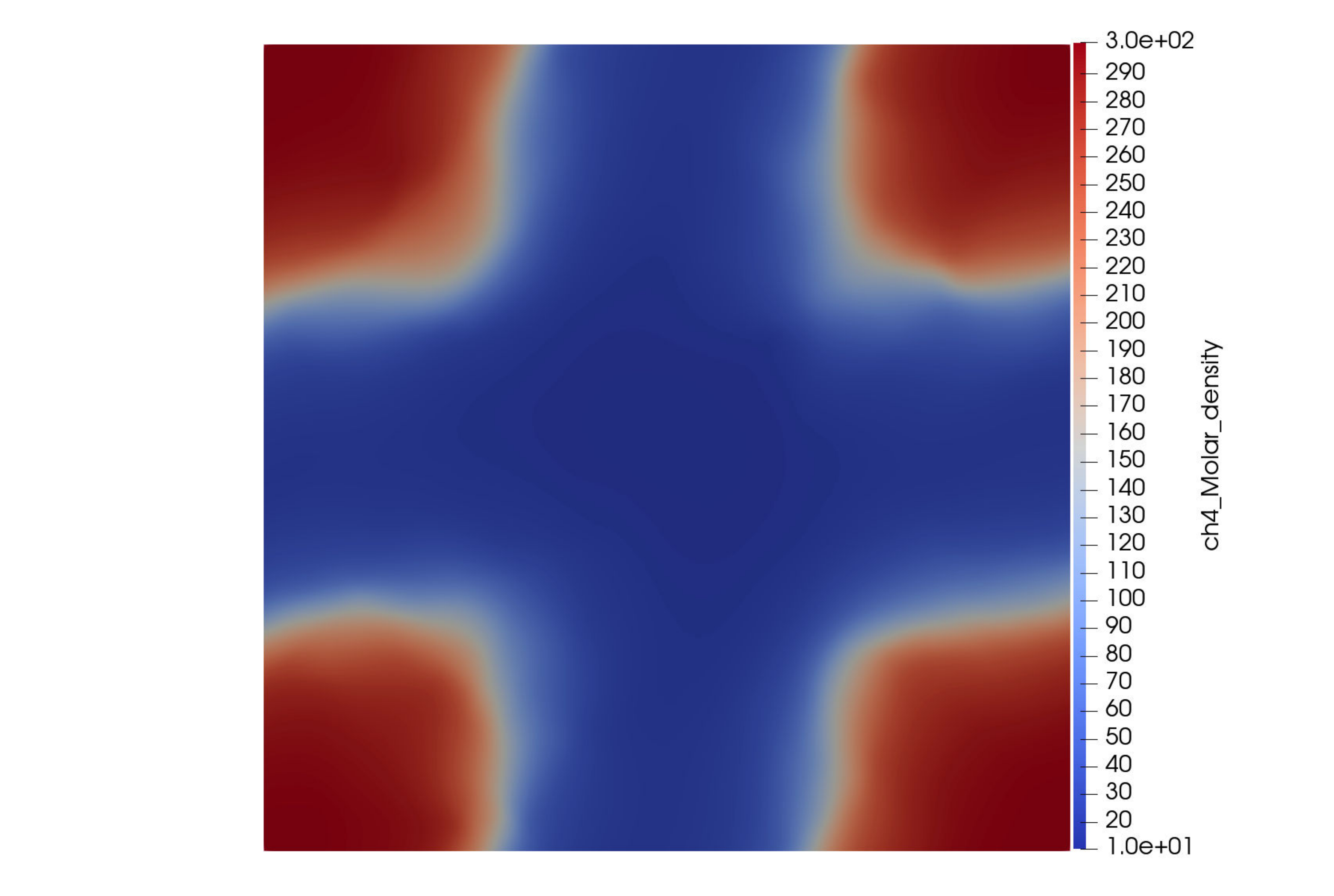}
	\includegraphics[width=5.5cm, height=4cm,trim=0.5cm 0cm 0cm 0cm,clip]{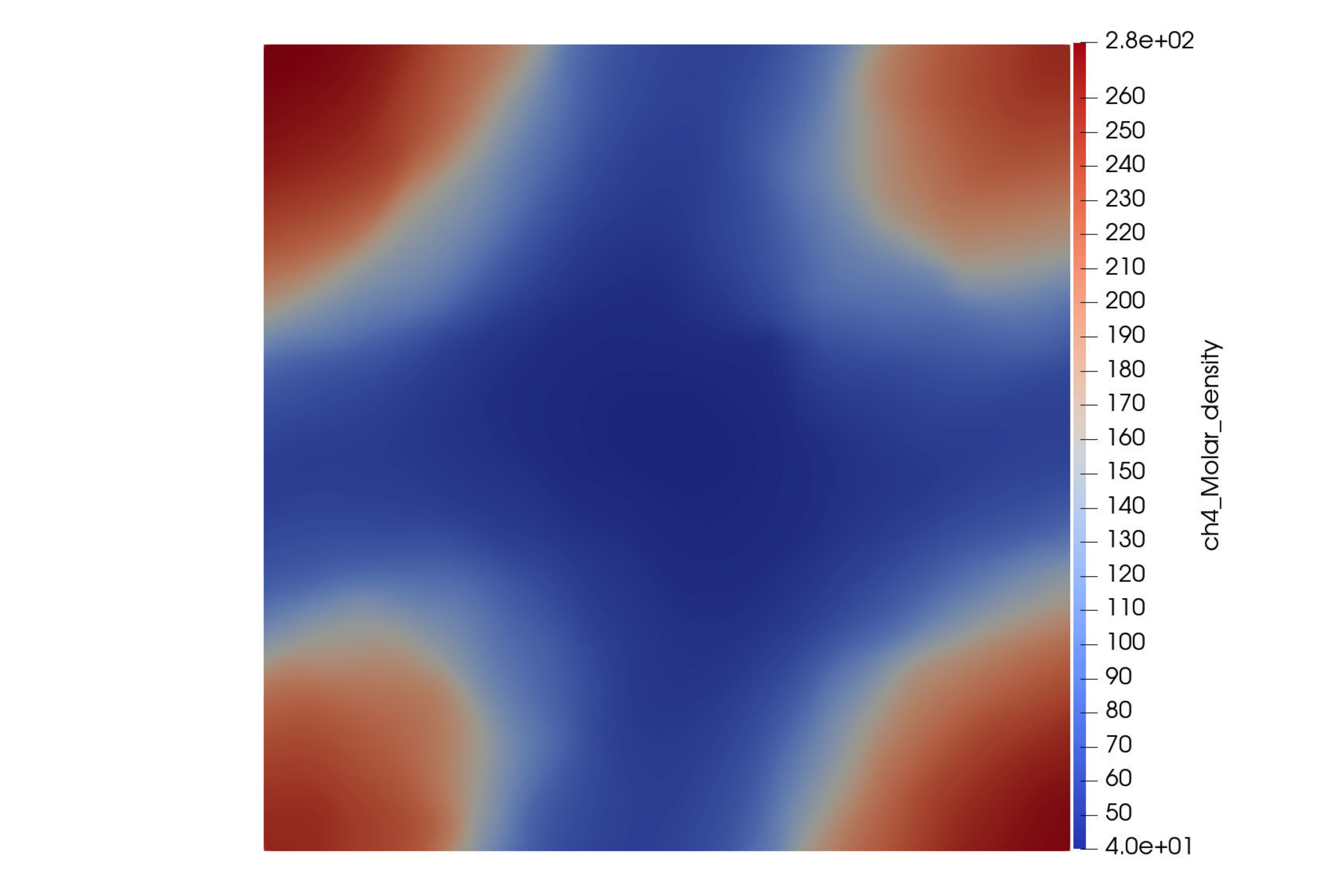}
	
	\includegraphics[width=5.5cm, height=4cm,trim=0.5cm 0cm 0cm 0cm,clip]{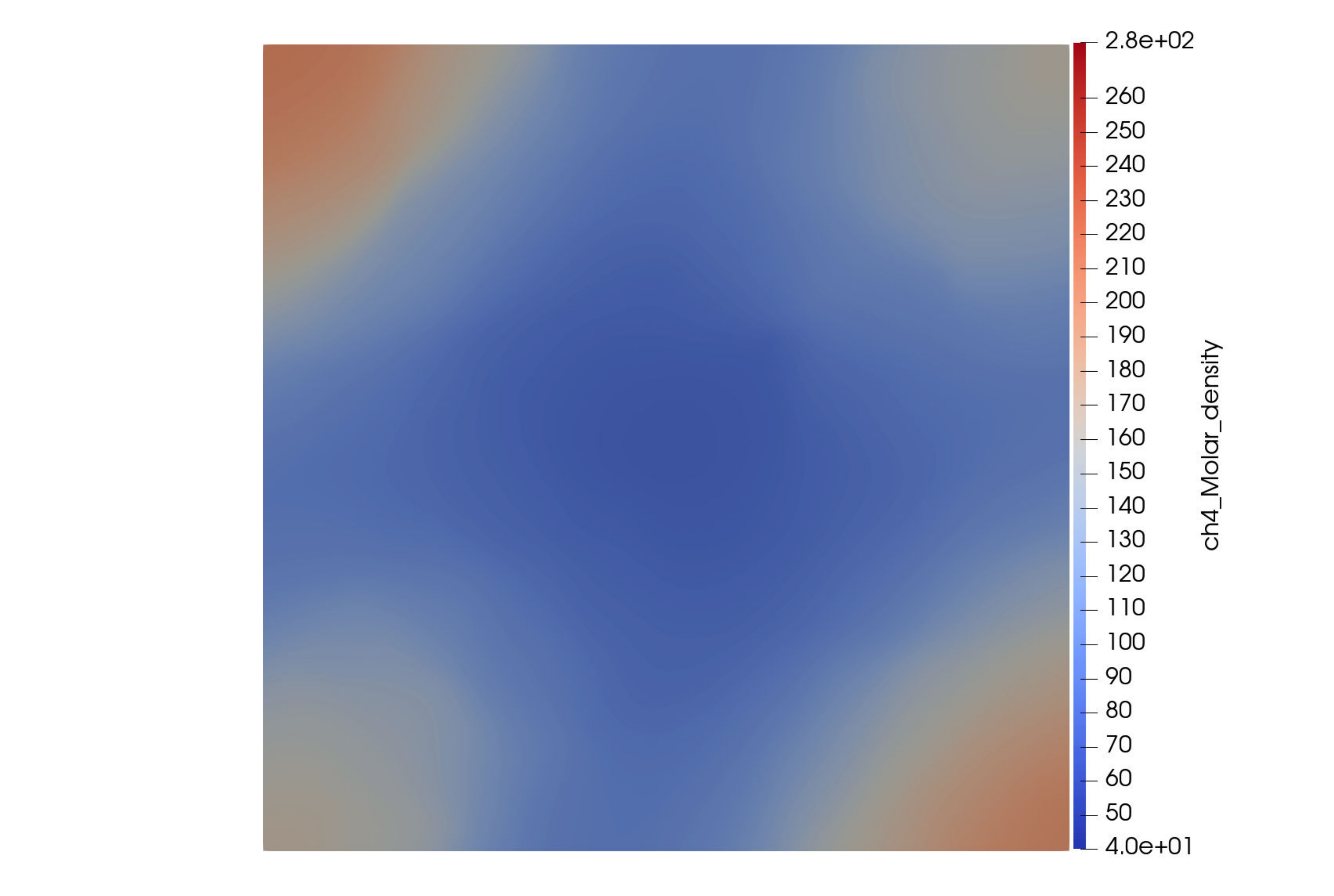}
	\includegraphics[width=5.5cm, height=4cm,trim=0.5cm 0cm 0cm 0cm,clip]{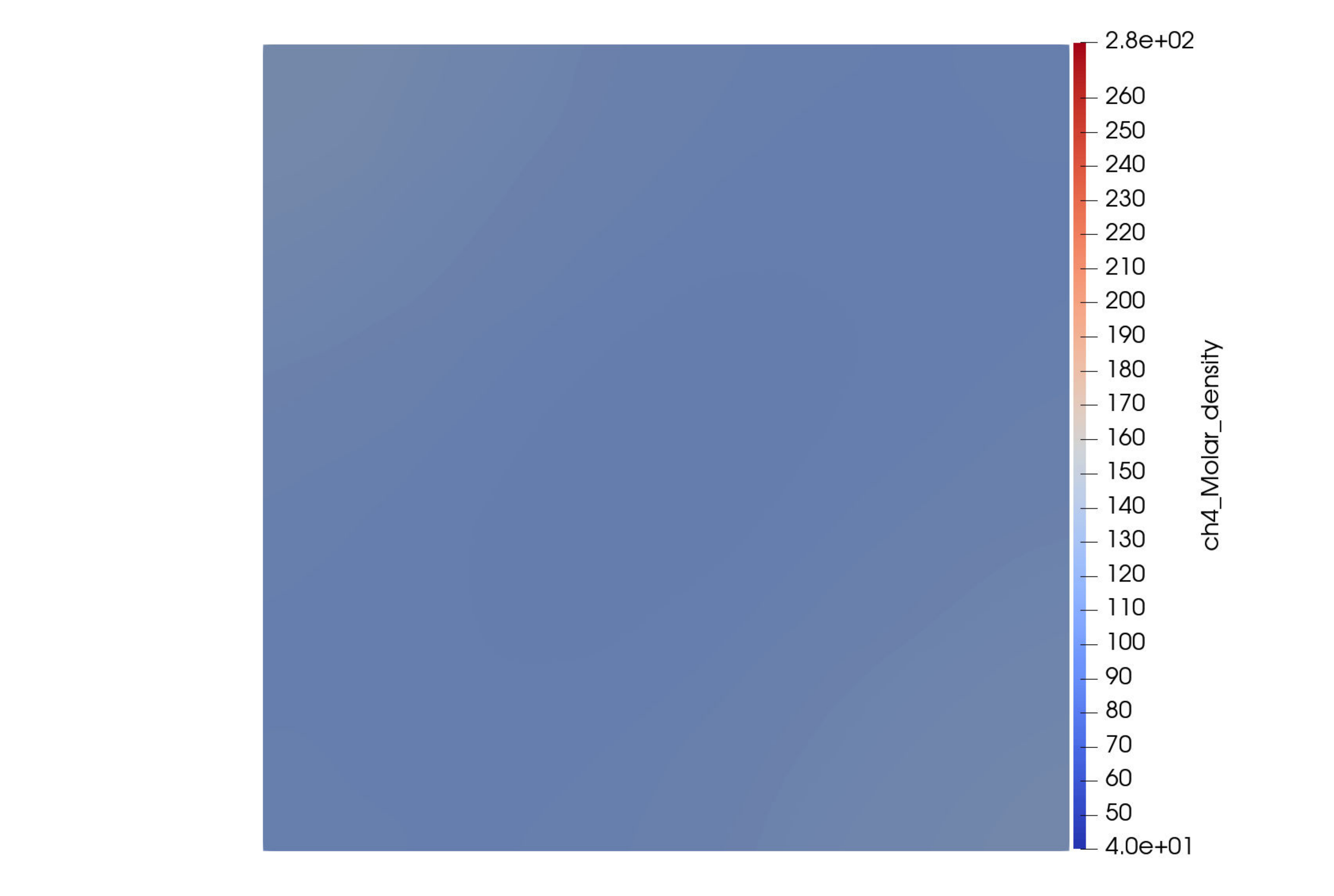}
	\caption{Distributions of molar density of CH$_4$ at different times in Example 1. Top-left: $n = 100$. Top-right: $n = 200$. Bottom-left: $n = 300$. Bottom-right: $n = 400$.}\label{fig1-ch4}
\end{figure}

\begin{figure}[htbp]
	\centering
	\includegraphics[width=5.5cm, height=4cm,trim=0.5cm 0cm 0cm 0cm,clip]{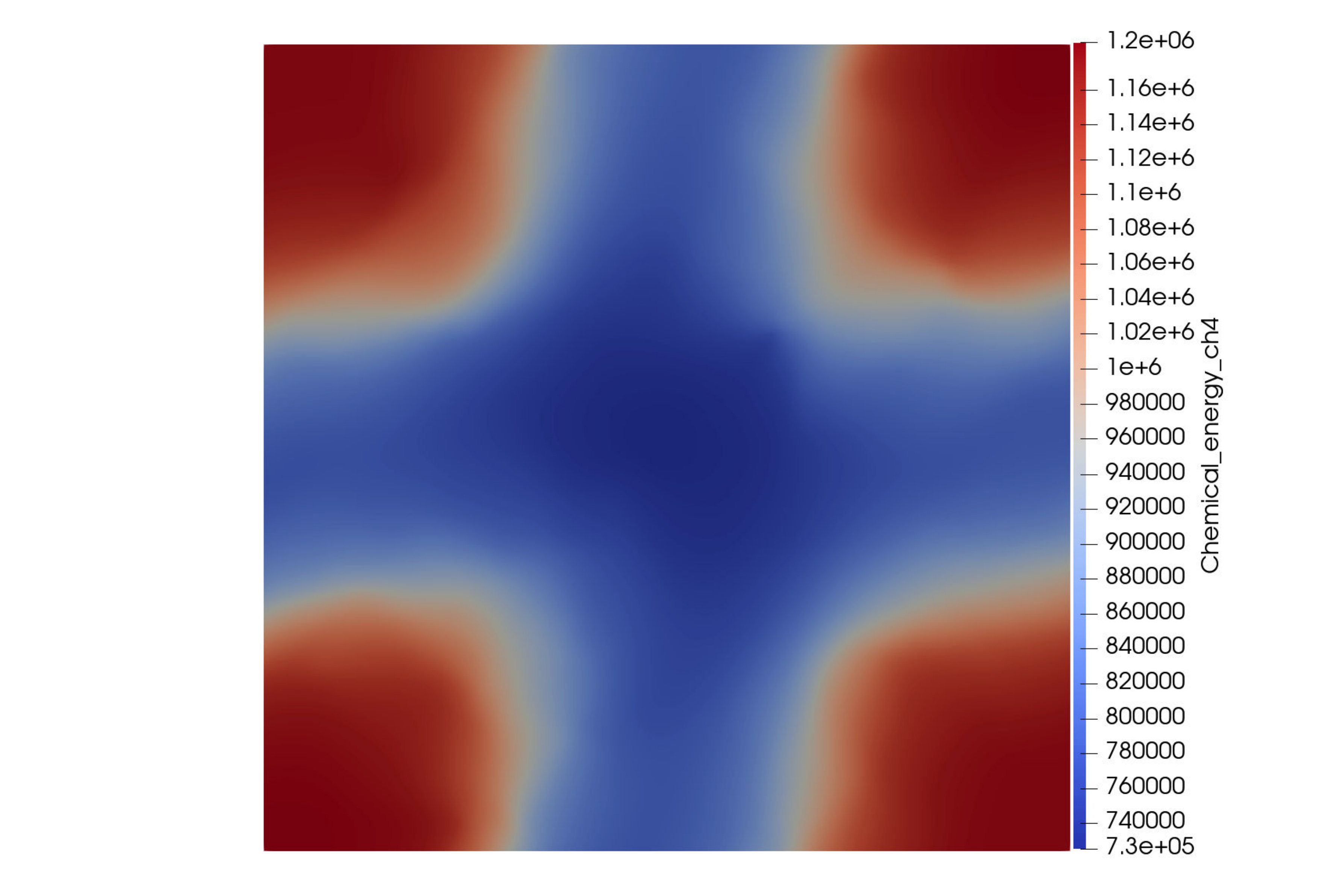}
	\includegraphics[width=5.5cm, height=4cm,trim=0.5cm 0cm 0cm 0cm,clip]{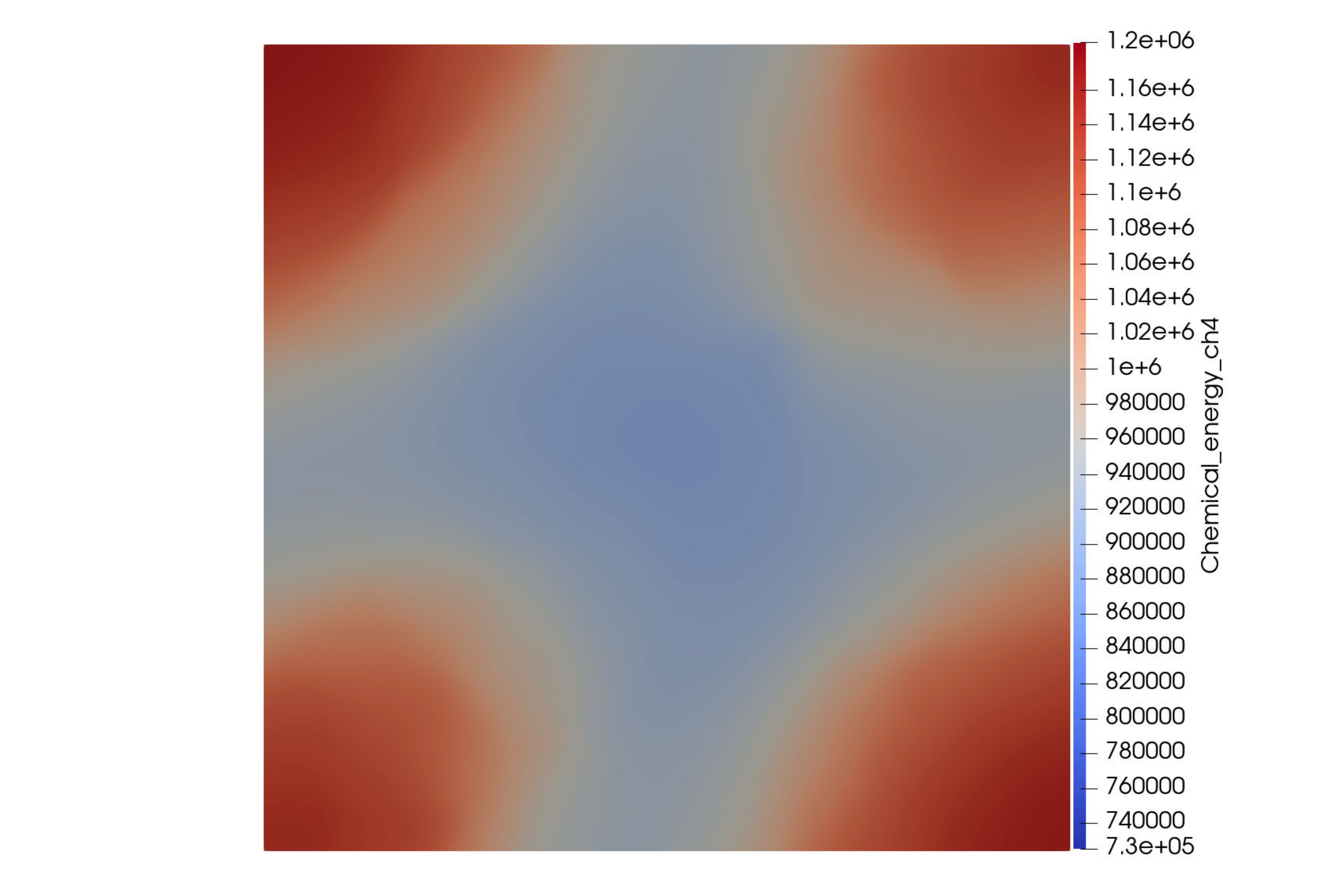}
	
	\includegraphics[width=5.5cm, height=4cm,trim=0.5cm 0cm 0cm 0cm,clip]{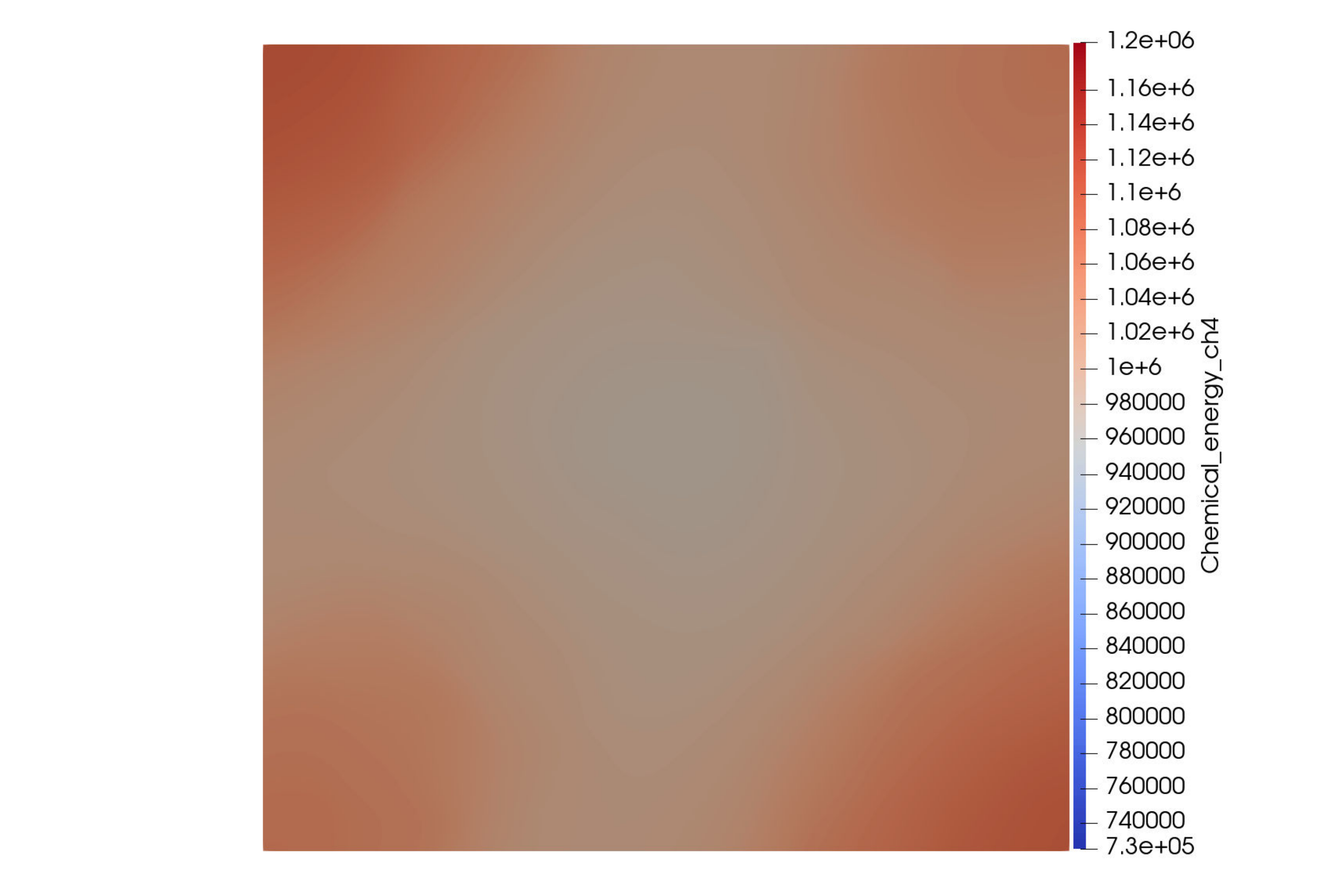}
	\includegraphics[width=5.5cm, height=4cm,trim=0.5cm 0cm 0cm 0cm,clip]{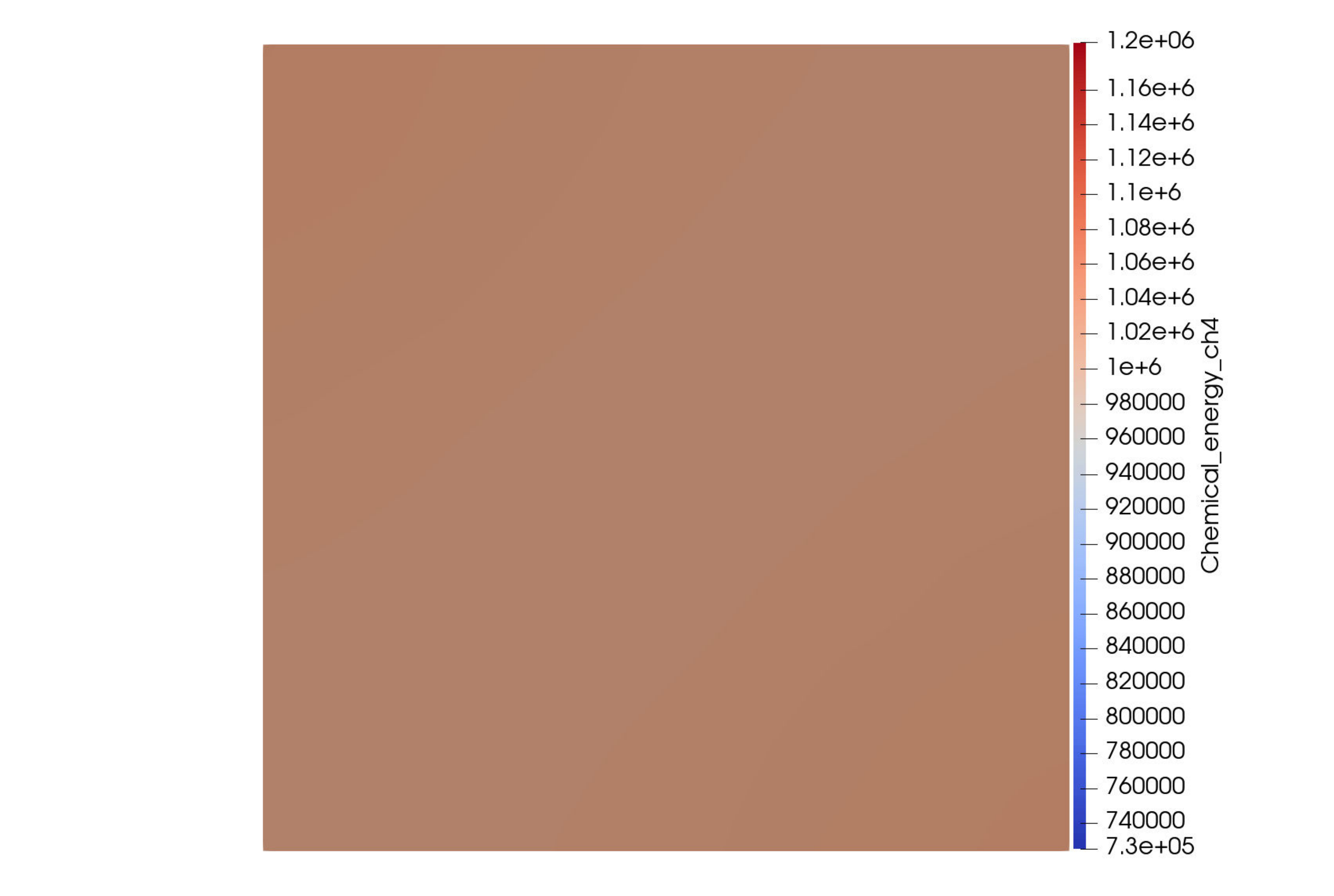}
	\caption{Distributions of chemical potential of CH$_4$ at different times in Example 1. Top-left: $n = 100$. Top-right: $n = 200$. Bottom-left: $n = 300$. Bottom-right: $n = 400$.}\label{fig1-ch4-che}
\end{figure}

\begin{figure}[htbp]
	\centering
	\includegraphics[width=5.5cm, height=4cm,trim=0.5cm 0cm 0cm 0cm,clip]{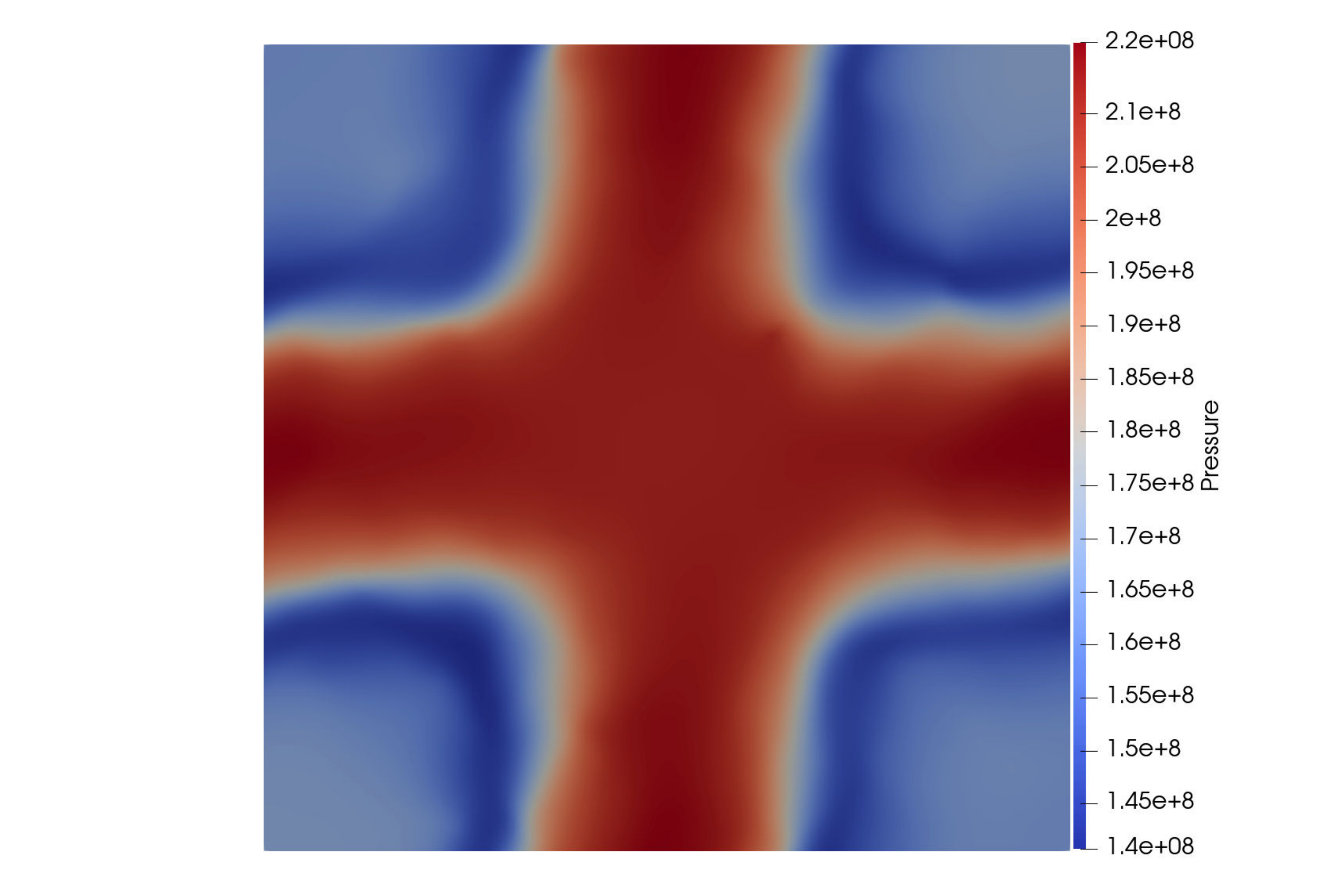}
	\includegraphics[width=5.5cm, height=4cm,trim=0.5cm 0cm 0cm 0cm,clip]{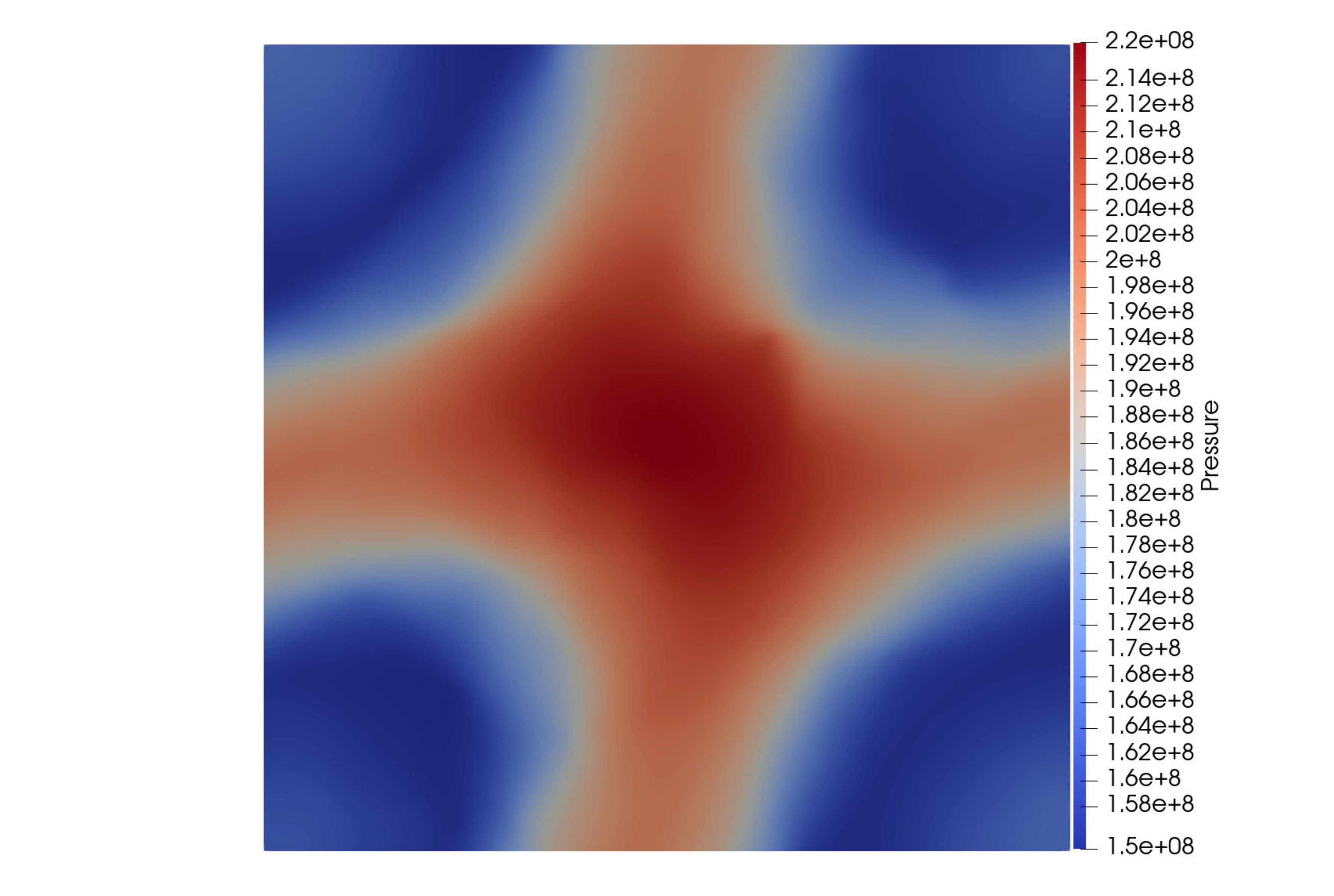}
	
	\includegraphics[width=5.5cm, height=4cm,trim=0.5cm 0cm 0cm 0cm,clip]{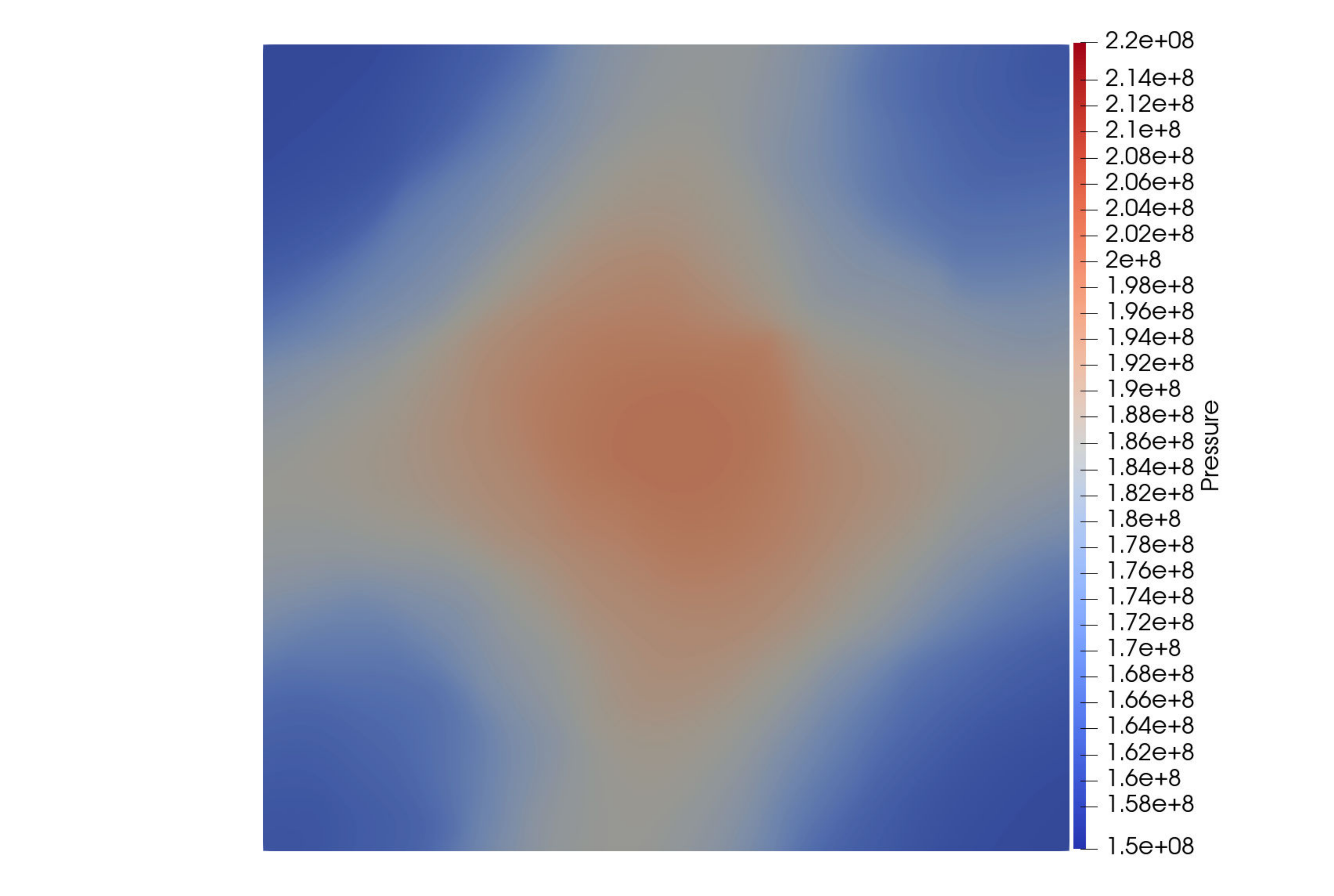}
	\includegraphics[width=5.5cm, height=4cm,trim=0.5cm 0cm 0cm 0cm,clip]{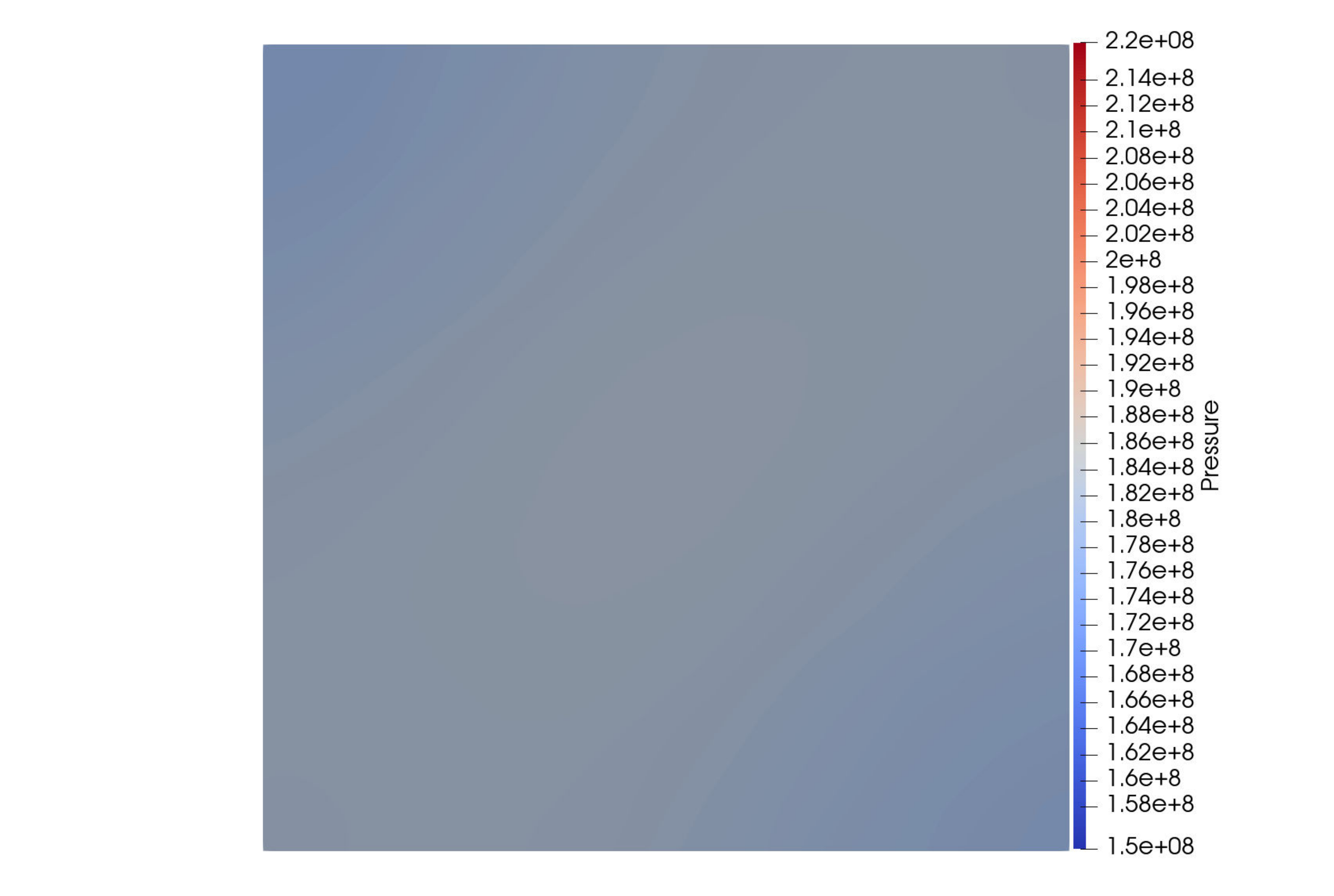}
	\caption{Distributions of pressure at different times in Example 1. Top-left: $n = 100$. Top-right: $n = 200$. Bottom-left: $n = 300$. Bottom-right: $n = 400$.}\label{fig1-pres}
\end{figure}
\begin{figure}[htbp]
	\centering
	\includegraphics[width=5.5cm, height=4cm,trim=0.5cm 0cm 0cm 0cm,clip]{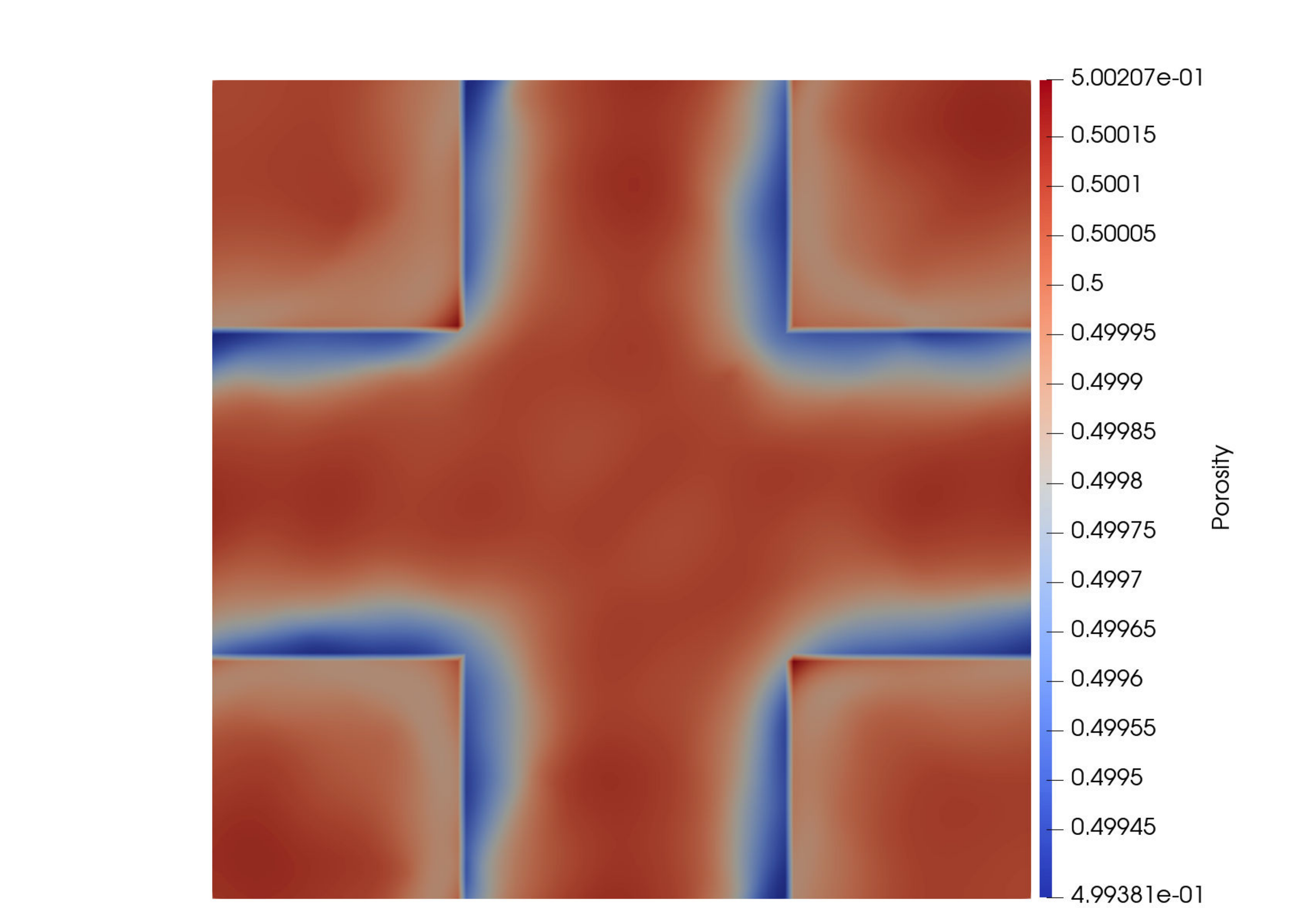}
	\includegraphics[width=5.5cm, height=4cm,trim=0.5cm 0cm 0cm 0cm,clip]{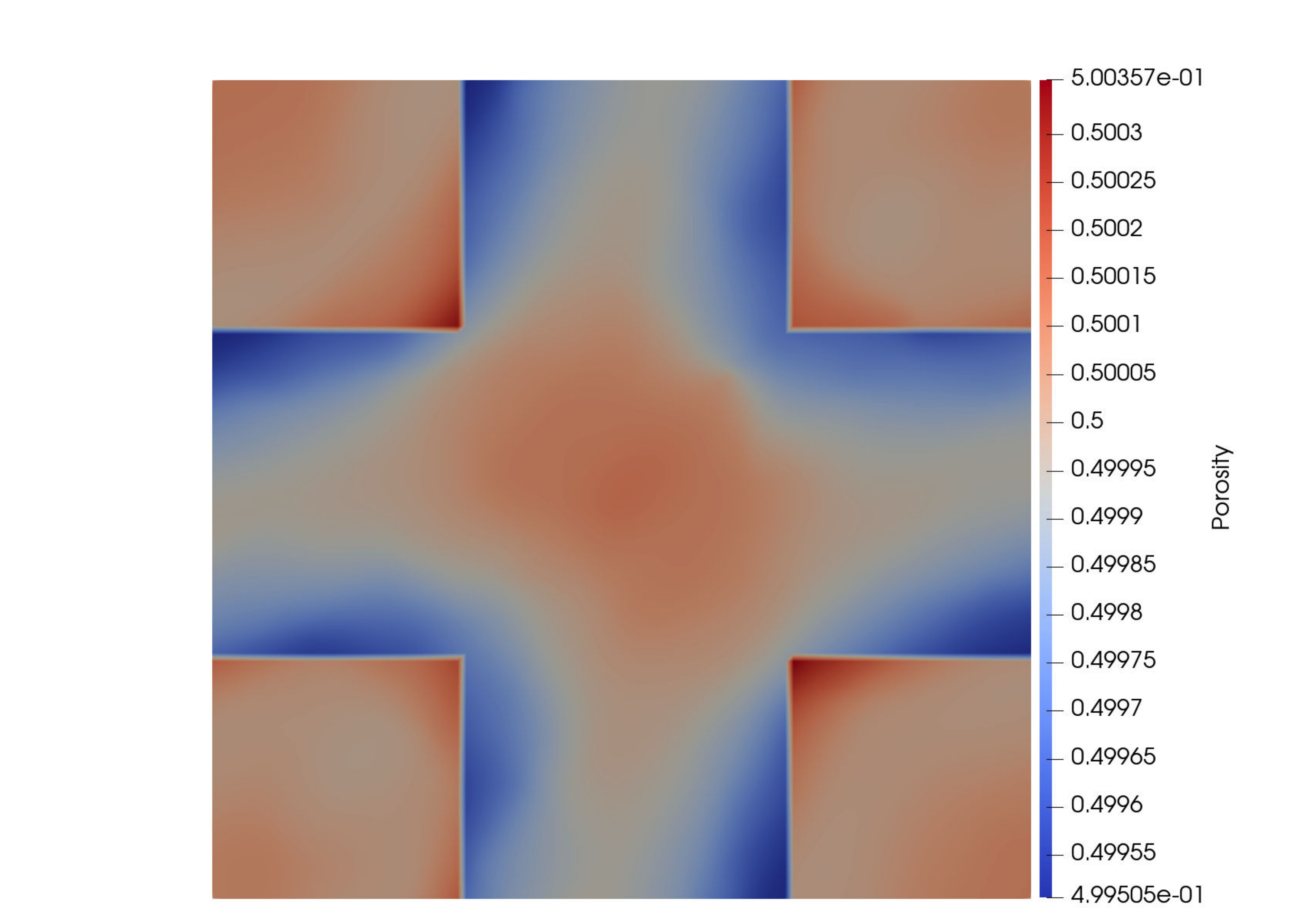}
	
	\includegraphics[width=5.5cm, height=4cm,trim=0.5cm 0cm 0cm 0cm,clip]{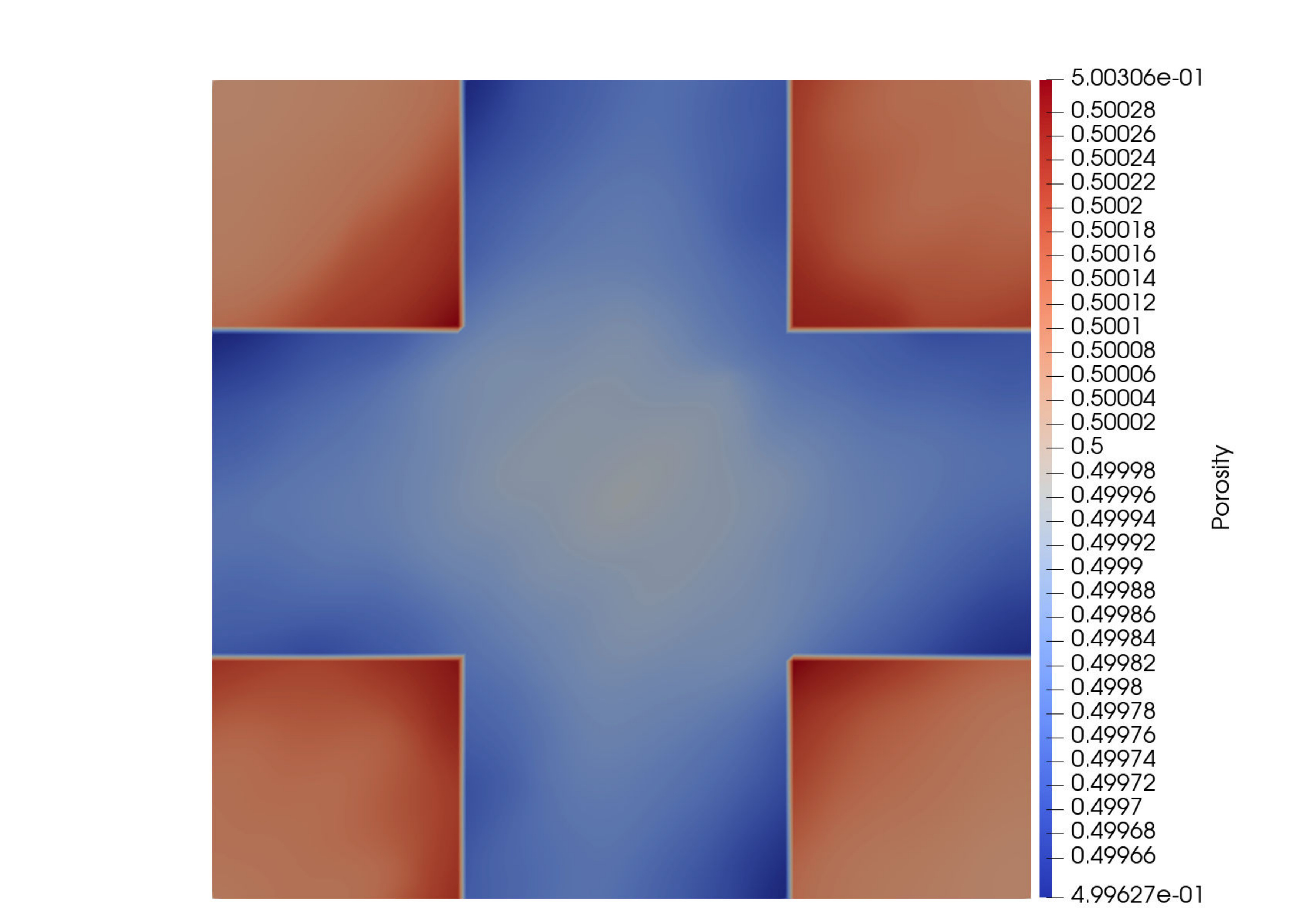}
	\includegraphics[width=5.5cm, height=4cm,trim=0.5cm 0cm 0cm 0cm,clip]{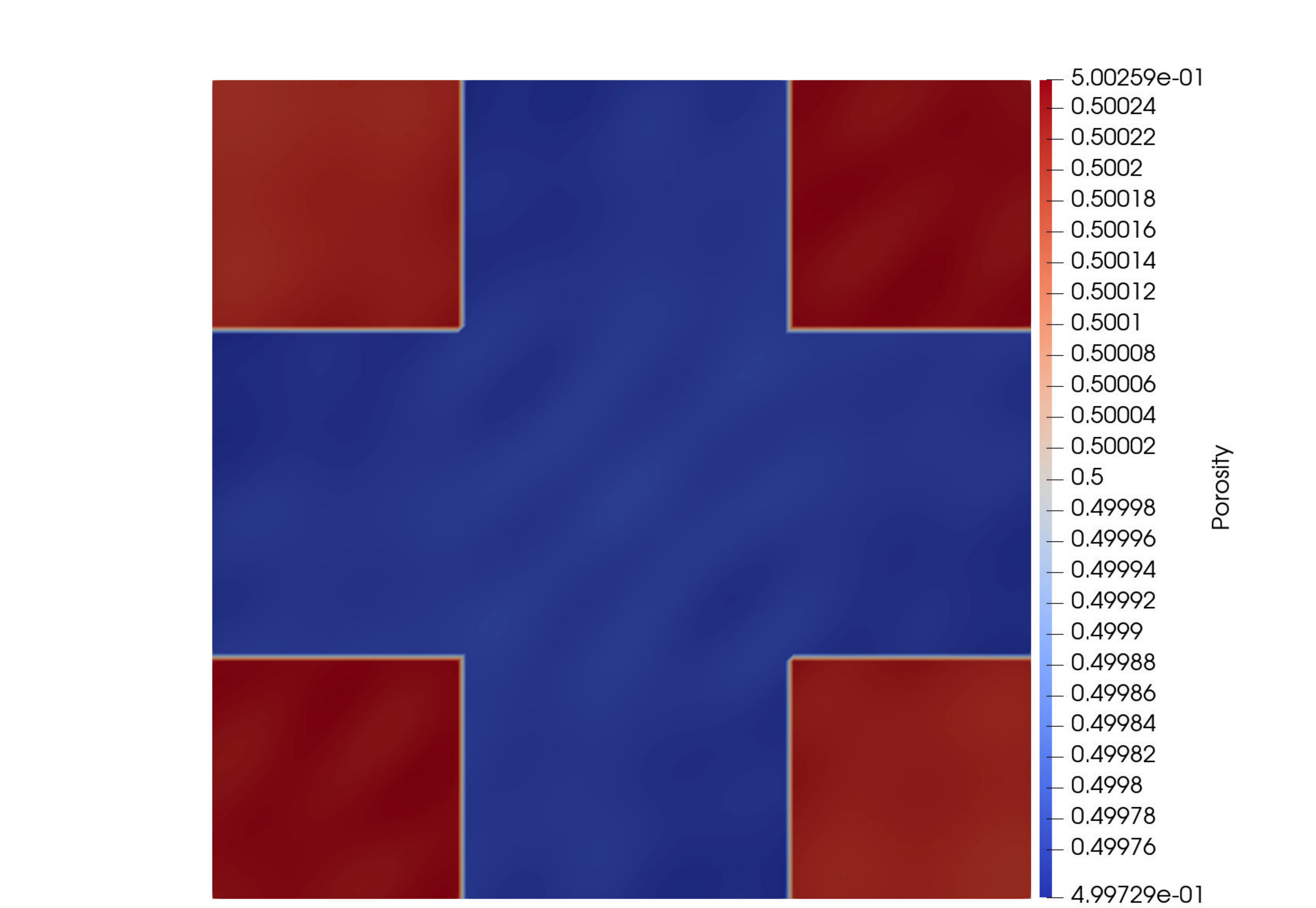}
	\caption{Distributions of porosity at different times in Example 1. Top-left: $n = 100$. Top-right: $n = 200$. Bottom-left: $n = 300$. Bottom-right: $n = 400$.}\label{fig1-poro}
\end{figure}
\subsection{Example 2}
In this numerical experiment, we simulate a multicomponent gas injection process in a poroelastic medium. The model considers a ternary mixture of CO$_2$, CH$_4$, and C$_2$H$_6$. 
The initial conditions are defined by a uniform molar density of gas, with values of $C_{\mathrm{CH}_4}^0 = C_{\mathrm{C}_2\mathrm{H}_6}^0 = 100$  $\mathrm{mol/m^3}$ and $C_{\mathrm{CO}_2}^0 = 10$  $\mathrm{mol/m}^3$. A Dirichlet boundary condition for the molar density of $\mathrm{CO_2}$ is prescribed at the left boundary ($x = 0$), driving the injection and displacement process. The permeability field is characterized by a highly heterogeneous, as illustrated on the left-hand side of Figure \ref{fig2-initial}, featuring two high-permeability zones ($\mathcal{K} = 200$ $\mathrm{md}$) within $[0~\text{m},~80~\text{m}] \times [65~\text{m},~70~\text{m}]$ and $[0~\text{m},~80~\text{m}] \times [35~\text{m},~40~\text{m}]$, embedded within a low-permeability matrix ($\mathcal{K}  = 1~\mathrm{md}$) that constitutes the remainder of the domain. As shown on the right-hand side of Figure \ref{fig2-initial}, the adaptive time step values at different computational stages are presented. It can be observed that the time step continues to increase, which can be attributed to the fact that the system has not yet reached an equilibrium state. Figures \ref{fig2-co2}, \ref{fig2-ch4}, and \ref{fig2-c2h6} depict the spatial distributions of molar density for $\mathrm{CO_2}$, $\mathrm{CH_4}$, and $\mathrm{C_2H_6}$, respectively, at different time steps $n = 50, 150, 350, 500$. The results demonstrate that the Dirichlet boundary condition on the left boundary induces a chemical potential gradient, which drives the  transport of $\mathrm{CO_2}$ into the domain. The advancing $\mathrm{CO_2}$ front efficiently displaces the native $\mathrm{CH_4}$ and $\mathrm{C_2H_6}$ mixtures towards the production outlet. The flow dynamics are dominantly channeled through the high-permeability layers, showcasing a clear bypassing effect characteristic of heterogeneous media. Figure \ref{fig2-poro} presents the evolution of porosity at time steps $n=50,150,350,500$. Driven by poroelastic coupling, changes in fluid pressure during injection alter the local effective stress, resulting in dynamic porosity adjustments. The variations are most pronounced within the high-permeability layers, where fluid flow and pressure transients are concentrated. 


\begin{figure}[htbp]
	\centering
	\includegraphics[width=5.5cm, height=4cm,trim=0.5cm 0cm 0cm 0cm,clip]{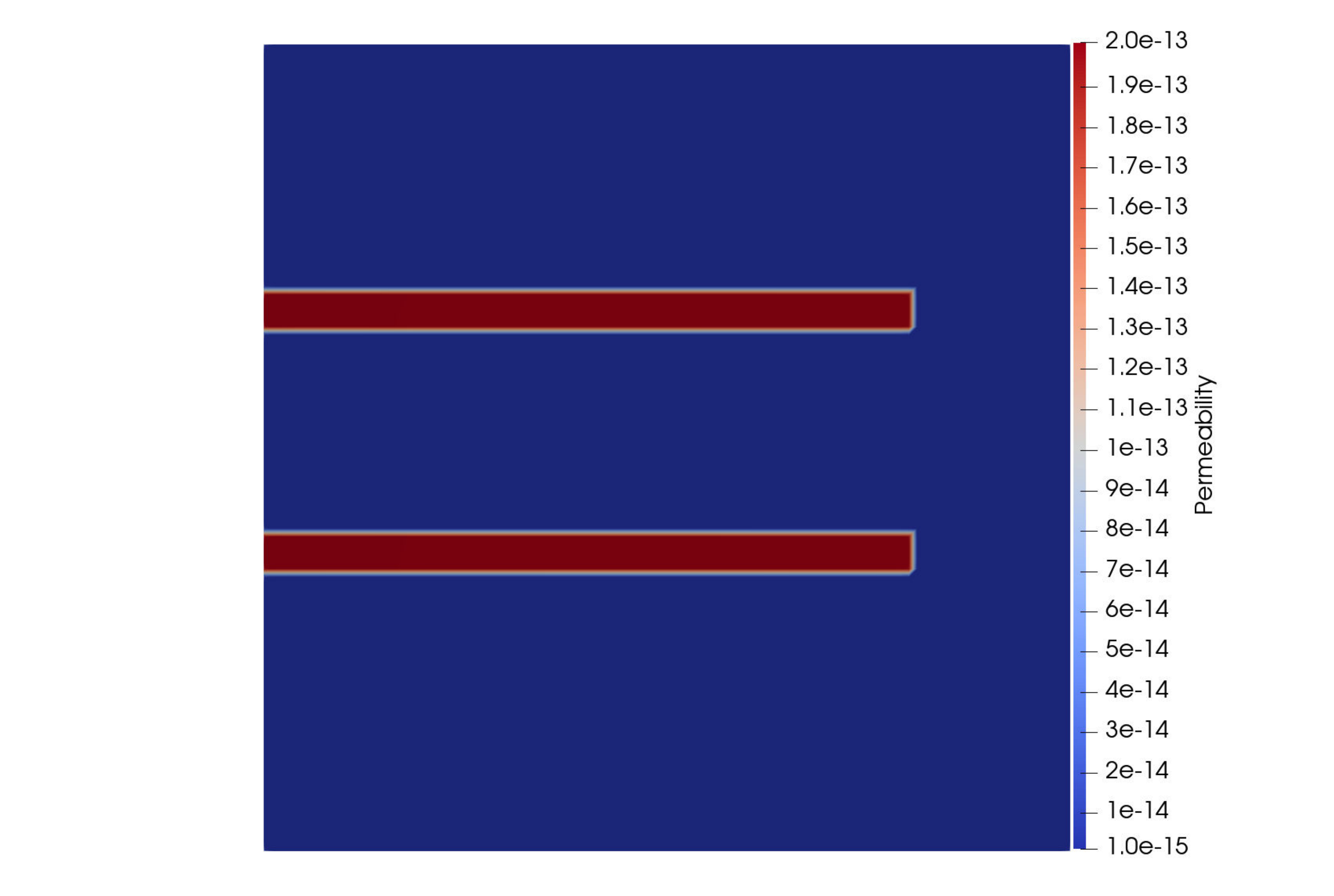}
	\includegraphics[width=5.0cm, height=4.5cm]{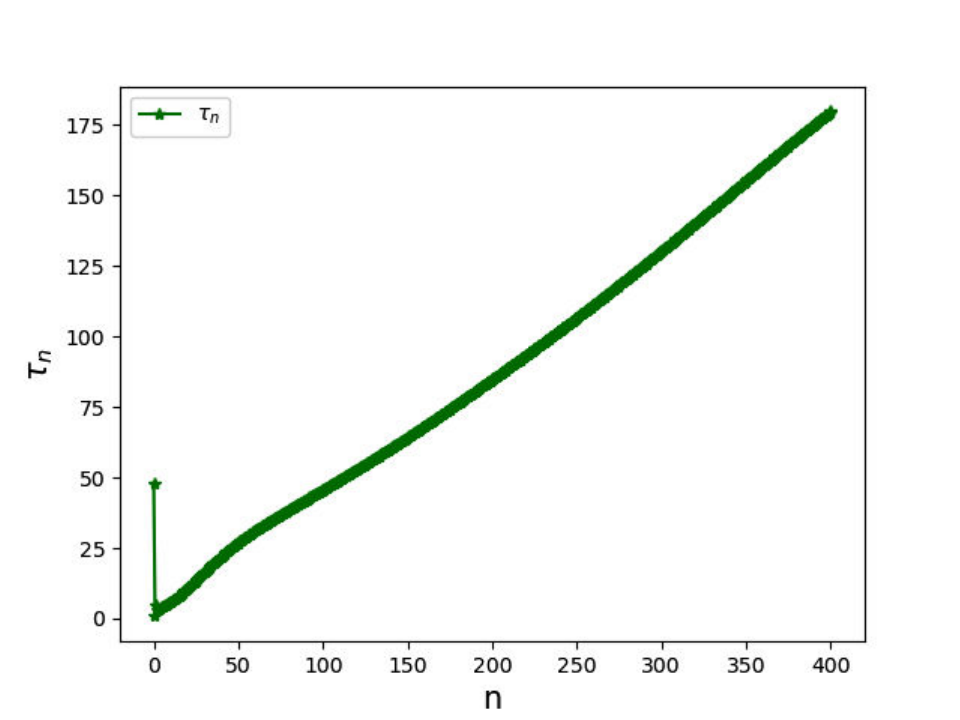}
	\caption{Example 2: Left: Initial distributions of permeability. Right: Adaptive values of the time step size.}\label{fig2-initial}
\end{figure}
\begin{figure}[htbp]
	\centering
	\includegraphics[width=5.5cm, height=4cm,trim=0.5cm 0cm 0cm 0cm,clip]{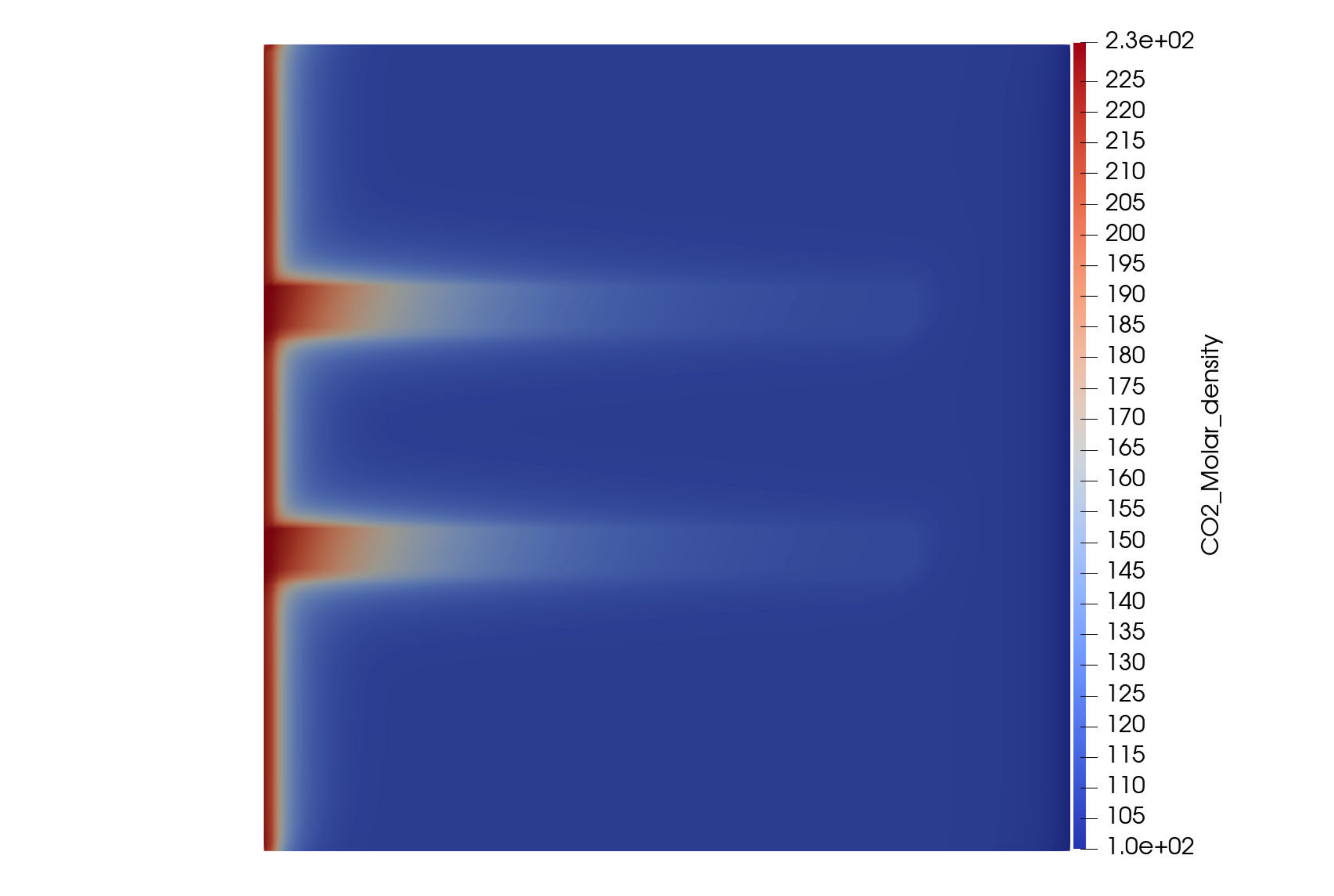}
	\includegraphics[width=5.5cm, height=4cm,trim=0.5cm 0cm 0cm 0cm,clip]{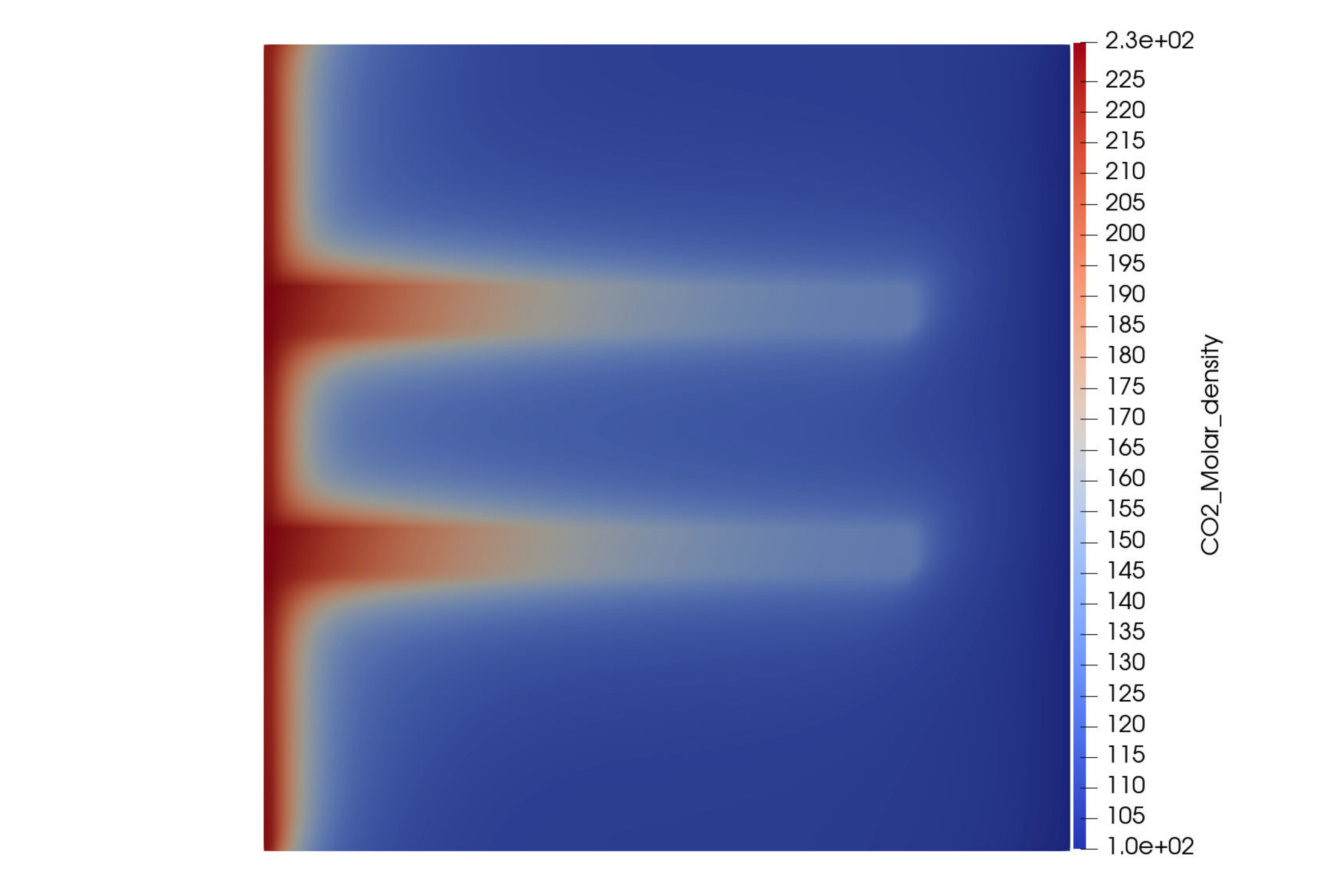}
	
	\includegraphics[width=5.5cm, height=4cm,trim=0.5cm 0cm 0cm 0cm,clip]{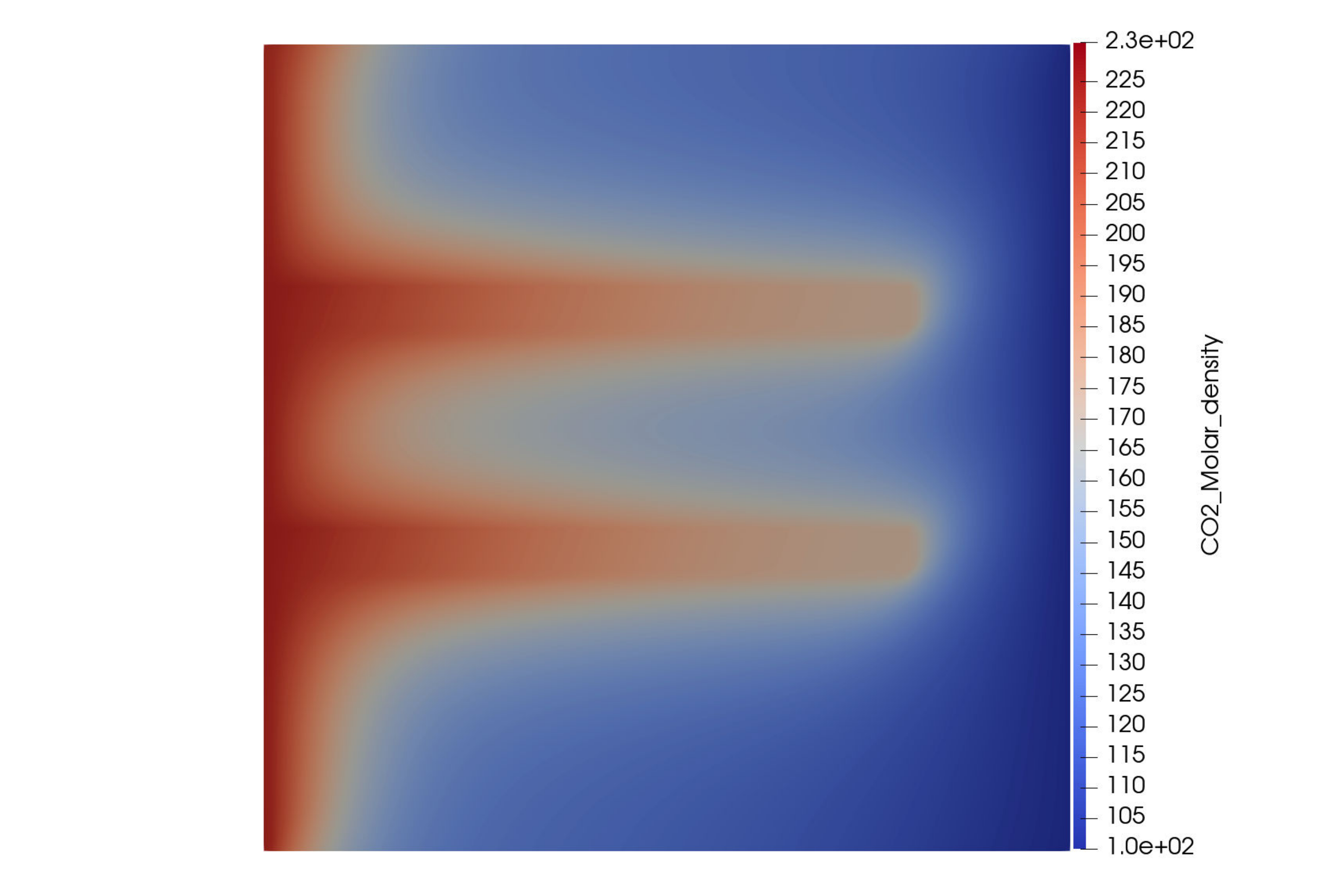}
	\includegraphics[width=5.5cm, height=4cm,trim=0.5cm 0cm 0cm 0cm,clip]{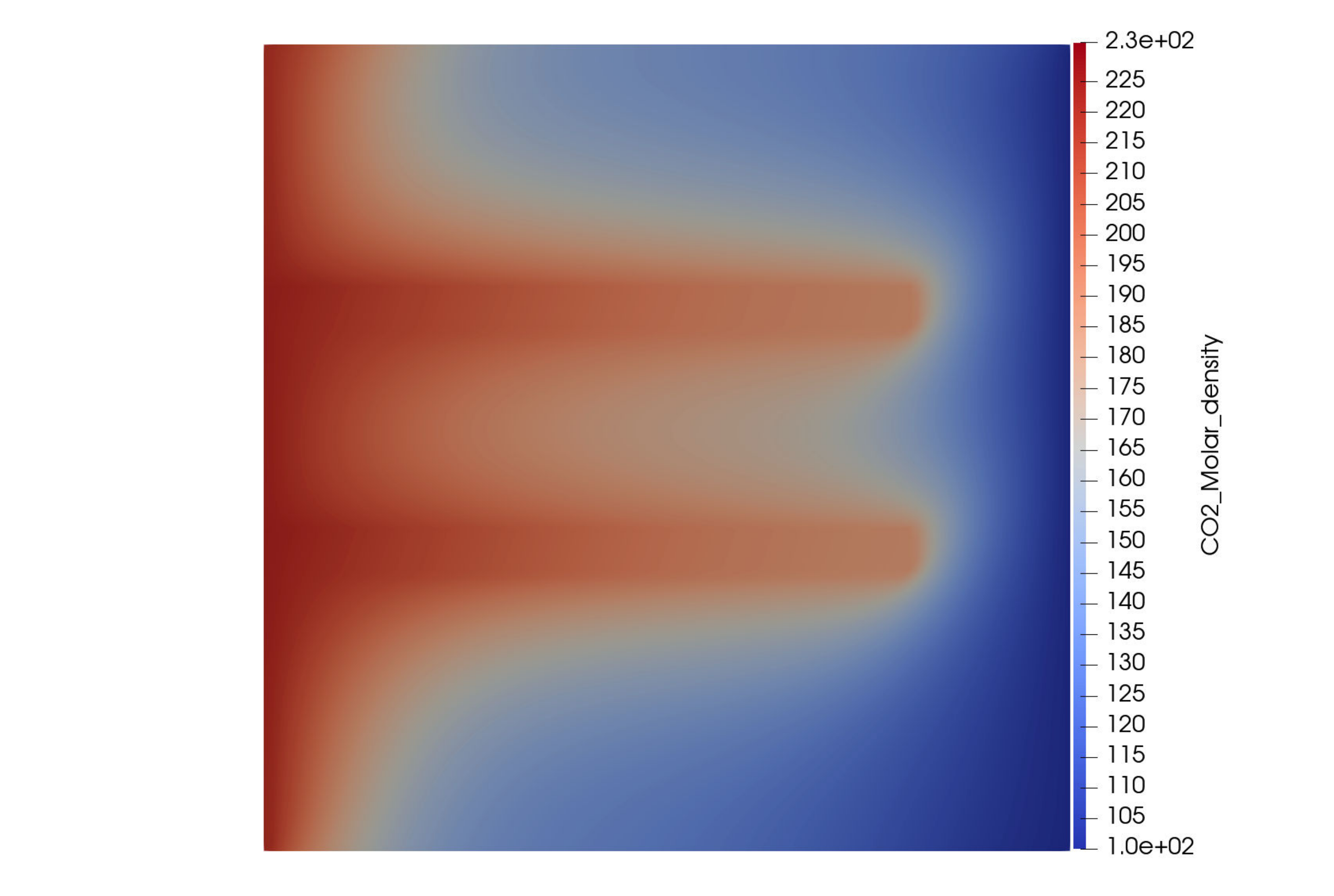}
	\caption{Distributions of molar density of CO$_2$ at different times in Example 2. Top-left: $n = 50$. Top-right: $n = 150$. Bottom-left: $n = 350$. Bottom-right: $n = 500$.}\label{fig2-co2}
\end{figure}

%

\begin{figure}[htbp]
	\centering
	\includegraphics[width=5.5cm, height=4cm,trim=0.5cm 0cm 0cm 0cm,clip]{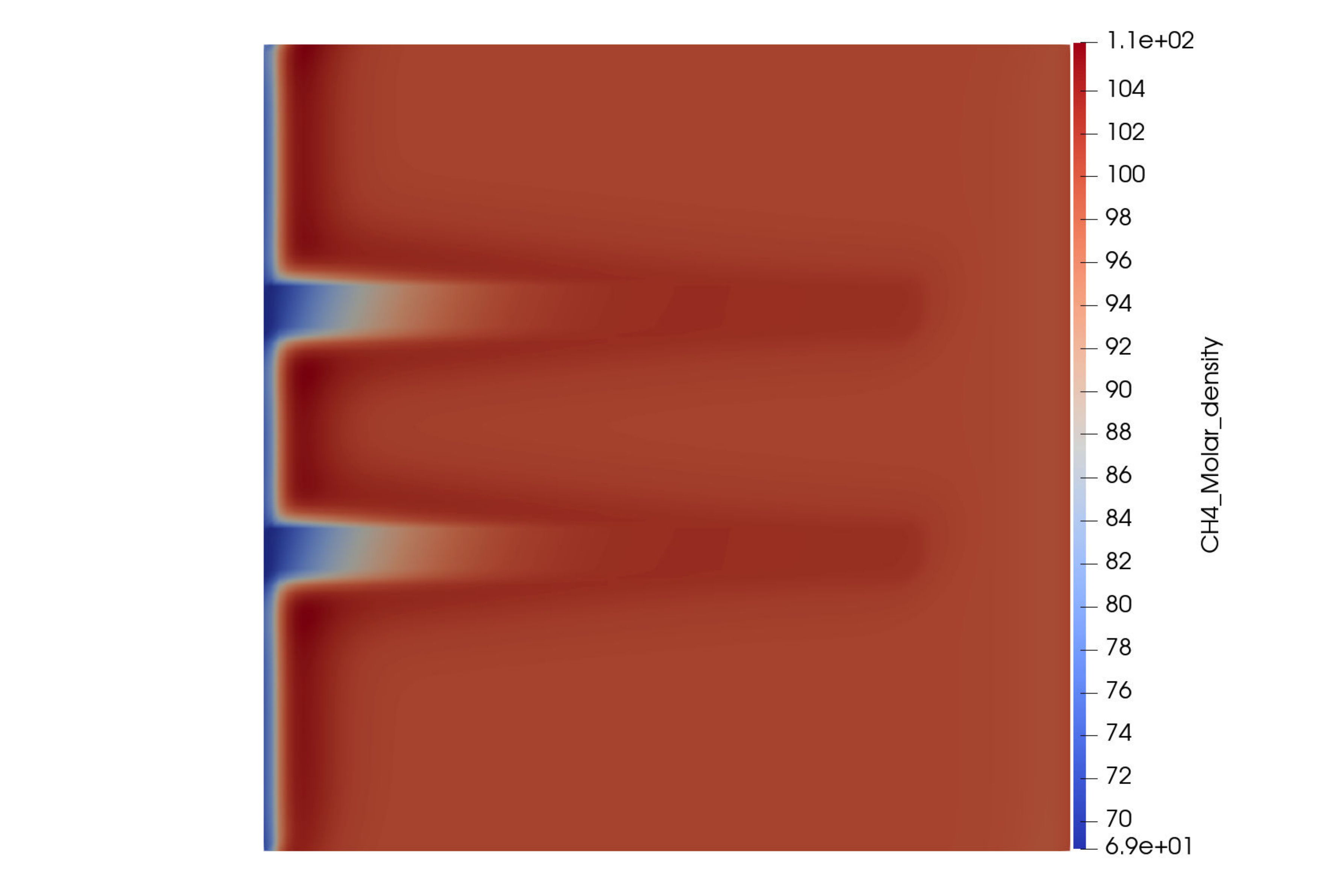}
	\includegraphics[width=5.5cm, height=4cm,trim=0.5cm 0cm 0cm 0cm,clip]{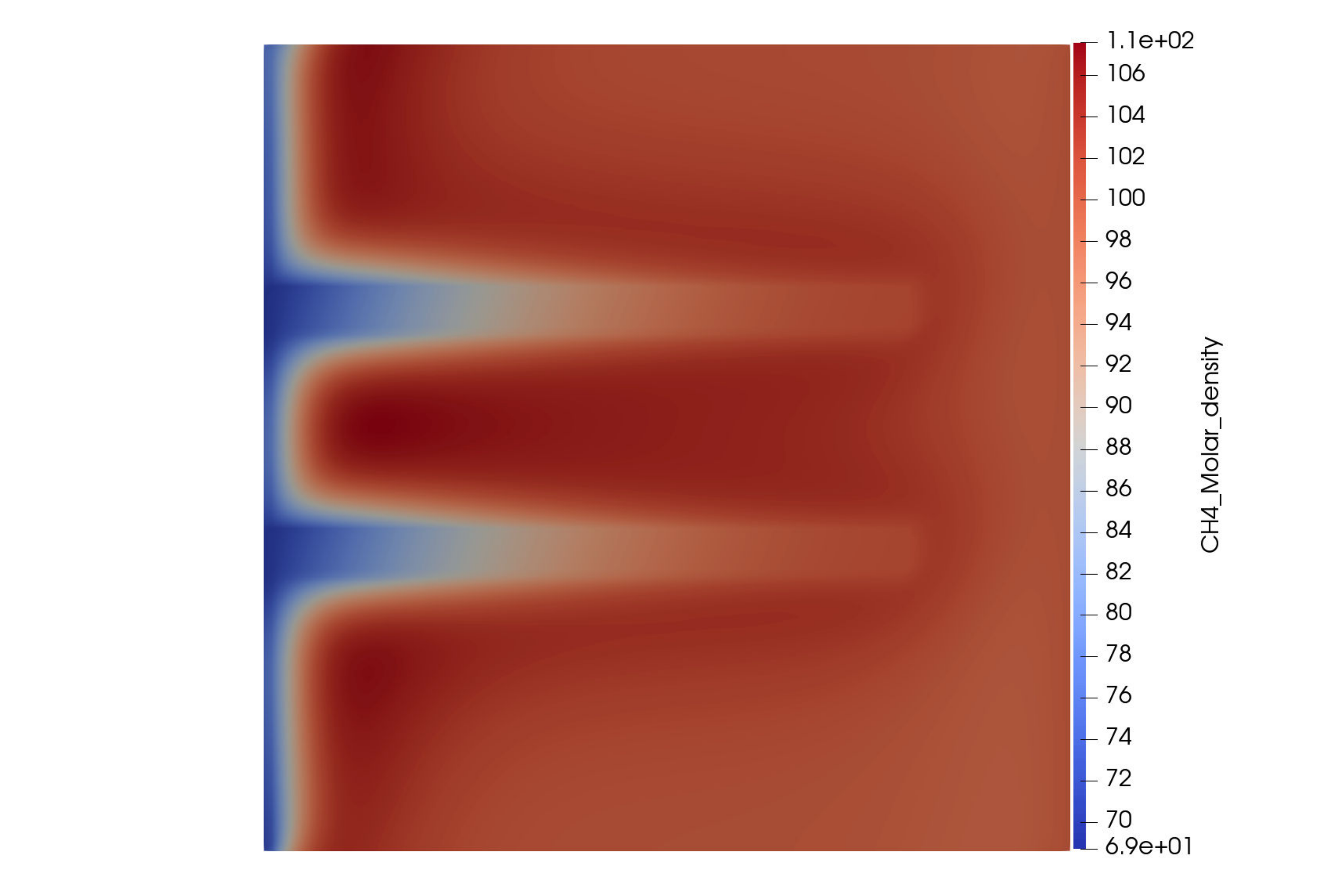}
	
	\includegraphics[width=5.5cm, height=4cm,trim=0.5cm 0cm 0cm 0cm,clip]{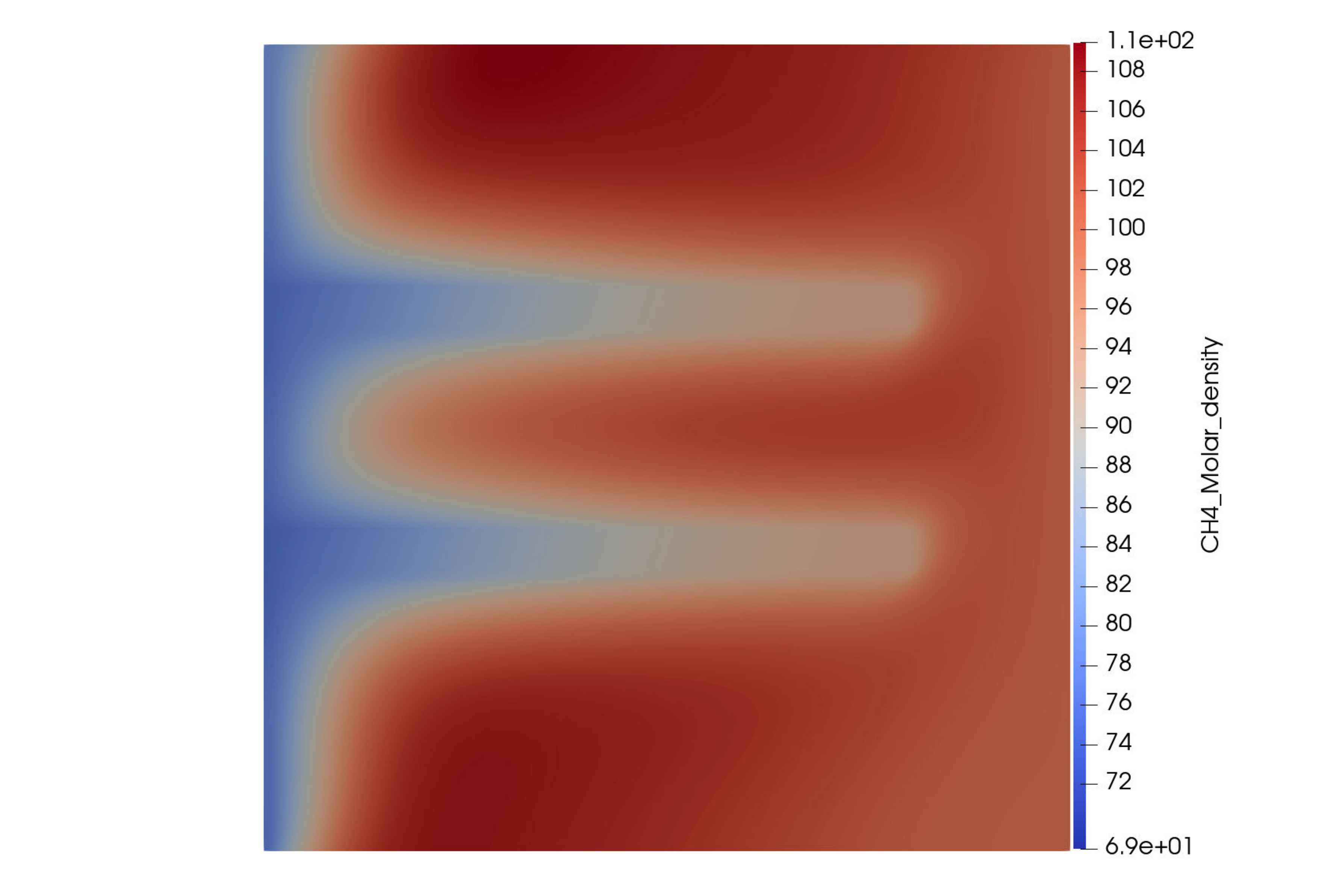}
	\includegraphics[width=5.5cm, height=4cm,trim=0.5cm 0cm 0cm 0cm,clip]{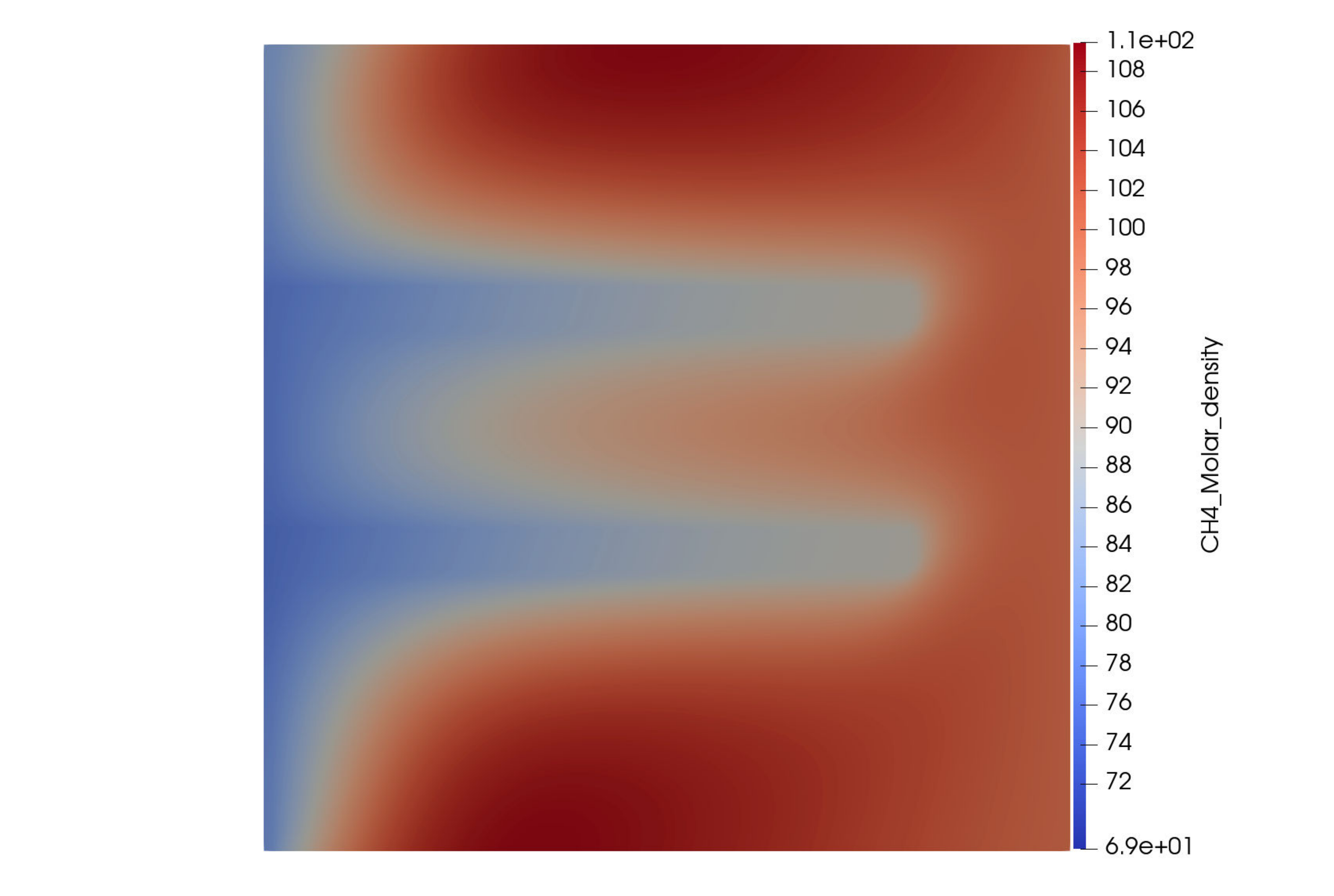}
	\caption{Distributions of molar density of CH$_4$ at different times in Example 2. Top-left: $n = 50$. Top-right: $n = 150$. Bottom-left: $n = 350$. Bottom-right: $n = 500$.}\label{fig2-ch4}
\end{figure}

%

\begin{figure}[htbp]
	\centering
	\includegraphics[width=5.5cm, height=4cm,trim=0.5cm 0cm 0cm 0cm,clip]{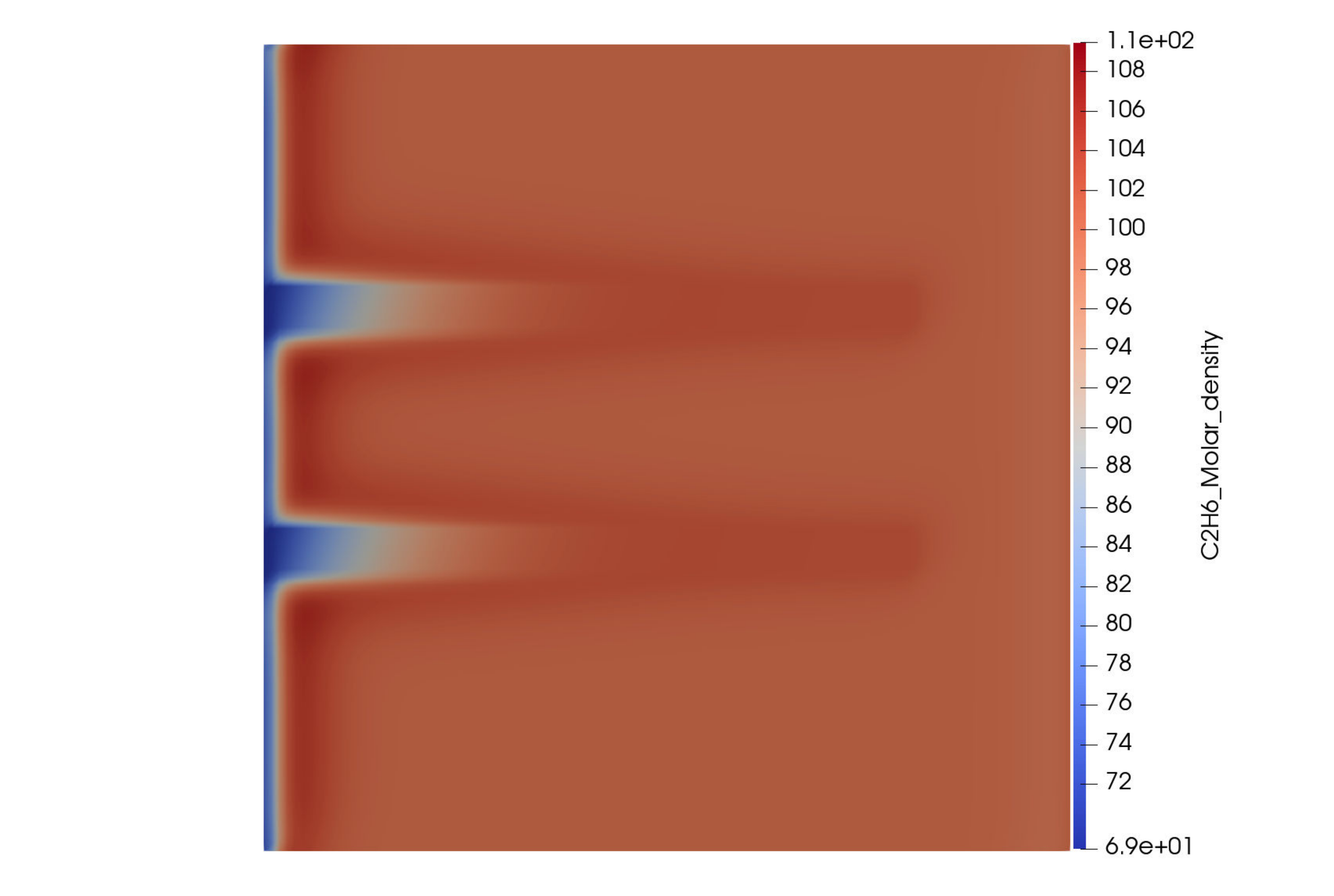}
	\includegraphics[width=5.5cm, height=4cm,trim=0.5cm 0cm 0cm 0cm,clip]{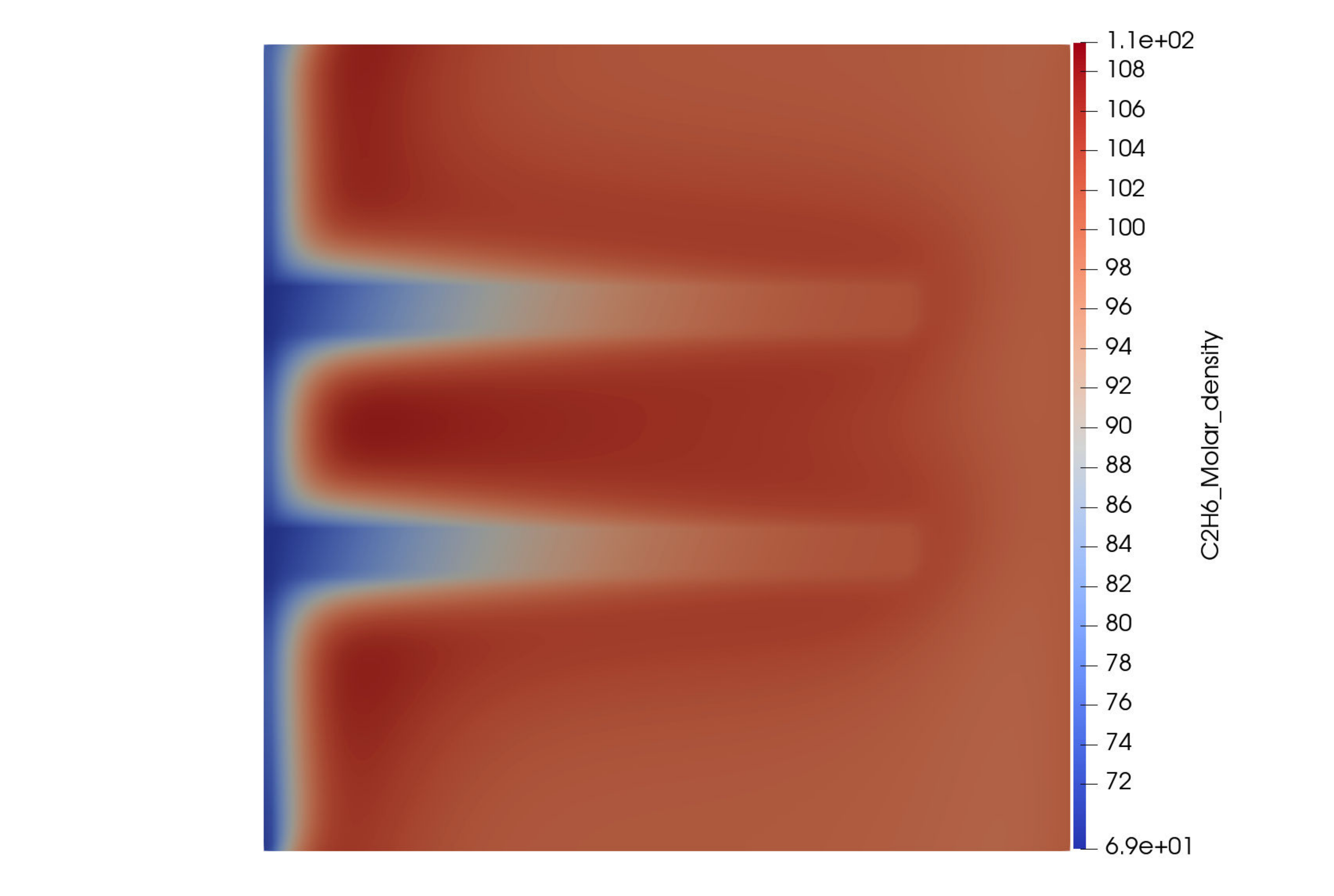}
	
	\includegraphics[width=5.5cm, height=4cm,trim=0.5cm 0cm 0cm 0cm,clip]{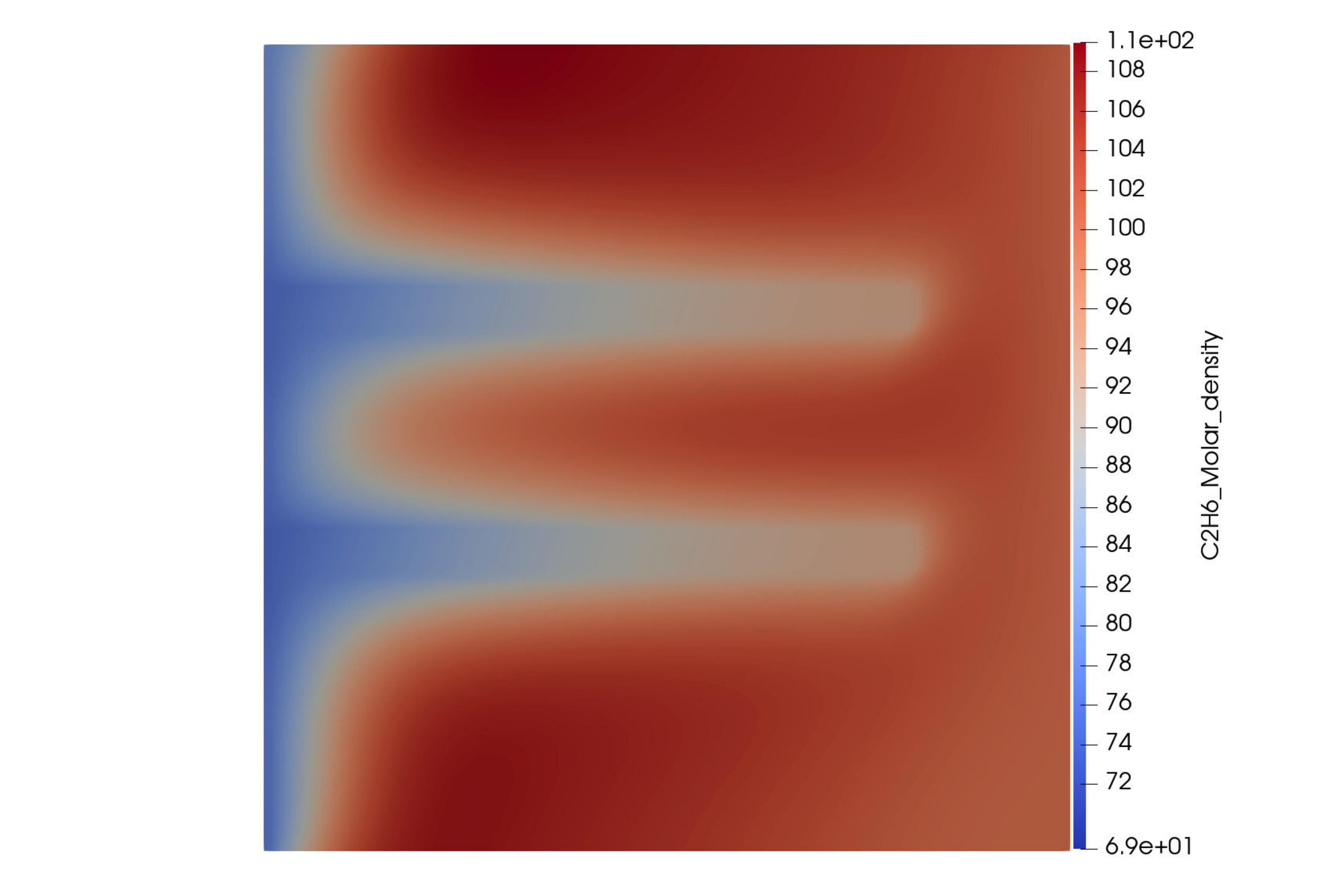}
	\includegraphics[width=5.5cm, height=4cm,trim=0.5cm 0cm 0cm 0cm,clip]{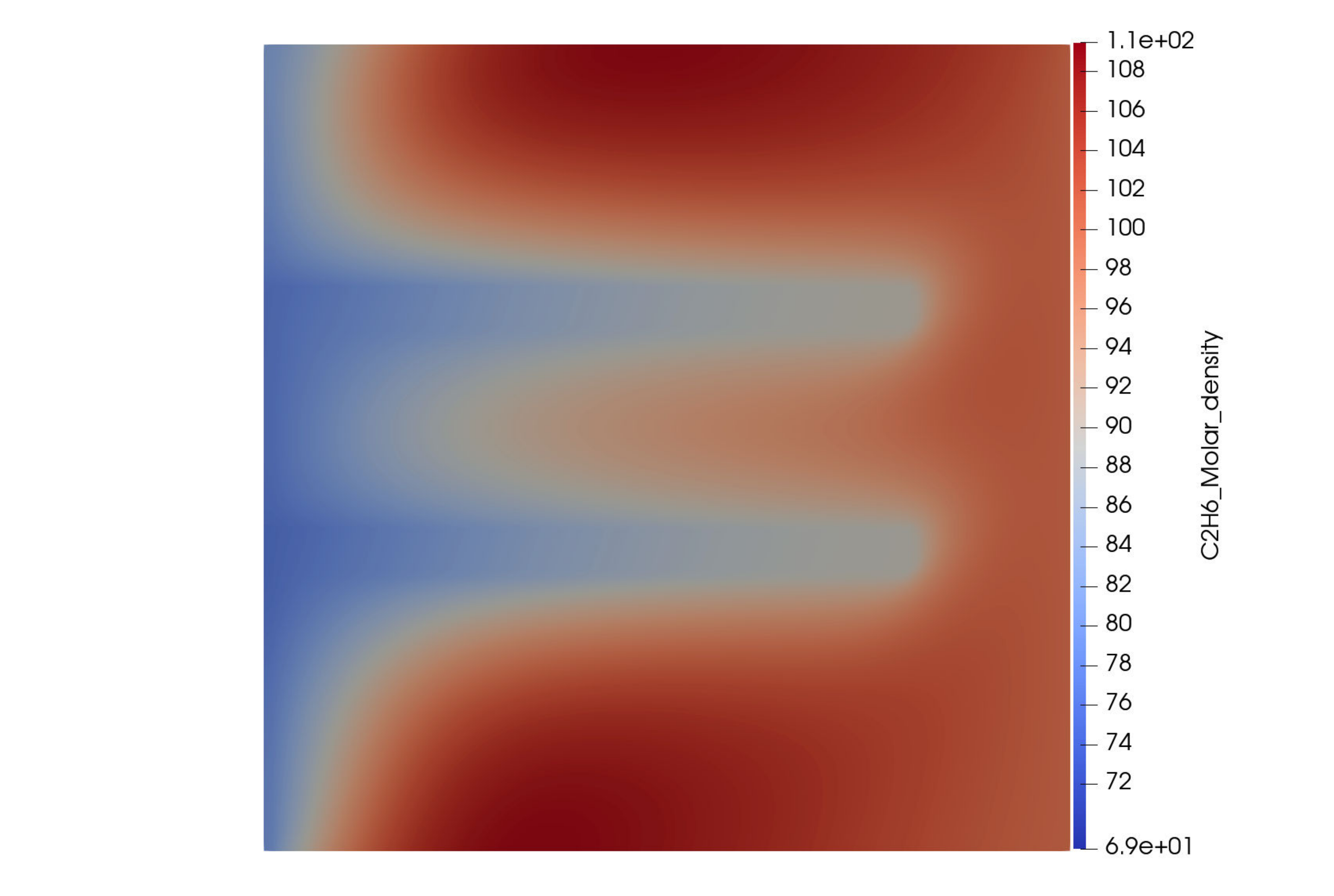}
	\caption{Distributions of molar density of C$_2$H$_6$ at different times in Example 2. Top-left: $n = 50$. Top-right: $n = 150$. Bottom-left: $n = 350$. Bottom-right: $n = 500$.}\label{fig2-c2h6}
\end{figure}

	%

\begin{figure}[htbp]
	\centering
	\includegraphics[width=5.5cm, height=4cm,trim=0.5cm 0cm 0cm 0cm,clip]{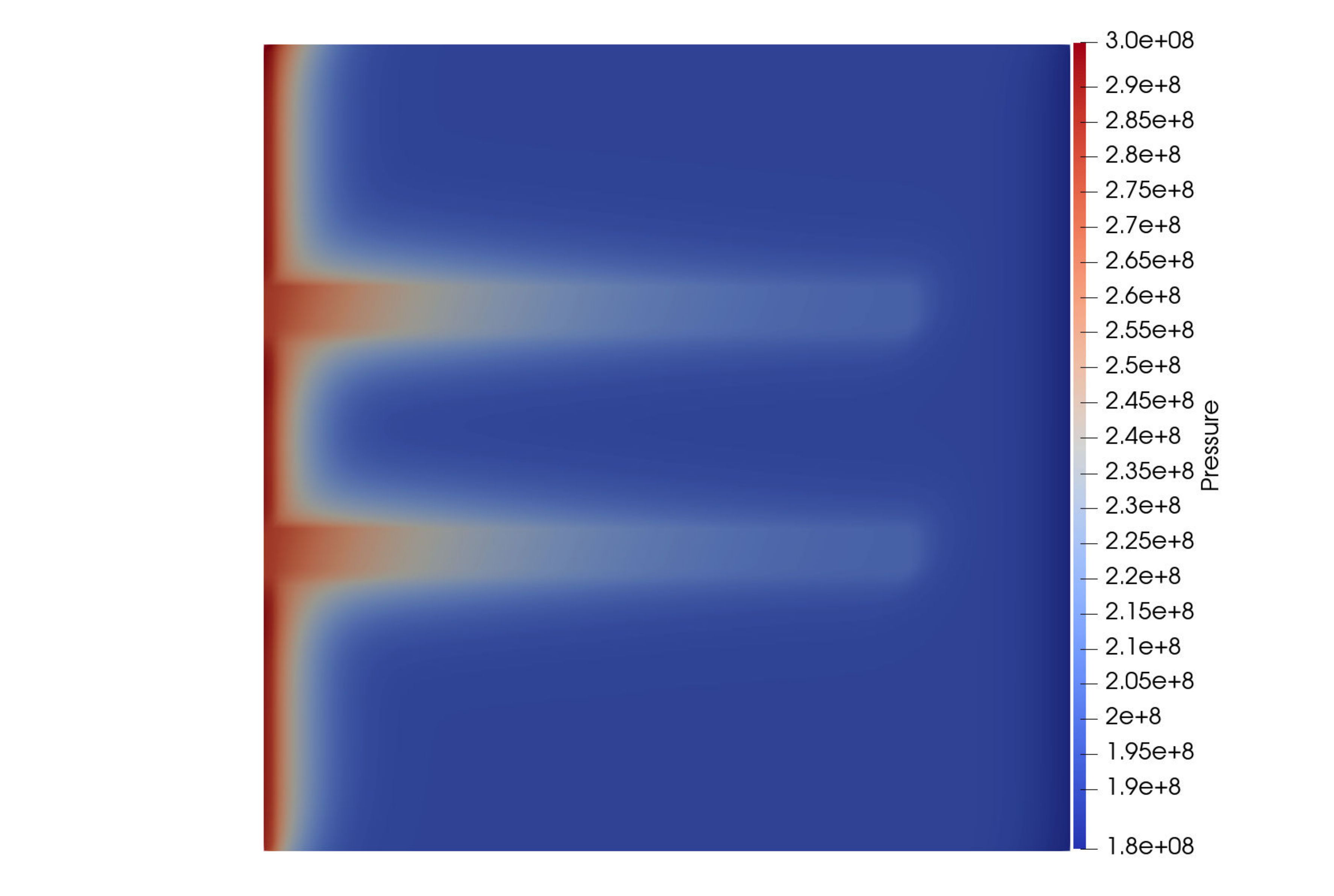}
	\includegraphics[width=5.5cm, height=4cm,trim=0.5cm 0cm 0cm 0cm,clip]{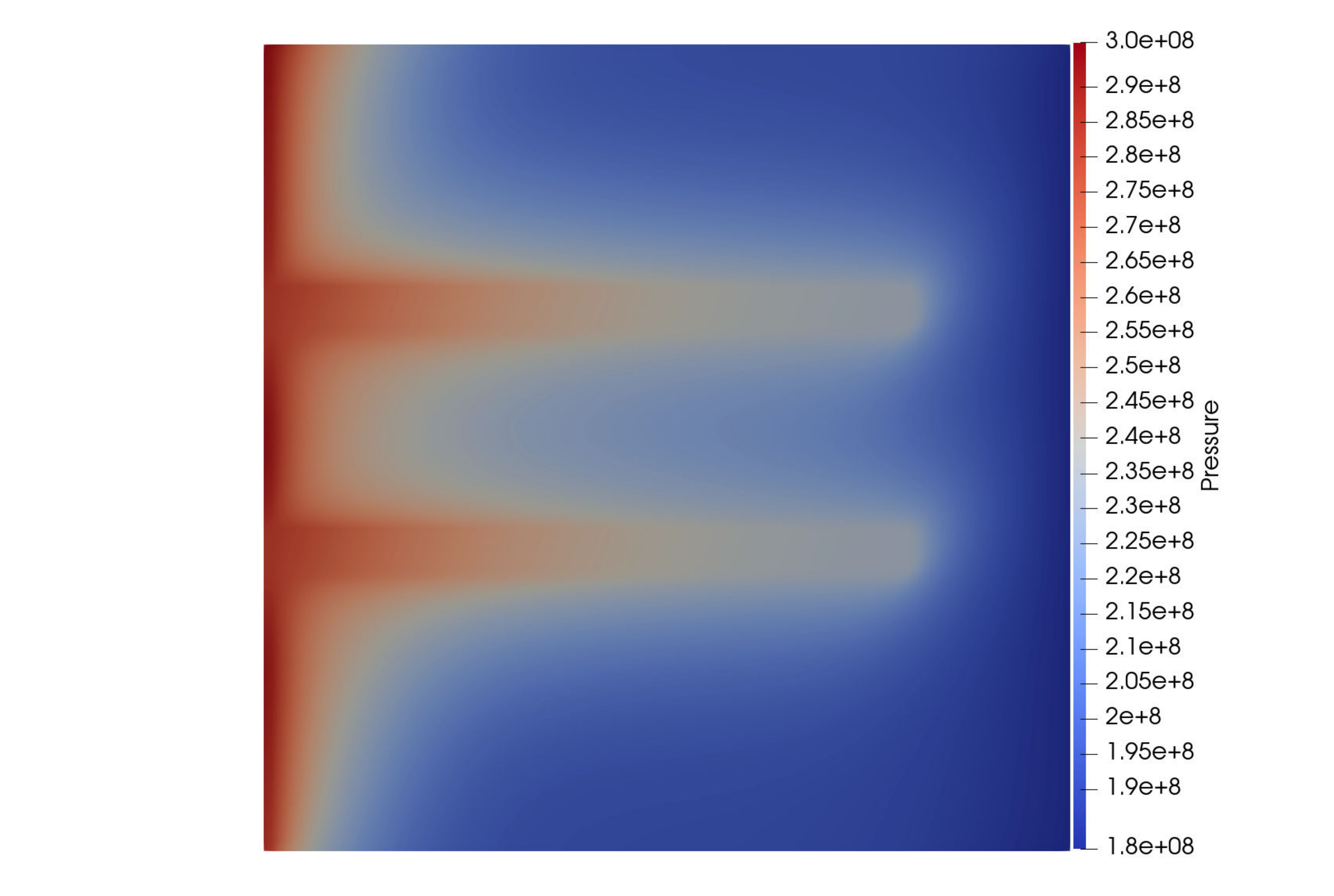}
	
	\includegraphics[width=5.5cm, height=4cm,trim=0.5cm 0cm 0cm 0cm,clip]{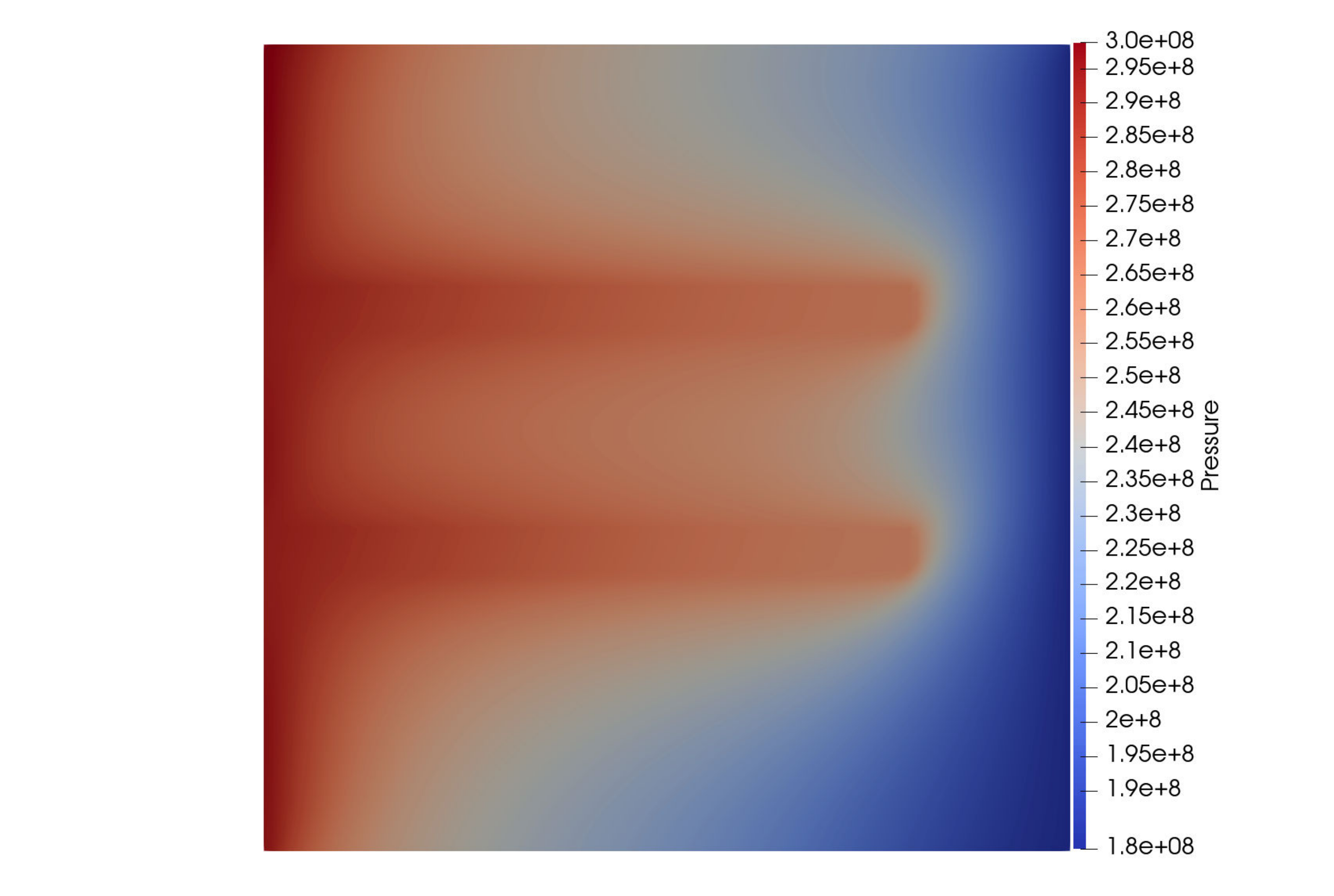}
	\includegraphics[width=5.5cm, height=4cm,trim=0.5cm 0cm 0cm 0cm,clip]{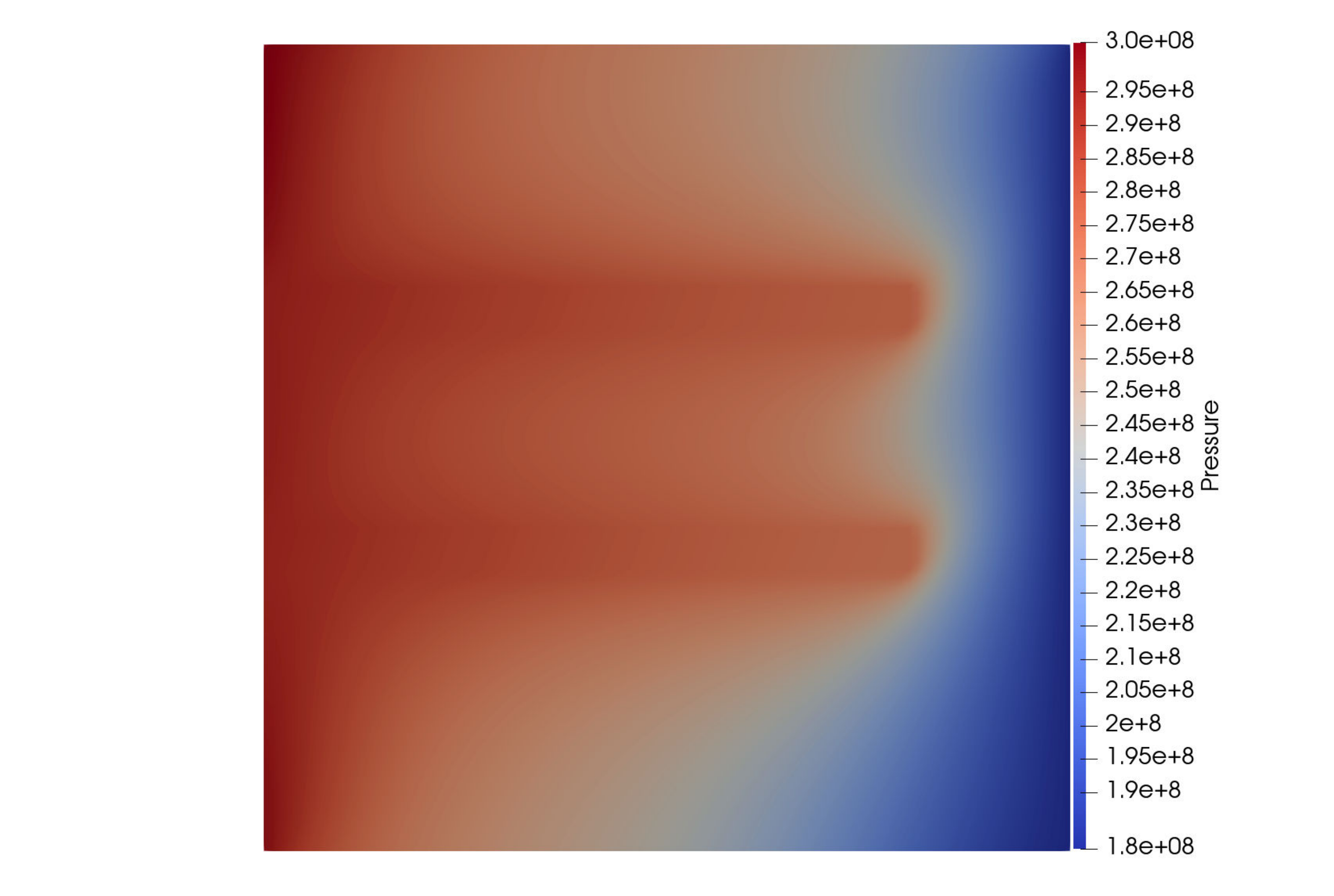}
	\caption{Distributions of pressure at different times in Example 2. Top-left: $n = 50$. Top-right: $n = 150$. Bottom-left: $n = 350$. Bottom-right: $n = 500$.}\label{fig2-pres}
\end{figure}
\begin{figure}[htbp]
	\centering
	\includegraphics[width=5.5cm, height=4cm,trim=0.5cm 0cm 0cm 0cm,clip]{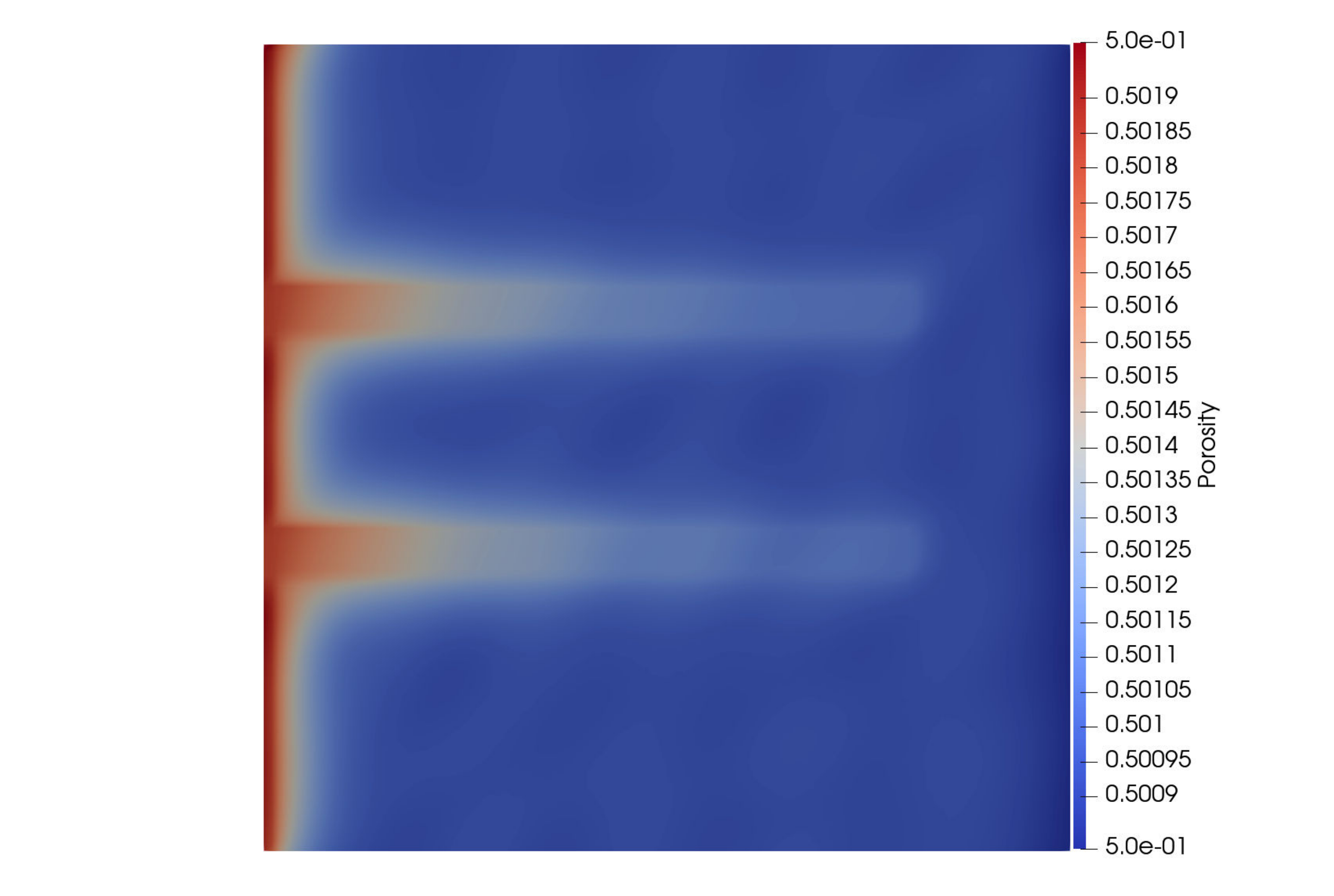}
	\includegraphics[width=5.5cm, height=4cm,trim=0.5cm 0cm 0cm 0cm,clip]{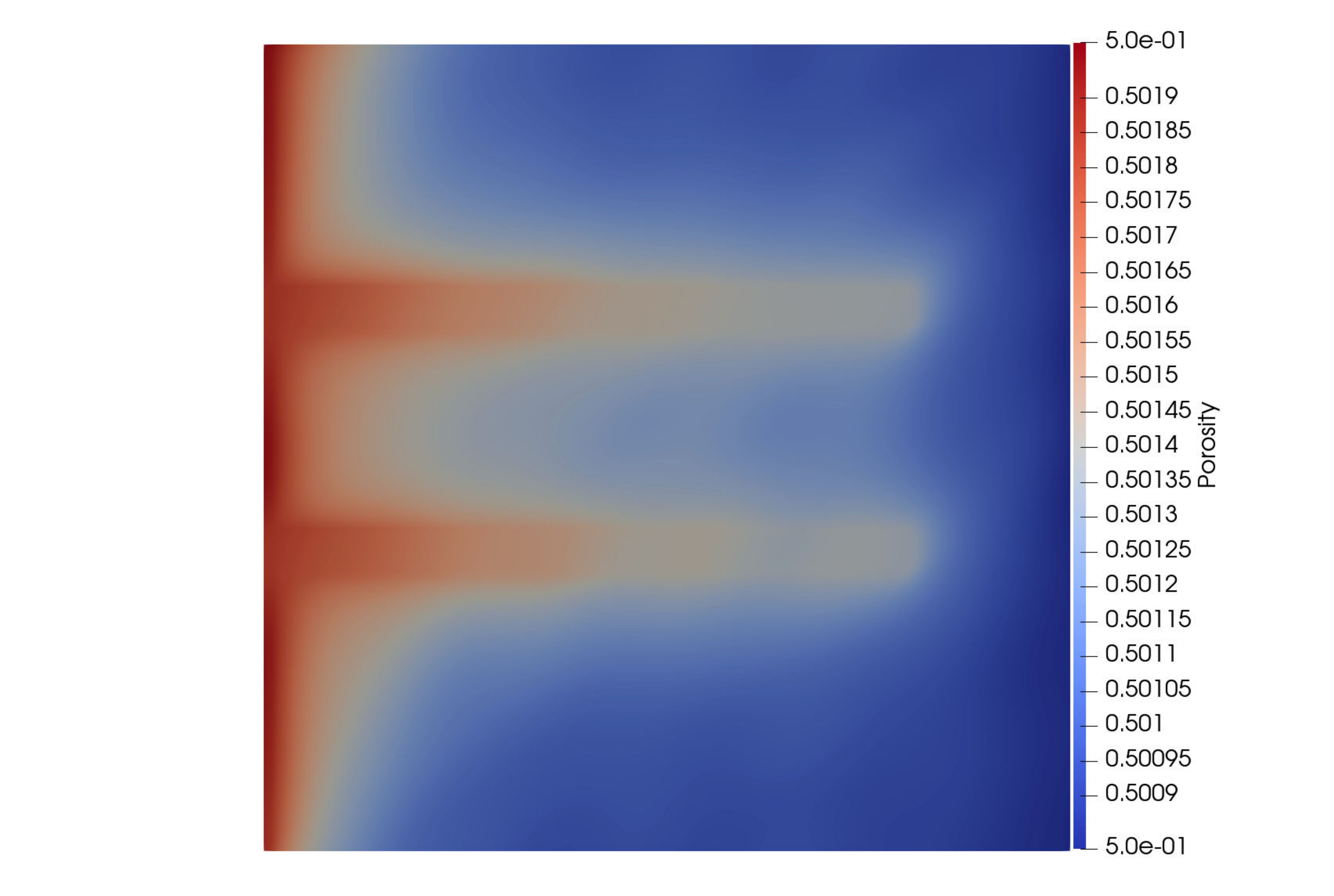}
	
	\includegraphics[width=5.5cm, height=4cm,trim=0.5cm 0cm 0cm 0cm,clip]{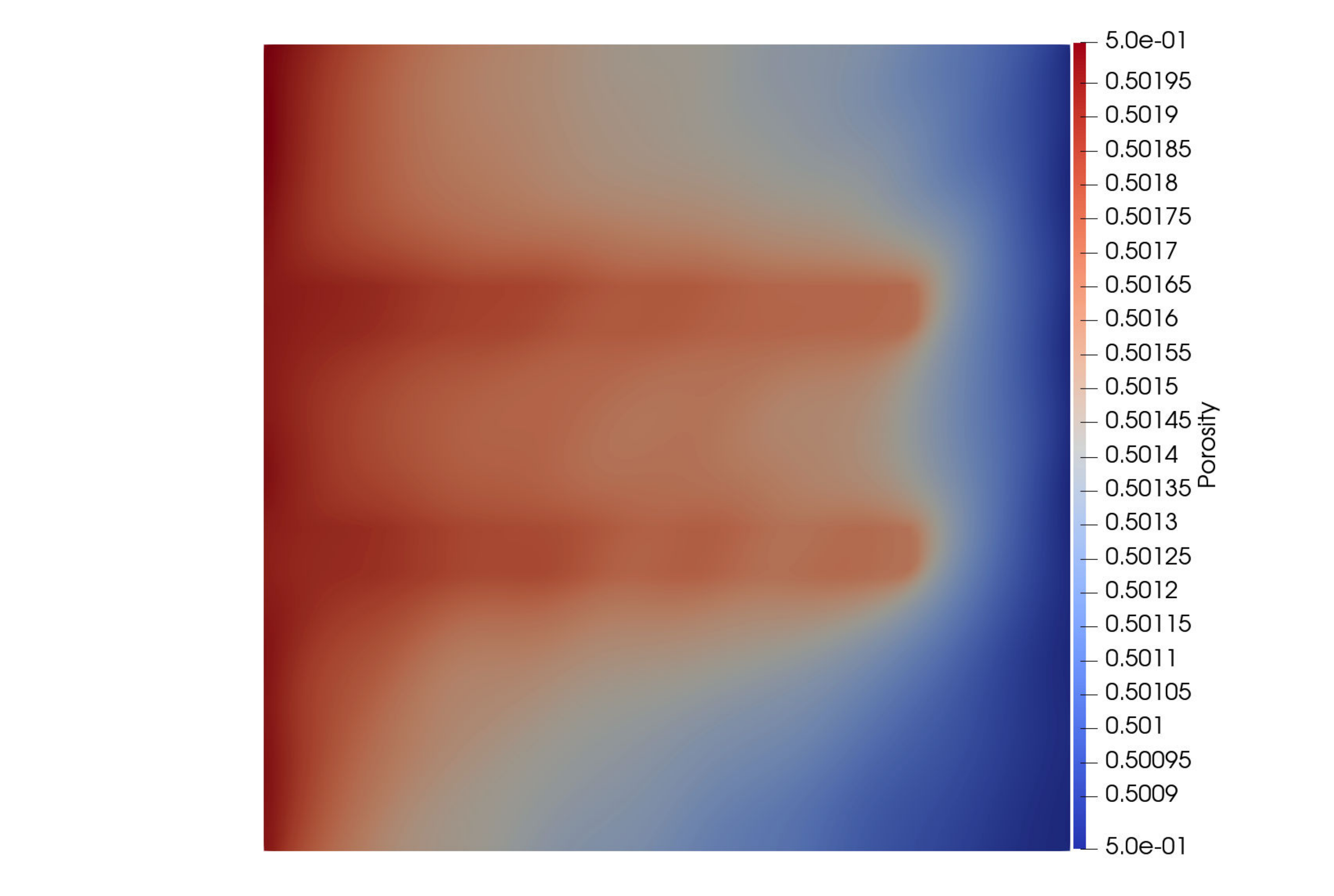}
	\includegraphics[width=5.5cm, height=4cm,trim=0.5cm 0cm 0cm 0cm,clip]{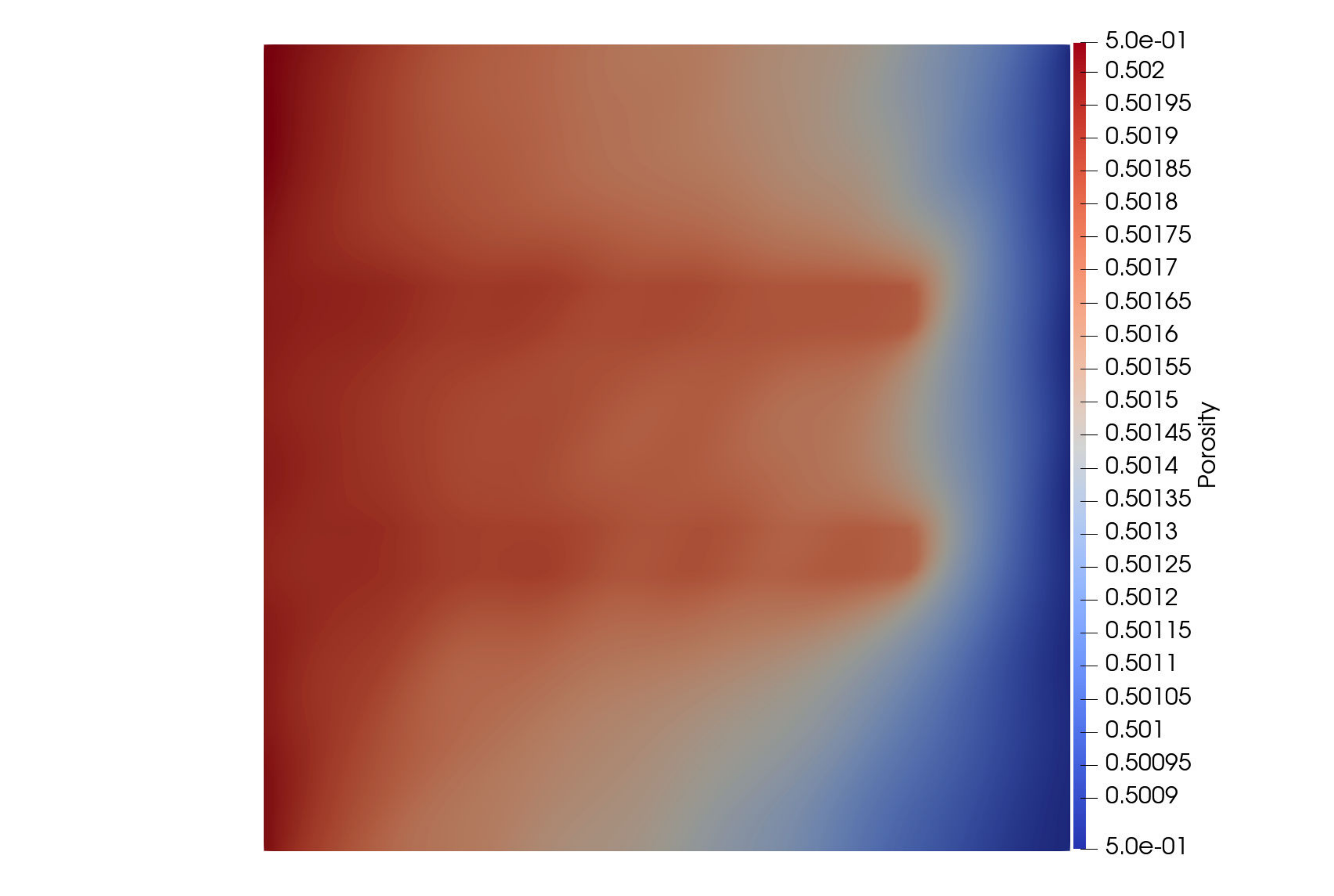}
	\caption{Distributions of porosity at different times in Example 2. Top-left: $n = 50$. Top-right: $n = 150$. Bottom-left: $n = 350$. Bottom-right: $n = 500$.}\label{fig2-poro}
\end{figure}
\subsection{Example 3}	

In this three-dimensional numerical experiment, we investigate the multicomponent gas flow process within a highly heterogeneous porous medium. The computational domain is extracted from the benchmark SPE10 geological model, which provides realistic spatial distributions of permeability and porosity representative of actual reservoir formations. The original SPE10 domain extends over $x \in [0, 60]~\mathrm{m}$, $y \in [0, 220]~\mathrm{m}$, and $z \in [0, 85]~\mathrm{m}$. To construct the present example, we extract the permeability and porosity data along the $x$-direction from the SPE10 model, forming a three-dimensional subregion defined by 
\[
x \in [15, 45]~\mathrm{m}, \quad
y \in [95, 125]~\mathrm{m}, \quad
z \in [0, 30]~\mathrm{m}.
\]
The computational mesh consists of $40 \times 40 \times 40$ tetrahedral elements, providing adequate spatial resolution to represent the fine-scale heterogeneity of the extracted subdomain. The permeability field in the selected subdomain spans several orders of magnitude, ranging from approximately $0.038~\mathrm{md}$ to $2.0\times10^4~\mathrm{md}$, reflecting a highly heterogeneous medium. The corresponding porosity field, obtained from the same dataset, exhibits spatial variability correlated with the permeability structure. The initial permeability and porosity distributions are illustrated in Figures~\ref{fig3-initial} and \ref{fig3-initial-slice}.

The simulated system involves a binary gas mixture composed of carbon dioxide (CO$_2$) and methane (CH$_4$). Initially, the domain is fully saturated with methane, with a uniform molar concentration of $C_{{CH_4}}^0 = 50~\mathrm{mol/m^3}$, while $C_{{CO_2}}^0 = 10~\mathrm{mol/m^3}$. A Dirichlet boundary condition is prescribed for the molar concentration of {CO2} at the bottom boundary ($z = 0$), where $C_{CO_2} = 50~\mathrm{mol/m^3}$, representing a continuous CO$_2$ injection process. All remaining boundaries are treated as impermeable (no-flux) conditions. The initial molar concentration distributions of both gas components are shown in Figure~\ref{fig3-initial-molar}.

As the simulation progresses, the injected CO$_2$ migrates upward, displacing the resident CH$_4$. Figures~\ref{fig3-co2} and~\ref{fig3-ch4} depict the temporal evolution of the molar concentration fields for CO$_2$ and CH$_4$, respectively, at representative time steps ($n = 50 , 500, 1000$, and $2000$). The results clearly demonstrate that CO$_2$ preferentially propagates along high-permeability channels.

The corresponding chemical potential fields of CO$_2$ and CH$_4$, presented in Figures~\ref{fig3-co2-che} and~\ref{fig3-ch4-che}, reveal the dominant thermodynamic gradients driving both advective and diffusive transport mechanisms. The evolution of the pressure field is illustrated in Figure~\ref{fig3-pres}. 	As illustrated in Figure~\ref{fig3-timestep}, the time step size fluctuates dynamically throughout the computation because the system does not reach a steady state within the simulated time frame. The spatial distributions of CO$_2$ and CH$_4$ from the displacement process at time step $n=5000$ are shown in Figure \ref{fig3-slice}, using both clip and slice views.

Overall, this three-dimensional numerical example demonstrates the capability of the proposed framework to accurately capture multicomponent gas transport in highly heterogeneous porous media.

\begin{figure}[htbp]
	\centering
	\includegraphics[width=5.5cm, height=4cm,trim=0.5cm 0cm 0cm 0cm,clip]{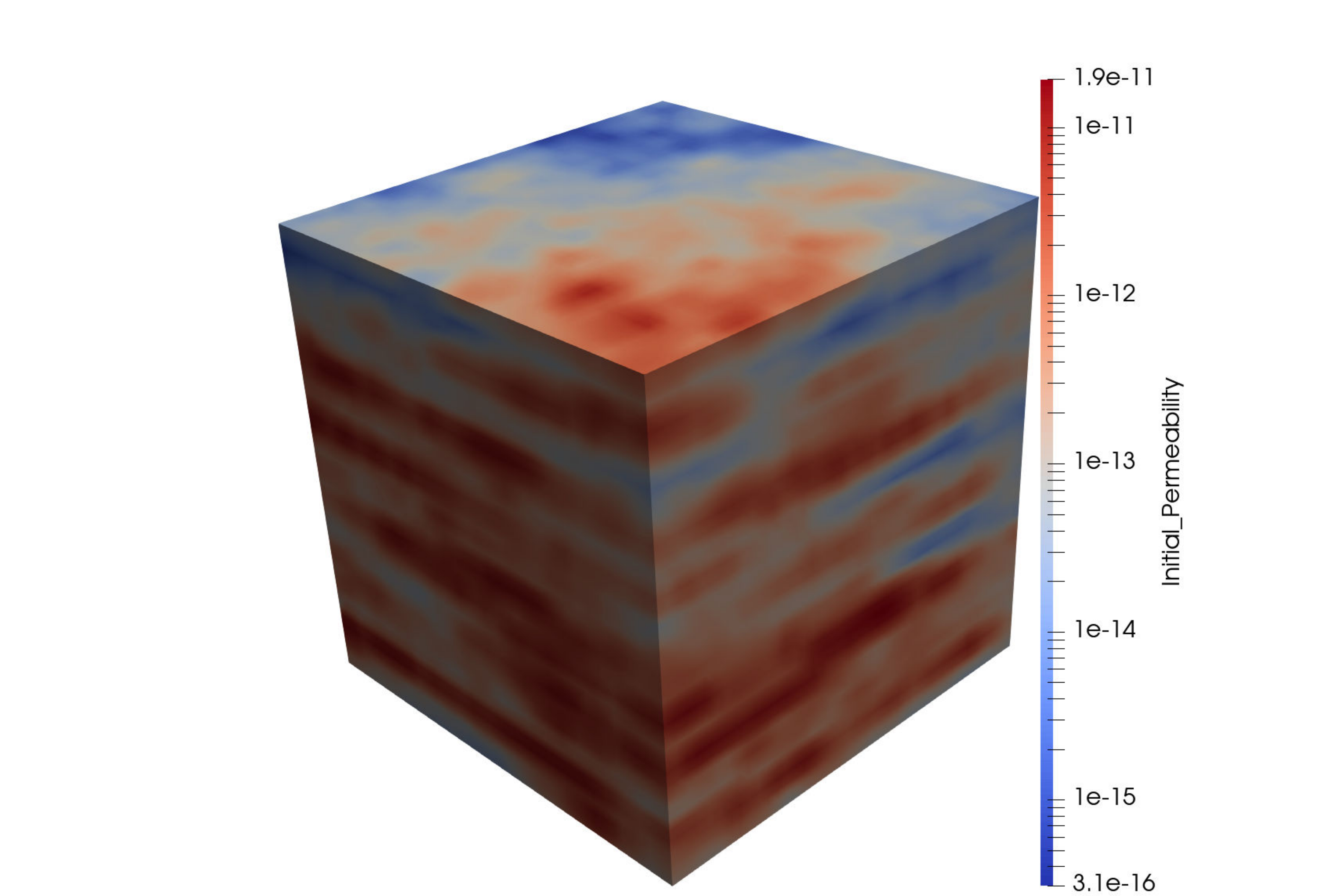}
	\includegraphics[width=5.5cm, height=4cm,trim=0.5cm 0cm 0cm 0cm,clip]{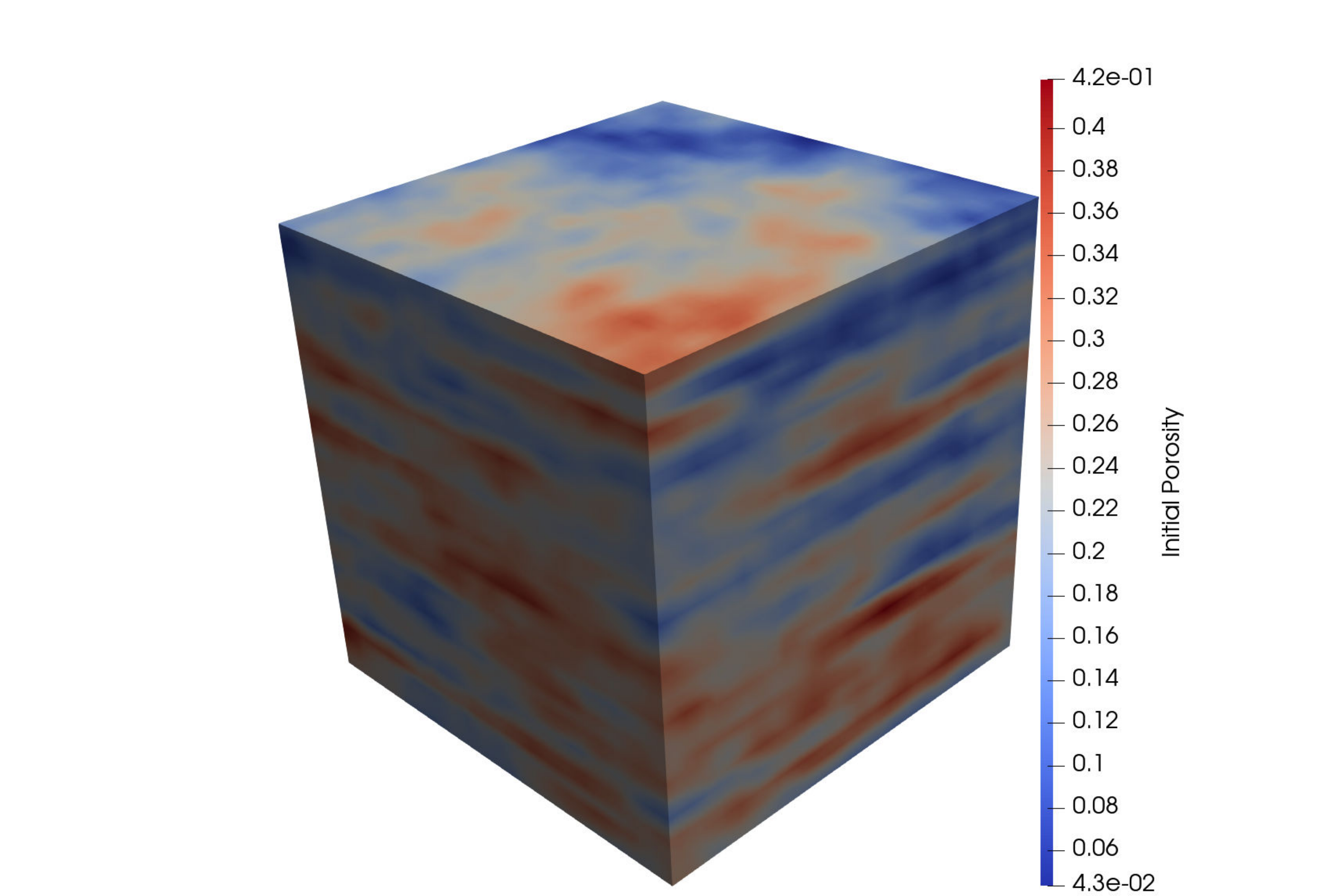}
	\caption{Example 3: Left: Initial distributions of permeability. Right: Initial distributions of porosity.}\label{fig3-initial}
\end{figure}
\begin{figure}[htbp]
	\centering
	\includegraphics[width=7cm, height=4cm, trim=6cm 0cm 6cm 0cm,clip]{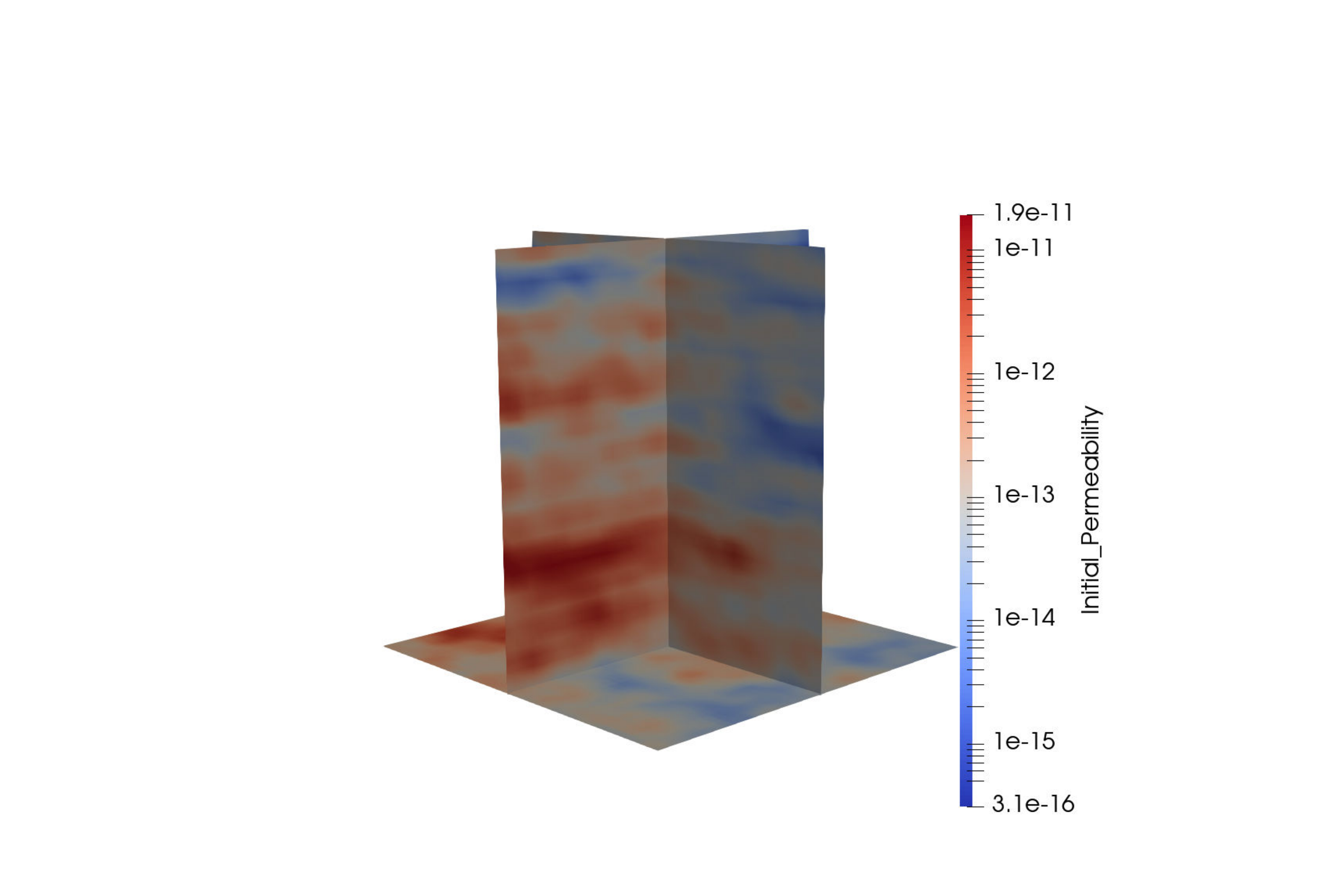}
	\includegraphics[width=7cm, height=4cm, trim=6cm 0cm 6cm 0cm,clip]{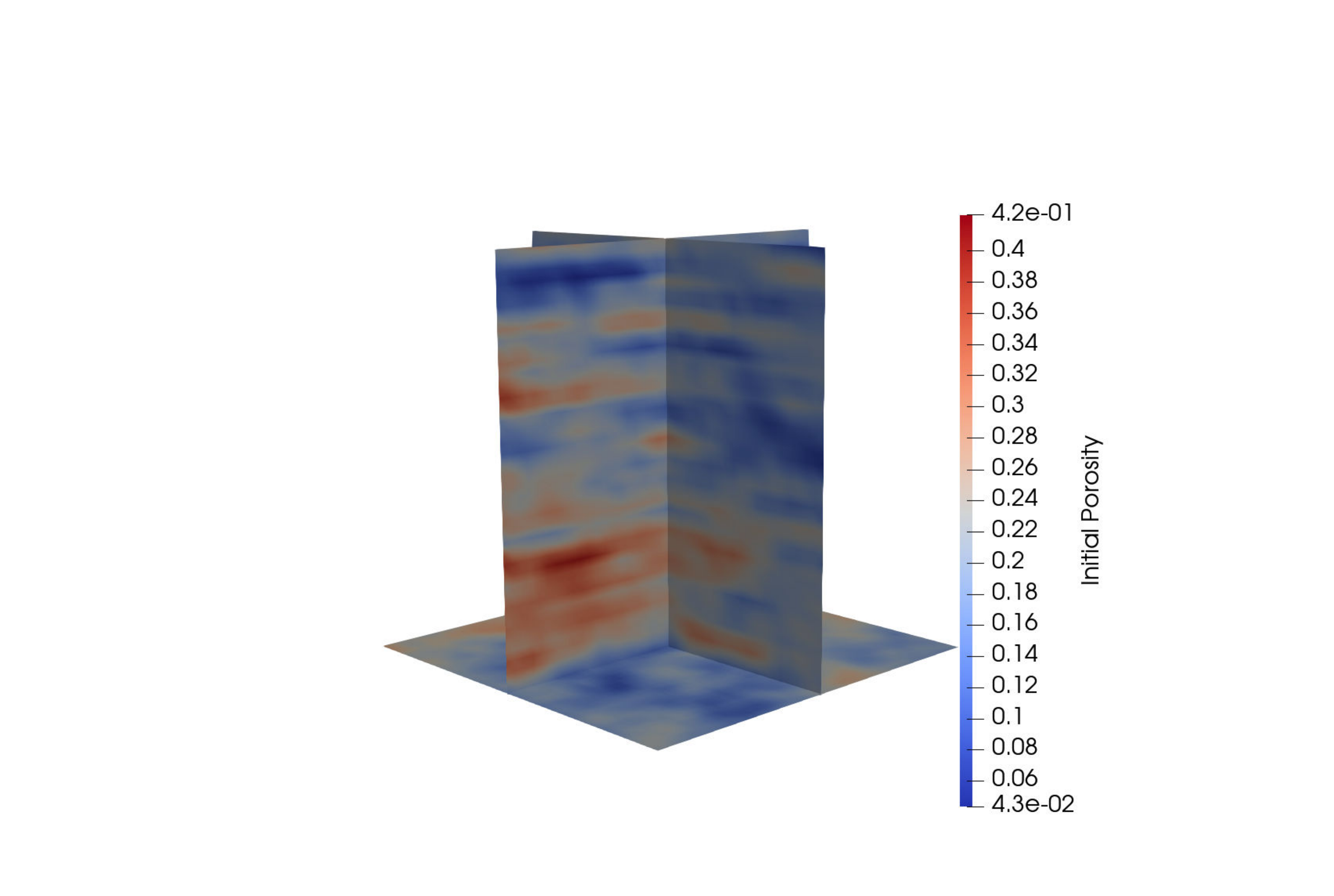}
	\caption{Example 3: Left: Slice of the initial permeability distribution. Right: Slice of the initial porosity distribution.}\label{fig3-initial-slice}
\end{figure}
\begin{figure}[htbp]
	\centering
	\includegraphics[width=5.0cm, height=4.5cm]{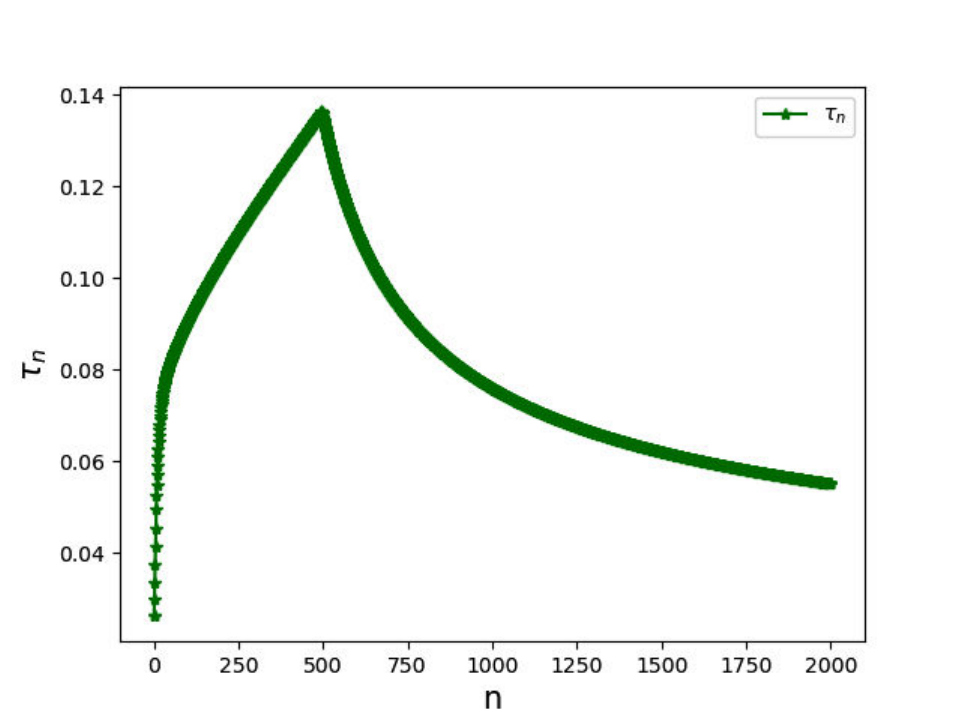}
	\caption{Example 3: Adaptive values of the time step size.}\label{fig3-timestep}
\end{figure}
\begin{figure}[htbp]
	\centering
	\includegraphics[width=7cm, height=4cm, trim=6cm 0cm 6cm 0cm,clip]{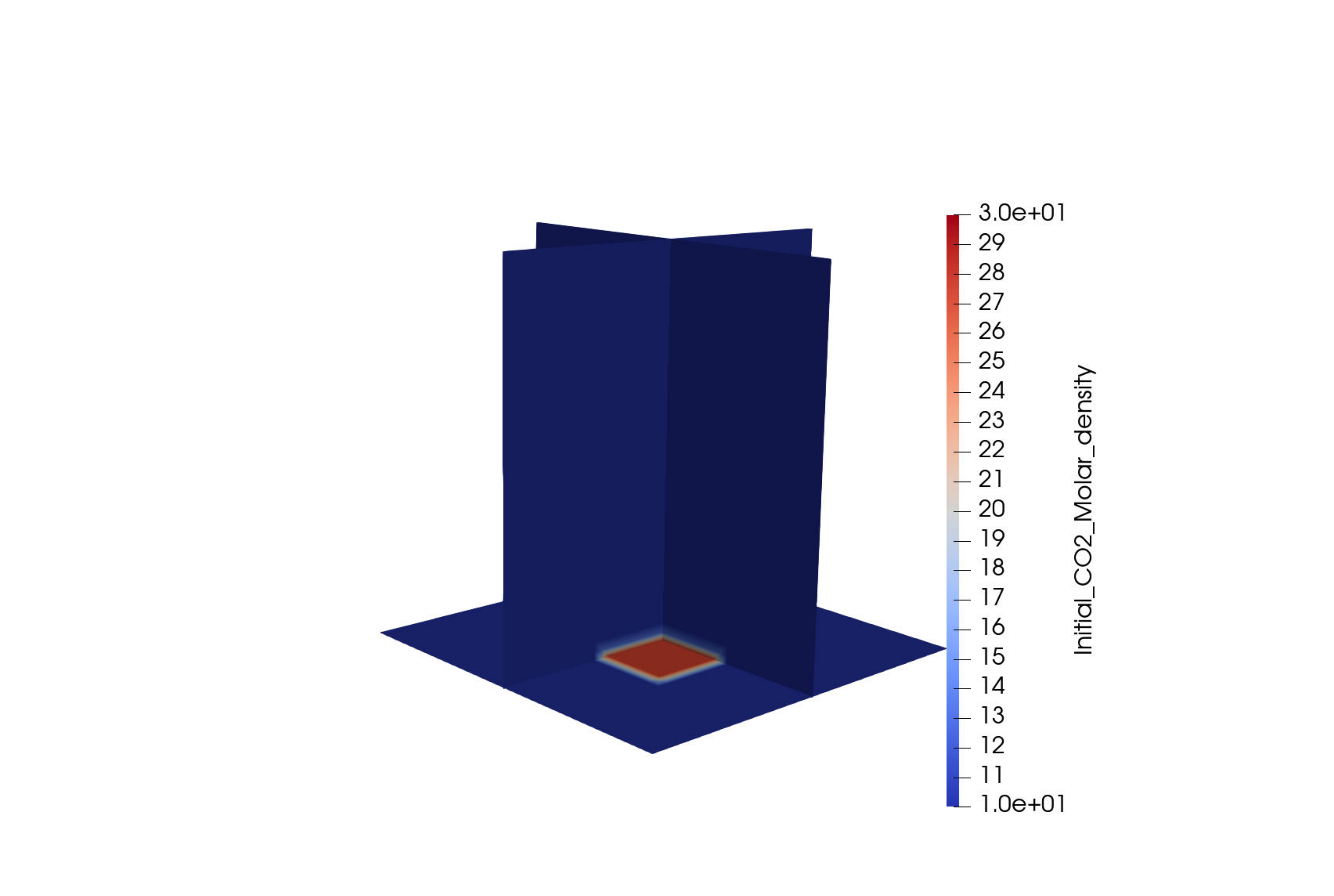}
	\includegraphics[width=7cm, height=4cm, trim=6cm 0cm 6cm 0cm,clip]{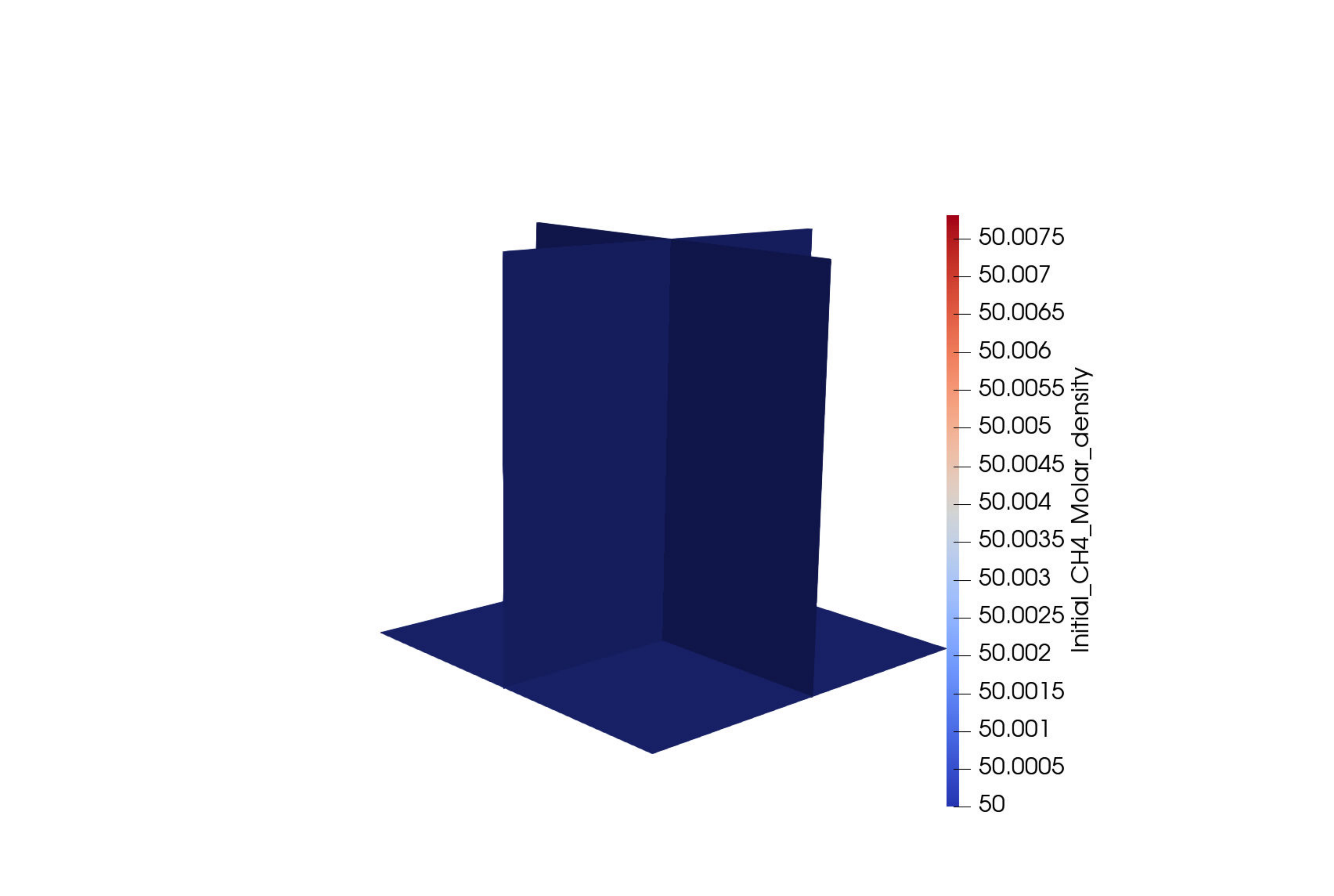}
	\caption{Example 3: Left: Initial molar distribution of CO$_2$. Right: Initial molar distribution of CH$_4$.}\label{fig3-initial-molar}
\end{figure}
\begin{figure}[htbp]
	\centering
	\includegraphics[width=7cm, height=4cm, trim=6cm 0cm 6cm 0cm,clip]{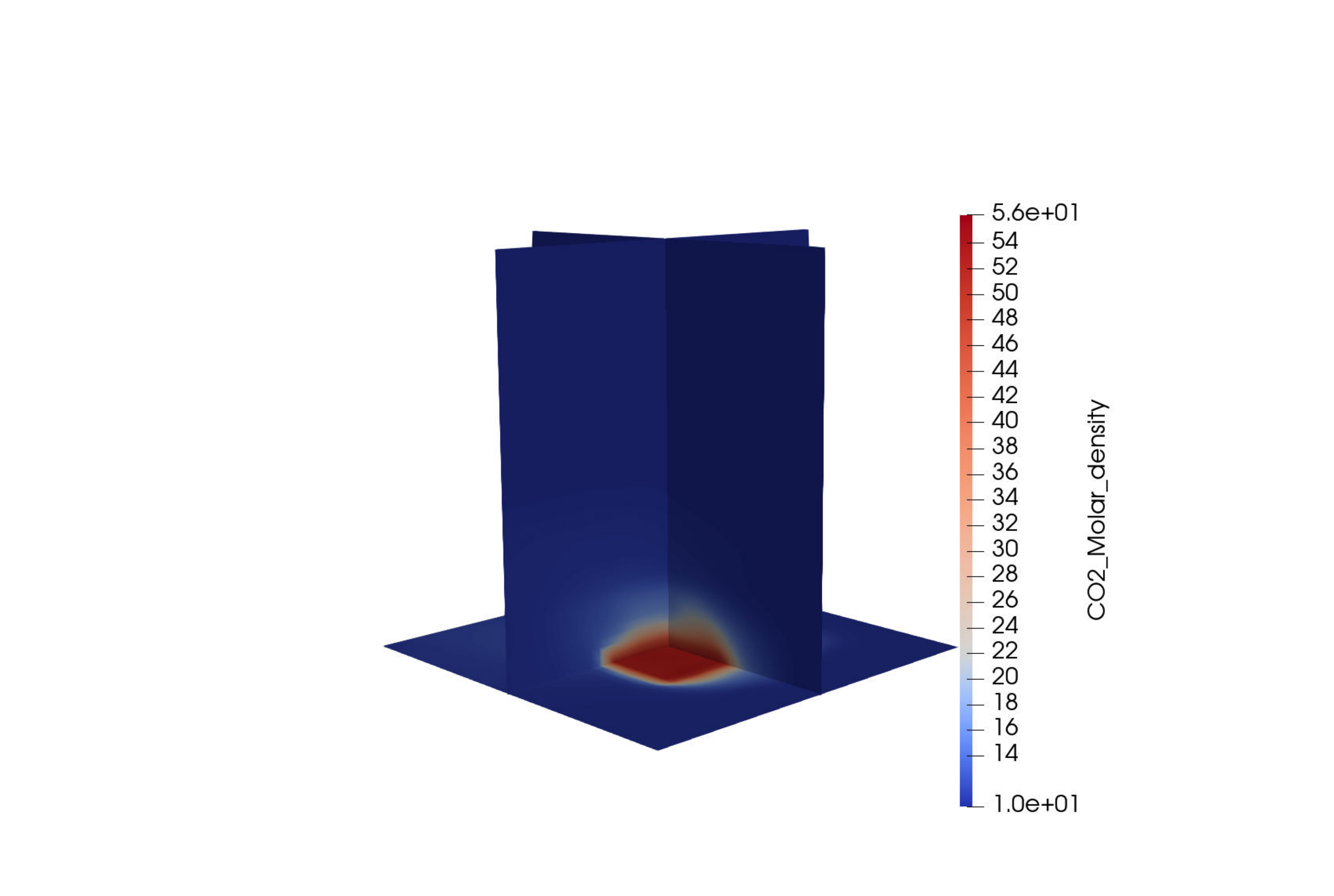}
	\includegraphics[width=7cm, height=4cm, trim=6cm 0cm 6cm 0cm,clip]{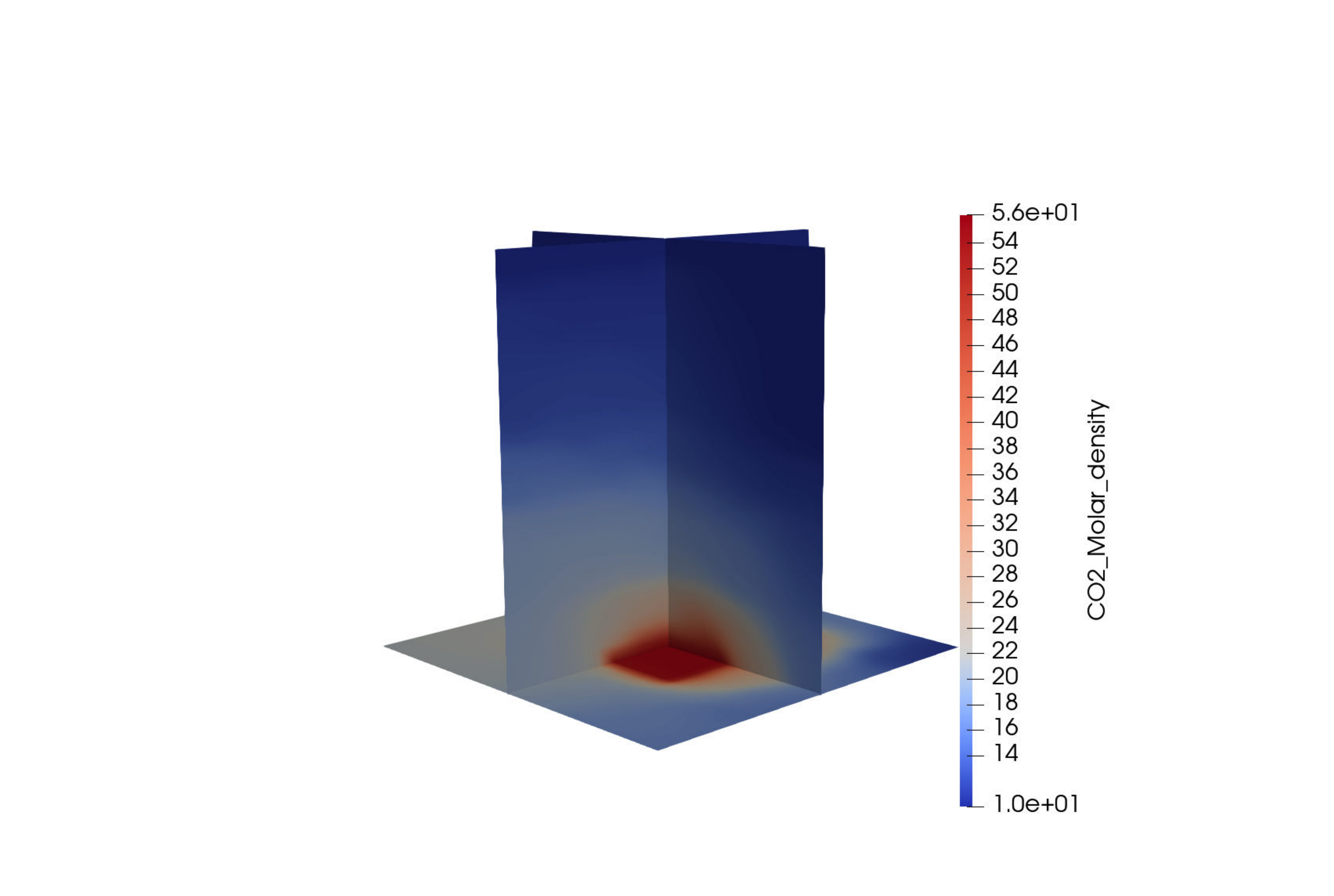}
	
	\includegraphics[width=7cm, height=4cm, trim=6cm 0cm 6cm 0cm,clip]{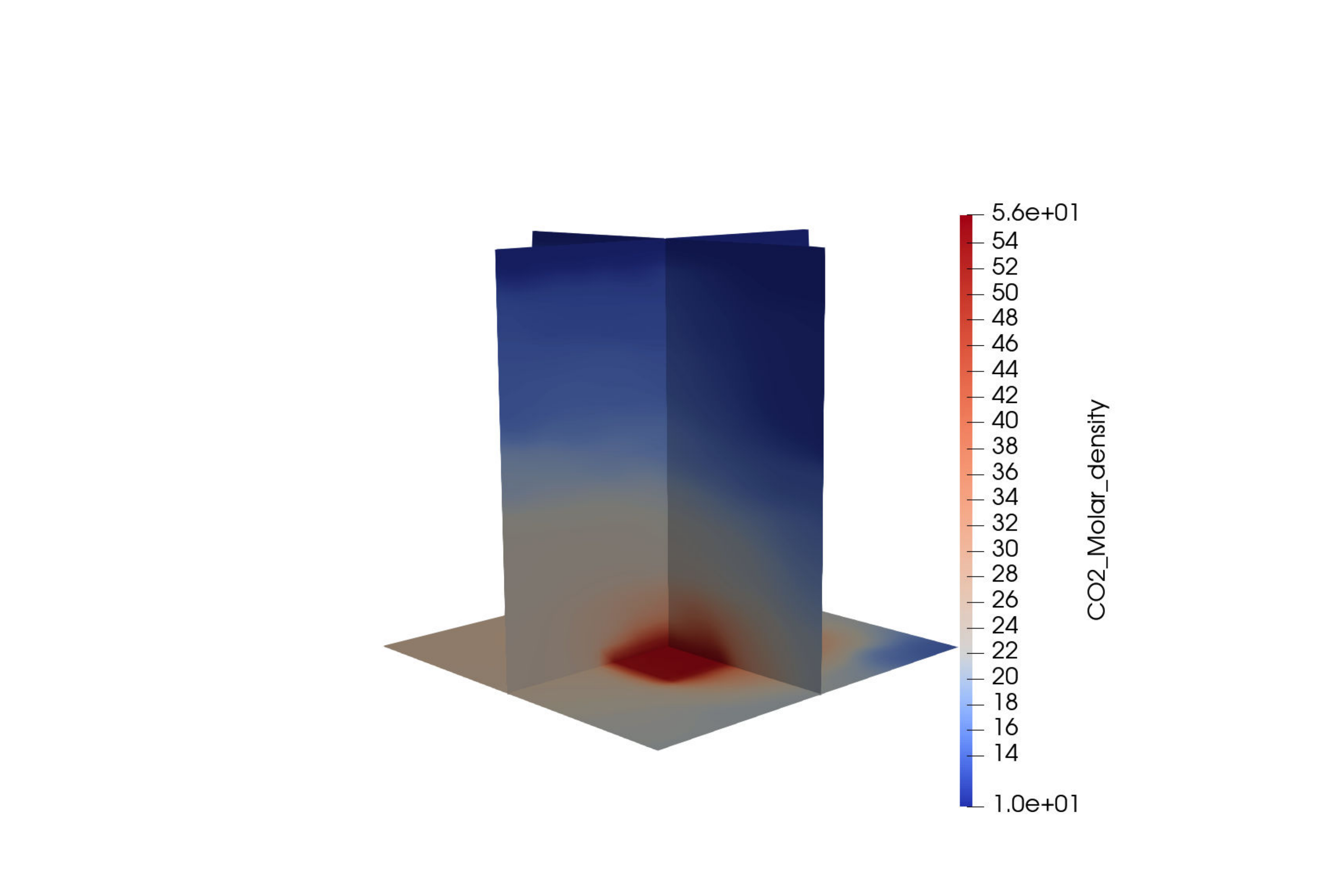}
	\includegraphics[width=7cm, height=4cm, trim=6cm 0cm 6cm 0cm,clip]{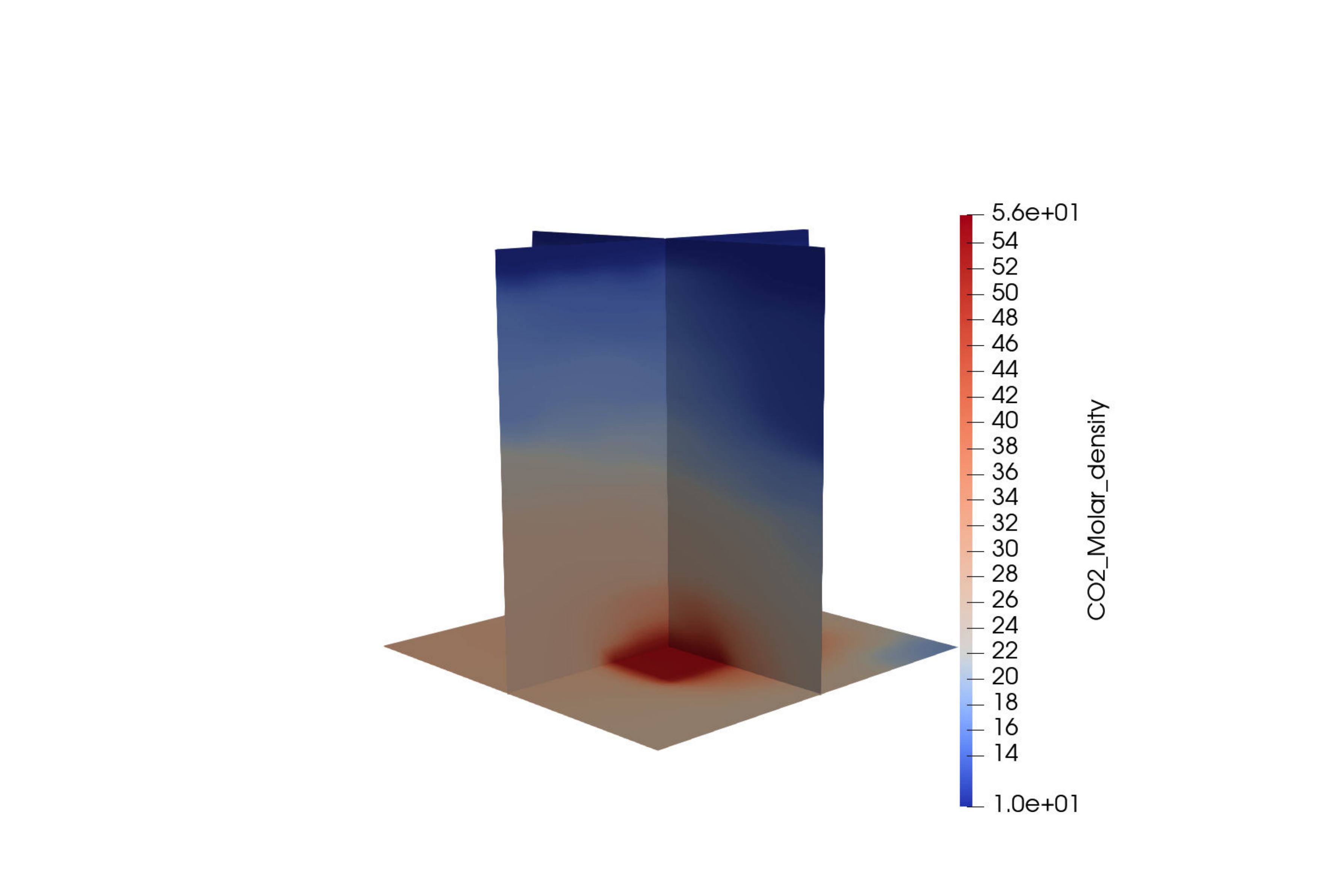}
	\caption{Distributions of molar density of CO$_2$ at different times in Example 3. Top-left: $n = 50$. Top-right: $n = 500$. Bottom-left: $n = 1000$. Bottom-right: $n = 2000$.}\label{fig3-co2}
\end{figure}

\begin{figure}[htbp]
	\centering
	\includegraphics[width=7cm, height=4cm, trim=6cm 0cm 6cm 0cm,clip]{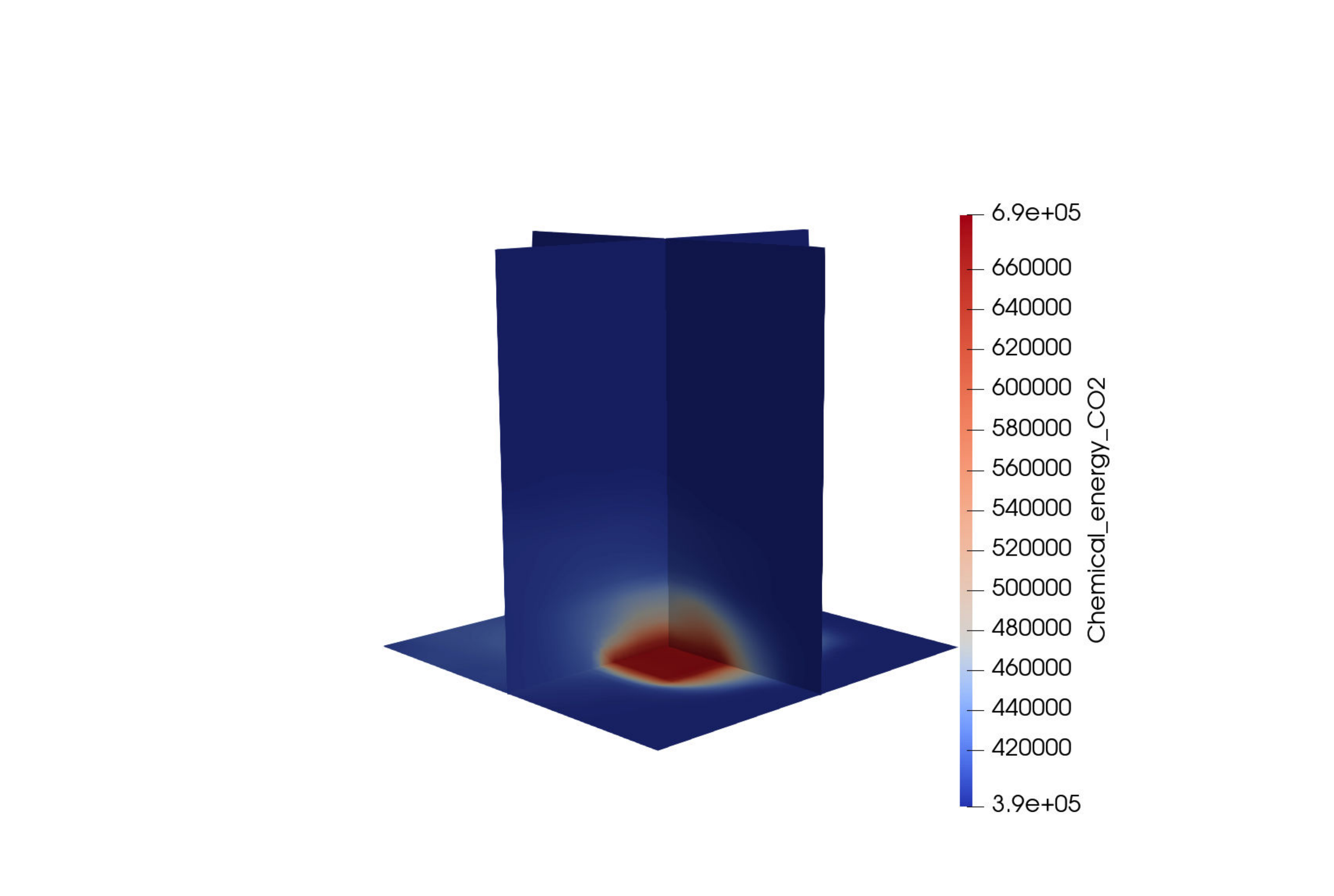}
	\includegraphics[width=7cm, height=4cm, trim=6cm 0cm 6cm 0cm,clip]{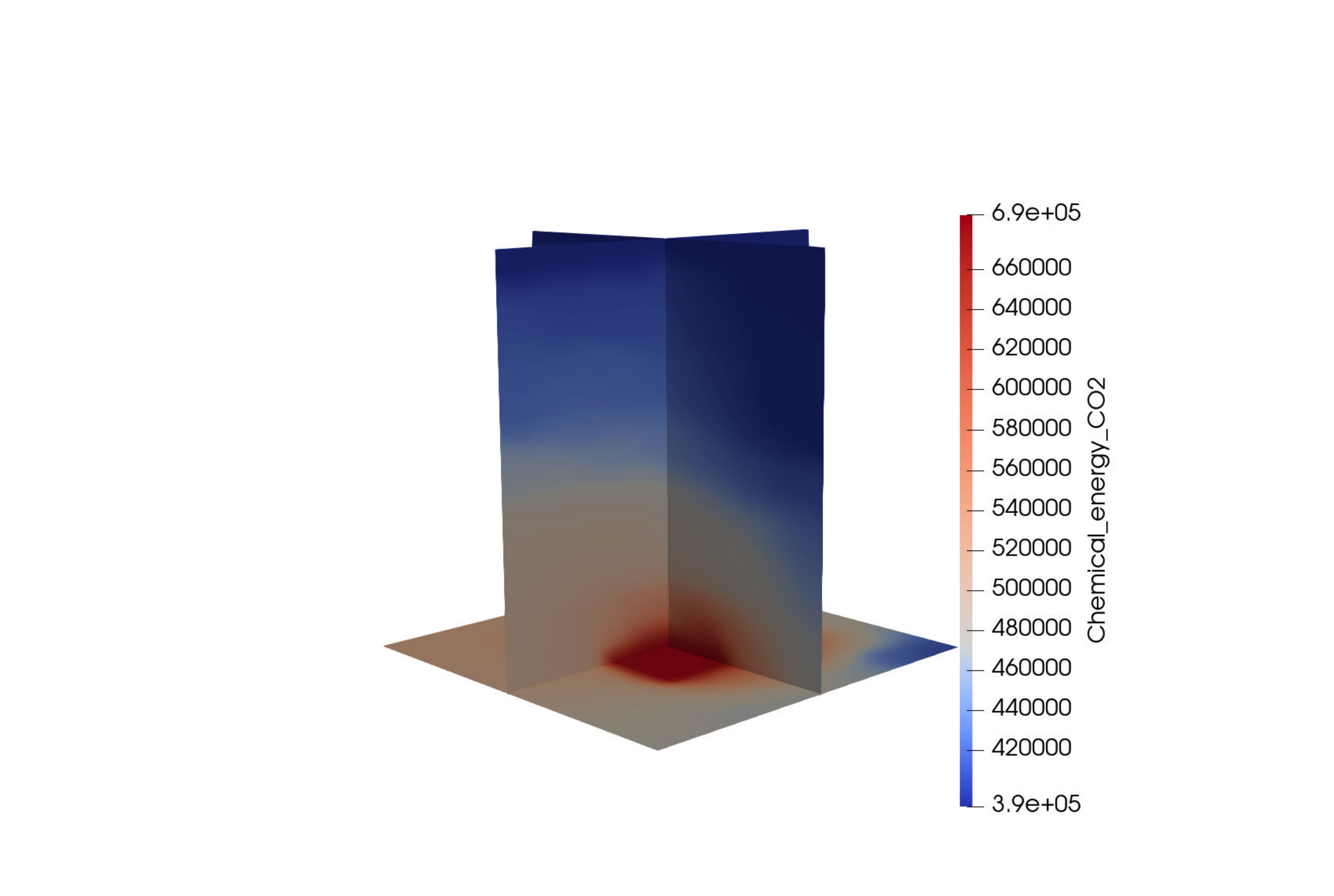}
	
	\includegraphics[width=7cm, height=4cm, trim=6cm 0cm 6cm 0cm,clip]{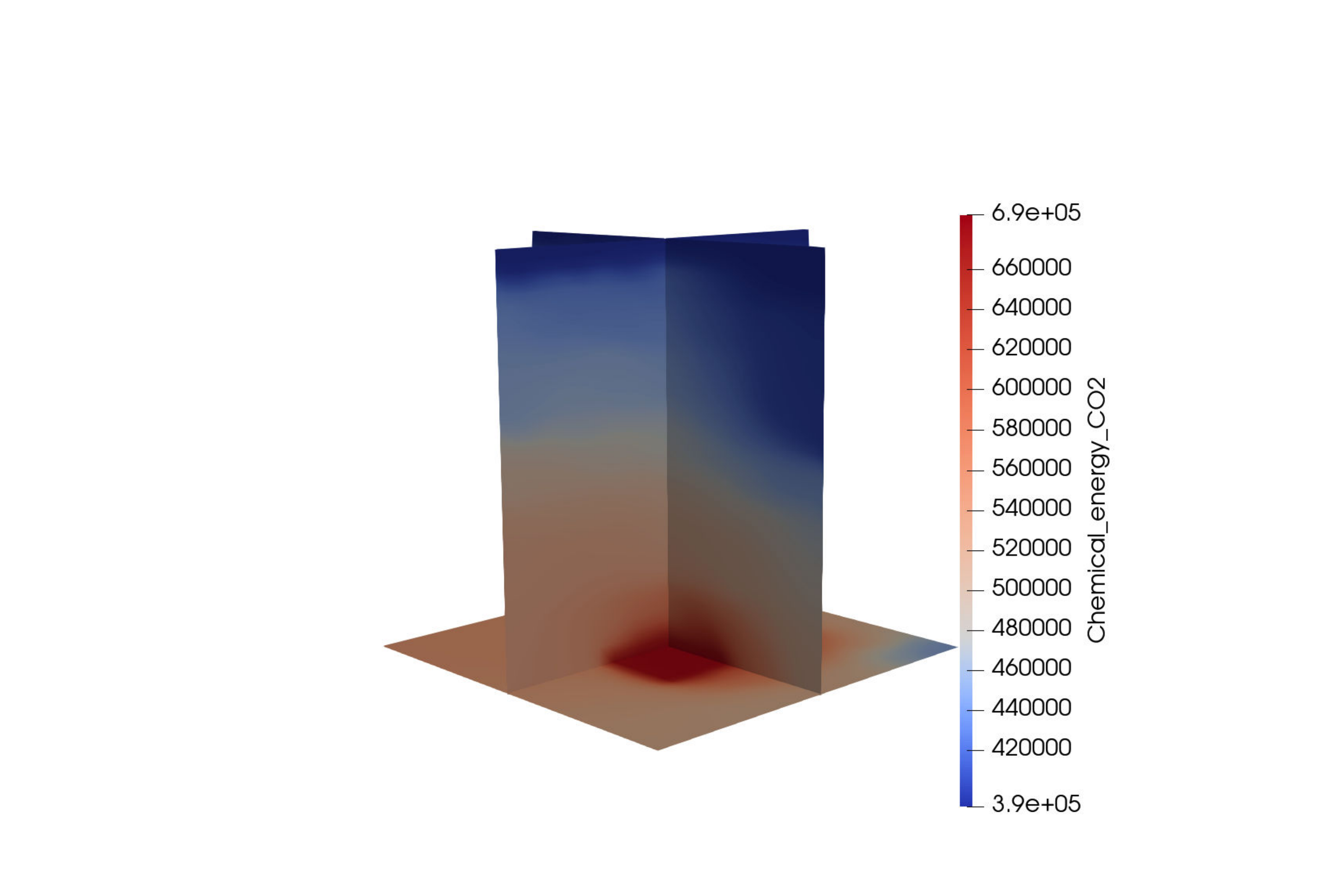}
	\includegraphics[width=7cm, height=4cm, trim=6cm 0cm 6cm 0cm,clip]{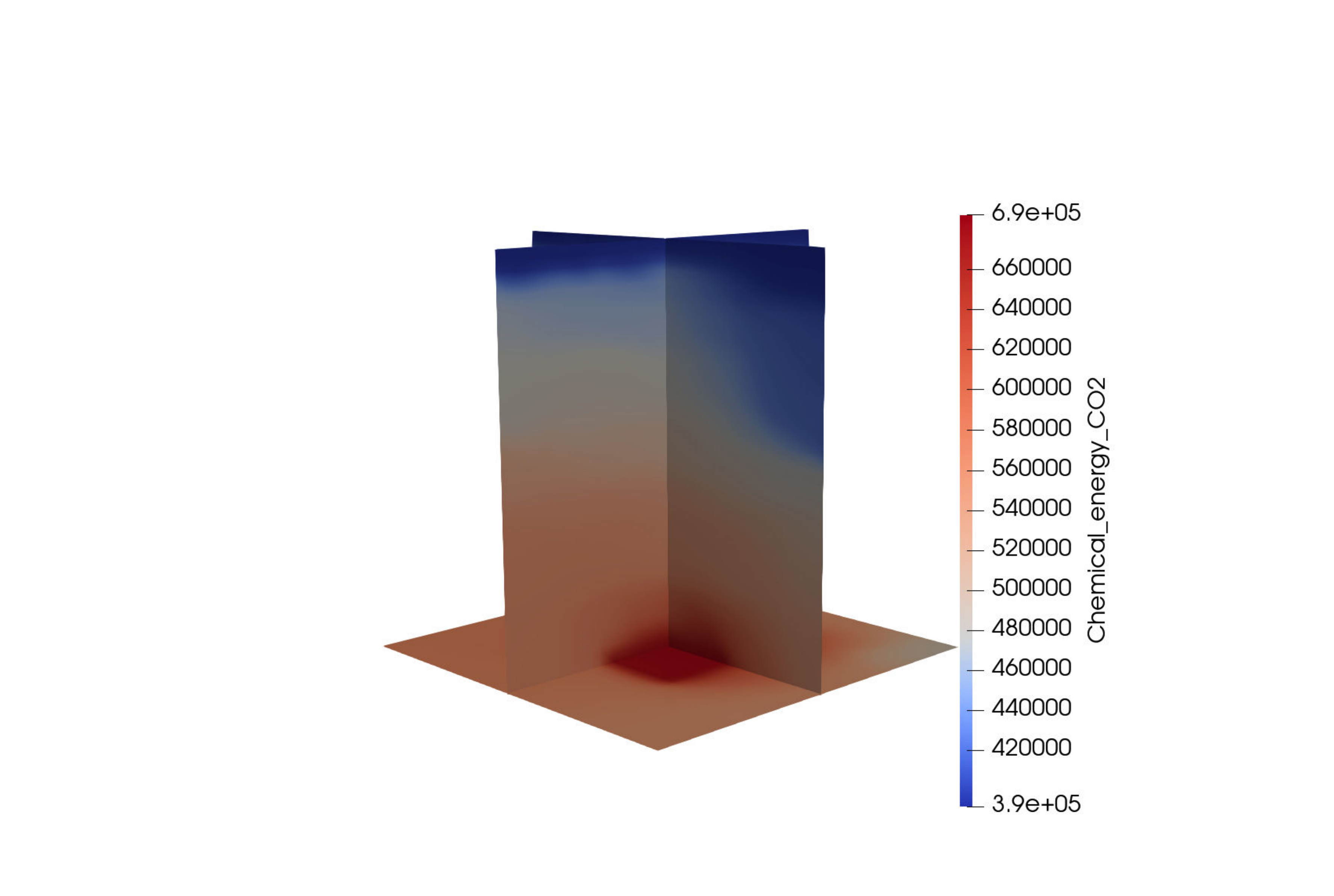}
	\caption{Distributions of chemical potential of CO$_2$ at different times in Example 3. Top-left: Top-left: $n = 50$. Top-right: $n = 500$. Bottom-left: $n = 1000$. Bottom-right: $n = 2000$.}\label{fig3-co2-che}
\end{figure}

\begin{figure}[htbp]
	\centering
	\includegraphics[width=7cm, height=4cm, trim=6cm 0cm 6cm 0cm,clip]{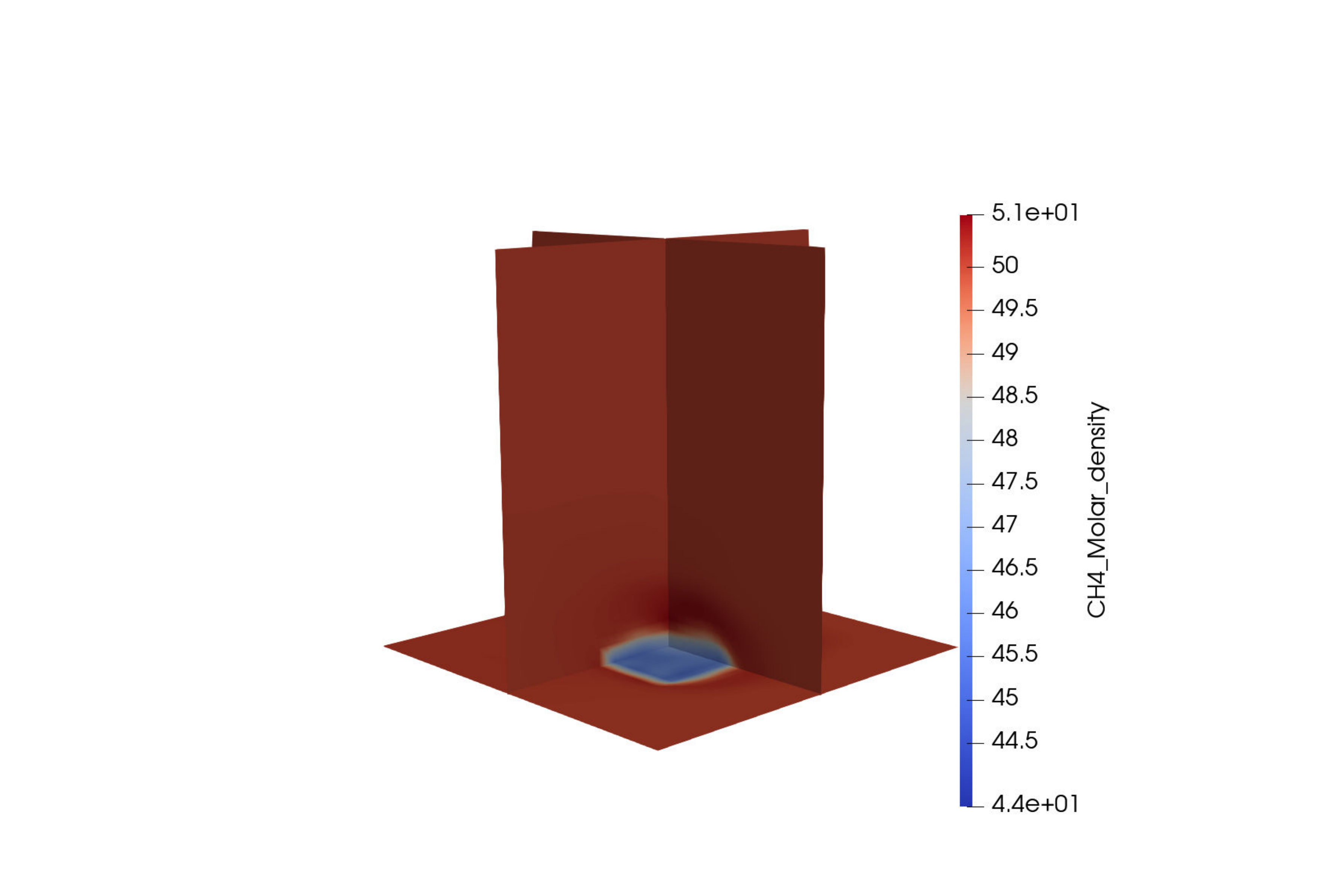}
	\includegraphics[width=7cm, height=4cm, trim=6cm 0cm 6cm 0cm,clip]{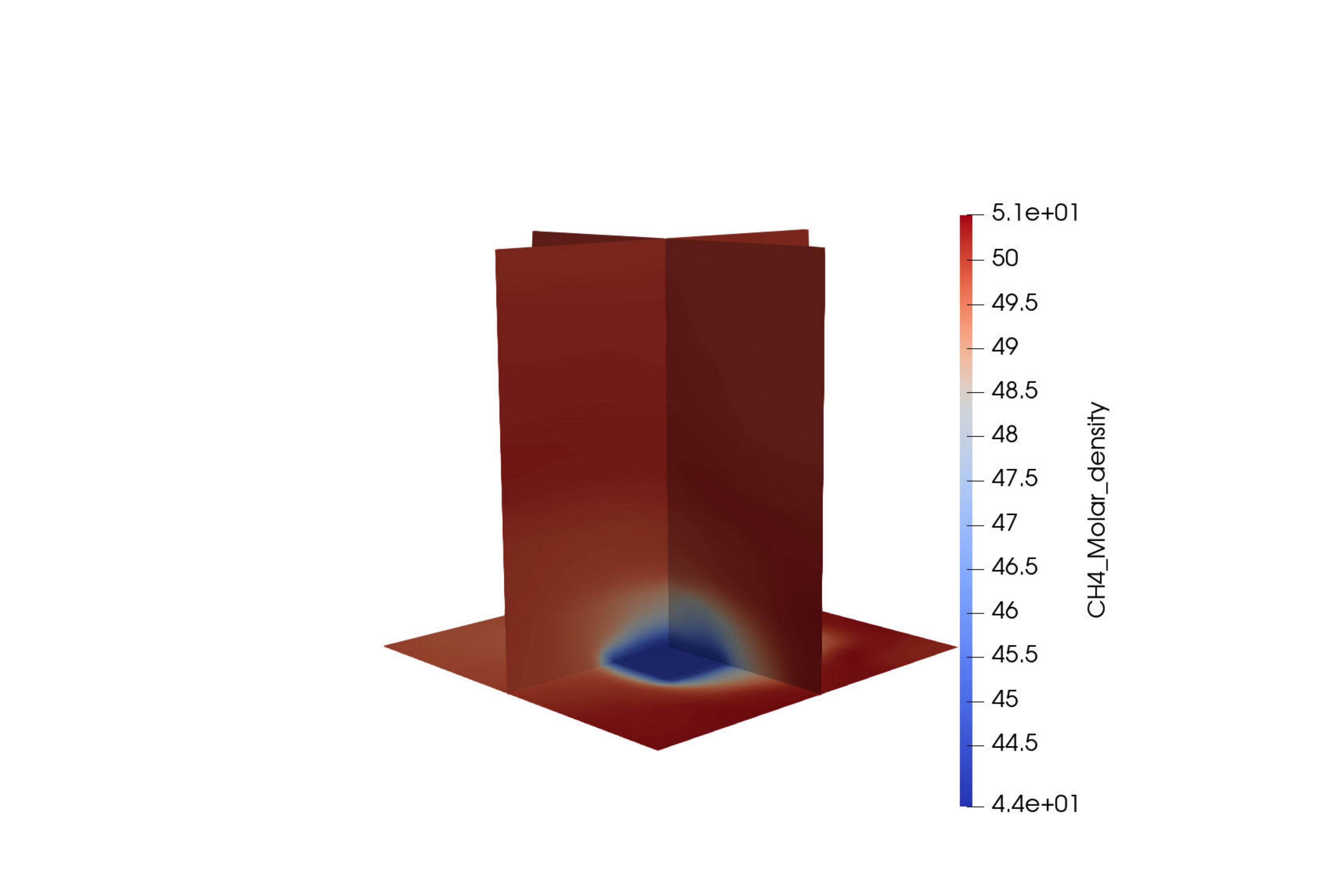}
	
	\includegraphics[width=7cm, height=4cm, trim=6cm 0cm 6cm 0cm,clip]{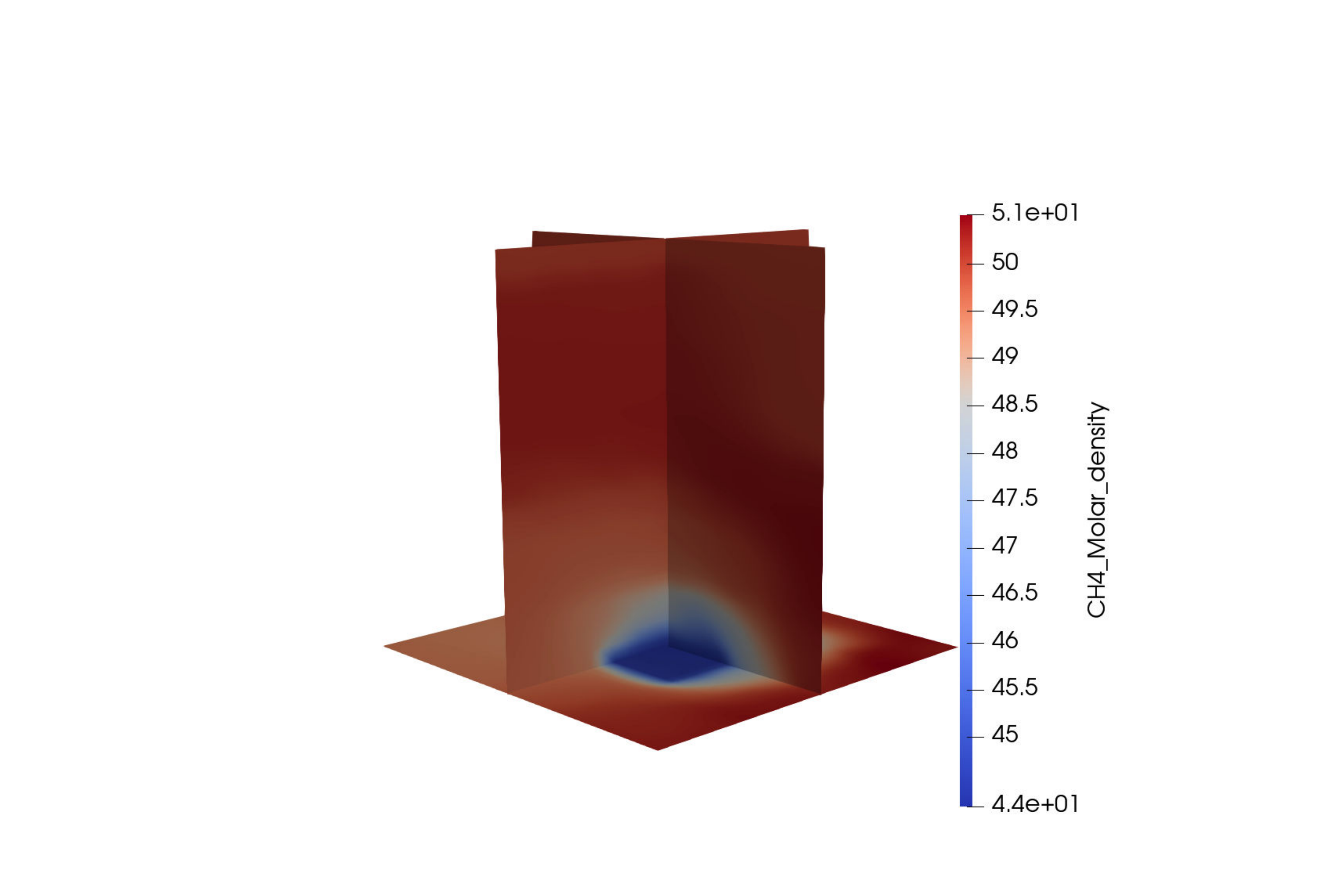}
	\includegraphics[width=7cm, height=4cm, trim=6cm 0cm 6cm 0cm,clip]{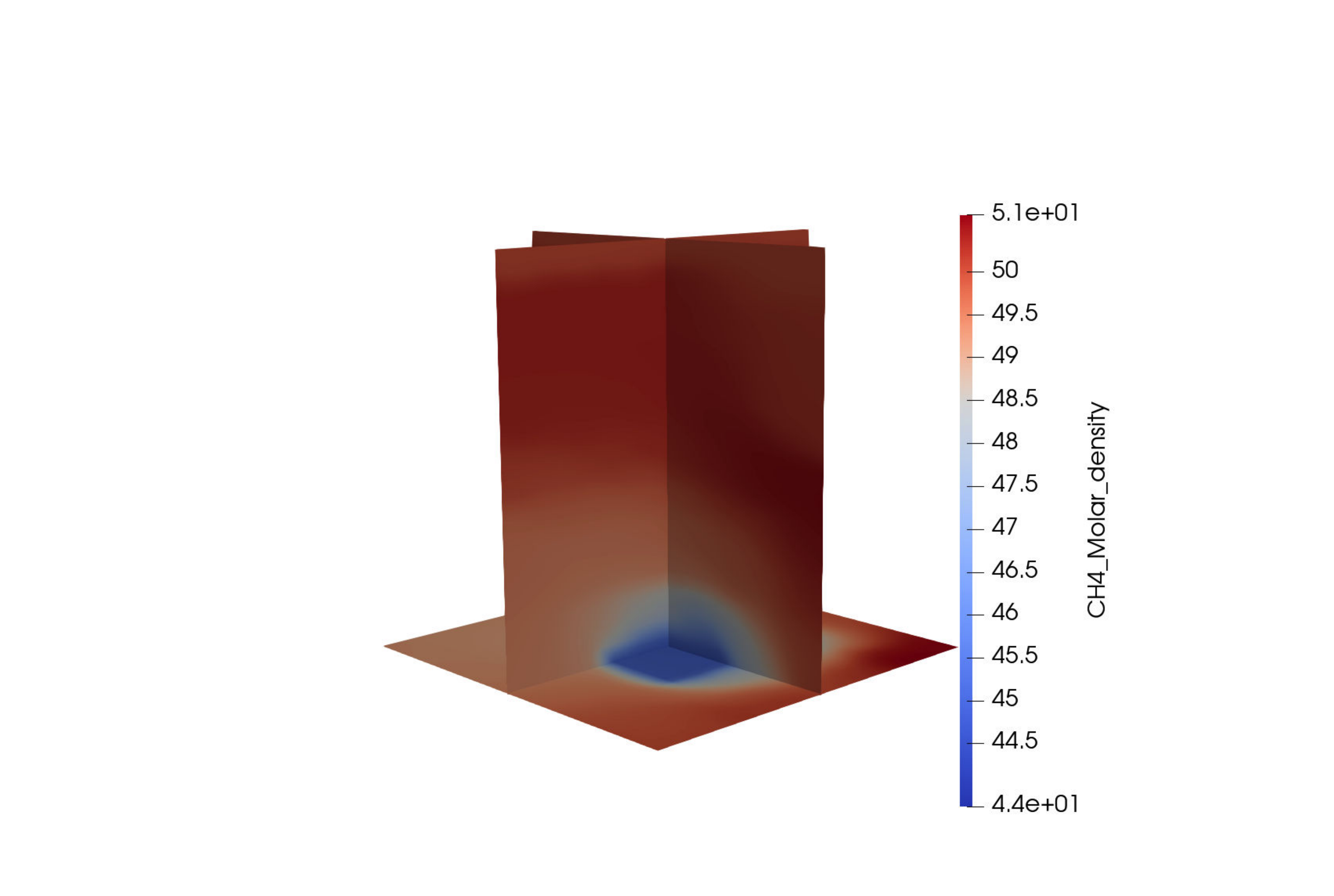}
	\caption{Distributions of molar density of CH$_4$ at different times in Example 3. Top-left: Top-left: $n = 50$. Top-right: $n = 500$. Bottom-left: $n = 1000$. Bottom-right: $n = 2000$.}\label{fig3-ch4}
\end{figure}

\begin{figure}[htbp]
	\centering
	\includegraphics[width=7cm, height=4cm, trim=6cm 0cm 6cm 0cm,clip]{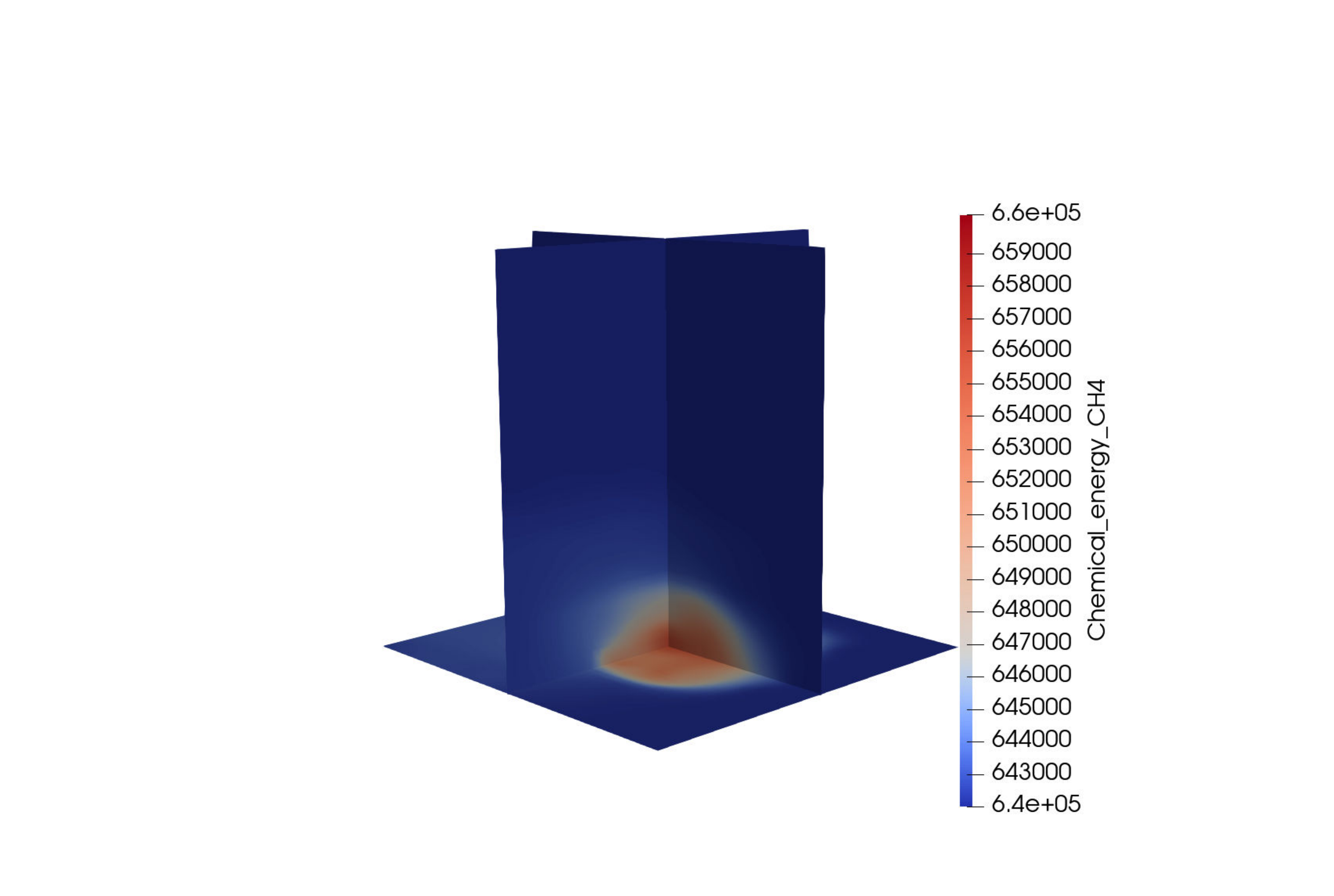}
	\includegraphics[width=7cm, height=4cm, trim=6cm 0cm 6cm 0cm,clip]{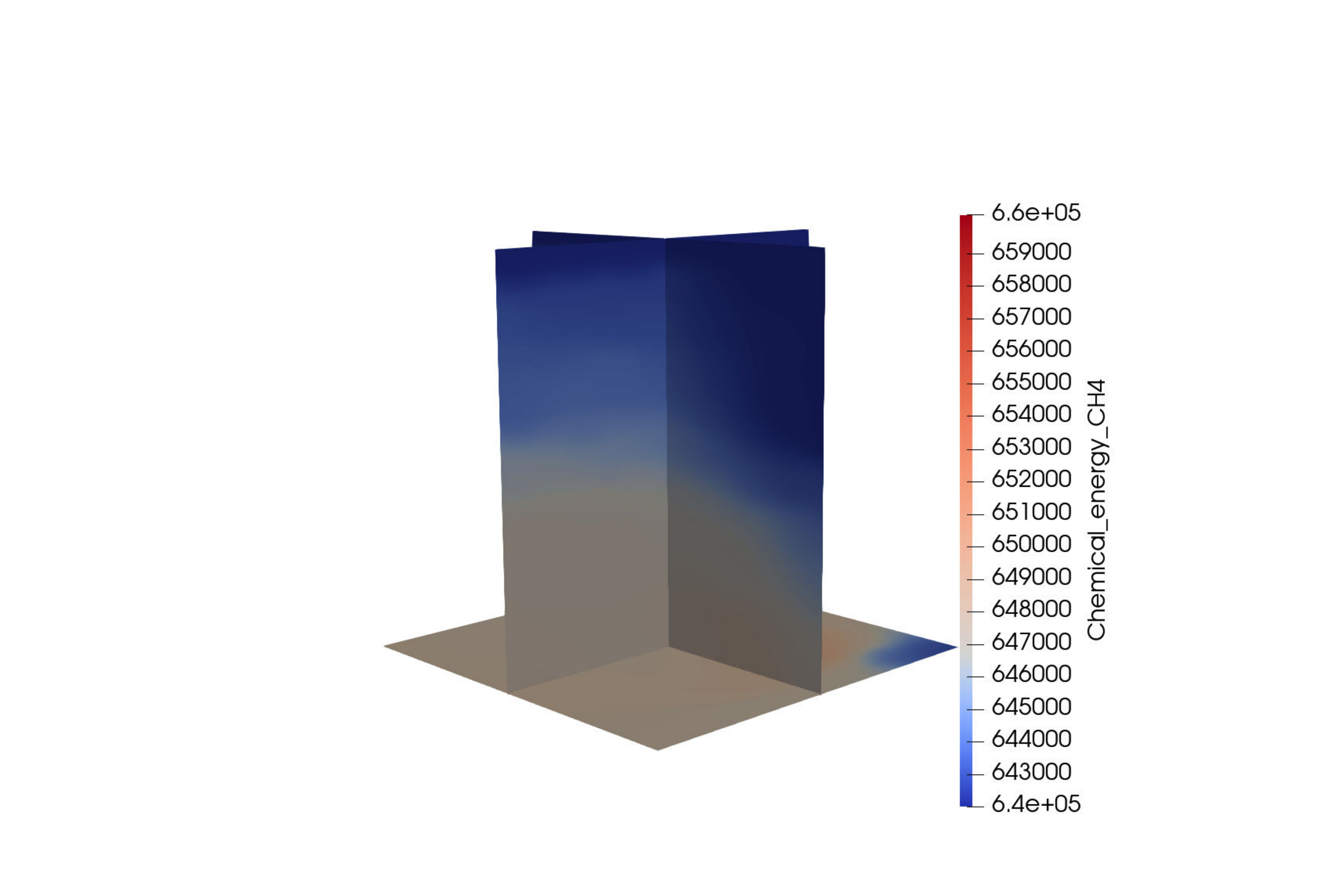}
	
	\includegraphics[width=7cm, height=4cm, trim=6cm 0cm 6cm 0cm,clip]{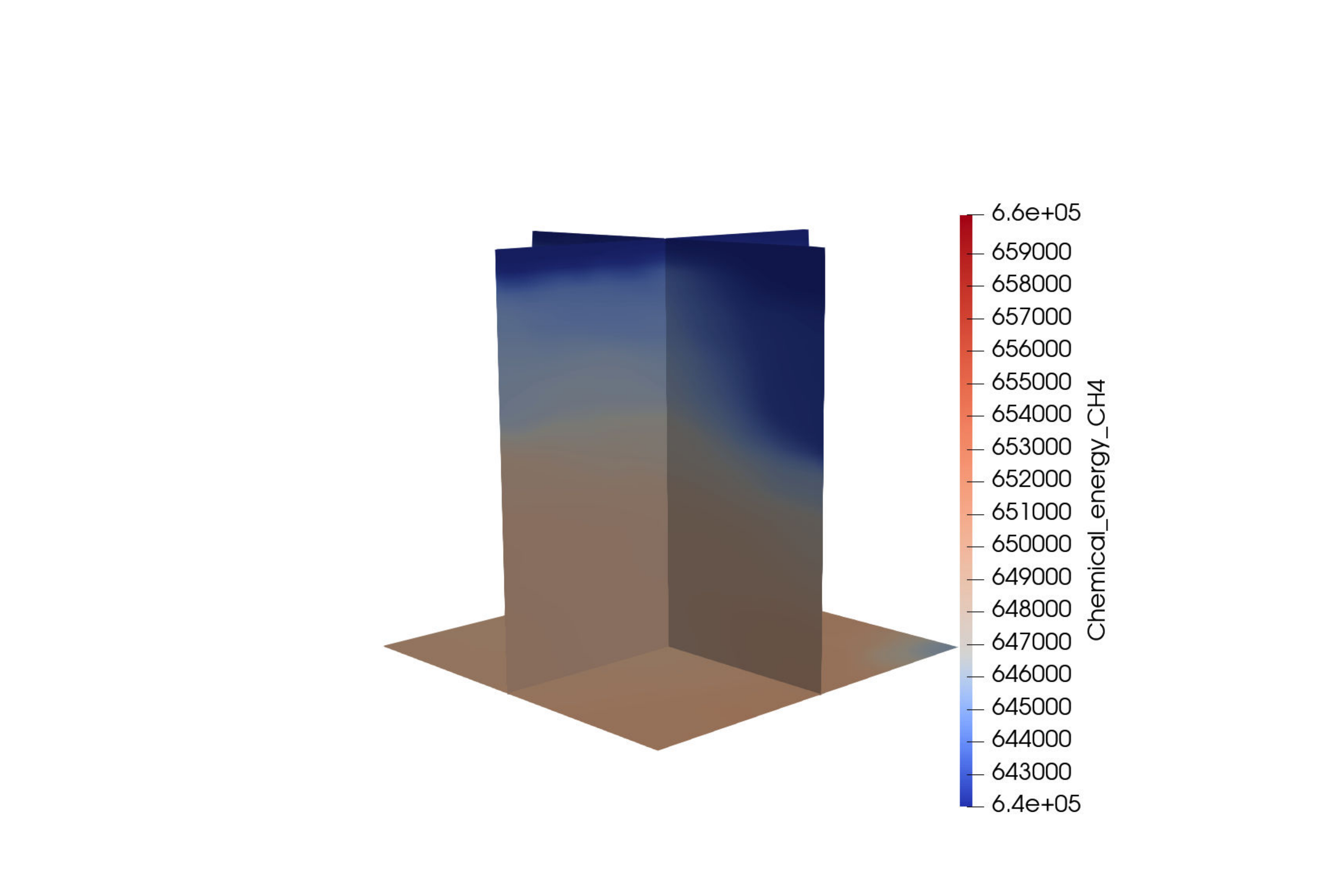}
	\includegraphics[width=7cm, height=4cm, trim=6cm 0cm 6cm 0cm,clip]{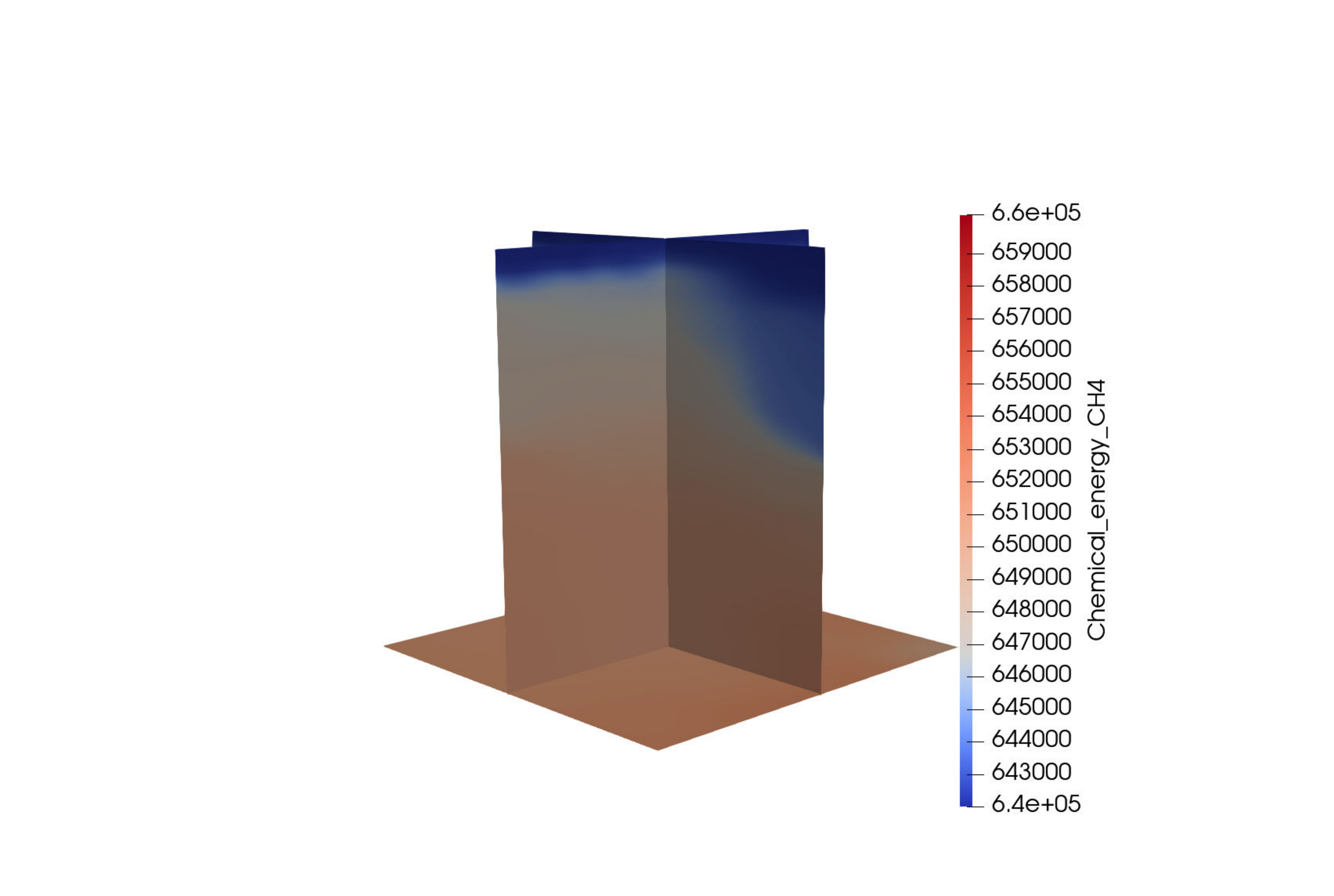}
	\caption{Distributions of chemical potential of CH$_4$ at different times in Example 3. Top-left: Top-left: $n = 50$. Top-right: $n = 500$. Bottom-left: $n = 1000$. Bottom-right: $n = 2000$.}\label{fig3-ch4-che}
\end{figure}

	%

\begin{figure}[htbp]
	\centering
	\includegraphics[width=7cm, height=4cm, trim=6cm 0cm 6cm 0cm,clip]{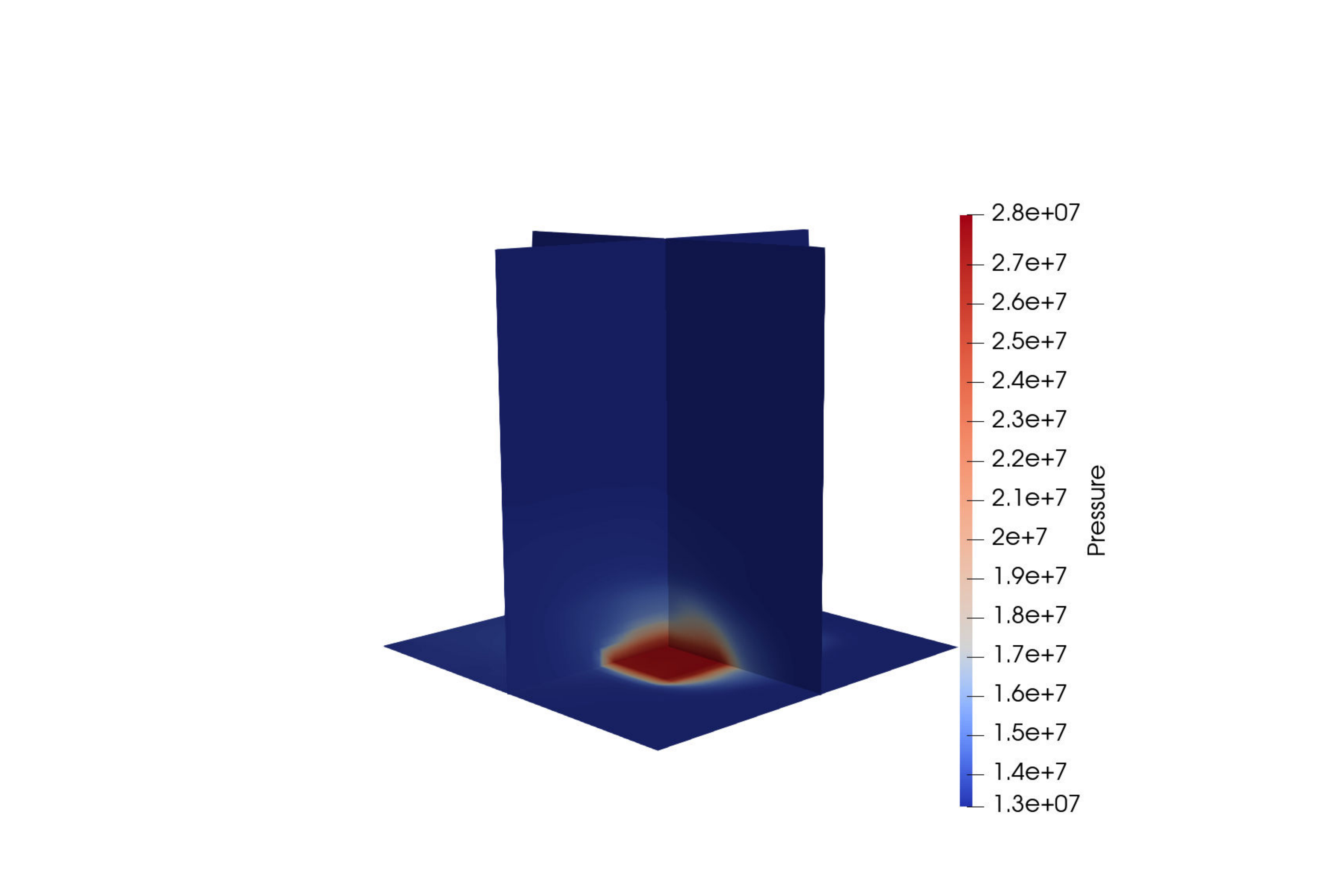}
	\includegraphics[width=7cm, height=4cm, trim=6cm 0cm 6cm 0cm,clip]{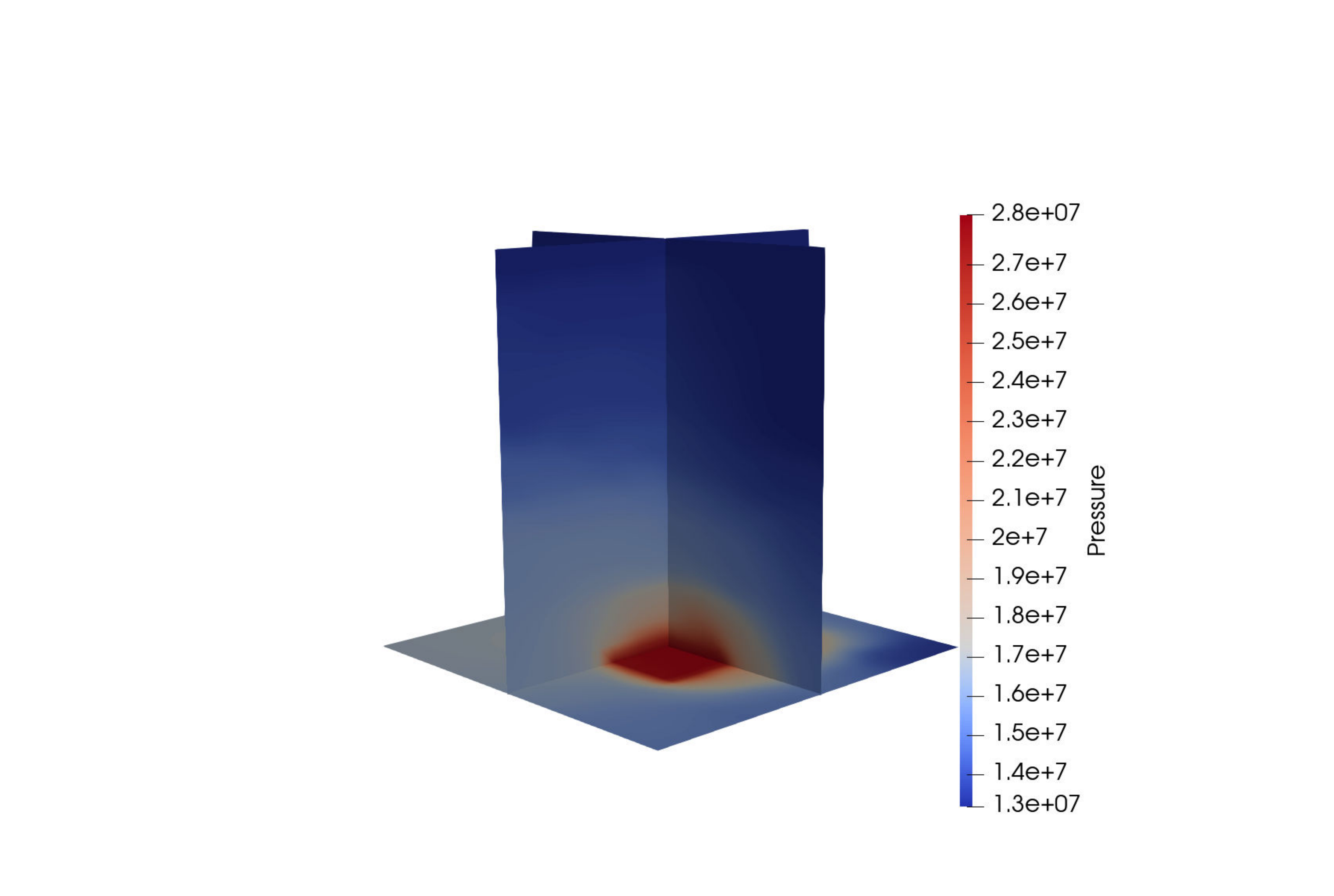}
	
	\includegraphics[width=7cm, height=4cm, trim=6cm 0cm 6cm 0cm,clip]{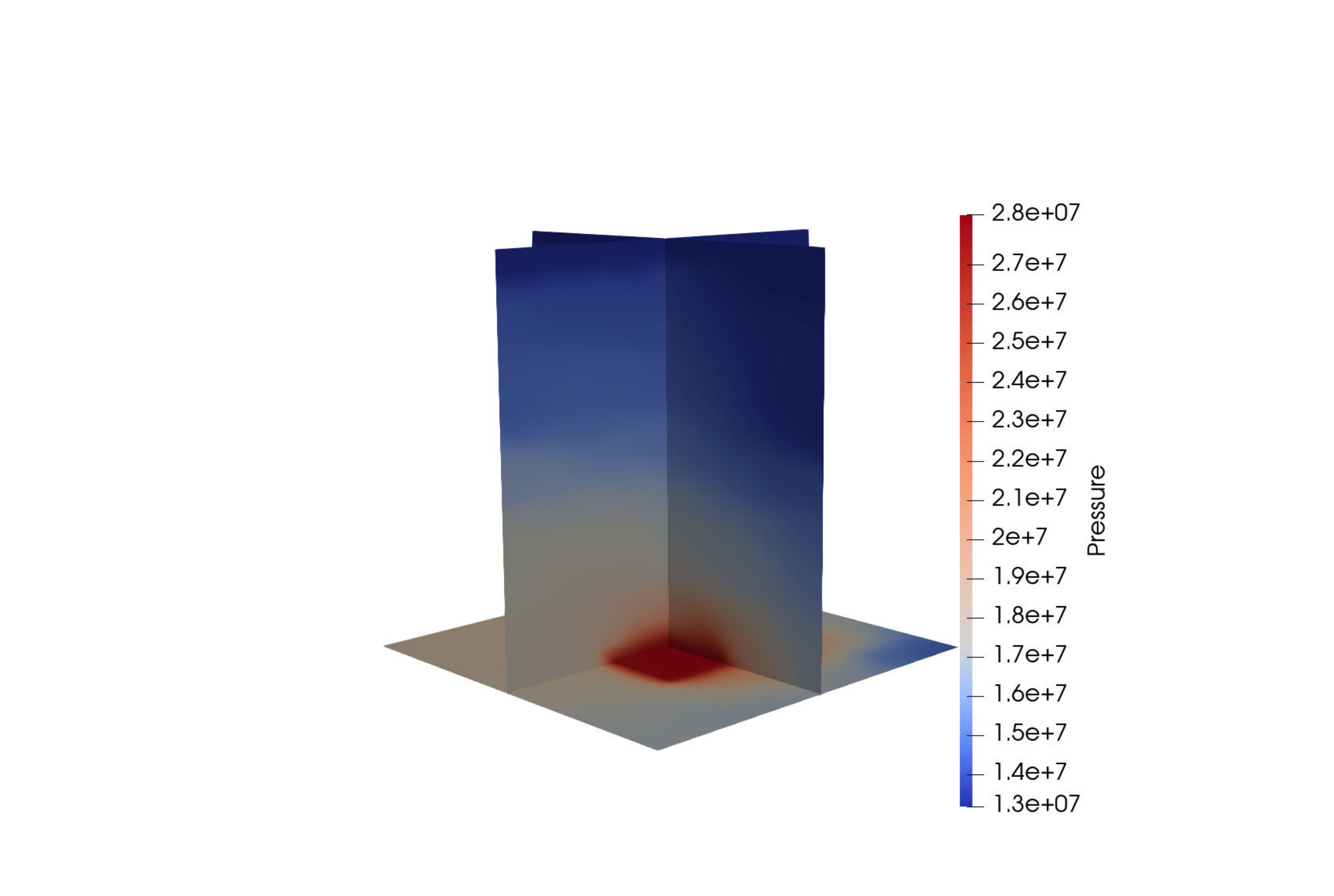}
	\includegraphics[width=7cm, height=4cm, trim=6cm 0cm 6cm 0cm,clip]{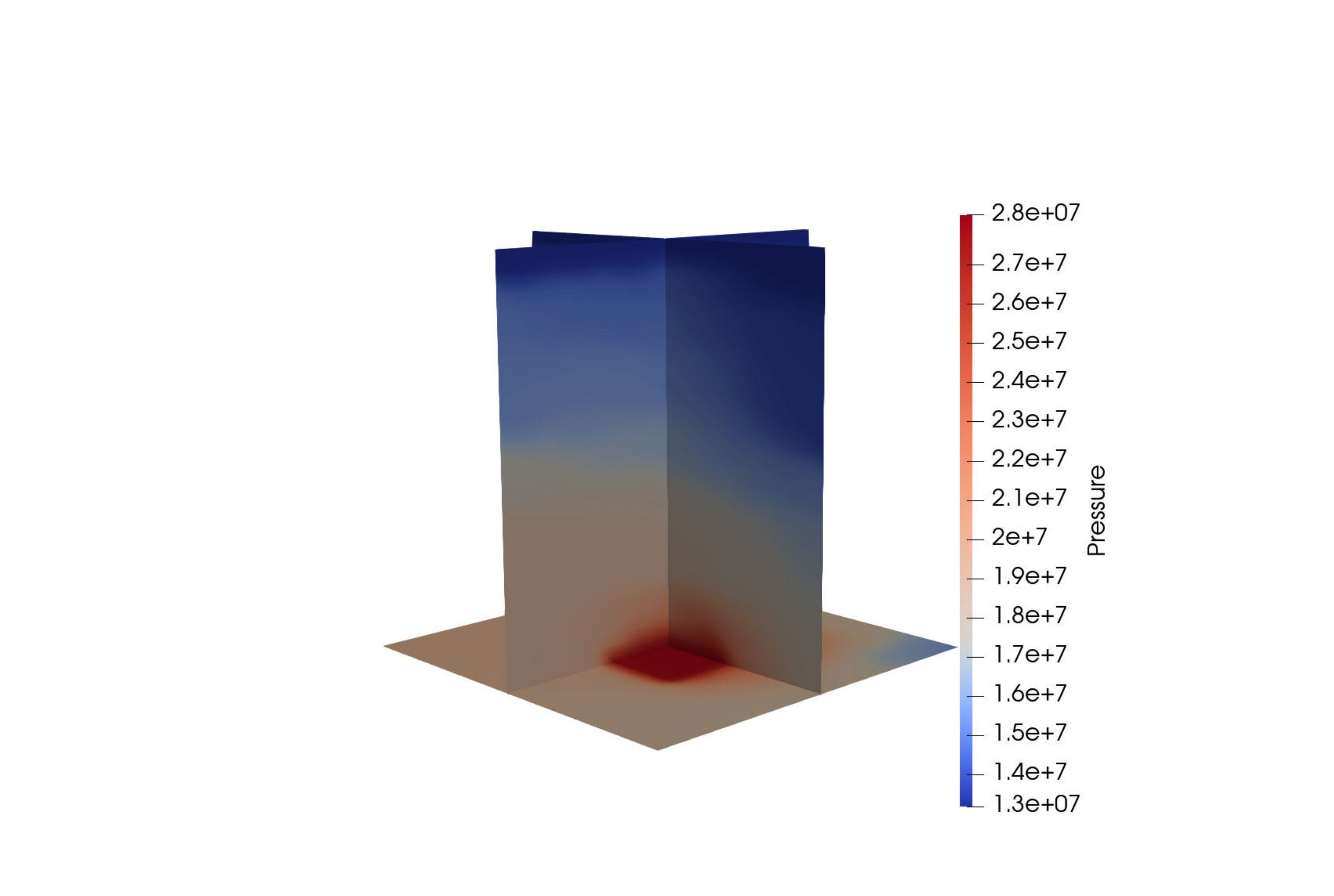}
	\caption{Distributions of pressure at different times in Example 3. Top-left: $n = 50$. Top-right: $n = 500$. Bottom-left: $n = 1000$. Bottom-right: $n = 2000$.}\label{fig3-pres}
\end{figure}

\begin{figure}[htbp]
	\centering
	\includegraphics[width=5cm, height=4cm, trim=6cm 0cm 6cm 0cm,clip]{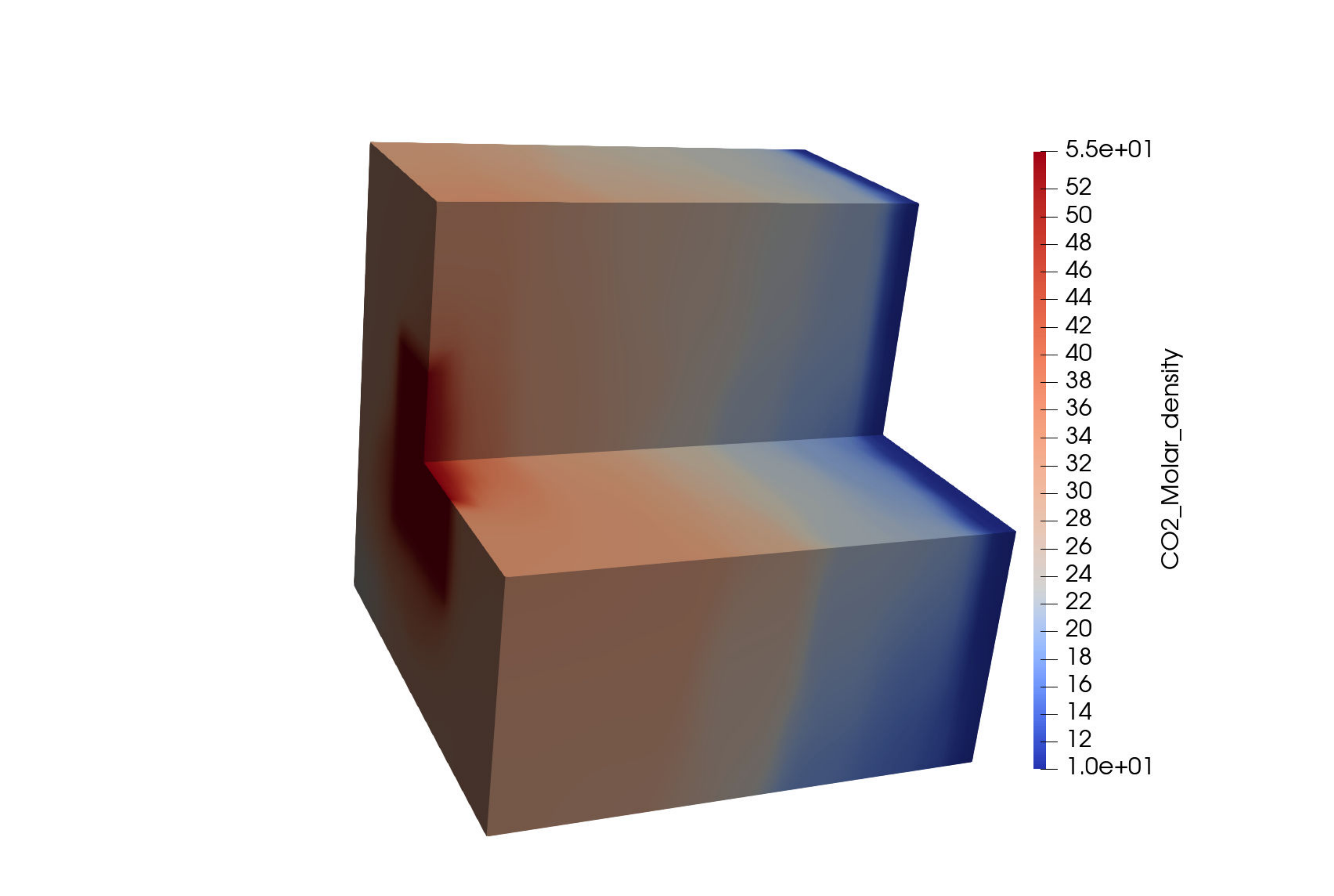}
	\includegraphics[width=5cm, height=4cm, trim=6cm 0cm 6cm 0cm,clip]{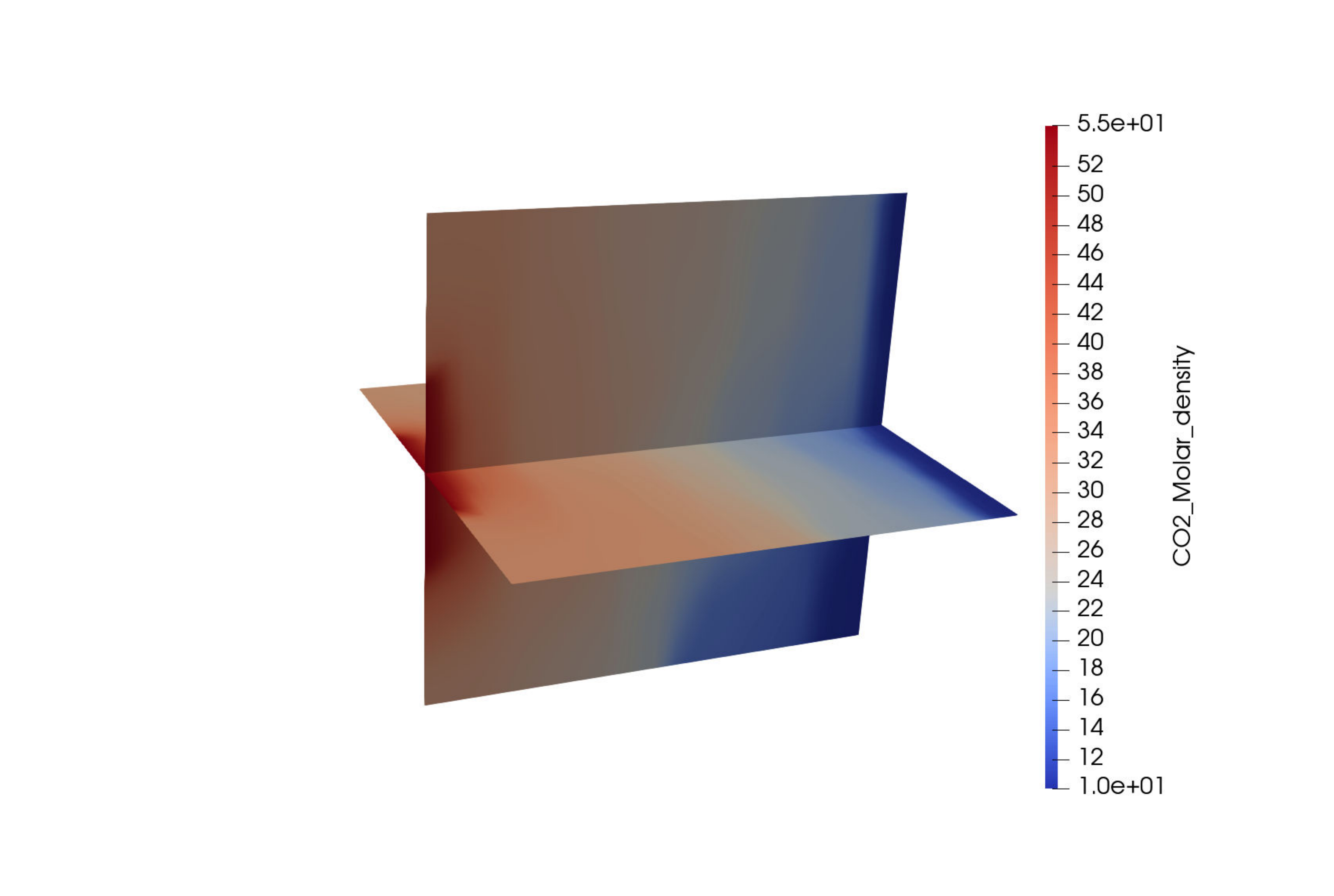}
	
	\includegraphics[width=5cm, height=4cm, trim=6cm 0cm 6cm 0cm,clip]{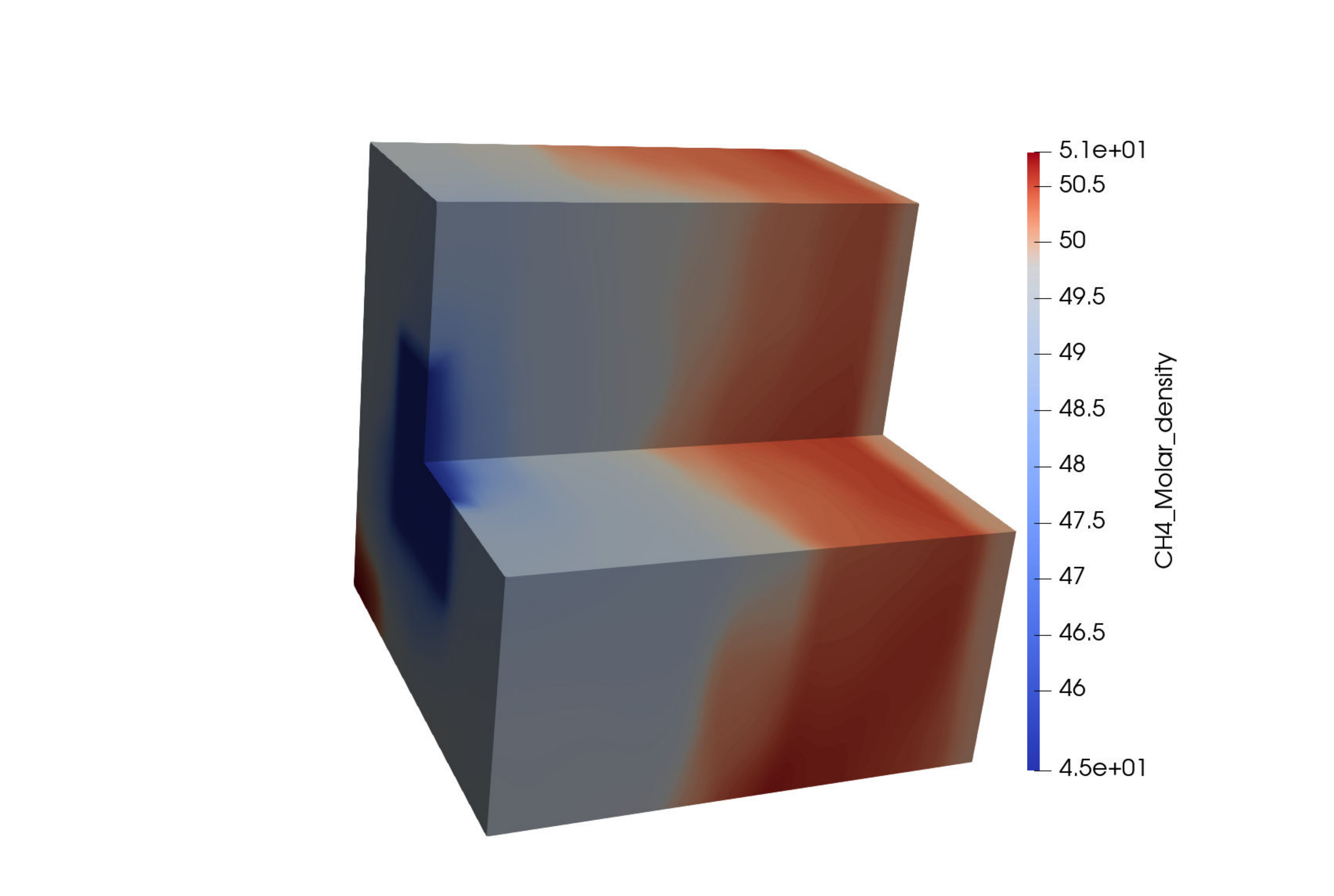}
	\includegraphics[width=5cm, height=4cm, trim=6cm 0cm 6cm 0cm,clip]{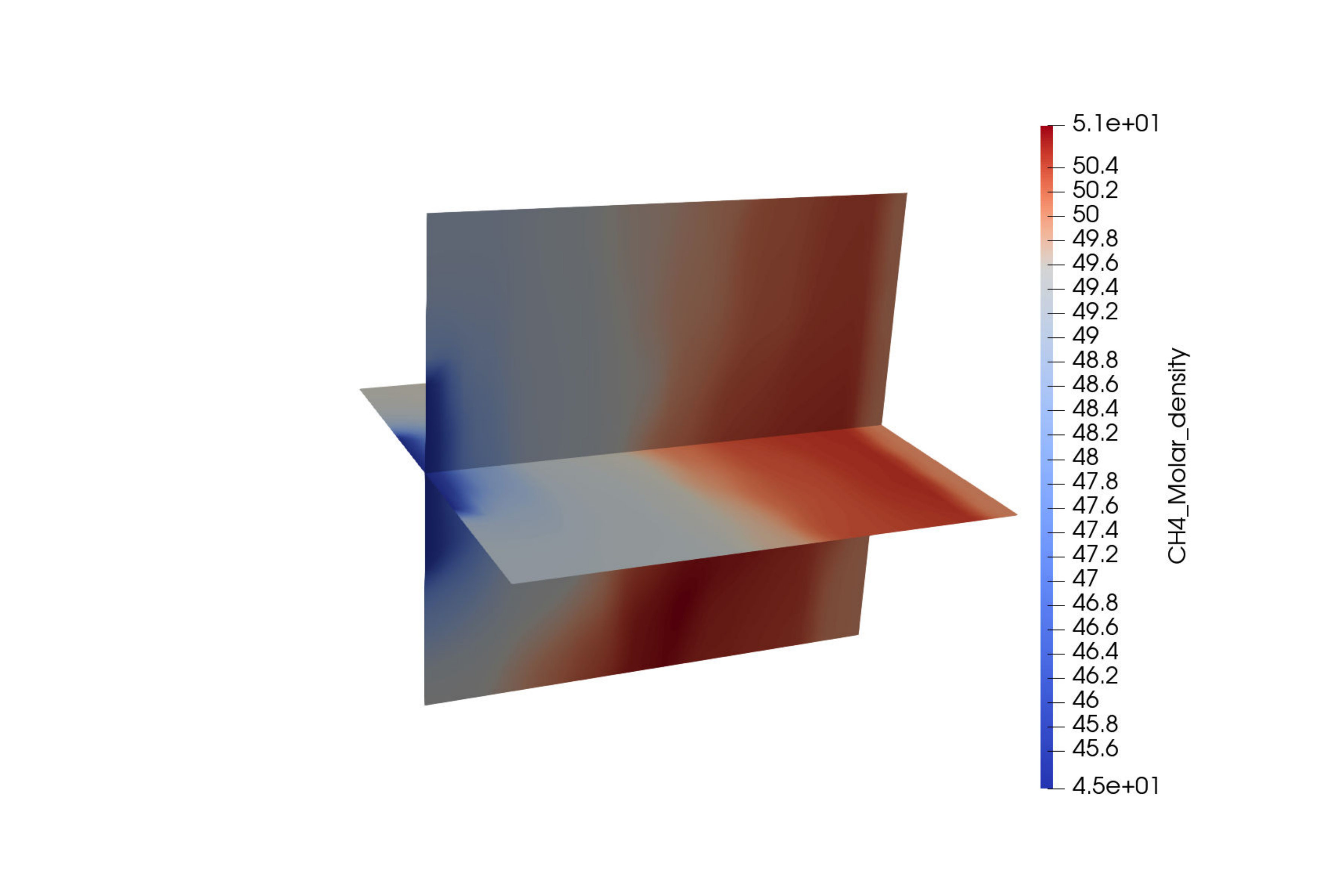}
	\caption{Distributions of molar densities of CO$_2$ and CH$_4$ at time step n = 5000. Top-left: CO$_2$  clip view. Top-right: CO$_2$  slice view. Bottom-left: CH$_4$  clip view. Bottom-right: CH$_4$  slice view.}\label{fig3-slice}
\end{figure}
\section{Conclusions}\label{sec-con}
In this work, we have developed a robust and efficient numerical framework for simulating multicomponent gas flow in poroelastic media. The proposed model systematically integrates multicomponent transport with the poroelastic response of the media. By introducing a stabilized discretization strategy and an adaptive time-stepping scheme, the numerical method ensures both numerical stability and computational efficiency. The use of a mixed finite element method with upwind stabilization for flow and transport, together with a discontinuous Galerkin formulation for the poroelastic momentum equation, further enhances accuracy and effectively mitigates numerical locking phenomena. Numerical experiments confirm the robustness and applicability of the framework, demonstrating its capability to handle complex multicomponent transport processes in poroelastic media. The proposed numerical framework provides a thermodynamically consistent discretization that ensures numerical stability and boundedness of molar densities, while the adaptive time-stepping strategy significantly improves computational efficiency. Together, these features result in a robust and reliable numerical scheme for simulating multicomponent transport in poroelastic media.
\section*{Acknowledgments}
Huangxin Chen was supported by the National Key Research and Development Project of China (Grant No. 2023YFA1011702) and the National Natural Science Foundation of China (Grant No. 12471345).
Shuyu Sun was supported by the National Key Research and Development Project of China (Grant No. 2023YFA1011701), the National Natural Science Foundation of China (Grant No. 12571466), the Fundamental Research Funds for the Central Universities, the Shanghai Magnolia Talent Fund (Innovation Talent Category) of Shanghai Municipal Human Resources and Social Security Bureau, and the Chang Jiang Scholars Program of the Ministry of Education of China. 

	\end{document}